\documentclass[a4paper]{amsart}

\usepackage[lmargin=1in,rmargin=1in,tmargin=1in,bmargin=1in]{geometry}
\RequirePackage{amsmath} 
\RequirePackage{amssymb}
\usepackage{amscd,latexsym,amsthm,amsfonts,amssymb,amsmath,amsxtra}
\usepackage[colorlinks=true,urlcolor=blue,citecolor=blue]{hyperref}
\usepackage{color}
\usepackage[all]{xy}
\usepackage[OT2,T1]{fontenc}
\usepackage{bm}
\usepackage{mathtools}
\usepackage{ mathrsfs }
\usepackage{xcolor}
\usepackage{comment}
\usepackage{marginnote}
\usepackage{enumitem}
\usepackage{thmtools, thm-restate}

\DeclareSymbolFont{cyrletters}{OT2}{wncyr}{m}{n}
\DeclareMathSymbol{\Sha}{\mathalpha}{cyrletters}{"58}

\let\Re\undefined
\let\Im\undefined

\DeclareMathOperator{\Re}{Re}
\DeclareMathOperator{\Im}{Im}

\DeclareMathOperator{\Tr}{Tr}

\DeclareMathOperator{\GL}{GL}

\newcommand{\bQ}{\mathbb{Q}}

\newcommand{\cF}{\mathcal{F}}

\newcommand{\cO}{\mathcal{O}}

\newcommand{\fn}{\mathfrak{n}}
\newcommand{\ff}{\mathfrak{f}}
\newcommand{\fl}{\mathfrak{l}}

\newcommand{\fp}{\mathfrak{p}}

\newcommand{\fq}{\mathfrak{q}}

\def\Re{\operatorname{Re}}

\newcommand{\floor}[1]{{\left\lfloor#1\right\rfloor}}

	\newcommand{\Res}{\operatorname{Res}}

	\newcommand{\K}{\operatorname{K}}

	\newcommand{\Ad}{\operatorname{Ad}}

	\newcommand{\fin}{\operatorname{fin}}
	\newcommand{\diag}{\operatorname{diag}}

	\newcommand{\Vol}{\operatorname{Vol}}

	\newcommand{\RNum}[1]{\uppercase\expandafter{\romannumeral #1\relax}}

\begin{document}
\theoremstyle{plain}
	\newtheorem{thm}{Theorem}[section]
	\newtheorem{cor}[thm]{Corollary}
	\newtheorem{thmy}{Theorem}
        \newtheorem{cory}{Corollary}
	\renewcommand{\thethmy}{\Alph{thmy}}
	\newenvironment{thmx}{\stepcounter{thm}\begin{thmy}}{\end{thmy}}
        \newenvironment{corx}{\stepcounter{cor}\begin{cory}}{\end{cory}}

	\renewcommand{\thecory}{\Alph{cory}}
	\newtheorem{hy}[thm]{Hypothesis}
	\newtheorem*{thma}{Theorem A}
	\newtheorem*{corb}{Corollary B}
	\newtheorem*{thmc}{Theorem C}
        \newtheorem*{thmd}{Theorem D}
	\newtheorem{lemma}[thm]{Lemma}  
	\newtheorem{prop}[thm]{Proposition}
	\newtheorem{conj}[thm]{Conjecture}
	\newtheorem{fact}[thm]{Fact}
	\newtheorem{claim}[thm]{Claim}
	
	\theoremstyle{definition}
	\newtheorem{defn}[thm]{Definition}
	\newtheorem{example}[thm]{Example}
	\theoremstyle{remark}
	
	\newtheorem{remark}[thm]{Remark}	
	\numberwithin{equation}{section}
	

\title[]{Relative Trace Formula and Uniform non-vanishing of Central $L$-values of Hilbert Modular Forms}%
\author{Zhining Wei, Liyang Yang and Shifan Zhao}

\address{Kassar House, 151 Thayer St, Providence, RI 02912 USA}
\email{zhining$\_$wei@brown.edu}

\address{Fine Hall, 304 Washington Rd, Princeton, 
		NJ 08544, USA}
\email{liyangy@princeton.edu}

\address{Math Building, 231 W. 18th Ave, Columbus, OH 43210 USA}
\email{zhao.3326@osu.edu}

\begin{abstract}
Let $\mathcal{F}(\mathbf{k},\mathfrak{q})$ be the set of normalized Hilbert newforms of weight $\mathbf{k}$ and prime level $\mathfrak{q}$. In this paper, utilizing regularized  relative trace formulas, we establish a positive proportion of $\#\{\pi\in\cF(\mathbf{k},\fq):L(1/2,\pi)\neq 0\}$ as $\#\mathcal{F}(\mathbf{k},\mathfrak{q})\to+\infty$. Moreover, our result matches the strength of the best known results in both the level and weight aspects. 
\end{abstract}

\date{\today}
\maketitle
\tableofcontents

\section{Introduction}
Let $\mathcal{F}$ be a finite set of automorphic representations. The Non-vanishing Problem (\textbf{NvP}) relative to $\mathcal{F}$ is to find some $c>0$ such that 
\begin{equation}\label{NvP}
\liminf_{\#\mathcal{F}\to+\infty}\frac{\#\{\pi\in \mathcal{F}:\ L(1/2,\pi)\neq 0\}}{\#\mathcal{F}}\geq c.
\end{equation}

Establishing such a positive proportion $c$ in \eqref{NvP}, or even demonstrating that $\#\{\pi \in \mathcal{F} : L(1/2, \pi) \neq 0\} \geq 1$, can have important arithmetic consequences. A primary case is $\mathcal{F} = \mathcal{F}(k,q)$, the set of  holomorphic newforms of weight $k$ and level $q$, which initiated systematic studies into the automorphic non-vanishing problem. This area has seen extensive investigation, including works such as \cite{Shimura1977}, \cite{Rohrlich1989}, \cite{FriedbergHoffstein1995}, \cite{OnoSkinner1998}, \cite{KowalskiMichel1999}, \cite{VanderKam1999}, \cite{IwaniecSarnak2000}, \cite{IwaniecLuoSarnak2000}, \cite{MichelVanderKam2002}, \cite{Trotabas2011}, \cite{Luo2015}, \cite{BM15}, \cite{BF21}, and \cite{BlomerFouvryKowalskiMichelMilicevicSawin2023}. 

For the family $\mathcal{F} = \mathcal{F}(k, q)$, the uniform limit $\#\mathcal{F} \to +\infty$, corresponding to $kq \to +\infty$, was first investigated in the seminal work of Iwaniec, Luo, and Sarnak \cite{IwaniecLuoSarnak2000}, who showed that $c = 1/8$ is admissible under the generalized Riemann hypothesis. Unconditionally, current research on \eqref{NvP} focuses on the following cases:
\begin{itemize}
    \item For fixed $k \geq 2$ and $q \to +\infty$, the best known result is $c = 1/4$ (cf. \cite{Trotabas2011}).
    \item For $k \to +\infty$ and $q = 1$, the best known result is $c = 1/5$ (cf. \cite{BF21}).
\end{itemize}
The uniform case of $kq \to +\infty$ remains unexplored unconditionally.

In this article, we consider the \textbf{NvP} for $\mathcal{F}=\cF(\mathbf{k},\fq)$, the set of normalized Hilbert newforms of level $\fq$ and holomorphic of weight $\mathbf{k}$. Our goal is to establish \eqref{NvP} \textit{uniformly} and unconditionally as $\#\mathcal{F} \to +\infty$, aiming for an explicit $c$ that matches the strength of the results in \cite{Trotabas2011} and \cite{BF21} in these hybrid cases. 

\subsection{Main Results}
Let $F$ be a totally real field of degree $d_F$. Let $\mathcal{O}_F$ and $\mathbb{A}_F$ denote the ring of integers and the ad\`{e}les ring of $F$, respectively. Let $\mathfrak{q} = \mathcal{O}_F$ or $\mathfrak{q} \subsetneq \mathcal{O}_F$ be a prime ideal with norm $N(\mathfrak{q})$. For a multi-index $\mathbf{k} = (k_v)_{v \mid \infty}\in (2\mathbb{Z})^{\otimes d_F}$, we write $\mathbf{k} \geq \mathbf{4}$ to indicate that $k_v \geq 4$ for all $v \mid \infty$, and we set $\|\mathbf{k}\| = \prod_{v \mid \infty} k_v$. Let $\mathcal{F}(\mathbf{k}, \mathfrak{q})$ denote the set of cuspidal automorphic representations $\pi$ of $\mathrm{PGL}_2$ over $F$ of level $\mathfrak{q}$ and holomorphic of weight $\mathbf{k}$, which is equivalent to the set of normalized Hilbert newforms of weight $\mathbf{k}$ and level $\mathfrak{q}$. 

Our main theorem establishes a uniform positive proportion non-vanishing result for central values $L(1/2,\pi)$, with $\pi\in \cF(\mathbf{k},\fq)$, as follows. 
\begin{thmx}\label{main theorem}
Let notation be as before. Suppose that $\mathbf{k}\geq\mathbf{4}.$ Let $0<\varepsilon<10^{-3}$, $(\log N(\mathfrak{q})\|\mathbf{k}\|)^{\varepsilon}\leq \xi\leq N(\mathfrak{q})^{1/2-\varepsilon}\|\mathbf{k}\|^{1/4-\varepsilon}$, and $(\log\xi)^{3/2+\varepsilon}<A\leq +\infty$. Then
\begin{equation}\label{uniform}
\frac{\#\{\pi\in\cF(\mathbf{k},\fq):L(1/2,\pi)> A^{-1}\}}{\#\cF(\mathbf{k},\fq)}\geq \frac{(1-\varepsilon)\cdot (\mathcal{M}_{\fq,\mathbf{k}}^{(1)})^2}{2N(\fq)\mathcal{M}_{\fq,\mathbf{k}}^{(2)}}
\end{equation}
when $\|\mathbf{k}\|N(\fq)$ is sufficiently large, where
\begin{align*}
\mathcal{M}_{\fq,\mathbf{k}}^{(1)}:=&4\delta_{\mathbf{k}}\textbf{1}_{\mathfrak{q}=\mathcal{O}_F}+(2(N(\mathfrak{q})+1)-4\delta_{\mathbf{k}})\zeta_{\mathfrak{q}}(1)\textbf{1}_{\mathfrak{q}\subsetneq\mathcal{O}_F},\\
\mathcal{M}_{\mathfrak{q},\mathbf{k}}^{(2)}:=&4(N(\mathfrak{q})+1)c_{\mathfrak{q}}\cdot\frac{\log \xi N(\mathfrak{q})^{\frac{1}{2}}\|\mathbf{k}\|}{\log\xi}- \frac{16\zeta_{\mathfrak{q}}(2)\delta_{\mathbf{k}}\cdot\textbf{1}_{\mathfrak{q}\subsetneq \mathcal{O}_F}}{1+N(\mathfrak{q})^{-1}}\cdot \frac{\log \xi\|\mathbf{k}\|}{\log\xi},
\end{align*} 
with $\delta_{\mathbf{k}}:=\textbf{1}_{\sum_{v\mid\infty}k_v\equiv 0\pmod{4}}$,  $c_{\mathfrak{q}}:=\zeta_{\mathfrak{q}}(1)^3\zeta_{\mathfrak{q}}(2)^{-1}\textbf{1}_{\mathfrak{q}\subsetneq \mathcal{O}_F}+\delta_{\mathbf{k}}\textbf{1}_{\mathfrak{q}=\mathcal{O}_F}$, and $\zeta_{\mathfrak{q}}(s):=(1-N(\mathfrak{q})^{-s})^{-1}$ being the local zeta factor.
\end{thmx}

\subsubsection{Uniform non-vanishing}
An explicit calculation of the right hand side of \eqref{uniform} yields the following result.
 
\begin{corx}\label{uniform nonvanishing corollary}
Let notation be as before. Suppose that $\mathbf{k}\geq\mathbf{4}.$ Then 
\begin{equation}\label{eq1.10}
\liminf_{N(\mathfrak{q})\|\mathbf{k}\|\to+\infty}\frac{\#\{\pi\in\cF(\mathbf{k},\mathfrak{q}):L(1/2,\pi)\geq (\log N(\mathfrak{q})\|\mathbf{k}\|)^{-2}\}}{\#\cF(\mathbf{k},\mathfrak{q})}\geq \frac{1}{100}.
\end{equation}	
\end{corx}
This can be regarded as an unconditional uniform non-vanishing result, whereas in \cite[Corollary 1.6, (1.51)]{IwaniecLuoSarnak2000}, the stronger result $c=1/8$ was proved under the assumption of the generalized Riemann hypothesis.

\subsubsection{The weight and level aspects}
By specifying the family $\cF(\mathbf{k},\fq)$ in Theorem \ref{main theorem}, we derive the following consequence. 
\begin{corx}
\label{corollary in the level aspect}
Let notation be as before. Let $\mathbf{k}\geq\mathbf{4}$. 
\begin{itemize}
\item Suppose $\mathbf{k}$ is fixed. Then 
\begin{equation}\label{eq1.2}
\liminf_{N(\fq)\to+\infty}\frac{\#\{\pi\in\cF(\mathbf{k},\fq):L(1/2,\pi)\geq (\log N(\mathfrak{q}))^{-2}\}}{\#\cF(\mathbf{k},\fq)}\geq \frac{1}{4}.
\end{equation}
\item Assume that $\sum_{v|\infty}k_v\equiv0\pmod{4}$. Then 
\begin{equation}\label{eq1.3}
\liminf_{\|\mathbf{k}\|\to+\infty}\frac{\#\{\pi\in\cF(\mathbf{k},\mathcal{O}_F):L(1/2,\pi)\geq (\log \|\mathbf{k}\|)^{-2}\}}{\#\cF(\mathbf{k},\mathcal{O}_F)}\geq \frac{1}{5}.
\end{equation} 
\item Suppose $\mathfrak{q}\subsetneq \mathcal{O}_F$ is fixed. Then  
\begin{equation}\label{eq1.4}
\liminf_{\|\mathbf{k}\|\to+\infty}\frac{\#\{\pi\in\cF(\mathbf{k},\mathfrak{q}):L(1/2,\pi)\geq (\log \|\mathbf{k}\|)^{-2}\}}{\#\cF(\mathbf{k},\mathfrak{q})}\geq \frac{(1-N(\mathfrak{q})^{-1})^3}{10 (1+N(\mathfrak{q})^{-1})}.
\end{equation} 
\item Along $\|\mathbf{k}\|\to+\infty$ and $N(\mathfrak{q})\to+\infty$, we have 
\begin{equation}\label{eq1.7}
\liminf_{N(\mathfrak{q})\to+\infty}\liminf_{\|\mathbf{k}\|\to+\infty}\frac{\#\{\pi\in\cF(\mathbf{k},\mathfrak{q}):L(1/2,\pi)\geq (\log \|\mathbf{k}\|)^{-2}\}}{\#\cF(\mathbf{k},\mathfrak{q})}\geq \frac{1}{10 }.
\end{equation}
\end{itemize}
\end{corx}

Note that \eqref{eq1.2} aligns with the result in \cite{Trotabas2011}; when $F = \mathbb{Q}$, \eqref{eq1.3} reduces to the main result in \cite{BF21}. For totally real fields $F$, \eqref{eq1.3} provides the first positive proportion non-vanishing result for Hilbert modular forms in the weight aspect. Furthermore, \eqref{eq1.4} and \eqref{eq1.7} appear to be new, even in the classical case of $F = \mathbb{Q}$.


\subsection{Strategy of the Proof}\label{sect1.2}
 In this section, we outline the proof for Theorem \ref{main theorem}. The key ingredients include the asymptotic formula for the $\lambda_{\pi}(\mathfrak{n})$-weighted second moment 
 \begin{equation}\label{equ1.5}
 \sum_{\substack{\pi\in \mathcal{F}(\mathbf{k},\mathfrak{q})}}\frac{\lambda_{\pi}(\mathfrak{n})L(1/2,\pi)^2}{L(1,\pi,\Ad)}=\mathrm{MT}_2^*(\mathfrak{n})+O(N(\mathfrak{n})^{1/2+\varepsilon}N(\mathfrak{q})^{\varepsilon}\|\mathbf{k}\|^{1/2+\varepsilon})  
 \end{equation}
as established in Theorem \ref{thm7.1} and Corollary \ref{cor7.2} in \textsection\ref{sec8.1}, along with the asymptotic formula for the $\lambda_{\pi}(\mathfrak{n})$-weighted first moment 
 \begin{equation}\label{equ1.6}
\sum_{\substack{\pi\in \mathcal{F}(\mathbf{k},\mathfrak{q})}}\frac{\lambda_{\pi}(\mathfrak{n})L(1/2,\pi)}{L(1,\pi,\Ad)}=\mathrm{MT}_1^*(\mathfrak{n})+O(N(\mathfrak{n})^{1/2+\varepsilon}N(\mathfrak{q})^{\varepsilon}\|\mathbf{k}\|^{\varepsilon})
\end{equation}
as established in Theorem \ref{thmD} and  Corollary \ref{cor9.9} in \textsection\ref{sec9.6}. Here, $\mathfrak{n}$ is an ideal coprime to $\mathfrak{q}$, and $\mathrm{MT}_i^*(\mathfrak{n})$ ($1 \leq i \leq 2$) refers to the main terms. Details on deriving these estimates are provided in \textsection\ref{sec1.2.1} and \textsection\ref{sec1.2.2}. 
\medskip

We may average \eqref{equ1.5} and \eqref{equ1.6} over $\mathfrak{n}$ with a suitable weight, obtaining 
\begin{thmx}\label{mollifie second moment generalized result}
Let notation be as before. Let $\xi>1$ and $0<\varepsilon<10^{-3}$. Let $\mathfrak{q}\subseteq \mathcal{O}_F$ be an integral ideal. Let $M_{\xi}(\pi)$ be the mollifier defined as in \eqref{M} in \textsection\ref{sec1.1.6}.  
\begin{itemize}
\item (Theorem \ref{thmC}) Let $\mathcal{M}_{\mathfrak{q},\mathbf{k}}^{(2)}$ be  defined as in Theorem \ref{main theorem}. We have 
\begin{equation}\label{A}
\sum_{\substack{\pi\in \mathcal{F}(\mathbf{k},\mathfrak{q})}}\frac{L(1/2,\pi)^2M_{\xi}(\pi)^2}{L(1,\pi,\Ad)}=\frac{\zeta_F(2)^2D_F^{\frac{3}{2}}}{(\Res_{s=1}\zeta_F(s))^2}\prod_{v\mid\infty}\frac{k_v-1}{4\pi^2}\cdot \mathcal{M}_{\mathfrak{q},\mathbf{k}}^{(2)}+\mathcal{E}_{\mathfrak{q},\mathbf{k}}^{(2)},
\end{equation}
where  
\begin{align*}
\mathcal{E}_{\mathfrak{q},\mathbf{k}}^{(2)}\ll (\log\xi)^{-1}N(\mathfrak{q})\|\mathbf{k}\|+\xi^{2+\varepsilon}N(\mathfrak{q})^{\varepsilon}\|\mathbf{k}\|^{1/2+\varepsilon},
\end{align*}
and the implied constant depends only on $\varepsilon$ and $F$.
\item (Theorem \ref{prop9.13}) Let $\mathcal{M}_{\mathfrak{q},\mathbf{k}}^{(1)}$ be  defined as in Theorem \ref{main theorem}. We have 
\begin{equation}\label{B}
\sum_{\substack{\pi\in \mathcal{F}(\mathbf{k},\mathfrak{q})}}\frac{L(1/2,\pi)M_{\xi}(\pi)}{L(1,\pi,\Ad)}=\frac{\zeta_F(2)D_F^{3/2}}{  \Res_{s=1}\zeta_F(s)}\prod_{v\mid\infty}\frac{(k_v-1)}{4\pi ^2}\cdot\mathcal{M}_{\fq,\mathbf{k}}^{(1)}+\mathcal{E}_{\fq,\mathbf{k}}^{(1)},
\end{equation}
where
\[\mathcal{E}_{\fq,\mathbf{k}}^{(1)} \ll (\log\xi)^{-1}N(\mathfrak{q})\|\mathbf{k}\|+N(\mathfrak{q})^{\varepsilon}\|\mathbf{k}\|^{\varepsilon}\xi^{1+2\varepsilon},\]
and the implied constant depends only on $F$ and $\varepsilon$.
\end{itemize}
\end{thmx} 

Utilizing Cauchy inequality we obtain 
\begin{equation}\label{equ1.9}
\bigg[\sum_{\substack{\pi}}\frac{L(1/2,\pi)M_{\xi}(\pi)}{L(1,\pi,\Ad)}\bigg]^2\leq\bigg[\sum_{\substack{\pi}}\frac{\textbf{1}_{L(1/2,\pi)\neq 0}}{L(1,\pi,\Ad)}\bigg]\cdot \bigg[\sum_{\substack{\pi}}\frac{L(1/2,\pi)^2M_{\xi}(\pi)^2}{L(1,\pi,\Ad)}\bigg],
\end{equation}
where $\pi\in \mathcal{F}(\mathbf{k},\mathfrak{q})$. Take $(\log N(\mathfrak{q})\|\mathbf{k}\|)^{\varepsilon}\leq \xi\leq N(\mathfrak{q})^{1/2-\varepsilon}\|\mathbf{k}\|^{1/4-\varepsilon}$ in Theorem \ref{mollifie second moment generalized result}. Substituting \eqref{A} and \eqref{B} into \eqref{equ1.9} yields 
\begin{equation}\label{equ1.10}
\sum_{\substack{\pi\in \mathcal{F}(\mathbf{k},\mathfrak{q})\\
L(1/2,\pi)\neq 0}}\frac{1}{L(1,\pi,\Ad)}\geq \frac{(1-\varepsilon)\cdot (\mathcal{M}_{\fq,\mathbf{k}}^{(1)})^2}{\mathcal{M}_{\fq,\mathbf{k}}^{(2)}},
\end{equation}
which, modulo the harmonic weight $L(1,\pi,\Ad)$, gives \eqref{uniform} in Theorem \ref{main theorem} with $A=+\infty$. We then take advantage of techniques in \cite{KowalskiMichel1999} to get rid of $L(1,\pi,\Ad)$ in \eqref{equ1.10}, which will be briefly discussed in \textsection\ref{sec1.2.3}.

\subsubsection{The weighted second moment via relative trace formula}\label{sec1.2.1} To address this, we adopt a different approach to derive \eqref{equ1.5} over number fields: the regularized relative trace formula ($\mathbf{RTF}$) developed in \cite{Yan23a} and \cite{Yan23c}. 

For a suitable test function $f \in C^{\infty}(\mathrm{PGL}_2(\mathbb{A}_F))$, let $\K(\cdot,\cdot)$ be the associated automorphic kernel (cf. \eqref{f1.5} in \textsection\ref{sec1.2}). Regularizing the integral
\begin{align*}
\int_{F^{\times}\backslash\mathbb{A}_F^{\times}}\int_{F^{\times}\backslash\mathbb{A}_F^{\times}}\K\left(\begin{pmatrix}
x\\
&1
\end{pmatrix},\begin{pmatrix}
y\\
&1
\end{pmatrix}\right)|x|^{s_1}|y|^{s_2}d^{\times}yd^{\times}x,
\end{align*}
in conjunction with the spectral-geometric expansions of $\K(\cdot,\cdot)$, we obtain 
\begin{equation}\label{fc1.12}
\sum_{\substack{\pi\in \mathcal{F}(\mathbf{k},\mathfrak{q})}}\frac{\lambda_{\pi}(\mathfrak{n})L(1/2,\pi)^2}{L(1,\pi,\Ad)}+\mathrm{oldforms}=\mathrm{MT}_2(\mathfrak{n})+\mathrm{Error}_2(\mathfrak{n}),
\end{equation}
where $``\mathrm{oldforms}"$ refers to the contribution from oldforms, $\mathrm{MT}_2(\mathfrak{n})$ is the main term, and $``\mathrm{Error}_2(\mathfrak{n})"$ arises from the regular orbital integrals. Roughly, 
\begin{align*}
\mathrm{Error}_2(\mathfrak{n})\asymp \sum_{\substack{t\in F-\{0,1\}\\
t(1-t)^{-1}\in \mathfrak{q}\mathfrak{n}^{-1}}}\int_{\mathbb{A}_F^{\times}}\int_{\mathbb{A}_F^{\times}}f\left(\begin{pmatrix}
x^{-1}& \\
&1
\end{pmatrix}\begin{pmatrix}
1& t\\
1& 1
\end{pmatrix}\begin{pmatrix}
xy& \\
&1
\end{pmatrix}\right)d^{\times}yd^{\times}x.
\end{align*}

We compute this integral explicitly: the Archimedean part of the inner integrals involves products of   Legendre functions, while the non-Archimedean part contributes a shift convolution of divisor functions, cf. Proposition \ref{prop6.6} in \textsection\ref{sec1.3.5}. A sharp bound on the Legendre functions, combined with certain combinatorial arguments, will yield the desired bound in \eqref{equ1.5}, as  established in \textsection\ref{sec7.3}.

For $F = \mathbb{Q}$, the error terms in \cite{BF21} involve various Gauss hypergeometric functions and their derivatives. In contrast, our formula is simpler and applies uniformly over number fields. The $``\mathrm{oldforms}"$ term will be discussed in \textsection\ref{sec1.2.3.}. 

\subsubsection{The weighted first moment via relative trace formula}\label{sec1.2.2}
Let $\psi$ be an additive character of $F\backslash\mathbb{A}_F$. We employ the relative trace formula associated with the integral
\begin{align*}
\int_{F^{\times}\backslash\mathbb{A}_F^{\times}}\int_{F\backslash\mathbb{A}_F}\K\left(\begin{pmatrix}
a\\
& 1
\end{pmatrix},\begin{pmatrix}
1& b\\
& 1
\end{pmatrix}\begin{pmatrix}
y\\
&1
\end{pmatrix}\right)\psi(b)|a|^{s}dbd^{\times}a
\end{align*}
to handle the weighted first moment, yielding  
\begin{equation}\label{fc1.13}
\sum_{\substack{\pi\in \mathcal{F}(\mathbf{k},\mathfrak{q})}}\frac{\lambda_{\pi}(\mathfrak{n})L(1/2,\pi)}{L(1,\pi,\Ad)}+\mathrm{oldforms}=\mathrm{MT}_1(\mathfrak{n})+\mathrm{Error}_1(\mathfrak{n}),
\end{equation}
where $``\mathrm{oldforms}"$ refers to the contribution from oldforms, $\mathrm{MT}_1(\mathfrak{n})$ is the main term, and $``\mathrm{Error}_1(\mathfrak{n})"$ arises from the regular orbital integrals. Roughly, 
\begin{equation}\label{equa1.12}
\mathrm{Error}_1(\mathfrak{n})\asymp \sum_{\substack{t\in \mathfrak{q}\mathfrak{n}^{-1}\\
t\neq 0}}\int_{\mathbb{A}_F^{\times}}\int_{\mathbb{A}_F}\left(\begin{pmatrix}
a^{-1}\\
& 1
\end{pmatrix}\begin{pmatrix}
-t& -1\\
1&
\end{pmatrix}\begin{pmatrix}
y& b\\
& 1
\end{pmatrix}\right)\psi(b)dbd^{\times}a.
\end{equation} 
 
This type of relative trace formula was first studied by Knightly and Li \cite{KL10} for $F = \mathbb{Q}$. As in loc. cit., we can explicitly compute the inner integrals in \eqref{equa1.12}: the Archimedean part involves products of Kummer confluent hypergeometric functions, while the non-Archimedean part corresponds to simple arithmetic functions, cf. Theorem \ref{thmD} in \textsection\ref{sec9.6}. However, the trivial bounding of Kummer confluent hypergeometric functions in \cite{KL10} and \cite{JK15} does not yield the estimate in \eqref{equ1.6} for all $\mathfrak{q}$ and $\mathbf{k}$, even when $F = \mathbb{Q}$.

In \textsection\ref{sec9.6}, we utilize the Hankel transform and stationary phase arguments to handle the Kummer confluent hypergeometric function in \eqref{equa1.12}, ultimately leading to the desired estimate \eqref{equ1.6}.

\subsubsection{Contributions from oldforms}\label{sec1.2.3.}
Recall that our goal is to establish a uniform non-vanishing result as $N(\mathfrak{q})\|\mathbf{k}\| \to \infty$. The contribution from oldforms may not be negligible when $\mathfrak{q}$ is a prime ideal whose norm is small compared to $\|\mathbf{k}\|$.

The classical application of the Petersson formula for Hilbert modular forms (cf. \cite{Luo2003}, \cite{Trotabas2011}) seems to be inconvenient in this scenario due to the root number. However, using the $\mathbf{RTF}$, the $``\mathrm{oldforms}"$ terms in \eqref{fc1.12} and \eqref{fc1.13} can be handled inductively and symmetrically (see \textsection\ref{sec8.2} and \textsection\ref{sec9.7}). This is another advantage of the relative trace formula.

\subsubsection{Removing the harmonic weight}\label{sec1.2.3}
To prove Theorem \ref{main theorem}, we follow the strategy from \cite{KowalskiMichel1999} to eliminate the harmonic weight $L(1,\pi,\Ad)$ from \eqref{equ1.10}. A key technical requirement is \cite[(25) in \textsection 3.3]{KowalskiMichel1999}, which, as noted in \cite{BF21}, follows from the hybrid bound 
\begin{equation}\label{equ1.15}
L(1/2,\pi)\ll N(\fq)^{\frac{1}{4}+\epsilon}\|\mathbf{k}\|^{\frac{3}{8}+\epsilon}
\end{equation}
for all $\pi\in\cF(\mathbf{k},\fq)$. It is important to note that \eqref{equ1.15} has been  established in either the weight aspect or the level aspect. However, in our setting, where both $\mathfrak{q}$ and $\mathbf{k}$ vary simultaneously, the bound \eqref{equ1.15} remains unproven.

Despite this, we observe that the bound
\begin{equation}\label{equ1.16}
M_{\xi}(\pi)L(1/2,\pi)\ll (N(\mathfrak{q})\|\mathbf{k}\|)^{\frac{1}{2}-\delta}
\end{equation}
for some $\delta>0$ also implies the condition (25) in   \cite[\textsection 3.3]{KowalskiMichel1999}. 

By amplifying according to the estimate \eqref{M} in Theorem \ref{mollifie second moment generalized result}, we establish \eqref{equ1.16} for all $0<\delta<1/4$, cf. Theorem \ref{thm11.1} in \textsection\ref{sec11.2}. Consequently, this allows us to apply the arguments in loc. cit.  to successfully remove the harmonic weights.


\bigskip

\subsection{Outline of the Paper} 
The structure of this paper is as follows. 
\subsubsection{The second moment via the regularized $\mathbf{RTF}$}
In \textsection\ref{sec2} we set up the test function $f=f_{\mathfrak{n},\mathfrak{q}}$ and recall the relevant regularized $\mathbf{RTF}$ from \cite{Yan23a} and \cite{Yan23c}, as summarized in Theorem \ref{thm2.3} of \textsection\ref{sec2.2}. This formula will be instrumental in deriving the asymptotic formula \eqref{equ1.5}.

\begin{itemize}
\item In \textsection\ref{sec3} we compute the spectral side explicitly, proving Theorems \ref{thm3.5} and \ref{spec} in \textsection\ref{sec2.3}, which gives the left hand side of \eqref{fc1.12} in \textsection\ref{sec1.2.1}. 
\item From \textsection\ref{section4} to \textsection\ref{sec7}, we handle the geometric side. In \textsection\ref{section4} and \textsection\ref{section5}, we compute the small cell orbital integral and the dual orbital integral, respectively. These constitute the main terms of the geometric side, as summarized in Proposition \ref{prop5.3} of \textsection\ref{sec5}. In \textsection\ref{sec7}, we analyze the regular orbital integrals, establishing the necessary upper bounds in Proposition \ref{prop6.12}.

\item In \textsection\ref{sec8}, we synthesize all the results from the previous sections to compute the mollified second moment \eqref{A} in Theorem \ref{mollifie second moment generalized result}. 
\end{itemize}
 
\subsubsection{The first moment via the $\mathbf{RTF}$}
In \textsection\ref{sec9}, we investigate the relative trace formula in  \cite{KL10} over totally real fields. After computing the spectral side in \textsection\ref{sec9.1} and \textsection\ref{sec9.7}, and the geometric side in \textsection\ref{sec9.2}--\textsection\ref{sec9.5}, we establish Theorem \ref{prop9.13} in \textsection\ref{sec9.8}, which leads to the mollified first moment \eqref{B} in Theorem \ref{mollifie second moment generalized result}.

\subsubsection{Uniform non-vanishing in harmonic average}
In \textsection\ref{sec10},  we establish uniform non-vanishing with the harmonic weight $L(1,\pi,\Ad)$; see  Theorem \ref{thm10.2}. A special case is the inequality \eqref{equ1.10} in \textsection\ref{sect1.2}. 

\subsubsection{Uniform non-vanishing in natural average}
In \textsection\ref{sec11}, we establish the main result, Theorem \ref{main theorem}, by removing the harmonic weight from Theorem \ref{mollifie second moment generalized result}. This approach builds on the method from \cite{KowalskiMichel1999}, combined with our novel technique for bounding individual mollified central values (cf. Theorem \ref{11.1} in \textsection \ref{sec11.2}).

\subsubsection{Appendix: $\mathbf{RTF}$'s of different types}
In Appendix \ref{classical vs RTF}, we summarize the strategy and challenges of applying the Petersson formula to establish nonvanishing results in both the uniform aspect and over totally real fields (as opposed to the case over $\mathbb{Q}$). We also compare this approach with the use of $\mathbf{RTF}$'s in this paper.

\subsection{Notation Guide}
\subsubsection{Number Fields and Measures}\label{1.1.1}
Let $F$ be a totally real field with ring of integers $\mathcal{O}_F$ and discriminant $D_F$. Let $N_F$ be the absolute norm. Let $\mathfrak{O}_F$ be the different of $F.$ Let $\mathbb{A}_F$ be the adele group of $F.$ Let $\Sigma_F$ be the set of places of $F.$ Denote by $\Sigma_{F,\fin}$ (resp. $\Sigma_{F,\infty}$) the set of non-Archimedean (resp. Archimedean) places. For $v\in \Sigma_F,$ we denote by $F_v$ the corresponding local field. For a non-Archimedean place $v,$ let $\mathcal{O}_v$ be the ring of integers of $F_v$, and $\mathfrak{p}_v$ be the maximal prime ideal in $\mathcal{O}_v$. Given an integral ideal $\mathcal{I},$ we say $v\mid \mathcal{I}$ if $\mathcal{I}\subseteq \mathfrak{p}_v.$ Fix a uniformizer $\varpi_{v}\in\mathfrak{p}_v.$ Denote by $e_v(\cdot)$ the evaluation relative to $\varpi_v$ normalized as $e_v(\varpi_v)=1.$ Let $q_v$ be the cardinality of $\mathbb{F}_v:=\mathcal{O}_v/\mathfrak{p}_v.$ We use $v\mid\infty$ to indicate an Archimedean place $v$ and write $v<\infty$ if $v$ is non-Archimedean. Let $|\cdot|_v$ be the norm in $F_v.$ Put $|\cdot|_{\infty}=\prod_{v\mid\infty}|\cdot|_v$ and $|\cdot|_{\fin}=\prod_{v<\infty}|\cdot|_v.$ Let $|\cdot|_{\mathbb{A}_F}=|\cdot|_{\infty}\otimes|\cdot|_{\fin}$. We will simply write $|\cdot|$ for $|\cdot|_{\mathbb{A}_F}$ in calculation over $\mathbb{A}_F^{\times}$ or its quotient by $F^{\times}$.

\subsubsection{Ideals}
Let $\mathfrak{q}=\mathcal{O}_F$ or $\mathfrak{q}\subsetneq \mathcal{O}_F$ be a prime ideal, and $\mathfrak{n}\subset \mathcal{O}_F$ be an integral ideal with ${\mathfrak{n}}+\mathfrak{q}= \mathcal{O}_F$. Let $\mathfrak{n}=\prod_{v<\infty}\mathfrak{p}_v^{r_v}$ be the primary decomposition. Denote by $e_v(\mathfrak{n})=r_v$ be the valuation of $\mathfrak{n}$ at the nonarchimedean place $v$. Let 
\begin{equation}\label{1.1}
V_{\mathfrak{q}}:=
\begin{cases}
N_F(\mathfrak{q})+1,\ \ &\text{if $\mathfrak{q}\subsetneq \mathcal{O}_F$}\\
1,\ \ &\text{if $\mathfrak{q}=\mathcal{O}_F$}.
\end{cases}
\end{equation}

For convenience, we denote $N(\mathfrak{q}) = N_F(\mathfrak{q})$ throughout this paper.    


\subsubsection{Additive Characters }
Let $\psi_{\mathbb{Q}}$ be the additive character on $\mathbb{Q}\backslash \mathbb{A}_{\mathbb{Q}}$ such that $\psi_{\mathbb{Q}}(t_{\infty})=\exp(2\pi it_{\infty}),$ for $t_{\infty}\in \mathbb{R}\hookrightarrow\mathbb{A}_{\mathbb{Q}}.$ Let $\psi=\psi_{\mathbb{Q}}\circ \Tr_F,$ where $\Tr_F$ is the trace map. Then $\psi(t)=\prod_{v\in\Sigma_F}\psi_v(t_v)$ for $t=(t_v)_v\in\mathbb{A}_F.$ 

At $v<\infty$, each $\psi_v$ has conductor $\mathfrak{p}_v^{-d_v}$, where $d_v$ is the nonnegative integer such that $\mathfrak{p}_v^{d_v}$ is the local different. 

\subsubsection{Haar Measure}\label{sec2.1.2}
For $v\in \Sigma_F,$ let $dt_v$ be the additive Haar measure on $F_v,$ self-dual relative to $\psi_v.$ Then $dt=\prod_{v\in\Sigma_F}dt_v$ is the standard Tamagawa measure on $\mathbb{A}_F$. Let $d^{\times}t_v=\zeta_{F_v}(1)dt_v/|t_v|_v,$ where $\zeta_{F_v}(\cdot)$ is the local Dedekind zeta factor. In particular, $\Vol(\mathcal{O}_v^{\times},d^{\times}t_v)=\Vol(\mathcal{O}_v,dt_v)=N_{F_v}(\mathfrak{D}_{F})^{-1/2}=q_v^{-d_v/2}$ for all finite place $v.$ Moreover, $\Vol(F\backslash\mathbb{A}_F; dt_v)=1$ and $\Vol(F\backslash\mathbb{A}_F^{(1)},d^{\times}t)=\underset{s=1}{\Res}\ \zeta_F(s),$ where $\mathbb{A}_F^{(1)}$ is the subgroup of ideles $\mathbb{A}_F^{\times}$ with norm $1,$ and $\zeta_F(s)=\prod_{v<\infty}\zeta_{F_v}(s)$ is the finite Dedekind zeta function. Denote by $\widehat{F^{\times}\backslash\mathbb{A}_F^{(1)}}$  the Pontryagin dual of $F^{\times}\backslash\mathbb{A}_F^{(1)}.$


\subsubsection{Algebraic Groups}
Let $G=\mathrm{GL}_2$ and $\overline{G}=\mathrm{PGL}_2$. Denote by $Z$ the center of $G$. Let $B$ be the Borel subgroup of $G$ and $B_0$ be the mirabolic subgroup of $G$. Let $T$ (resp. $N$) be the Levi (resp. unipotent radical) of $B$. Let $A=\diag(\GL(1),1).$

\subsubsection{Automorphic Representations}
Let $\mathbf{k}=(k_v)_{v\mid\infty}\in \mathbb{Z}_{>2}^{d_F}$, where $k_v$ is even, $v\mid\infty$. Let $\Pi_{\mathbf{k}}(\mathfrak{q})$ is the set of unitary cuspidal automorphic representations $\pi=\otimes_{v\leq\infty}\pi_v$ of $\mathrm{PGL}_2(\mathbb{A}_F)$ such that
\begin{itemize}
\item for $v\mid\infty$, $\pi_v$ is the discrete series of weight $k_v$;
\item the arithmetic conductor of $\pi_{\mathrm{fin}}:=\otimes_{v<\infty}\pi_v$ divides $\mathfrak{q}$.   
\end{itemize} 

Let $\mathcal{F}(\mathbf{k},\mathfrak{q})\subseteq \Pi_{\mathbf{k}}(\mathfrak{q})$ be the subset of unitary cuspidal automorphic representations $\pi=\otimes_{v\leq\infty}\pi_v$ such that the arithmetic conductor of $\pi_{\mathrm{fin}}:=\otimes_{v<\infty}\pi_v$ is \textit{equal} to $\mathfrak{q}$.

\subsubsection{$L$-function} 
Let $\pi\in \Pi_{\mathbf{k}}(\mathfrak{q})$. For $\Re(s)\gg 1$, let
\begin{align*}
L(s,\pi):=\sum_{\mathfrak{m}\subseteq\mathcal{O}_F} \frac{\lambda_{\pi}(\mathfrak{m})}{N(\mathfrak{m})^s}
\end{align*}
be the Dirichlet series expression for the standard $L$-function of $\pi$. In particular, the Dirichlet coefficients $\lambda_{\pi}(\mathfrak{m})$ are consistent with the Hecke eigenvalue.

\subsubsection{Mollifier}\label{sec1.1.6}
Let $\rho$ be a multiplicative arithmetic function. Suppose $\rho(\mathfrak{p})\ll 1$ for all prime ideals $\mathfrak{p}$, with the implied constant being absolute. Let $\xi>1$ be a parameter to be determined. Define 
\begin{equation}\label{M}
M_{\xi,\rho}(\pi)=\frac{1}{\log \xi} \sum_{\substack{\mathfrak{n} \subseteq \mathcal{O}_F \\ (\mathfrak{n},\mathfrak{q}) = 1}} \frac{\lambda_\pi(\mathfrak{n})\mu_F(\mathfrak{n})\rho(\mathfrak{n})}{\sqrt{N(\mathfrak{n})}}\cdot  \frac{1}{2\pi i}\int_{(2)} \frac{\xi^s}{N(\mathfrak{n})^s} \frac{ds}{s^3}.
\end{equation}
where $\mu_F$ is the M\"{o}bius function. 

One special choice of $\rho$ is given by
\begin{equation}\label{rho}
\rho(\mathfrak{n}):=\textbf{1}_{\mathfrak{n}=\mathcal{O}_F}+\textbf{1}_{\mathfrak{n}\subsetneq \mathcal{O}_F}\prod_{\substack{\mathfrak{p}\mid\mathfrak{n}\\ \text{$\mathfrak{p}$ prime}}}(1+N(\mathfrak{p})^{-1})^{-1}.
\end{equation}

We write $M_{\xi}(\pi)$ for $M_{\xi,\rho}(\pi)$ if $\rho$ is defined by \eqref{rho}. 

\subsubsection{Fan-shaped Contours}\label{sec1.1.1}
Let $z_0=r_0e^{i\theta_0}\in \mathbb{C}^{\times}$, where $r_0=|z_0|$ and $-\pi\leq \theta_0<\pi$. For $R>0$, we define the contour $\mathcal{C}_R(z_0):=\mathcal{C}_R^{(1)}(z_0)\cup \mathcal{C}_R^{(2)}(z_0)\cup \mathcal{C}_R^{(3)}(z_0)$, where $\mathcal{C}_R^{(j)}(z_0)$, $1\leq j\leq 3$, is defined as follows.
\begin{itemize}
\item Define $\mathcal{C}_R^{(1)}(z_0)$ as the line segment along the real axis from $0$ to $R$.
\item Define $\mathcal{C}_R^{(2)}(z_0)=\{Re^{i\theta}:\ 0\leq \theta\leq \theta_0\}$ as the directed arc of a circle with radius $R$. The orientation of this arc is counterclockwise if $0 \leq \theta_0< \pi$, and clockwise if $-\pi\leq \theta_0<0$.
\item Define $\mathcal{C}_R^{(3)}(z_0)=\{re^{i\theta_0}:\ R\geq r\geq 0\}$ as the line segment from the point $z_0$ on the circle of radius $R$ back to the origin along the line defined by the angle $\theta_0$.
\end{itemize}

\section{The Regularized Relative Trace Formula}\label{sec2}

\subsection{Test Functions and the  Automorphic Kernel}\label{sec1.2}
Let $f_{\mathfrak{n},\mathfrak{q}}:=\otimes_{v\leq \infty}f_v\in L^1(\overline{G}(\mathbb{A}_F))$ be the test function defined as follows. 
\begin{itemize}
\item Let $v\mid\infty$. Define  
\begin{equation}\label{t1.5}
f_{v}(g_v):=\frac{k_v-1}{4\pi}\cdot \frac{(2i)^{k_v}\det (g_v)^{k_v/2}\textbf{1}_{\det g_v>0}}{(-b_v+c_v+i(a_v+d_v))^{k_v}},\ \ g_v=\begin{pmatrix}
a_v & b_v\\
c_v & d_v
\end{pmatrix}\in G(F_v).
\end{equation} 
Note that $f_{v}$ is the normalized matrix coefficient relative to a lowest weight unit vector in the discrete series $\pi_v$ (cf. e.g., \cite{KL08}). It is integrable over $\overline{G}(F_v)$ if and only if $k_v>2$. 

\item Let $v\mid\mathfrak{n}$. For $g_v\in G(F_v)$, we define 
\begin{equation}\label{t1.6}
f_v(g_v):=q_v^{-\frac{e_v(\mathfrak{n})}{2}}\sum_{\substack{i+j=e_v(\mathfrak{n})\\
i\geq j\geq 0}}\textbf{1}_{Z(F_v)K_v\diag(\varpi_v^i,\varpi_v^j)K_v}(g_v).
\end{equation}

\item For $q\subsetneq \mathcal{O}_F$, and $v=\mathfrak{q}$, i.e., $\mathfrak{p}_v=\mathfrak{q}$. For $g_v\in G(F_v)$, we define 
\begin{equation}\label{t1.7}
f_v(g_v):=\Vol(K_0(\mathfrak{p}_v))^{-1}\textbf{1}_{Z(F_v)K_0(\mathfrak{p}_v)}(g_v),
\end{equation}
where 
\begin{align*}
K_0(\mathfrak{p}_v):=\bigg\{\begin{pmatrix}
a_v& b_v\\
c_v & d_v
\end{pmatrix}\in K_v:\ c_v\in \mathfrak{p}_v\bigg\}.
\end{align*}

\item Let $v<\infty$ and $v\nmid \mathfrak{n}\mathfrak{q}$. We define, for $g_v\in G(F_v)$, 
\begin{equation}\label{t1.8}
f_v(g_v):=\textbf{1}_{Z(F_v)K_v}(g_v). 
\end{equation}
\end{itemize}

With the above test function $f_{\mathfrak{n},\mathfrak{q}}$, we define the kernel function as
\begin{equation}\label{f1.5}
\K(x,y):=\sum_{\gamma\in \overline{G}(F)}f_{\mathfrak{n},\mathfrak{q}}(x^{-1}\gamma y),\ \ x, y\in G(\mathbb{A}_F). 
\end{equation}
This is conventionally called the \textit{geometric} expansion of the kernel function $\K$. 

We also have the spectral decomposition 
\begin{equation}\label{eq1.6}
\K(x,y)=\sum_{\pi\in \Pi_{\mathbf{k}}(\mathfrak{q})}\sum_{\phi\in\mathfrak{B}_{\pi}}\pi(f_{\mathfrak{n},\mathfrak{q}})\phi(x)\overline{\phi(y)},
\end{equation}
where $\mathfrak{B}_{\pi}$ is an orthonormal basis of $\pi$. Notice that the right hand side of \eqref{eq1.6} is a finite sum.  

\subsection{The Relative Trace Formula}\label{sec2.2}
Let $f_{\mathfrak{n},\mathfrak{q}}$ be defiend as in \textsection\ref{sec1.2}. Let $\textbf{s}=(s_1,s_2)\in \mathbb{C}^2$. Consider the function 
\begin{equation}\label{eq1.5}
J(f_{\mathfrak{n},\mathfrak{q}},\textbf{s}):=\int_{F^{\times}\backslash\mathbb{A}_F^{\times}}\int_{F^{\times}\backslash\mathbb{A}_F^{\times}}\K\left(\begin{pmatrix}
x\\
&1
\end{pmatrix},\begin{pmatrix}
y\\
&1
\end{pmatrix}\right)|x|^{s_1}|y|^{s_2}d^{\times}yd^{\times}x.
\end{equation}

A regularized relative trace formula based on the spectral-geometric expansion of the kernel function $\K(\cdot,\cdot)$ was established in \cite{Yan23a} or \cite[Theorem 3.2]{Yan23c}. However, by utilizing \eqref{eq1.6} and the rapid decay of cusp forms, the function $J(f_{\mathfrak{n},\mathfrak{q}},\textbf{s})$ converges absolutely in $(s_1,s_2)\in \mathbb{C}^2$. Thus, we can provide a simplified argument to derive the specific relative trace formula required for the purposes of this paper. 

\subsubsection{The Spectral Side}
Let $\Re(s_1)\gg 1$ and $\Re(s_2)\gg 1$. Substituting \eqref{eq1.6} into \eqref{eq1.5} and sapping integrals, we obtain 
\begin{equation}\label{eq1.8}
J(f_{\mathfrak{n},\mathfrak{q}},\textbf{s})=\sum_{\pi\in \Pi_{\mathbf{k}}(\mathfrak{q})}\sum_{\phi\in\mathfrak{B}_{\pi}}\mathcal{P}(s_1,\pi(f_{\mathfrak{n},\mathfrak{q}})\phi)\mathcal{P}(s_2,\overline{\phi}),
\end{equation}
where for $s\in \mathbb{C}$, 
\begin{equation}\label{eq2.2}
\mathcal{P}(s,\phi):=\int_{F^{\times}\backslash \mathbb{A}_F^{\times}}\phi\left(\begin{pmatrix}
a\\
&1
\end{pmatrix}\right)|a|^{s}d^{\times}a.
\end{equation}

For accuracy, we denote $J_{\mathrm{Spec}}(f_{\mathfrak{n},\mathfrak{q}},\textbf{s})$ by the integral \eqref{eq1.8}, which is referred to as the \textit{spectral side} of the relative trace formula. 

\subsubsection{The Geometric Side}\label{sec2.3.2}
Substituting the geometric expansion \eqref{f1.5} into \eqref{eq1.5}, the function $J(f_{\mathfrak{n},\mathfrak{q}},\textbf{s})$ boils down to 
\begin{equation}\label{f1.10}
\iint_{(F^{\times}\backslash\mathbb{A}_F^{\times})^2}\sum_{\gamma\in \overline{G}(F)}f_{\mathfrak{n},\mathfrak{q}}\left(\begin{pmatrix}
x^{-1}\\
&1	
\end{pmatrix}
\gamma \begin{pmatrix}
xy&\\
&1
\end{pmatrix}\right)|x|^{s_1+s_2}|y|^{s_2}d^{\times}yd^{\times}x,
\end{equation}
which converges in the region $\Re(s_1)\gg 1$ and $\Re(s_2)\gg 1$.

We denote by $J_{\mathrm{Geom}}(f_{\mathfrak{n},\mathfrak{q}},\textbf{s})$ the integral \eqref{f1.10}, which is conventionally called the \textit{geometric side} of the relative trace formula. 


Write $w=\begin{pmatrix}
& -1\\
1&
\end{pmatrix}$. The Bruhat decomposition is the following disjoint union:
\begin{align*}
\mathrm{GL}_2(F)=&B(F)\bigsqcup B(F)w\bigsqcup T(F)\begin{pmatrix}
1& \\
1&1
\end{pmatrix}wA(F)\bigsqcup\bigsqcup_{t\in F-\{1\}} T(F)\begin{pmatrix}
1& t\\
1&1
\end{pmatrix}A(F).
\end{align*}


Therefore, for $\Re(s_1)\gg 1$ and $\Re(s_2)\gg 1$, substituting the above decomposition into \eqref{f1.10}, we obtain  
\begin{equation}\label{fc2.14}
J_{\mathrm{Geom}}(f_{\mathfrak{n},\mathfrak{q}},\textbf{s})=\sum_{\delta\in \{I_2,w\}}J_{\mathrm{small}}^{\delta}(f_{\mathfrak{n},\mathfrak{q}},\textbf{s})+\sum_{\delta\in \{I_2,w\}}J_{\mathrm{dual}}^{\delta}(f_{\mathfrak{n},\mathfrak{q}},\textbf{s})+J_{\mathrm{reg}}(f_{\mathfrak{n},\mathfrak{q}},\textbf{s}),
\end{equation}
where $J_{\mathrm{small}}^{\delta}(f_{\mathfrak{n},\mathfrak{q}},\textbf{s})$ is defined by 
\begin{align*}
\int_{F^{\times}\backslash\mathbb{A}_F^{\times}}\int_{F^{\times}\backslash\mathbb{A}_F^{\times}}\sum_{\gamma\in B_0(F)}f_{\mathfrak{n},\mathfrak{q}}\left(\begin{pmatrix}
x^{-1}& \\
&1
\end{pmatrix}\gamma \delta\begin{pmatrix}
y& \\
&1
\end{pmatrix}\right)|x|^{s_1}|y|^{s_2}d^{\times}yd^{\times}x,
\end{align*}
and $J_{\mathrm{dual}}^{\delta}(f_{\mathfrak{n},\mathfrak{q}},\textbf{s})$ is defined by 
\begin{align*}
\int_{\mathbb{A}_F^{\times}}\int_{\mathbb{A}_F^{\times}}f_{\mathfrak{n},\mathfrak{q}}\left(\begin{pmatrix}
x^{-1}& \\
&1
\end{pmatrix}\begin{pmatrix}
1\\
1& 1	
\end{pmatrix}\delta\begin{pmatrix}
y& \\
&1
\end{pmatrix}\right)|x|^{s_1}|y|^{s_2}d^{\times}yd^{\times}x,
\end{align*}
and $J_{\mathrm{reg}}(f_{\mathfrak{n},\mathfrak{q}},\textbf{s})$ is defined by 
\begin{align*}
\sum_{t\in F-\{0,1\}}\int_{\mathbb{A}_F^{\times}}\int_{\mathbb{A}_F^{\times}}f_{\mathfrak{n},\mathfrak{q}}\left(\begin{pmatrix}
x^{-1}& \\
&1
\end{pmatrix}\begin{pmatrix}
1& t\\
1& 1
\end{pmatrix}\begin{pmatrix}
xy& \\
&1
\end{pmatrix}\right)|x|^{s_1+s_2}|y|^{s_2}d^{\times}yd^{\times}x.
\end{align*}

\begin{lemma}\label{lemma2.1}
Suppose $\Re(s_1)\gg 1$, $\Re(s_2)\gg 1$, and $\Re(s_1-s_2)\gg 1$. 
\begin{itemize}
\item The function  $J_{\mathrm{small}}^{I_2}(f_{\mathfrak{n},\mathfrak{q}},\textbf{s})$ converges absolutely and is equal to 
\begin{equation}\label{equ2.14}
\int_{\mathbb{A}_F^{\times}}\int_{\mathbb{A}_F^{\times}}\int_{\mathbb{A}_F}f_{\mathfrak{n},\mathfrak{q}}\left(\begin{pmatrix}
1& u\\
&1
\end{pmatrix} 
\begin{pmatrix}
y& \\
&1
\end{pmatrix}\right)\psi(xu)|x|^{1+s_1+s_2}|y|^{s_2}d^{\times}yd^{\times}x.
\end{equation} 
\item The function  $J_{\mathrm{small}}^{w}(f_{\mathfrak{n},\mathfrak{q}},\textbf{s})$ converges absolutely and is equal to 
\begin{equation}\label{equ2.15}
\int_{\mathbb{A}_F^{\times}}\int_{\mathbb{A}_F^{\times}}\int_{\mathbb{A}_F}f_{\mathfrak{n},\mathfrak{q}}\left(\begin{pmatrix}
1& u\\
&1
\end{pmatrix} 
\begin{pmatrix}
y& \\
&1
\end{pmatrix}w\right)\psi(xu)|x|^{1+s_1-s_2}|y|^{-s_2}d^{\times}yd^{\times}x.
\end{equation} 
\end{itemize}
\end{lemma}
\begin{proof}
Let $g\in G(\mathbb{A}_F)$. Define the function $F(g)$ by 
\begin{align*}
\int_{F^{\times}\backslash\mathbb{A}_F^{\times}}\int_{F^{\times}\backslash\mathbb{A}_F^{\times}}\sum_{\gamma\in B_0(F)}f_{\mathfrak{n},\mathfrak{q}}\left(\begin{pmatrix}
x^{-1}& \\
&1
\end{pmatrix}g^{-1}\gamma \begin{pmatrix}
xy& \\
&1
\end{pmatrix}\right)|x|^{s_1+s_2}|y|^{s_2}d^{\times}yd^{\times}x.
\end{align*}

Then $F(u g)=F(g)$ for all $u\in B_0(F)$. Utilizing Fourier expansion we deduce 
\begin{align*}
F(g)=\int_{F\backslash\mathbb{A}_F}F\left(\begin{pmatrix}
1& u\\
&1
\end{pmatrix}
g\right)du+\sum_{\alpha\in F^{\times}}\int_{F\backslash\mathbb{A}_F}F\left(\begin{pmatrix}
1& u\\
&1
\end{pmatrix}\begin{pmatrix}
\alpha\\
&1
\end{pmatrix}
g\right)\overline{\psi}(u)du.
\end{align*}

Since $f_{v}$, for $v\mid\infty$, is a matrix coefficient of a discrete series, then 
\begin{equation}\label{equ2.16}
\int_{F\backslash\mathbb{A}_F}F\left(\begin{pmatrix}
1& u\\
&1
\end{pmatrix}
g\right)du\equiv 0.	
\end{equation}

Moreover, by changing of variables, the sum 
\begin{align*}
\sum_{\alpha\in F^{\times}}\int_{F\backslash\mathbb{A}_F}F\left(\begin{pmatrix}
1& u\\
&1
\end{pmatrix}\begin{pmatrix}
\alpha\\
&1
\end{pmatrix}
\right)\overline{\psi}(u)du
\end{align*}
is equal to 
\begin{equation}\label{equ2.17}
\int_{\mathbb{A}_F^{\times}}\int_{\mathbb{A}_F^{\times}}\int_{\mathbb{A}_F}f_{\mathfrak{n},\mathfrak{q}}\left(\begin{pmatrix}
x^{-1}\\
&1
\end{pmatrix}\begin{pmatrix}
1& u\\
&1
\end{pmatrix} 
\begin{pmatrix}
y& \\
&1
\end{pmatrix}\right)\psi(u)|x|^{s_1}|y|^{s_2}d^{\times}yd^{\times}x.
\end{equation} 

Therefore, \eqref{equ2.14} follows from \eqref{equ2.16} and \eqref{equ2.17}, along with the change of variable $y\mapsto xy$. The absolute convergence will be proved in \textsection\ref{section4}, where we also obtain a meromorphic continuation of $J_{\mathrm{small}}^{I_2}(f_{\mathfrak{n},\mathfrak{q}},\textbf{s})$.  

By definition, and the change of variable $y\mapsto y^{-1}$, $J_{\mathrm{small}}^{w}(f_{\mathfrak{n},\mathfrak{q}},\textbf{s})$ becomes
\begin{align*}
\int_{F^{\times}\backslash\mathbb{A}_F^{\times}}\int_{F^{\times}\backslash\mathbb{A}_F^{\times}}\sum_{\gamma\in B_0(F)}f_{\mathfrak{n},\mathfrak{q}}\left(\begin{pmatrix}
x^{-1}& \\
&1
\end{pmatrix}\gamma \begin{pmatrix}
y& \\
&1
\end{pmatrix}w\right)|x|^{s_1}|y|^{-s_2}d^{\times}yd^{\times}x.
\end{align*}

Therefore, \eqref{equ2.15} follows similarly from the proof of \eqref{equ2.14}.
\end{proof}

\begin{lemma}\label{lemma2.2}
Suppose $\Re(s_1)\gg 1$, $\Re(s_2)\gg 1$, and $\Re(s_1-s_2)\gg 1$.
\begin{itemize}
\item The function  $J_{\mathrm{dual}}^{I_2}(f_{\mathfrak{n},\mathfrak{q}},\textbf{s})$ converges absolutely, and 
\begin{equation}\label{f2.18}
J_{\mathrm{dual}}^{I_2}(f_{\mathfrak{n},\mathfrak{q}},\textbf{s})=\int_{\mathbb{A}_F^{\times}}\int_{\mathbb{A}_F^{\times}}f_{\mathfrak{n},\mathfrak{q}}\left(\begin{pmatrix}
1\\
x& 1	
\end{pmatrix}\begin{pmatrix}
y& \\
&1
\end{pmatrix}\right)|x|^{s_1+s_2}|y|^{s_2}d^{\times}yd^{\times}x.
\end{equation} 
\item The function  $J_{\mathrm{dual}}^{w}(f_{\mathfrak{n},\mathfrak{q}},\textbf{s})$ converges absolutely,
\begin{equation}\label{f2.19}
J_{\mathrm{dual}}^{w}(f_{\mathfrak{n},\mathfrak{q}},\textbf{s})=\int_{\mathbb{A}_F^{\times}}\int_{\mathbb{A}_F^{\times}}f_{\mathfrak{n},\mathfrak{q}}\left(\begin{pmatrix}
1\\
x& 1	
\end{pmatrix}\begin{pmatrix}
y& \\
&1
\end{pmatrix}w\right)|x|^{s_1-s_2}|y|^{-s_2}d^{\times}yd^{\times}x.
\end{equation} 
\end{itemize}
\end{lemma}
\begin{proof}
The expressions \eqref{f2.18} and \eqref{f2.19} follow from a direct change of variables, and the absolute convergence will be verified in \textsection\ref{section5}.
\end{proof}

\subsubsection{The Relative Trace Formula}
Let notation be as in \textsection\ref{sec2.3.2}. Define 
\begin{equation}\label{1.12}
J_{\mathrm{sing}}(f_{\mathfrak{n},\mathfrak{q}},\textbf{s}):=\sum_{\delta\in \{I_2,w\}}J_{\mathrm{small}}^{\delta}(f_{\mathfrak{n},\mathfrak{q}},\textbf{s})+\sum_{\delta\in \{I_2,w\}}J_{\mathrm{dual}}^{\delta}(f_{\mathfrak{n},\mathfrak{q}},\textbf{s}).
\end{equation}

As a consequence of 
\begin{align*}
J_{\mathrm{Spec}}(f_{\mathfrak{n},\mathfrak{q}},\textbf{s})=J_{\mathrm{Geom}}(f_{\mathfrak{n},\mathfrak{q}},\textbf{s})
\end{align*} 
and the orbital decomposition \eqref{fc2.14}, we derive the following abstract relative trace formula. 
\begin{thm}\label{thm2.3}
Let notation be as before. Let $\textbf{s}=(s_1,s_2)\in \mathbb{C}^2$ satisfy $\Re(s_1)\gg 1$, $\Re(s_2)\gg 1$, and $\Re(s_1-s_2)\gg 1$. Then
\begin{equation}\label{1.13}
J_{\mathrm{Spec}}(f_{\mathfrak{n},\mathfrak{q}},\textbf{s})=J_{\mathrm{sing}}(f_{\mathfrak{n},\mathfrak{q}},\textbf{s})+J_{\mathrm{reg}}(f_{\mathfrak{n},\mathfrak{q}},\textbf{s}).
\end{equation}
Moreover, both sides of  \eqref{1.13} admit a meromorphic continuation to $\mathbb{C}^2$.
\end{thm}

Theorem \ref{thm2.3} is a special case of \cite[Corollary 3.3]{Yan23c}. The formula \eqref{1.13} will be computed explicitly as the weighted second moment in \textsection\ref{sec8}; see Theorem \ref{thm7.1} in \textsection\ref{sec8.1}.

\section{The Spectral Side}\label{sec3}
In this section we aim to derive a meromorphic continuation of  $J_{\mathrm{Spec}}(f_{\mathfrak{n},\mathfrak{q}},\textbf{s})$ to $\textbf{s}\in \mathbb{C}^2$, and obtain a precise formula for it explicitly as a second moment of automorphic $L$-functions. 
\subsection{Explicit Period Integrals}\label{sec3.1}
Let $\textbf{s}=(s_1,s_2)\in \mathbb{C}$. Let $\phi$ be a cusp form. Define
\begin{equation}\label{eq2.1}
\mathcal{Z}(\textbf{s},\phi):=\langle \phi,\phi\rangle^{-1}\mathcal{P}(s_1,\phi)\mathcal{P}(s_2,\overline{\phi}),
\end{equation}
where for $s\in \mathbb{C}$, the function $\mathcal{P}(s,\phi)$ is defined as in \eqref{eq2.2}.

\begin{lemma}\label{lem2.1}
Let $\pi\in \mathcal{A}_0(\mathbf{k},\mathfrak{q})$. Let $\phi$ be the cusp form corresponding to the vector $\textbf{v}=\otimes_{v\leq\infty}\textbf{v}_v\in \pi$ such that 
\begin{itemize}
\item at $v\mid\infty$, $\textbf{v}_v$ is the lowest weight vector in $\pi_v$;
\item at $v=\mathfrak{q}$, $\textbf{v}_v$ is the local new vector in $\pi_v$, which is a Steinberg representation twisted by a unramified character of $F_v^{\times}$;
\item at $v\nmid \mathfrak{q}\infty$, $\textbf{v}_v$ is a spherical vector.  
\end{itemize}
Let $\textbf{s}=(s_1,s_2)\in \mathbb{C}$. Then $\mathcal{Z}(\textbf{s},\phi)$ admits a meromorphic continuation to $\textbf{s}\in \mathbb{C}^2$. Moreover, we have the explicit calculation 
\begin{equation}\label{f2.2}
\mathcal{Z}(\textbf{s},\phi)=\frac{\zeta_{\mathfrak{q}}(1)L(s_1+1/2,\pi)L(s_2+1/2,\widetilde{\pi})}{2\zeta_{\mathfrak{q}}(2)^2L^{(\mathfrak{q})}(1,\pi,\Ad)D_F^{1/2-s_1-s_2}}\prod_{v\mid\infty}\frac{2^{k_v}\pi \Gamma(k_v/2+s_1)\Gamma(k_v/2+s_2)}{(2\pi)^{s_1+s_2}\Gamma(k_v)},
\end{equation}
where $\zeta_{\mathfrak{q}}(s)=(1-N(\mathfrak{q})^{-s})^{-1}$ and $L^{(\mathfrak{q})}(s,\pi,\Ad)$ is the partial adjoint $L$-function with the local factors $L_v(s,\pi_v,\Ad)$ being removed at $v=\mathfrak{q}\infty$. 
\end{lemma}
\begin{proof}
For $v\leq\infty$, let $W_v$ be the vector in the Whittaker model of $\pi_v$ corresponding to $\textbf{v}_v$. By Rankin-Selberg convolution (cf. \cite[\textsection 4.4.2]{MV10}), 
\begin{equation}\label{2.1}
\langle \phi,\phi\rangle=2\Lambda(1,\pi,\Ad)\prod_{v\leq \infty}\frac{\zeta_v(2)\langle W_v,W_v\rangle_v}{L_v(1,\pi_v\times\widetilde{\pi}_v)},
\end{equation}
where $\langle W_v,W_v\rangle_v:=\int_{F_v^{\times}}|W_v\left(\diag(a_v,1)\right)|^2d^{\times}a_v$. 

Suppose $\Re(s_1)\gg 1$ and $\Re(s_2)\gg 1$, where $\mathcal{Z}(\textbf{s},\phi)$ converges absolutely. By Fourier-Whittaker expansion, along with \eqref{2.1}, we obtain 
\begin{equation}\label{f2.3}
\mathcal{Z}(\textbf{s},\phi)=\frac{1}{2\Lambda(1,\pi,\Ad)}\prod_{v\leq\infty}\frac{L_v(1,\pi_v\times\widetilde{\pi}_v)}{\zeta_v(2)}\cdot \frac{\mathcal{P}_v(s_1,W_v)\mathcal{P}_v(s_2,\overline{W_v})}{\langle W_v,W_v\rangle_v},
\end{equation}
where for $s\in \mathbb{C}$ with $\Re(s)\gg 1$, we define 
\begin{align*}
\mathcal{P}_v(s_1,W_v):=\int_{F_v^{\times}}W_v\left(\begin{pmatrix}
a_v\\
&1
\end{pmatrix}\right)|a_v|_v^{s}d^{\times}a_v.
\end{align*}

We consider the local integrals $\langle W_v,W_v\rangle_v^{-1}\cdot \mathcal{P}_v(s_1,W_v)\mathcal{P}_v(s_2,\overline{W_v})$ in the following scenarios. 
\begin{itemize}
\item Let $v\nmid\mathfrak{q}\infty$. Utilizing \cite[Proposition 4.6.8]{Bum97}, we obtain by a straight forward calculation that 
\begin{align*}
\mathcal{P}_v(s_1,W_v)\mathcal{P}_v(s_2,\overline{W_v})=L_v(s_1+1/2,\pi_v)L_v(s_2+1/2,\widetilde{\pi}_v)q_v^{d_v(s_1+s_2)}\Vol(\mathcal{O}_v^{\times})^2,
\end{align*}
where $L_v(s+1/2,\pi_v)$ is the local $L$-factor; and 
\begin{align*}
\langle W_v,W_v\rangle_v=&\int_{N(F_v)\backslash G(F_v)}|W_v(g_v)|^2|\det g_v|_vdg_v=\int_{F_v^{\times}}|W_v(\diag(a_v,1))|^2d^{\times}a_v
\\
=&\Vol(\mathcal{O}_v^{\times})\sum_{j\geq -d_v}|W_v(\diag(\varpi_v^j,1))|^2=\frac{L_v(1,\pi_v\times\widetilde{\pi}_v)\Vol(\mathcal{O}_v^{\times})}{\zeta_v(2)}.
\end{align*}




Combining the above calculations we derive that  
\begin{equation}\label{2.2}
\frac{L_v(1,\pi_v\times\widetilde{\pi}_v)\mathcal{P}_v(s_1,W_v)\mathcal{P}_v(s_2,\overline{W_v})}{\zeta_v(2)\langle W_v,W_v\rangle_v}=\frac{L_v(s_1+1/2,\pi_v)L_v(s_2+1/2,\widetilde{\pi}_v)}{q_v^{d_v(1/2-s_1-s_2)}}.
\end{equation}

\item Let $v=\mathfrak{q}$. By assumption, $\pi_v=\chi\otimes\mathrm{St}$, where $\mathrm{St}$ is the Steinberg representation, and $\chi$ is a unramified character, and $W_v$ is a local new form. Hence, for $\gamma\in \mathcal{O}_v^{\times}$,  
\begin{align*}
W_v(\diag(\varpi_v^j\gamma,1))=\chi(\varpi_v)^jq_v^{-j}W_v(I_2)\textbf{1}_{j\geq 0}.
\end{align*}
As a result, we obtain 
\begin{align*}
\mathcal{P}_v(s_1,W_v)=W_v(I_2)\Vol(\mathcal{O}_v^{\times})\sum_{j\geq 0}\frac{\chi(\varpi_v)^j}{q_v^{j(s_1+1)}}=W_v(I_2)\Vol(\mathcal{O}_v^{\times})L_v(s_1+1/2,\pi_v),
\end{align*}
and 
\begin{align*}
\langle W_v,W_v\rangle_v=\int_{F_v^{\times}}|W_v(\diag(a_v,1))|^2d^{\times}a_v=|W_v(I_2)|^2\Vol(\mathcal{O}_v^{\times})\sum_{j\geq 0}q_v^{-2m},
\end{align*}
which is equal to $|W_v(I_2)|^2\Vol(\mathcal{O}_v^{\times})\zeta_v(2)$. Note that in the above calculation we have used the assumption that $d_v=0$, i.e., $\mathfrak{q}$ is not ramified. 

Gathering together the above calculation we derive that 
\begin{equation}\label{2.3}
\frac{L_v(1,\pi_v\times\widetilde{\pi}_v)\mathcal{P}_v(s_1,W_v)\mathcal{P}_v(s_2,\overline{W_v})}{\zeta_v(2)\langle W_v,W_v\rangle_v}=\frac{L_v(s_1+1/2,\pi_v)L_v(s_2+1/2,\widetilde{\pi}_v)}{\zeta_v(2)^2L_v(1,\pi_v\times\widetilde{\pi}_v)^{-1}}.
\end{equation}

\item Let $v\mid\infty$. Then $W_v\left(\begin{pmatrix}
a_v&\\
&1
\end{pmatrix}\right)=a_v^{\frac{k_v}{2}}e^{-2\pi a_v}e^{2\pi}W_v(I_2)\textbf{1}_{a_v>0}$. Hence, 
\begin{equation}\label{2.4}
\frac{\mathcal{P}_v(s_1,W_v)}{e^{2\pi}W_v(I_2)}=\int_0^{\infty}\frac{a_v^{\frac{k_v}{2}+s_1}}{e^{2\pi a_v}}d^{\times}a_v=\frac{\Gamma(k_v/2+s_1)}{(2\pi)^{k_v/2+s_1}},
\end{equation}
and 
\begin{equation}\label{2.5}
\frac{\langle W_v,W_v\rangle_v}{e^{4\pi}|W_v(I_2)|^2}=\int_{F_v^{\times}}\frac{|W_v(\diag(a_v,1))|^2}{e^{4\pi}|W_v(I_2)|^2}d^{\times}a_v=\int_0^{\infty}\frac{a^{k_v}}{e^{4\pi a}}d^{\times}a=\frac{\Gamma(k_v)}{(4\pi)^{k_v}}.
\end{equation}

It follows from \eqref{2.4} and \eqref{2.5} that 
\begin{equation}\label{2.6}
\frac{L_v(1,\pi_v\times\widetilde{\pi}_v)\mathcal{P}_v(s_1,W_v)\mathcal{P}_v(s_2,\overline{W_v})}{\zeta_v(2)\langle W_v,W_v\rangle_v}=\frac{L_v(1,\pi_v\times\widetilde{\pi}_v)\Gamma(k_v/2+s_1)\Gamma(k_v/2+s_2)}{(2\pi)^{s_1+s_2-k_v}\zeta_v(2)\Gamma(k_v)}.
\end{equation}
\end{itemize}

Substituting \eqref{2.2}, \eqref{2.3}, and \eqref{2.6} into \eqref{f2.3}, the function $\mathcal{Z}(\textbf{s},\phi)$ becomes 
\begin{align*}
\frac{L(s_1+1/2,\pi)L(s_2+1/2,\widetilde{\pi})}{2L(1,\pi,\Ad)D_F^{1/2-s_1-s_2}}\prod_{v\mid\infty}\frac{\Gamma(k_v/2+s_1)\Gamma(k_v/2+s_2)}{(2\pi)^{s_1+s_2}2^{-k_v}\zeta_v(2)\Gamma(k_v)}\prod_{v=\mathfrak{q}}\frac{L_v(1,\pi_v\times\widetilde{\pi}_v)}{\zeta_v(2)^2}.
\end{align*}

Therefore, \eqref{f2.2} follows from the above expression along with the facts that 
\begin{align*}
L_v(1,\pi_v\times\widetilde{\pi}_v)=L_v(1,\pi_v,\Ad)\zeta_v(1)=L_v(1,\pi_v,\Ad),
\end{align*}
and $\zeta_v(2)=\pi^{-1}$ at all $v\mid\infty$. Here we have made use of the fact that $\zeta_v(1)=1$ at a real place $v$. 
\end{proof}
\begin{remark}
For $v<\infty$, the local new Whittaker vector (in the above proof) is normalized via
\begin{equation}\label{eq3.10}
W_v\left(\begin{pmatrix}
\varpi_v^{-d_v}\\
&1
\end{pmatrix}\right)=1.
\end{equation}	
\end{remark}

\subsection{Contribution From Old Forms}
Suppose $\mathfrak{q}\subsetneq \mathcal{O}_F$ is a prime. Let $v=\mathfrak{q}$. Let $\pi_v$ be a unitary irreducible admissible unramified representation of $\mathrm{PGL}_2(F_v)$. Let $W_v^{\circ}$ be the normalized spherical vector in the Whittaker model of $\pi_v$ such that $\langle W_v^{\circ}, W_v^{\circ}\rangle=1$.

Denote by $K_v[1]:=\big\{\begin{pmatrix}
a& b\\
c& d
\end{pmatrix}\in K_v:\ c\in \mathfrak{q}\big\}$. Then the subspace $\mathcal{W}_v^{K_v[1]}$ of right-$K_v[1]$-invariant vectors in the Whittaker model of $\pi_v$ is $2$-dimensional. By Atkin-Lehner decomposition, an orthonormal basis of $\mathcal{W}_v^{K_v[1]}$ is of the form $W_v^{\circ}$ and $\alpha_{\pi_v} W_v^{\circ}+\beta_{\pi_v}\pi_v(\diag{1,\varpi_v})W_v^{\circ}$ for some $\alpha_{\pi_v}, \beta_{\pi_v}\in \mathbb{C}$ satisfying 
\begin{equation}\label{2.10}
\begin{cases}
\langle W_v^{\circ},\alpha_{\pi_v} W_v^{\circ}+\beta_{\pi_v}\pi_v(\diag{1,\varpi_v})W_v^{\circ}\rangle_v=0\\
\langle \alpha_{\pi_v} W_v^{\circ}+\beta_{\pi_v}\pi_v(\diag{1,\varpi_v})W_v^{\circ},\alpha_{\pi_v} W_v^{\circ}+\beta_{\pi_v}\pi_v(\diag{1,\varpi_v})W_v^{\circ}\rangle_v=1.
\end{cases}
\end{equation}

\begin{lemma}\label{lem2.2}
Let notation be as before. Let $\mathfrak{q}\subsetneq \mathcal{O}_F$ be a prime ideal. Let $\pi=\otimes_v\pi_v\in \mathcal{F}(\mathbf{k},\mathcal{O}_F)$. Let $\alpha=\alpha_{\pi_{\mathfrak{q}}}$ and $\beta=\beta_{\pi_{\mathfrak{q}}}$ be the coefficients defined by \eqref{2.10}. Let $\chi_v$ be the nontrivial unramified quadratic character of $F_v^{\times}$. 
\begin{itemize}
\item Let $\lambda_{\pi}(\mathfrak{q})$ be the $\mathfrak{q}$-th Dirichlet coefficient of $L(s,\pi)$. Then 
\begin{equation}\label{eq2.11}
\begin{cases}
\alpha^2=\lambda_{\pi}(\mathfrak{q})^2N(\mathfrak{q})^{-1}L_{\mathfrak{q}}(1/2,\pi_{\mathfrak{q}})L_{\mathfrak{q}}(1/2,\pi_{\mathfrak{q}}\times\chi_{\mathfrak{q}}),\\
\beta^2=\zeta_{\mathfrak{q}}(1)^2\zeta_{\mathfrak{q}}(2)^{-2}L_{\mathfrak{q}}(1/2,\pi_{\mathfrak{q}})L_{\mathfrak{q}}(1/2,\pi_{\mathfrak{q}}\times\chi_{\mathfrak{q}}). 
\end{cases}
\end{equation}

\item Moreover, we have 
\begin{equation}\label{2.11}
1+(\alpha+\beta)^2=2(1+N(\mathfrak{q})^{-1})L_{\mathfrak{q}}(1/2,\pi_{\mathfrak{q}}\times\chi_{\mathfrak{q}}).
\end{equation}
\end{itemize} 	
\end{lemma}
\begin{proof}
Write $\gamma=\langle W_v^{\circ},\pi_v(\diag{1,\varpi_v})W_v^{\circ}\rangle_v$. The constraints \eqref{2.10} amounts to 
\begin{align*}
\begin{cases}
\alpha+\beta\gamma=0\\
\alpha^2+\beta^2+2\alpha\beta\gamma=1
\end{cases}\ \ \Leftrightarrow\ \ \ \begin{cases}
\alpha=-\beta\gamma\\
\beta^2(1-\gamma^2)=1.
\end{cases}
\end{align*}

Therefore, we have 
\begin{equation}\label{2.12}
1+(\alpha+\beta)^2=1+\frac{(1-\gamma)^2}{1-\gamma^2}=\frac{2}{1+\gamma}.
\end{equation}

Utilizing Macdonald's formula we have 
\begin{equation}\label{eq2.14}
\gamma=\lambda_{\pi}(\mathfrak{q})N(\mathfrak{q})^{-1/2}(1+N(\mathfrak{q})^{-1})^{-1},
\end{equation}
from which we obtain \eqref{eq2.11}. Substituting \eqref{eq2.14} into \eqref{2.12} leads to 
\begin{align*}
1+(\alpha+\beta)^2=\frac{2(1+N(\mathfrak{q})^{-1})}{1-\lambda_{\pi}(\mathfrak{q})N(\mathfrak{q})^{-1/2}+N(\mathfrak{q})^{-1}}.
\end{align*}

Therefore, \eqref{2.11} follows from the fact that $\chi_v(\varpi_v)=-1$.
\end{proof}

Let $\pi=\otimes_v\pi_v$ be a unitary automorphic representation of $\mathrm{PGL}_2$ over $F$. By the definition of the test function $f=\otimes_vf_v$, we have $\pi(f_{\mathfrak{n},\mathfrak{q}})\phi\equiv 0$ for all $\phi\in \pi$ unless 
\begin{itemize}
\item $\phi$ is right invariant by the compact subgroup $K_{\mathfrak{q}}[1]\otimes\prod_{v\nmid\mathfrak{q\infty}}K_v$;
\item $\pi_v$ is the discrete series of weight $k_v$ for all $v\mid\infty$.
\end{itemize}

By the definition of $f_{\mathfrak{n},\mathfrak{q}}$, we obtain 
\begin{equation}\label{f2.15}
J_{\mathrm{Spec}}(f_{\mathfrak{n},\mathfrak{q}},\textbf{s})=J_{\mathrm{Spec}}^{\mathrm{new}}(f_{\mathfrak{n},\mathfrak{q}},\textbf{s})+J_{\mathrm{Spec}}^{\mathrm{old}}(f_{\mathfrak{n},\mathfrak{q}},\textbf{s}),
\end{equation}
where 
\begin{align*}
J_{\mathrm{Spec}}^{\mathrm{new}}(f_{\mathfrak{n},\mathfrak{q}},\textbf{s}):=&\sum_{\substack{\pi\in \mathcal{F}(\mathbf{k},\mathfrak{q}),\ \phi\in \mathcal{B}_{\pi}^{\mathrm{new}}}}\frac{\mathcal{P}(s_1,\pi(f_{\mathfrak{n},\mathfrak{q}})\phi)\mathcal{P}(s_2,\overline{\phi})}{\langle\phi,\phi\rangle},\\
J_{\mathrm{Spec}}^{\mathrm{old}}(f_{\mathfrak{n},\mathfrak{q}},\textbf{s}):=&\sum_{\substack{\pi\in \mathcal{F}(\mathbf{k},\mathcal{O}_F),\ \phi\in \mathfrak{B}_{\pi}^{K_{\mathfrak{q}}[1]}}}\frac{\mathcal{P}(s_1,\pi(f_{\mathfrak{n},\mathfrak{q}})\phi)\mathcal{P}(s_2,\overline{\phi})}{\langle\phi,\phi\rangle}.
\end{align*}
Here $\mathcal{P}(s_1,\pi(f_{\mathfrak{n},\mathfrak{q}})\phi)$ and $\mathcal{P}(s_2,\overline{\phi})$ are defined by \eqref{eq2.2}, and 
\begin{itemize}
\item $\mathcal{B}_{\pi}^{\mathrm{new}}$ is an orthonormal basis of $\pi$ consisting of vectors $\phi$ satisfying the constraints in Lemma \ref{lem2.1}; 
\item $\mathfrak{B}_{\pi}^{K_{\mathfrak{q}}[1]}$ is an orthonormal basis of $\pi$ consisting of right-$K_{\mathfrak{q}}[1]\otimes\prod_{v\nmid\mathfrak{q\infty}}K_v$-invariant vectors.
\end{itemize}
Hence, $\#\mathcal{B}_{\pi}^{\mathrm{new}}=1$ and $\#\mathfrak{B}_{\pi}^{K_{\mathfrak{q}}[1]}=2$. 
\begin{lemma}\label{lem2.3}
Let notation be as before. Then $J_{\mathrm{Spec}}^{\mathrm{new}}(f_{\mathfrak{n},\mathfrak{q}},\textbf{s})$ admits a meromorphic continuation to $\textbf{s}\in \mathbb{C}^2$. Moreover, $J_{\mathrm{Spec}}^{\mathrm{new}}(f_{\mathfrak{n},\mathfrak{q}},\textbf{s})$ is equal to
is equal to 
\begin{align*}
\sum_{\substack{\pi\in \mathcal{F}(\mathbf{k},\mathfrak{q})}}\frac{\lambda_{\pi}(\mathfrak{n})L(s_1+1/2,\pi)L(s_2+1/2,\widetilde{\pi})}{2\zeta_{\mathfrak{q}}(2)^2L^{(\mathfrak{q})}(1,\pi,\Ad)D_F^{2-s_1-s_2}}\prod_{v\mid\infty}\frac{2^{k_v}\pi\Gamma(k_v/2+s_1)\Gamma(k_v/2+s_2)}{(2\pi)^{s_1+s_2}\Gamma(k_v)}.
\end{align*}
\end{lemma}
\begin{proof}
By the construction of $f_{\mathfrak{n},\mathfrak{q}}$, we have
\begin{equation}\label{2.14}
\pi(f_{\mathfrak{n},\mathfrak{q}})\phi=\lambda_{\pi}(\mathfrak{n})D_F^{-3/2}\cdot \phi
\end{equation}
for $\phi\in \mathcal{B}_{\pi}^{\mathrm{new}}$. Hence, Lemma \ref{lem2.3} follows from \eqref{2.14} and Lemma \ref{lem2.1}. 
\end{proof}

\begin{lemma}\label{lem2.4}
Let notation be as before. Let $\chi_{\mathfrak{q}}$ be the nontrivial unramified quadratic character of $F_{\mathfrak{q}}^{\times}$. We have the following assertions. 
\begin{itemize}
\item $J_{\mathrm{Spec}}^{\mathrm{old}}(f_{\mathfrak{n},\mathfrak{q}},\textbf{s})$ admits a meromorphic continuation to $\textbf{s}\in \mathbb{C}^2$. Moreover,   $J_{\mathrm{Spec}}^{\mathrm{old}}(f_{\mathfrak{n},\mathfrak{q}},\textbf{s})$ can be expressed as  
\begin{equation}\label{eq2.17}
\sum_{\substack{\pi\in \mathcal{F}(\mathbf{k},\mathcal{O}_F)}}\frac{\lambda_{\pi}(\mathfrak{n})C_{\pi_{\mathfrak{q}}}(\textbf{s})L(s_1+1/2,\pi)L(s_2+1/2,\widetilde{\pi})}{2L(1,\pi,\Ad)D_F^{2-s_1-s_2}}\prod_{v\mid\infty}\frac{\Gamma(k_v/2+s_1)\Gamma(k_v/2+s_2)}{(2\pi)^{s_1+s_2}2^{-k_v}\pi^{-1}\Gamma(k_v)},
\end{equation}
where $C_{\pi_{\mathfrak{q}}}(\textbf{s})$ is defined by 
\begin{equation}\label{equ2.18}
1+L_{\mathfrak{q}}(1/2,\pi_{\mathfrak{q}})L_{\mathfrak{q}}(1/2,\pi_{\mathfrak{q}}\times\chi_{\mathfrak{q}})\prod_{j=1}^2(N(\mathfrak{q})^{-s_j}+N(\mathfrak{q})^{-s_j-1}-\lambda_{\pi}(\mathfrak{q})N(\mathfrak{q})^{-1/2}).
\end{equation}
 
 \item In particular, $J_{\mathrm{Spec}}^{\mathrm{old}}(f_{\mathfrak{n},\mathfrak{q}},\textbf{0})$ is equal to  
\begin{equation}\label{eq2.18}
\frac{1+N(\mathfrak{q})^{-1}}{D_F^{2}}\sum_{\substack{\pi\in \mathcal{F}(\mathbf{k},\mathcal{O}_F)}}\frac{\lambda_{\pi}(\mathfrak{n})L_{\mathfrak{q}}(1/2,\pi_{\mathfrak{q}}\times\chi_{\mathfrak{q}})|L(1/2,\pi)|^2}{L(1,\pi,\Ad)}\prod_{v\mid\infty}\frac{2^{k_v}\pi\Gamma(k_v/2)^2}{\Gamma(k_v)}.
\end{equation}
\end{itemize}
\end{lemma} 
\begin{proof}
Let $\pi=\otimes_{v\leq\infty}\pi_v\in \mathcal{F}(\mathbf{k},\mathcal{O}_F)$. Let $\phi^{\circ}$ be the new form in $\pi$ such that $\langle \phi^{\circ},\phi^{\circ}\rangle=1$. Let $\alpha_{\pi_{\mathfrak{q}}}$ and $\beta_{\pi_{\mathfrak{q}}}$ be the coefficients in Lemma \ref{lem2.2}. Then   
\begin{equation}\label{2.17}
\mathfrak{B}_{\pi}^{K_{\mathfrak{q}}[1]}=\mathrm{Span}\{\phi^{\circ}, \alpha_{\pi_{\mathfrak{q}}}\phi^{\circ}+\beta_{\pi_{\mathfrak{q}}}\pi_{\mathfrak{q}}(\diag(1,\varpi_{\mathfrak{q}}))\phi^{\circ}\}.
\end{equation}

Let $\phi\in \mathfrak{B}_{\pi}^{K_{\mathfrak{q}}[1]}$. Parallel to \eqref{2.14} we have $\pi(f_{\mathfrak{n},\mathfrak{q}})\phi=\lambda_{\pi}(\mathfrak{n})D_F^{-3/2}\phi$. Hence, for $\Re(s_1)\gg 1$ and $\Re(s_2)\gg 1$, we obtain 
\begin{equation}\label{2.18}
J_{\mathrm{Spec}}^{\mathrm{old}}(f_{\mathfrak{n},\mathfrak{q}},\textbf{s})=\sum_{\substack{\pi\in \mathcal{F}(\mathbf{k},\mathcal{O}_F),\ \phi\in \mathfrak{B}_{\pi}^{K_{\mathfrak{q}}[1]}}}\frac{\lambda_{\pi}(\mathfrak{n})\cdot \mathcal{P}(s_1,\phi)\mathcal{P}(s_2,\overline{\phi})}{\langle\phi,\phi\rangle}.
\end{equation}

By a change of variable, 
\begin{align*}
\mathcal{P}(s_1,\pi_{\mathfrak{q}}(\diag(1,\varpi_{\mathfrak{q}}))\phi^{\circ})=N(\mathfrak{q})^{-s_1}\mathcal{P}(s_1,\phi^{\circ}).
\end{align*}
Therefore, we obtain, for $\phi=\alpha_{\pi_{\mathfrak{q}}}\phi^{\circ}+\beta_{\pi_{\mathfrak{q}}}\pi_{\mathfrak{q}}(\diag(1,\varpi_{\mathfrak{q}}))\phi^{\circ}$, that 
\begin{equation}\label{2.19}
\mathcal{P}(s_1,\phi)\mathcal{P}(s_2,\overline{\phi})=\mathcal{P}(s_1,\phi^{\circ})\mathcal{P}(s_2,\overline{\phi^{\circ}})\cdot \prod_{j=1}^{2}(\alpha_{\pi_{\mathfrak{q}}}+\beta_{\pi_{\mathfrak{q}}}N(\mathfrak{q})^{-s_j}).
\end{equation}

Substituting \eqref{2.19} into \eqref{2.18}, the function $J_{\mathrm{Spec}}^{\mathrm{old}}(f_{\mathfrak{n},\mathfrak{q}},\textbf{s})$ boils down to 
\begin{equation}\label{2.20}
\sum_{\substack{\pi\in \mathcal{F}(\mathbf{k},\mathcal{O}_F),\ \phi^{\circ}\in \mathfrak{B}_{\pi}^{\mathrm{new}}}}\bigg[1+\prod_{j=1}^{2}(\alpha_{\pi_{\mathfrak{q}}}+\beta_{\pi_{\mathfrak{q}}}N(\mathfrak{q})^{-s_j})\bigg]\cdot \frac{\lambda_{\pi}(\mathfrak{n})\cdot \mathcal{P}(s_1,\phi^{\circ})\mathcal{P}(s_2,\overline{\phi^{\circ}})}{\langle\phi^{\circ},\phi^{\circ}\rangle}.
\end{equation}

By \eqref{eq2.11} in Lemma \ref{lem2.2}, $1+\prod_{j=1}^{2}(\alpha_{\pi_{\mathfrak{q}}}+\beta_{\pi_{\mathfrak{q}}}N(\mathfrak{q})^{-s_j})$ amounts to  
\begin{align*}
&1+\zeta_{\mathfrak{q}}(1)^2\zeta_{\mathfrak{q}}(2)^{-2}(N(\mathfrak{q})^{-s_1}-\gamma)(N(\mathfrak{q})^{-s_2}-\gamma)L_{\mathfrak{q}}(1/2,\pi_{\mathfrak{q}})L_{\mathfrak{q}}(1/2,\pi_{\mathfrak{q}}\times\chi_{\mathfrak{q}})\\
=&1+L_{\mathfrak{q}}(1/2,\pi_{\mathfrak{q}})L_{\mathfrak{q}}(1/2,\pi_{\mathfrak{q}}\times\chi_{\mathfrak{q}})\prod_{j=1}^2(N(\mathfrak{q})^{-s_j}+N(\mathfrak{q})^{-s_j-1}-\lambda_{\pi}(\mathfrak{q})N(\mathfrak{q})^{-1/2}).
\end{align*}
Here $\gamma$ is defined as in \eqref{eq2.14}. 

Therefore, \eqref{eq2.17} follows from Lemma \ref{lem2.1} and \eqref{2.20}. Taking $\textbf{s}=(0,0)$ in \eqref{eq2.14}, along with \eqref{2.11} in Lemma \ref{lem2.2}, we then obtain \eqref{eq2.18}. 
\end{proof}

\subsection{The Spectral Side $J_{\mathrm{Spec}}(f_{\mathfrak{n},\mathfrak{q}},\textbf{s})$}\label{sec2.3}
Substituting Lemmas \ref{lem2.3} and \ref{lem2.4} into \eqref{f2.15}, we obtain the following results. 
\begin{thm}\label{thm3.5}
Let notation be as before. Let $\mathfrak{q}=\mathcal{O}_F$. Then $J_{\mathrm{Spec}}(f_{\mathfrak{n},\mathfrak{q}},\textbf{s})$ admits a meromorphic continuation to $\textbf{s}\in \mathbb{C}^2$. Moreover, $J_{\mathrm{Spec}}(f_{\mathfrak{n},\mathfrak{q}},\textbf{s})$ can be expressed explicitly as 
\begin{align*}
\prod_{v\mid\infty}\frac{\Gamma(k_v/2+s_1)\Gamma(k_v/2+s_2)}{2^{-k_v}\pi^{-1}(2\pi)^{s_1+s_2}\Gamma(k_v)}\sum_{\substack{\pi\in \mathcal{F}(\mathbf{k},\mathcal{O}_F)}}\frac{\lambda_{\pi}(\mathfrak{n})L(s_1+1/2,\pi)L(s_2+1/2,\widetilde{\pi})}{2L(1,\pi,\Ad)D_F^{2-s_1-s_2}}.
\end{align*}	
\end{thm}

\begin{cor}\label{cor2.6}
Let notation be as before. Let $\mathfrak{q}=\mathcal{O}_F$. Then 
\begin{align*}
J_{\mathrm{Spec}}(f_{\mathfrak{n},\mathfrak{q}},\textbf{0})=\frac{1}{2D_F^{2}}\prod_{v\mid\infty}\frac{\pi\cdot 2^{k_v} \Gamma(k_v/2)^2}{\Gamma(k_v)}\sum_{\substack{\pi\in \mathcal{F}(\mathbf{k},\mathcal{O}_F)}}\frac{\lambda_{\pi}(\mathfrak{n})|L(1/2,\pi)|^2}{L(1,\pi,\Ad)}.
\end{align*}	
\end{cor}

\begin{thm}\label{spec}
Let notation be as before. Let $\mathfrak{q}\subsetneq \mathcal{O}_F$ be a prime ideal. Then $J_{\mathrm{Spec}}(f_{\mathfrak{n},\mathfrak{q}},\textbf{s})$ admits a meromorphic continuation to $\textbf{s}\in \mathbb{C}^2$. Moreover, 
\begin{align*}
J_{\mathrm{Spec}}(f_{\mathfrak{n},\mathfrak{q}},\textbf{s})=J_{\mathrm{Spec}}^{\mathrm{new}}(f_{\mathfrak{n},\mathfrak{q}},\textbf{s})+J_{\mathrm{Spec}}^{\mathrm{old}}(f_{\mathfrak{n},\mathfrak{q}},\textbf{s}),\tag{\ref{f2.15}}
\end{align*}
where $J_{\mathrm{Spec}}^{\mathrm{new}}(f_{\mathfrak{n},\mathfrak{q}},\textbf{s})$ is defined by  
\begin{align*}
\prod_{v\mid\infty}\frac{2^{k_v}\pi\Gamma(k_v/2+s_1)\Gamma(k_v/2+s_2)}{(2\pi)^{s_1+s_2}\Gamma(k_v)}\sum_{\substack{\pi\in \mathcal{F}(\mathbf{k},\mathfrak{q})}}\frac{\lambda_{\pi}(\mathfrak{n})L(s_1+1/2,\pi)L(s_2+1/2,\widetilde{\pi})}{2\zeta_{\mathfrak{q}}(2)^2L^{(\mathfrak{q})}(1,\pi,\Ad)D_F^{2-s_1-s_2}},
\end{align*}
and $J_{\mathrm{Spec}}^{\mathrm{old}}(f_{\mathfrak{n},\mathfrak{q}},\textbf{s})$ is defined by 
\begin{align*}
\prod_{v\mid\infty}\frac{\pi \Gamma(k_v/2+s_1)\Gamma(k_v/2+s_2)}{(2\pi)^{s_1+s_2}2^{-k_v}\Gamma(k_v)}\sum_{\substack{\pi\in \mathcal{F}(\mathbf{k},\mathcal{O}_F)}}\frac{\lambda_{\pi}(\mathfrak{n})C_{\pi_{\mathfrak{q}}}(\textbf{s})L(s_1+1/2,\pi)L(s_2+1/2,\widetilde{\pi})}{2L(1,\pi,\Ad)D_F^{2-s_1-s_2}}.
\end{align*}
Here $C_{\pi_{\mathfrak{q}}}(\textbf{s})$ is defined as in \eqref{equ2.18} in Lemma \ref{lem2.4}.
\end{thm}

\begin{cor}\label{cor2.7}
Let notation be as before. Let $\mathfrak{q}\subsetneq \mathcal{O}_F$ be a prime ideal. Then
\begin{equation}\label{2.23}
J_{\mathrm{Spec}}(f_{\mathfrak{n},\mathfrak{q}},\textbf{0})=J_{\mathrm{Spec}}^{\mathrm{new}}(f_{\mathfrak{n},\mathfrak{q}},\textbf{0})+J_{\mathrm{Spec}}^{\mathrm{old}}(f_{\mathfrak{n},\mathfrak{q}},\textbf{0}),
\end{equation}
where  
\begin{align*}
J_{\mathrm{Spec}}^{\mathrm{new}}(f_{\mathfrak{n},\mathfrak{q}},\textbf{0})=\sum_{\substack{\pi\in \mathcal{F}(\mathbf{k},\mathfrak{q})}}\frac{\lambda_{\pi}(\mathfrak{n})|L(1/2,\pi)|^2}{2\zeta_{\mathfrak{q}}(2)^2L^{(\mathfrak{q})}(1,\pi,\Ad)D_F^{2}}\prod_{v\mid\infty}\frac{2^{k_v}\pi \Gamma(k_v/2)^2}{\Gamma(k_v)},
\end{align*}
and $J_{\mathrm{Spec}}^{\mathrm{old}}(f_{\mathfrak{n},\mathfrak{q}},\textbf{0})$ is defined by 
\begin{align*}
\frac{1+N(\mathfrak{q})^{-1}}{D_F^{2}}\sum_{\substack{\pi\in \mathcal{F}(\mathbf{k},\mathcal{O}_F)}}\frac{\lambda_{\pi}(\mathfrak{n})L_{\mathfrak{q}}(1/2,\pi_{\mathfrak{q}}\times\chi_{\mathfrak{q}})|L(1/2,\pi)|^2}{L(1,\pi,\Ad)}\prod_{v\mid\infty}\frac{2^{k_v}\pi \Gamma(k_v/2)^2}{\Gamma(k_v)}.
\end{align*}
\end{cor}

\section{The Small Cell Orbital Integrals}\label{section4}
Let $\delta\in \{I_2,w\}$, where $w=\begin{pmatrix}
	& -1\\
	1
\end{pmatrix}$. Let $\textbf{s}=(s_1,s_2)$. Recall the definition:
\begin{align*}
J_{\mathrm{small}}^{I_2}(f_{\mathfrak{n},\mathfrak{q}},\textbf{s}):=&\int_{\mathbb{A}_F^{\times}}\int_{\mathbb{A}_F^{\times}}\int_{\mathbb{A}_F}f_{\mathfrak{n},\mathfrak{q}}\left(\begin{pmatrix}
y& u\\
&1
\end{pmatrix}\delta\right)\psi(xu)du|x|^{1+s_1+s_2}|y|^{s_2}d^{\times}yd^{\times}x,
\end{align*}
and 
\begin{equation}\label{equa4.1}
J_{\mathrm{small}}^{w}(f_{\mathfrak{n},\mathfrak{q}},\textbf{s})=J_{\mathrm{small}}^{I_2}(R(w)f_{\mathfrak{n},\mathfrak{q}},\textbf{s}'),	
\end{equation}
with $\textbf{s}'=(s_1,-s_2).$ Here $R(w)f_{\mathfrak{n},\mathfrak{q}}(g):=f_{\mathfrak{n},\mathfrak{q}}(gw)$, $g\in G(\mathbb{A}_F)$.

We will show that this integral converges absolutely in  $\Re(s_1)-|\Re(s_2|)>0$, and admits a meromorphic in $\textbf{s}\in \mathbb{C}^2$. 

\subsection{Calculation of Local Integrals}
For each place $v$, we define the local integral 
\begin{align*}
J_{\mathrm{small},v}^{\delta}(\textbf{s}):=\int_{F_v^{\times}}\int_{F_v^{\times}}\Phi_v^{\delta}(x_v,y_v)|x_v|_v^{1+s_1+s_2}|y_v|_v^{s_2}d^{\times}y_vd^{\times}x_v,
\end{align*}
where 
\begin{align*}
\Phi_v^{\delta}(x_v,y_v):=\int_{F_v}f_v\left(\begin{pmatrix}
y_v& u_v\\
&1
\end{pmatrix}\delta\right)\psi_v(x_vu_v)du_v.
\end{align*}

\begin{lemma}\label{lemma0.3}
Let $J_{\mathrm{small},\fin}^{\delta}(\textbf{s}):=\prod_{v<\infty}J_{\mathrm{small},v}^{\delta}(\textbf{s})$.
\begin{itemize}
\item Let $\Re(s_1+s_2)>0$. Then 
\begin{equation}\label{f3.1}
\prod_{v<\infty}\int_{F_v^{\times}}\int_{F_v^{\times}}\big|\Phi_v^{\delta}(x_v,y_v)\big||x_v|_v^{1+\Re(s_1+s_2)}|y_v|_v^{\Re(s_2)}d^{\times}y_vd^{\times}x_v<\infty.
\end{equation}
\item The function $J_{\mathrm{small},\fin}^{\delta}(\textbf{s})$ admits a meromorphic continuation to $\textbf{s}\in \mathbb{C}^2$, given explicitly by 
\begin{equation}\label{0.5}
J_{\mathrm{small},\fin}^{\delta}(\textbf{s})=\frac{(\textbf{1}_{\mathfrak{q}=\mathcal{O}_F}+\textbf{1}_{\mathfrak{q}\subsetneq \mathcal{O}_F, 
\delta=I_2})\cdot V_{\mathfrak{q}}\cdot \zeta_F(1+s_1+s_2)}{D_F^{1/2-s_1-s_2}N(\mathfrak{n})^{1/2+s_2}}\prod_{v\mid\mathfrak{n}}\sum_{l=0}^{e_v(\mathfrak{n})}q_v^{l(s_2-s_1)}.
\end{equation}
\end{itemize}
\end{lemma}
\begin{proof}

Let $v<\infty$ be a finite place. We consider the following scenarios.   
\begin{itemize}
\item Suppose $v\nmid \mathfrak{q}$. Let $r=\max\{e_v(\mathfrak{n}),0\}$, and $r_2=e_v(y_v)$. By the definition of $f_v$ in and \eqref{t1.6} and \eqref{t1.8}, $f_v\left(\begin{pmatrix}
y_v& u_v\\
&1
\end{pmatrix}\delta\right)$ is a multiple of the characteristic function for the restrictions    
\begin{equation}\label{3.2}
\varpi_v^l\begin{pmatrix}
\varpi_v^{r_2}& u_v\\
&1
\end{pmatrix}\in \bigsqcup_{\substack{i+j=r\\
i\geq j\geq 0}}K_v\begin{pmatrix}
\varpi_v^i\\
&\varpi_v^j
\end{pmatrix}K_v
\end{equation}
for some $l\in\mathbb{Z}$. Notice that \eqref{3.2} reduces to the following
\begin{align*}
\begin{cases}
2l+r_2=r,\ \ l\geq 0\\
l+r_2\geq 0,\ \ l+e_v(u_v)\geq 0
\end{cases}\ \ \Leftrightarrow\ \ \
\begin{cases}
0\leq l\leq r,\ \ r_2=r-2l\\
e_v(u_v)\geq -l.
\end{cases} 
\end{align*}

Substituting the above restrictions into the definition of $J_{\mathrm{small},v}^{\delta}(\textbf{s})$ yields
\begin{align*}
J_{\mathrm{small},v}^{\delta}(\textbf{s})=&q_v^{-\frac{r}{2}}\Vol(\mathcal{O}_v^{\times})\int_{F_v^{\times}}\sum_{l=0}^rq_v^{s_2(2l-r)}\int_{\mathfrak{p}_v^{-l}}\psi_v(x_vu_v)du_v|x_v|_v^{1+s_1+s_2}d^{\times}x_v.
\end{align*}

Swapping the integrals, along with orthogonality of characters, we otbain 
\begin{align*}
J_{\mathrm{small},v}^{\delta}(\textbf{s})=&q_v^{-\frac{r}{2}}\Vol(\mathcal{O}_v^{\times})\sum_{l=0}^rq_v^{s_2(2l-r)}\sum_{j\geq l-d_v}q_v^{-j(1+s_1+s_2)}\int_{\mathfrak{p}_v^{-l}}\int_{\varpi_v^j\mathcal{O}_v^{\times}}\psi_v(x_vu_v)d^{\times}x_vdu_v.
\end{align*}

As a consequence, we derive that  
\begin{equation}\label{3.3}
J_{\mathrm{small},v}^{\delta}(\textbf{s})=\Vol(\mathcal{O}_v)\Vol(\mathcal{O}_v^{\times})^2q_v^{-(s_2+1/2)r}\sum_{l=0}^rq_v^{(s_2-s_1)l}q_v^{d_v(1+s_1+s_2)}\zeta_v(1+s_1+s_2).
\end{equation}

Likewise, the integral 
\begin{align*}
\int_{F_v^{\times}}\int_{F_v^{\times}}\big|\Phi_v^{\delta}(x_v,y_v)\big||x_v|_v^{1+\Re(s_1+s_2)}|y_v|_v^{\Re(s_2)}d^{\times}y_vd^{\times}x_v
\end{align*}
is equal to 
\begin{align*}
q_v^{-\frac{r}{2}}\Vol(\mathcal{O}_v^{\times})\sum_{l=0}^rq_v^{\Re(s_2)(2l-r)}\sum_{j\geq l-d_v}q_v^{-j(1+\Re(s_1+s_2))}\int_{\mathfrak{p}_v^{-l}}\Big|\int_{\varpi_v^j\mathcal{O}_v^{\times}}\psi_v(x_vu_v)du_v\Big|d^{\times}x_v,
\end{align*}
which further simplifies to 
\begin{equation}\label{f3.4}
\frac{\Vol(\mathcal{O}_v)\Vol(\mathcal{O}_v^{\times})^2}{q_v^{(\Re(s_2)+1/2)r}}\sum_{l=0}^rq_v^{(\Re(s_2-s_1))l}q_v^{d_v(1+\Re(s_1+s_2))}\zeta_v(1+\Re(s_1+s_2)).
\end{equation}

\item Suppose $\mathfrak{q}\subsetneq \mathcal{O}_F$ and $v=\mathfrak{q}$. By the definition of $f_v$ in \eqref{t1.7}, $f_v\left(\begin{pmatrix}
y_v& u_v\\
&1
\end{pmatrix}\delta\right)$ is a multiple of the characteristic function for the restrictions    
\begin{equation}\label{3.4}
\varpi_v^l\begin{pmatrix}
\varpi_v^{r_2}& u_v\\
&1
\end{pmatrix}\delta\in K_0(\mathfrak{p}_v)
\end{equation}
for some $l\in \mathbb{Z}$. Notice that \eqref{3.4} is empty if $\delta=w=\begin{pmatrix}
	& -1\\
	1
\end{pmatrix}$. Hence, we assume $\delta=I_2$ for the moment. In this case, \eqref{3.4} amounts to 
\begin{align*}
\begin{cases}
2l+r_2=r,\ \ l= 0\\
l+r_2= 0,\ \ l+e_v(u_v)\geq 0
\end{cases}\ \ \Leftrightarrow\ \ \
\begin{cases}
l=r_2=r=0\\
e_v(u_v)\geq 0.
\end{cases} 
\end{align*}

Substituting these constraints into the definition of $J_{\mathrm{small},v}^{\delta}(\textbf{s})$ leads to 
\begin{align*}
J_{\mathrm{small},v}^{\delta}(\textbf{s})=\frac{\textbf{1}_{\delta=I_2}}{\Vol(K_0(\mathfrak{p}_v))}\int_{F_v^{\times}}\int_{\mathcal{O}_v}\psi_v(x_vu_v)du_v|x_v|_v^{1+s_1+s_2}d^{\times}x_v=c_v\sum_{j\geq 0}q_v^{-j(1+s_1+s_2)},
\end{align*}
where $c_v:=\textbf{1}_{\delta=I_2}\cdot \Vol(\mathcal{O}_v,dx_v)\Vol(K_0(\mathfrak{p}_v))^{-1}$. Thus,
\begin{equation}\label{3.5}
J_{\mathrm{small},v}^{\delta}(\textbf{s})=\frac{\textbf{1}_{\delta=I_2}\cdot \Vol(\mathcal{O}_v,dx_v)}{\Vol(K_0(\mathfrak{p}_v))}\cdot \zeta_v(1+s_1+s_2).
\end{equation}

Likewise, the integral 
\begin{align*}
\int_{F_v^{\times}}\int_{F_v^{\times}}\big|\Phi_v^{\delta}(x_v,y_v)\big||x_v|_v^{1+\Re(s_1+s_2)}|y_v|_v^{\Re(s_2)}d^{\times}y_vd^{\times}x_v
\end{align*}
is equal to 
\begin{equation}\label{f3.7}
\textbf{1}_{\delta=I_2}\cdot \Vol(\mathcal{O}_v)\Vol(\mathcal{O}_v^{\times})^2\Vol(K_0(\mathfrak{p}_v))^{-1}\cdot \zeta_v(1+\Re(s_1+s_2)).
\end{equation}

\end{itemize}

Therefore, the estimate \eqref{f3.1} follows from \eqref{f3.4} and \eqref{f3.7}; and the equality \eqref{0.5} follows from \eqref{3.4} and \eqref{3.5}. 
\end{proof}

\begin{lemma}\label{lem0.4}
Let $x\in \mathbb{R}$ and $y>0$. Let $k$ be a positive even integer. Then 
\begin{equation}\label{f0.6}
\int_{\mathbb{R}}\frac{e^{2\pi i xu}}{(-u+i(y+1))^{k}}du=\frac{(2\pi i)^{k}}{\Gamma(k)}x^{k-1}e^{-2\pi x(y+1)}\cdot \textbf{1}_{x\geq 0}. 
\end{equation}
\end{lemma}
\begin{proof}
By the change of variable $u\mapsto -u$, we have
\begin{align*}
I(x):=\int_{\mathbb{R}}\frac{e^{2\pi i xu}}{(-u+i(y+1))^{k}}du=\int_{\mathbb{R}}\frac{e^{-2\pi i xu}}{(u+i(y+1))^{k}}du.
\end{align*}

If $x<0$, then $I(x)\equiv 0$ by shifting contour along the $y$-axis direction. Suppose $x\geq 0$. Shifting contour and by Cauchy integral, we obtain 
\begin{align*}
I(x)=\int_{\mathbb{R}}\frac{e^{-2\pi i xu}}{(u+i(y+1))^{k}}du=-\frac{2\pi i}{(k-1)!}\cdot \frac{\partial^{k-1}(e^{-2\pi ixu})}{\partial u^{k-1}}\bigg|_{u=-i(y+1)},
\end{align*}
from which \eqref{f0.6} holds. 
\end{proof}

\begin{lemma}\label{lem3.3}
Let $v\mid\infty$. Let $\delta\in \{I_2,w\}$. 
\begin{itemize}
\item The function  $J_{\mathrm{small},v}^{\delta}(s)$ converges in $\Re(s_1)>-k_v/2$ and $\Re(s_2)>-k_v/2$. 
\item For $\Re(s_1)>-k_v/2$ and $\Re(s_2)>-k_v/2$, we have 
\begin{equation}\label{0.6}
J_{\mathrm{small},v}^{\delta}(\textbf{s})=\frac{2^{k_v}(k_v-1)(\textbf{1}_{\delta=I_2}+i^{k_v}\textbf{1}_{\delta=w})}{4\pi\cdot (2\pi)^{s_1+s_2}}\cdot \frac{\Gamma(k_v/2+s_1)\Gamma(k_v/2+s_2)}{\Gamma(k_v)}.
\end{equation}
\end{itemize}
\end{lemma}
\begin{proof}
By definition of $f_v$ in \eqref{t1.5}, the function $J_{\mathrm{small},v}^{I_2}(\textbf{s})$ boils down to  
\begin{align*}
\frac{(2i)^{k_v}(k_v-1)}{4\pi}\int_{F_v^{\times}}\int_0^{\infty}\int_{\mathbb{R}}\frac{y_v^{\frac{k_v}{2}}\psi_v(x_vu_v)}{(-u_v+i(y_v+1))^{k_v}}du_v|x_v|_v^{1+s_1+s_2}|y_v|_v^{s_2}d^{\times}y_vd^{\times}x_v.
\end{align*}
Moreover, taking advantage of the definition of $f_v$ as the matrix coefficient of lowest weight vectors, we obtain $f_v\left(\begin{pmatrix}
y_v& u_v\\
&1
\end{pmatrix}w\right)=i^{k_v}f_v\left(\begin{pmatrix}
y_v& u_v\\
&1
\end{pmatrix}\right)$. Hence, 
\begin{equation}\label{3.8}
J_{\mathrm{small},v}^{w}(\textbf{s})=i^{k_v}J_{\mathrm{small},v}^{I_2}(\textbf{s}).
\end{equation}

By Lemma \ref{lem0.4}, along with the triangle inequality, we have
\begin{equation}\label{3.9}
|J_{\mathrm{small},v}^{I_2}(\textbf{s})|\ll \int_0^{\infty}\int_0^{\infty}y_v^{\frac{k_v}{2}}x^{k-1}e^{-2\pi x(y+1)}|x_v|_v^{\Re(s_1+s_2)}|y_v|_v^{\Re(s_2)}d^{\times}y_vdx_v,
\end{equation}
which converges absolutely in $\Re(s_1)>-k_v/2$ and $\Re(s_2)>-k_v/2$. Here the implied constant in \eqref{3.9} depends only on $k_v$. 

Henceforth we may assume $\Re(s_1)>-k_v/2$ and $\Re(s_2)>-k_v/2$. By Lemma \ref{lem0.4} we obtain 
\begin{align*}
J_{\mathrm{small},v}^{I_2}(\textbf{s})=\frac{(2i)^{k_v}(k_v-1)}{4\pi}\frac{(2\pi i)^{k_v}}{\Gamma(k_v)}\int_0^{\infty}\int_0^{\infty}y^{\frac{k_v}{2}}x^{k_v-1}e^{-2\pi x(y+1)}\cdot x^{s_1+s_2}y^{s_2-1}dydx.
\end{align*}

Changing the variable $y\mapsto yx^{-1}$ yields
\begin{align*}
J_{\mathrm{small},v}^{I_2}(\textbf{s})=\frac{(2i)^{k_v}(k_v-1)}{4\pi}\frac{(2\pi i)^{k_v}}{\Gamma(k_v)}\int_0^{\infty}\int_0^{\infty}y^{\frac{k_v}{2}+s_2-1}x^{\frac{k_v}{2}+s_1-1}e^{-2\pi (y+x)}dydx.
\end{align*}

Therefore, \eqref{0.6} follows from \eqref{3.8} and the above calculation of $J_{\mathrm{small},v}^{I_2}(\textbf{s})$. 
\end{proof}

\subsection{Meromorphic Continuation of $J_{\mathrm{small}}^{\delta}(f_{\mathfrak{n},\mathfrak{q}},\textbf{s})$}
Utilizing Lemmas \ref{lemma0.3} and \ref{lem3.3} with $\delta=I_2$, along with the relation \eqref{equa4.1}, we obtain the following corollary. 
\begin{prop}\label{cor3.4} 
Let notation be as before. 
\begin{itemize}
\item The function  $J_{\mathrm{small}}^{I_2}(f_{\mathfrak{n},\mathfrak{q}},\textbf{s})$ converges absolutely in the region 
\begin{align*}
\begin{cases}
\Re(s_1+s_2)>0\\
\min\{\Re(s_1), \Re(s_2)\}>-\min_{v\mid\infty}\{k_v/2\}.
\end{cases}
\end{align*}
Moreover, it admits a meromorphic continuation to $\mathbb{C}^2$ given explicitly by  
\begin{align*}
J_{\mathrm{small}}^{I_2}(f_{\mathfrak{n},\mathfrak{q}},\textbf{s})=&\prod_{v\mid\infty}\frac{2^{k_v}(k_v-1)}{4\pi\cdot (2\pi)^{s_1+s_2}}\cdot \frac{\Gamma(k_v/2+s_1)\Gamma(k_v/2+s_2)}{\Gamma(k_v)}\\
&\cdot \frac{V_{\mathfrak{q}}\cdot \zeta_F(1+s_1+s_2)}{D_F^{1/2-s_1-s_2}N(\mathfrak{n})^{1/2+s_2}}\prod_{v\mid\mathfrak{n}}\sum_{l=0}^{e_v(\mathfrak{n})}q_v^{l(s_2-s_1)}.
\end{align*}
\item The function  $J_{\mathrm{small}}^{w}(f_{\mathfrak{n},\mathfrak{q}},\textbf{s})$ converges absolutely in the region 
\begin{align*}
\begin{cases}
\Re(s_1-s_2)>0\\
\Re(s_1)>-\min_{v\mid\infty}\{k_v/2\},\ \ \Re(s_2)<\min_{v\mid\infty}\{k_v/2\}.
\end{cases}
\end{align*}
Moreover, it admits a meromorphic continuation to $\mathbb{C}^2$ given explicitly by  
\begin{align*}
J_{\mathrm{small}}^{w}(f_{\mathfrak{n},\mathfrak{q}},\textbf{s})=&\textbf{1}_{\mathfrak{q}=\mathcal{O}_F}\cdot\prod_{v\mid\infty}\frac{(2i)^{k_v}(k_v-1)}{4\pi\cdot (2\pi)^{s_1-s_2}}\cdot \frac{\Gamma(k_v/2+s_1)\Gamma(k_v/2-s_2)}{\Gamma(k_v)}\\
&\cdot \frac{ V_{\mathfrak{q}}\cdot \zeta_F(1+s_1-s_2)}{D_F^{1/2-s_1+s_2}N(\mathfrak{n})^{1/2-s_2}}\prod_{v\mid\mathfrak{n}}\sum_{l=0}^{e_v(\mathfrak{n})}q_v^{l(-s_2-s_1)}.
\end{align*}
\end{itemize}
\end{prop}

In particular, $J_{\mathrm{small}}^{w}(f_{\mathfrak{n},\mathfrak{q}},\textbf{s})\equiv 0$ if $\mathfrak{q}\subsetneq \mathcal{O}_F$; and for $\textbf{s}=(s_1,0)$, we have  $J_{\mathrm{small}}^{I_2}(f_{\mathfrak{n},\mathfrak{q}},\textbf{s})+J_{\mathrm{small}}^{w}(f_{\mathfrak{n},\mathfrak{q}},\textbf{s})\equiv 0$ if $\sum_{v\mid\infty}k_v\equiv 2\pmod{4}$ and $\mathfrak{q}=\mathcal{O}_F$.

\section{The Dual Orbital Integrals}\label{section5}
Let $\delta\in \{I_2,w\}$. Let $\textbf{s}=(s_1,s_2)\in \mathbb{C}^2$. Recall the dual orbital integral
\begin{align*}
J_{\mathrm{dual}}^{I_2}(f_{\mathfrak{n},\mathfrak{q}},\textbf{s}):=\int_{\mathbb{A}_F^{\times}}\int_{\mathbb{A}_F^{\times}}f_{\mathfrak{n},\mathfrak{q}}\left(\begin{pmatrix}
1& \\
x&1
\end{pmatrix}\begin{pmatrix}
y& \\
&1
\end{pmatrix}\delta\right)|x|^{s_1+s_2}|y|^{s_2}d^{\times}yd^{\times}x,
\end{align*}
and 
\begin{equation}\label{equa5.1}
J_{\mathrm{dual}}^{w}(f_{\mathfrak{n},\mathfrak{q}},\textbf{s})=J_{\mathrm{dual}}^{I_2}(R(w)f_{\mathfrak{n},\mathfrak{q}},\textbf{s}'),	
\end{equation}
with $\textbf{s}'=(s_1,-s_2).$ Here $R(w)f_{\mathfrak{n},\mathfrak{q}}(g):=f_{\mathfrak{n},\mathfrak{q}}(gw)$, $g\in G(\mathbb{A}_F)$.

We will show that this integral converges absolutely in  $\Re(s_1)-|\Re(s_2|)>1$, and admits a meromorphic in $\textbf{s}\in \mathbb{C}^2$.

\subsection{Calculation of Local Integrals}
Let $v$ be a place. Define 
\begin{align*}
J_{\mathrm{dual},v}^{\delta}(f_{\mathfrak{n},\mathfrak{q}},\textbf{s}):=\int_{F_v^{\times}}\int_{F_v^{\times}}f_v\left(\begin{pmatrix}
1& \\
x_v&1
\end{pmatrix}\begin{pmatrix}
y_v& \\
&1
\end{pmatrix}\delta\right)|x_v|_v^{s_1+s_2}|y_v|_v^{s_2}d^{\times}y_vd^{\times}x_v.
\end{align*}
\begin{lemma}\label{lemma4.1}
Let $J_{\mathrm{dual},\fin}^{\delta}(f_{\mathfrak{n},\mathfrak{q}},\textbf{s}):=\prod_{v<\infty}J_{\mathrm{dual},v}^{\delta}(f_{\mathfrak{n},\mathfrak{q}},\textbf{s})$. 
\begin{itemize}
\item Let $\Re(s_1+s_2)>1$. Then 
\begin{equation}\label{eq4.1}
\prod_{v<\infty}\int_{F_v^{\times}}\int_{F_v^{\times}}\bigg|f_{v}\left(\begin{pmatrix}
1& \\
x_v&1
\end{pmatrix}\begin{pmatrix}
y_v& \\
&1
\end{pmatrix}\delta\right)|x_v|_v^{s_1+s_2}|y_v|_v^{s_2}\bigg|d^{\times}y_vd^{\times}x_v<\infty.
\end{equation}
\item The function $J_{\mathrm{dual},\fin}^{\delta}(\textbf{s})$ admits a meromorphic continuation to $\textbf{s}\in \mathbb{C}^2$, given explicitly by 
\begin{equation}\label{f4.1}
J_{\mathrm{dual},\fin}^{\delta}(f_{\mathfrak{n},\mathfrak{q}},\textbf{s})=\frac{(\textbf{1}_{\mathfrak{q}=\mathcal{O}_F}+\textbf{1}_{\mathfrak{q}\subsetneq \mathcal{O}_F, 
\delta=I_2})\cdot V_{\mathfrak{q}}\cdot \zeta_F(s_1+s_2)}{N(\mathfrak{n})^{1/2-s_1}N(\mathfrak{q})^{s_1+s_2}D_F}\prod_{v\mid\mathfrak{n}}\sum_{l=0}^{e_v(\mathfrak{n})}q_v^{l(s_2-s_1)}.
\end{equation}
\end{itemize}
\end{lemma}
\begin{proof}
Let $v<\infty$ be a finite place. We consider the following scenarios.
\begin{itemize}
\item Suppose $v\nmid\mathfrak{q}$. Let $r=\max\{e_v(\mathfrak{n}),0\}$, $r_1=e_v(x_v)$, and $r_2=e_v(y_v)$. By the definition of $f_v$ in and \eqref{t1.6} and \eqref{t1.8}, $f_{v}\left(\begin{pmatrix}
1& \\
x_v&1
\end{pmatrix}\begin{pmatrix}
y_v& \\
&1
\end{pmatrix}\delta\right)$ is a multiple of the characteristic function for the restrictions    
\begin{equation}\label{4.1}
\varpi_v^l\begin{pmatrix}
\varpi_v^{r_2}& \\
\varpi_v^{r_1+r_2} &1
\end{pmatrix}\in \bigsqcup_{\substack{i+j=r\\
i\geq j\geq 0}}K_v\begin{pmatrix}
\varpi_v^i\\
&\varpi_v^j
\end{pmatrix}K_v
\end{equation}
for some $l\in\mathbb{Z}$. Notice that \eqref{4.1} reduces to the following
\begin{equation}\label{4.2}
\begin{cases}
2l+r_2=r,\ \ l\geq 0\\
l+r_2\geq 0,\ \ l+r_1+r_2\geq 0
\end{cases}\ \ \Leftrightarrow\ \ \
\begin{cases}
0\leq l\leq r,\ \ r_2=r-2l\\
r_1\geq l-r.
\end{cases} 
\end{equation}

Consequently, it follows from \eqref{4.2} that 
\begin{equation}\label{4.3}
\int_{F_v^{\times}}f_{v}\left(\begin{pmatrix}
1& \\
x_v&1
\end{pmatrix}\begin{pmatrix}
y_v& \\
&1
\end{pmatrix}\delta\right)|y_v|_v^{s_2}d^{\times}y_v=q_v^{-\frac{r}{2}}\Vol(\mathcal{O}_v^{\times})\sum_{l=0}^rq_v^{(2l-r)s_2}\cdot \textbf{1}_{r_1\geq l-r}.
\end{equation}

Substituting \eqref{4.3} into the definition of $J_{\mathrm{dual},v}^{\delta}(f_{\mathfrak{n},\mathfrak{q}},\textbf{s})$, we obtain 
\begin{equation}\label{4.4}
J_{\mathrm{dual},v}^{\delta}(f_{\mathfrak{n},\mathfrak{q}},\textbf{s})=q_v^{(s_1-1/2)r}\Vol(\mathcal{O}_v^{\times})^2\sum_{l=0}^rq_v^{(s_2-s_1)l}\zeta_v(s_1+s_2).
\end{equation}

\item Suppose $\mathfrak{q}\subsetneq \mathcal{O}_F$ and $v=\mathfrak{q}$. By the definition of $f_v$ in \eqref{t1.7}, the function $f_{v}\left(\begin{pmatrix}
1& \\
x_v&1
\end{pmatrix}\begin{pmatrix}
y_v& \\
&1
\end{pmatrix}\delta\right)$ is a multiple of the characteristic function for the restrictions    
\begin{equation}\label{4.5}
\varpi_v^l\begin{pmatrix}
\varpi_v^{r_2}& \\
\varpi_v^{r_1+r_2} &1
\end{pmatrix}\delta\in K_0(\mathfrak{p}_v)
\end{equation}
for some $l\in \mathbb{Z}$. Notice that \eqref{4.5} is empty if $\delta=w=\begin{pmatrix}
	& -1\\
	1
\end{pmatrix}$. Hence, we assume $\delta=I_2$ for the moment. In this case, \eqref{4.5} amounts to 
\begin{equation}\label{4.6}
\begin{cases}
2l+r_2=r,\ \ l= 0\\
l+r_2= 0,\ \ l+r_1+r_2\geq 1
\end{cases}\ \ \Leftrightarrow\ \ \
\begin{cases}
l=r_2=r=0\\
r_1\geq 1.
\end{cases} 
\end{equation}

Substituting \eqref{4.6} into the definition of $J_{\mathrm{dual},v}^{\delta}(f_{\mathfrak{n},\mathfrak{q}},\textbf{s})$, we obtain 
\begin{equation}\label{4.7}
J_{\mathrm{dual},v}^{\delta}(f_{\mathfrak{n},\mathfrak{q}},\textbf{s})=\frac{\Vol(\mathcal{O}_v^{\times})^2\textbf{1}_{\delta=I_2}}{\Vol(K_0(\mathfrak{p}_v))}\cdot q_v^{-s_1-s_2}\zeta_v(s_1+s_2).
\end{equation}
\end{itemize}

Therefore, \eqref{f4.1} follows from \eqref{4.4} and \eqref{4.7}. Let $\Re(s_1+s_2)>1$. Replacing $(s_1,s_2)$ with $(\Re(s_1), \Re(s_2))$ into \eqref{f4.1}, the integral 
\begin{align*}
\prod_{v<\infty}\int_{F_v^{\times}}\int_{F_v^{\times}}\bigg|f_{v}\left(\begin{pmatrix}
1& \\
x_v&1
\end{pmatrix}\begin{pmatrix}
y_v& \\
&1
\end{pmatrix}\delta\right)|x_v|_v^{s_1+s_2}|y_v|_v^{s_2}\bigg|d^{\times}y_vd^{\times}x_v
\end{align*} 
is equal to 
\begin{align*}
\frac{(\textbf{1}_{\mathfrak{q}=\mathcal{O}_F}+\textbf{1}_{\mathfrak{q}\subsetneq \mathcal{O}_F, 
\delta=I_2})\cdot V_{\mathfrak{q}}\cdot \zeta_F(\Re(s_1+s_2))}{N(\mathfrak{n})^{1/2-\Re(s_1)}N(\mathfrak{q})^{\Re(s_1+s_2)}D_F}\prod_{v\mid\mathfrak{n}}\sum_{l=0}^{e_v(\mathfrak{n})}q_v^{l(\Re(s_2-s_1))}<\infty.
\end{align*}
Hence, \eqref{eq4.1} follows. 
\end{proof}

\subsubsection{Archimedean Integrals}

\begin{lemma}\label{lemma4.2}
Let $v\mid\infty$. We have the following assertions.
\begin{itemize}
\item The integral $J_{\mathrm{dual},v}(\textbf{s})$ converges absolutely in 
\begin{equation}\label{eq4.11}
\begin{cases}
\Re(s_1+s_2)>1\\
\Re(s_1)<k_v/2-1,\ \Re(s_2)<k_v/2.
\end{cases}
\end{equation}
\item Let $\textbf{s}=(s_1,s_2)$ be in the region defined by \eqref{eq4.11}. The function $J_{\mathrm{dual},v}(\textbf{s})$ is equal to 
\begin{equation}\label{eq4.10}
\frac{2^{k_v}(k_v-1)\cos\frac{\pi  (s_1+s_2)}{2}}{2\pi}B(s_1+s_2,k_v-s_1-s_2)B(k_v/2-s_1,k_v/2-s_2).
\end{equation}
In particular, \eqref{eq4.10} gives an explicit meromorphic continuation of $J_{\mathrm{dual},v}(\textbf{s})$ to $\mathbb{C}^2$, which is holomorphic in 
\begin{align*}
\begin{cases}
\Re(s_1+s_2)>0\\
\Re(s_1)<k_v/2,\ \Re(s_2)<k_v/2.
\end{cases}
\end{align*}
\end{itemize}
\end{lemma}
\begin{proof}
By definition of $f_v$ in \eqref{t1.5}, along with a change of variable $x_v\mapsto x_vy_v^{-1}$, we have
\begin{align*}
J_{\mathrm{dual},v}(\textbf{s})=\frac{(2i)^{k_v}(k_v-1)}{4\pi}\int_{F_v^{\times}}\int_0^{\infty}\frac{y_v^{\frac{k_v}{2}}}{(x_v+i(y_v+1))^{k_v}}|x_v|_v^{s_1+s_2}|y_v|_v^{-s_1}d^{\times}y_vd^{\times}x_v.
\end{align*}

Since $k_v\in \mathbb{Z}_{\geq 4}$, we conclude that  
\begin{align*}
\int_{F_v^{\times}}\int_0^{\infty}\bigg|\frac{y_v^{\frac{k_v}{2}}}{(x_v+i(y_v+1))^{k_v}}|x_v|_v^{s_1+s_2}|y_v|_v^{-s_1}\bigg|d^{\times}y_vd^{\times}x_v<\infty
\end{align*}
in the region described in \eqref{eq4.11}.

Hence, $J_{\mathrm{dual},v}(\textbf{s})$ converges absolutely in the region \eqref{eq4.11}, where  
\begin{equation}\label{0.9}
J_{\mathrm{dual},v}(\textbf{s})=J_{\mathrm{dual},v}^+(\textbf{s})+\overline{J_{\mathrm{dual},v}^+(\overline{\textbf{s}})},
\end{equation}
where $\overline{\textbf{s}}=(\overline{s_1},\overline{s_2})$, and 
\begin{equation}\label{4.12}
J_{\mathrm{dual},v}^+(\textbf{s}):=\frac{(2i)^{k_v}(k_v-1)}{4\pi}\int_0^{\infty}\int_0^{\infty}\frac{y^{\frac{k_v}{2}-s_1-1}}{(x+i(y+1))^{k_v}}x^{s_1+s_2-1}dydx.
\end{equation}

Let $h(z):=(z+i(y+1))^{-k_v}e^{(s_1+s_2-1)\log z}$, where $y>0$. Then $h(z)$ is holomorphic in the region $-10^{-1}<\arg(z)<\pi /2+10^{-1}$. Integrating $h(z)$ along the sector contour $\mathcal{C}_R(i(y+1))$ as defined in \textsection\ref{sec1.1.1}, and taking $R\to\infty$, we obtain 
\begin{equation}\label{4.13}
\int_0^{\infty}\frac{x^{s_1+s_2-1}}{(x+i(y+1))^{k_v}}dx=\frac{e^{\frac{\pi i (s_1+s_2)}{2}}}{i^{k_v}(1+y)^{k_v-s_1-s_2}}\int_0^{\infty}\frac{x^{s_1+s_2-1}}{(x+1)^{k_v}}dx.
\end{equation}

Substituting \eqref{4.13} into \eqref{4.12} leads to 
\begin{align*}
J_{\mathrm{dual},v}^+(\textbf{s})=\frac{2^{k_v}(k_v-1)e^{\frac{\pi i (s_1+s_2)}{2}}}{4\pi}B(s_1+s_2,k_v-s_1-s_2)B(k_v/2-s_1,k_v/2-s_2).
\end{align*}
In conjunction with \eqref{0.9} we then derive \eqref{eq4.10}. 
\end{proof}

\subsection{Meromorphic Continuation of $J_{\mathrm{dual}}^{\delta}(f_{\mathfrak{n},\mathfrak{q}},\textbf{s})$}
\begin{prop}\label{cor4.3}
Let notation be as before. Let notation be as before. 
\begin{itemize}
\item Then function $J_{\mathrm{dual}}^{I_2}(f_{\mathfrak{n},\mathfrak{q}},\textbf{s})$ converges absolutely in the region 
\begin{equation}\label{eq4.15}
\begin{cases}
\Re(s_1+s_2)>1\\
\Re(s_1)<\min_{v\mid\infty}\{k_v/2\}-1,\ \Re(s_2)<\min_{v\mid\infty}\{k_v/2\}.
\end{cases}
\end{equation}
Moreover, it admits a meromorphic continuation to $\mathbb{C}^2$ given explicitly by
\begin{align*}
J_{\mathrm{dual}}^{I_2}(f_{\mathfrak{n},\mathfrak{q}},\textbf{s})=&\prod_{v\mid\infty}\frac{2^{k_v}(k_v-1)}{4\pi\cdot (2\pi)^{-(s_1+s_2)}}\cdot\frac{\Gamma(k_v/2-s_1)\Gamma(k_v/2-s_2)}{\Gamma(k_v)}\\
&\frac{ V_{\mathfrak{q}}\cdot \zeta_F(1-s_1-s_2)}{D_F^{1/2+s_1+s_2}N(\mathfrak{q})^{s_1+s_2}N(\mathfrak{n})^{1/2-s_1}}\prod_{v\mid\mathfrak{n}}\sum_{l=0}^{e_v(\mathfrak{n})}q_v^{l(s_2-s_1)}.
\end{align*} 
\item Then function $J_{\mathrm{dual}}^{w}(f_{\mathfrak{n},\mathfrak{q}},\textbf{s})$ converges absolutely in the region 
\begin{align*}
\begin{cases}
\Re(s_1-s_2)>1\\
\Re(s_1)<\min_{v\mid\infty}\{k_v/2\}-1,\ \Re(s_2)>-\min_{v\mid\infty}\{k_v/2\}.
\end{cases}
\end{align*}
Moreover, it admits a meromorphic continuation to $\mathbb{C}^2$ given explicitly by
\begin{align*}
J_{\mathrm{dual}}^{w}(f_{\mathfrak{n},\mathfrak{q}},\textbf{s})=&\textbf{1}_{\mathfrak{q}=\mathcal{O}_F}\cdot \prod_{v\mid\infty}\frac{(2i)^{k_v}(k_v-1)}{4\pi\cdot (2\pi)^{-(s_1-s_2)}}\cdot\frac{\Gamma(k_v/2-s_1)\Gamma(k_v/2+s_2)}{\Gamma(k_v)}\\
&\frac{V_{\mathfrak{q}}\cdot \zeta_F(1-s_1+s_2)}{D_F^{1/2+s_1-s_2}N(\mathfrak{q})^{s_1-s_2}N(\mathfrak{n})^{1/2-s_1}}\prod_{v\mid\mathfrak{n}}\sum_{l=0}^{e_v(\mathfrak{n})}q_v^{l(-s_2-s_1)}.
\end{align*} 
\end{itemize}

\end{prop}
\begin{proof}
Combining Lemma \ref{lemma4.1} with Lemma \ref{lemma4.2} we conclude that $J_{\mathrm{dual}}^{I_2}(f_{\mathfrak{n},\mathfrak{q}},\textbf{s})$ converges absolutely in the region defined by \eqref{eq4.15}, and admits the expression 	
\begin{align*}
J_{\mathrm{dual}}^{\delta}(f_{\mathfrak{n},\mathfrak{q}},\textbf{s})=&\frac{(\textbf{1}_{\mathfrak{q}=\mathcal{O}_F}+\textbf{1}_{\mathfrak{q}\subsetneq \mathcal{O}_F, 
\delta=I_2})\cdot V_{\mathfrak{q}}\cdot \zeta_F(s_1+s_2)}{N(\mathfrak{n})^{1/2-s_1}N(\mathfrak{q})^{s_1+s_2}D_F}\prod_{v\mid\mathfrak{n}}\sum_{l=0}^{e_v(\mathfrak{n})}q_v^{l(s_2-s_1)}\\
&\prod_{v\mid\infty}\frac{2^{k_v-1}(\textbf{1}_{\delta=I_2}+i^{k_v}\textbf{1}_{\delta=w})\Gamma(k_v/2-s_1)\Gamma(k_v/2-s_2)\cos\frac{\pi  (s_1+s_2)}{2}}{\pi(k_v-1)^{-1}\Gamma(s_1+s_2)^{-1}\Gamma(k_v)}.
\end{align*}

Recall the functional equation of the Dedekind zeta function:
\begin{equation}\label{4.15}
\pi^{-\frac{(1-s)d_F}{2}}\Gamma((1-s)/2)^{d_F}\zeta_F(1-s)=D_F^{s-1/2}\pi^{-\frac{sd_F}{2}}\Gamma(s/2)^{d_F}\zeta_F(s),
\end{equation}
where $d_F$ is the degree of $F$, and $D_F$ is the absolutely discriminant. 

Utilizing Euler's reflection formula and the Legendre duplication formula, we obtain 
\begin{equation}\label{4.16}
\frac{\Gamma(s/2)}{\Gamma((1-s)/2)}=\frac{\Gamma(s/2)\Gamma((1+s)/2)\cos\frac{\pi s}{2}}{\pi}=\frac{2^{1-s}\Gamma(s)\cos\frac{\pi s}{2}}{\sqrt{\pi}}
\end{equation}

Substituting \eqref{4.16} into \eqref{4.15} leads to 
\begin{equation}\label{4.17}
\zeta_F(1-s)=D_F^{s-1/2}\bigg[\pi^{-s}\cdot 2^{1-s}\Gamma(s)\cos\frac{\pi s}{2}\bigg]^{d_F}\zeta_F(s).
\end{equation}

Then Corollary \ref{cor4.3} follows from substituting \eqref{4.17} into the above formula for $J_{\mathrm{dual}}^{I_2}(f_{\mathfrak{n},\mathfrak{q}},\textbf{s})$, in conjunction with \eqref{equa5.1}.
\end{proof}

\section{The Singular Orbital Integrals}\label{sec5}
Let $\Re(s_1)\gg 1$ and $\Re(s_2)\gg 1$.  Recall the definition:
\begin{align*}
J_{\mathrm{sing}}(f_{\mathfrak{n},\mathfrak{q}},\textbf{s}):=\sum_{\delta\in \{I_2,w\}}J_{\mathrm{small}}^{\delta}(f_{\mathfrak{n},\mathfrak{q}},\textbf{s})+\sum_{\delta\in \{I_2,w\}}J_{\mathrm{dual}}^{\delta}(f_{\mathfrak{n},\mathfrak{q}},\textbf{s}).\tag{\ref{1.12}}
\end{align*}

By Proposition \ref{cor3.4} and Proposition \ref{cor4.3}, $J_{\mathrm{sing}}(f_{\mathfrak{n},\mathfrak{q}},\textbf{s})$ admits a meromorphic continuation to $(s_1,s_2)\in \mathbb{C}^2$.

For $s\in \mathbb{C}$, we define the $s$-divisor function $\tau_s(\mathfrak{n}):=\prod_{v<\infty}\tau_{s,v}(\mathfrak{n})$, where $\tau_{s,v}(\mathfrak{n}):=\sum_{l=0}^{e_v(\mathfrak{n})}q_v^{-ls}$. Denote by $\tau(\mathfrak{n})=\tau_s(\mathfrak{n})\big|_{s=0}$. Note that $\tau(\mathfrak{n})$ is the generalization of the classical divisor function. 

\begin{lemma}\label{lem5.1}
Let notation be as before. Let $\mathfrak{q}=\mathcal{O}_F$. Let $\textbf{s}=(s,0)\in \mathbb{C}^2$. Let $\varepsilon>0$, and $\mathcal{C}_{\varepsilon}:=\big\{z\in\mathbb{C}:\ |z|=\varepsilon\big\}$.  
\begin{itemize}
\item $J_{\mathrm{sing}}(f_{\mathfrak{n},\mathfrak{q}},\textbf{s})\equiv 0$ if $\sum_{v\mid\infty}k_v\equiv 2\pmod{4}$.
\item Suppose $\sum_{v\mid\infty}k_v\equiv 0\pmod{4}$. Then 
\begin{equation}\label{5.1}
J_{\mathrm{sing}}(f_{\mathfrak{n},\mathfrak{q}},\textbf{s})=A_{\mathfrak{n}}(s)\tau_s(\mathfrak{n})\zeta_F(1+s)+A_{\mathfrak{n}}(-s)N(\mathfrak{n})^s\tau_s(\mathfrak{n})\zeta_F(1-s),
\end{equation}
where 
\begin{equation}\label{a6.2}
A_{\mathfrak{n}}(s):=2\prod_{v\mid\infty}\frac{2^{k_v}(k_v-1)\Gamma(k_v/2+s)\Gamma(k_v/2)}{2\cdot (2\pi)^{1+s}\Gamma(k_v)}\cdot \frac{1}{D_F^{1/2-s}N(\mathfrak{n})^{1/2}}.
\end{equation}

\item Suppose $\sum_{v\mid\infty}k_v\equiv 0\pmod{4}$. The function $J_{\mathrm{sing}}(f_{\mathfrak{n},\mathfrak{q}},\textbf{s})$ is holomorphic at $\textbf{s}=\textbf{0}=(0,0)$, with 
\begin{equation}\label{5.2}
J_{\mathrm{sing}}(f_{\mathfrak{n},\mathfrak{q}},\textbf{0})=2\cdot\frac{d(s\zeta_F(1+s)H_{\mathfrak{n}}(s))}{ds}\bigg|_{s=0}=2\cdot \frac{1}{2\pi i}\oint_{\mathcal{C}_{\varepsilon}}\frac{\zeta_F(1+s)H_{\mathfrak{n}}(s)}{s}ds,
\end{equation}
where 
\begin{equation}\label{eq5.2}
H_{\mathfrak{n}}(s):=2\prod_{v\mid\infty}\frac{2^{k_v}(k_v-1)\Gamma((k_v+s)/2)^2}{2\cdot (2\pi)^{1+s}\Gamma(k_v)}\cdot \frac{\tau(\mathfrak{n})}{D_F^{1/2-s}N(\mathfrak{n})^{(1+s)/2}}.
\end{equation}
\end{itemize}	
\end{lemma}
\begin{proof}
It follows from Propositions \ref{cor3.4} and \ref{cor4.3} that 
$J_{\mathrm{small}}^{\delta}(f_{\mathfrak{n},\mathfrak{q}},\textbf{s})$ is equal to 
\begin{equation}\label{5.3}
\Big[\textbf{1}_{\delta=I_2}+\textbf{1}_{\delta=w}\prod_{v\mid\infty}i^{k_v}\Big]\prod_{v\mid\infty}\frac{2^{k_v}(k_v-1)\Gamma(k_v/2+s)\Gamma(k_v/2)}{4\pi\cdot (2\pi)^{s}\Gamma(k_v)}\cdot \frac{\zeta_F(1+s)\tau_s(\mathfrak{n})}{D_F^{1/2-s}N(\mathfrak{n})^{1/2}},
\end{equation}
and $J_{\mathrm{dual}}^{\delta}(f_{\mathfrak{n},\mathfrak{q}},\textbf{s})$ is equal to
\begin{equation}\label{5.4}
\Big[\textbf{1}_{\delta=I_2}+\textbf{1}_{\delta=w}\prod_{v\mid\infty}i^{k_v}\Big]\prod_{v\mid\infty}\frac{2^{k_v}(k_v-1)\Gamma(k_v/2-s)\Gamma(k_v/2)}{4\pi\cdot (2\pi)^{-s}\Gamma(k_v)}\cdot \frac{\zeta_F(1-s)\tau_s(\mathfrak{n})}{D_F^{1/2+s}N(\mathfrak{n})^{1/2-s}}.
\end{equation}

Therefore, $J_{\mathrm{sing}}(f_{\mathfrak{n},\mathfrak{q}},\textbf{s})\equiv 0$ if $\sum_{v\mid\infty}k_v\equiv 0\pmod{4}$. Assuming $\sum_{v\mid\infty}k_v\equiv 2\pmod{4}$, then \eqref{5.1} follows from  \eqref{5.3} and \eqref{5.4}. 

Although the right hand side of \eqref{5.1} is not an even function, we can still rewrite it into a even form at $s=0$. Consider the Taylor expansions
\begin{align*}
A_{\mathfrak{n}}(s)=a_0+a_1s+O(s^2),\ \ \tau_s(\mathfrak{n})=b_0+b_1s+O(s^2),\ \ \zeta_F(1+s)=\frac{R}{s}+c_0+O(s),
\end{align*} 
together with $N(\mathfrak{n})^s=1+s\log N(\mathfrak{n})+O(s^2)$, we obtain from \eqref{5.1} that 
\begin{align*}
J_{\mathrm{sing}}(f_{\mathfrak{n},\mathfrak{q}},\textbf{s})=&(a_0+a_1s)(b_0+b_1s)(Rs^{-1}+c_0)\\
&+(a_0-a_1s)(1+s\log N(\mathfrak{n}))(b_0+b_1s)(-Rs^{-1}+c_0)+O(s)\\
=&2(a_0b_0c_0+a_1b_0R)-a_0b_0R\log N(\mathfrak{n})+O(s).
\end{align*}

Comparing $H_{\mathfrak{n}}(s)$ defined in \eqref{eq5.2} and $A_{\mathfrak{n}}(s)$ defined in \eqref{a6.2}, we have 
\begin{align*}
(N(\mathfrak{n})^{s/2}H_{\mathfrak{n}}(s))\big|_{s=0}=A_{\mathfrak{n}}(0)\tau(\mathfrak{n}),\ \ \ \frac{d(N(\mathfrak{n})^{s/2}H_{\mathfrak{n}}(s))}{ds}\Big|_{s=0}=\tau(\mathfrak{n})\cdot\frac{d A_{\mathfrak{n}}(s)}{ds}\Big|_{s=0}.
\end{align*}
Notice that $b_0=\tau(\mathfrak{n})$. Therefore, 
\begin{align*}
H_{\mathfrak{n}}(s)=a_0b_0+(a_1b_0-2^{-1}a_0b_0\log N(\mathfrak{n}))s+O(s^2),
\end{align*}
from which we deduce that 
\begin{align*}
H_{\mathfrak{n}}(s)\zeta_F(1+s)+H_{\mathfrak{n}}(-s)\zeta_F(1-s)=2(a_0b_0c_0+a_1b_0R)-a_0b_0R\log N(\mathfrak{n})+O(s).
\end{align*}

As a consequence, we conclude that 
\begin{align*}
J_{\mathrm{sing}}(f_{\mathfrak{n},\mathfrak{q}},\textbf{s})=&H_{\mathfrak{n}}(s)\zeta_F(1+s)+H_{\mathfrak{n}}(-s)\zeta_F(1-s)+O(s).
\end{align*}

Taking $s\to 0$, we obtain 
\begin{align*}
J_{\mathrm{sing}}(f_{\mathfrak{n},\mathfrak{q}},\textbf{0})=&2H_{\mathfrak{n}}(0)\frac{d(s\zeta_F(1+s))}{ds}\bigg|_{s=0}+2\Res_{s=1}\zeta_F(s)\cdot \frac{dH_{\mathfrak{n}}(s)}{ds}\bigg|_{s=0},
\end{align*}
which is the formula \eqref{5.2}. 
\end{proof}

\begin{lemma}\label{lem5.2}
Let notation be as before. Let $\mathfrak{q}\subsetneq\mathcal{O}_F$ be a prime ideal. Let $\textbf{s}=(s/2,s/2)\in \mathbb{C}^2$. Let $\varepsilon>0$, and $\mathcal{C}_{\varepsilon}:=\big\{z\in\mathbb{C}:\ |z|=\varepsilon\big\}$.  
\begin{itemize}
\item We have 
\begin{equation}\label{5.5}
J_{\mathrm{sing}}(f_{\mathfrak{n},\mathfrak{q}},\textbf{s})=\frac{N(\mathfrak{q})+1}{2}\cdot \Big[H_{\mathfrak{n}}(s)\zeta_F(1+s)+N(\mathfrak{q})^{-s}H_{\mathfrak{n}}(-s)\zeta_F(1-s)\Big],
\end{equation}
where $H_{\mathfrak{n}}(s)$ is defined as in \eqref{eq5.2}.

\item The function $J_{\mathrm{sing}}(f_{\mathfrak{n},\mathfrak{q}},\textbf{s})$ is holomorphic at $\textbf{s}=\textbf{0}=(0,0)$, with 
\begin{equation}\label{5.6}
J_{\mathrm{sing}}(f_{\mathfrak{n},\mathfrak{q}},\textbf{0})=\frac{N(\mathfrak{q})+1}{2}\cdot\frac{1}{2\pi i}\oint_{\mathcal{C}_{\varepsilon}}\frac{\zeta_F(1+s)H_{\mathfrak{n}}(s)(1+N(\mathfrak{q})^s)}{s}ds.	
\end{equation}
\end{itemize}	
\end{lemma}
\begin{proof}
By Propositions \ref{cor3.4} and \ref{cor4.3}, we derive from \eqref{1.12} that 
\begin{equation}\label{5.8}
J_{\mathrm{sing}}(f_{\mathfrak{n},\mathfrak{q}},\textbf{s})=J_{\mathrm{small}}^{I_2}(f_{\mathfrak{n},\mathfrak{q}},\textbf{s})+J_{\mathrm{dual}}^{I_2}(f_{\mathfrak{n},\mathfrak{q}},\textbf{s}),
\end{equation}
where 
\begin{align*}
J_{\mathrm{small}}^{I_2}(f_{\mathfrak{n},\mathfrak{q}},\textbf{s})=&V_{\mathfrak{q}}\prod_{v\mid\infty}\frac{2^{k_v}(k_v-1)\Gamma((k_v+s)/2)^2}{2\cdot (2\pi)^{1+s}\Gamma(k_v)}\cdot \frac{\tau(\mathfrak{n})\cdot \zeta_F(1+s)}{D_F^{1/2-s}N(\mathfrak{n})^{(1+s)/2}},\\
J_{\mathrm{dual}}^{I_2}(f_{\mathfrak{n},\mathfrak{q}},\textbf{s})=&V_{\mathfrak{q}}\prod_{v\mid\infty}\frac{2^{k_v}(k_v-1)\Gamma((k_v-s)/2)^2}{2\cdot (2\pi)^{1-s}\Gamma(k_v)}\cdot\frac{\tau(\mathfrak{n})\cdot \zeta_F(1-s)}{D_F^{1/2+s}N(\mathfrak{q})^{s}N(\mathfrak{n})^{(1-s)/2}}.
\end{align*}

Hence \eqref{5.5} follows from \eqref{5.8}. Moreover, we have 
\begin{align*}
J_{\mathrm{sing}}(f_{\mathfrak{n},\mathfrak{q}},\textbf{0})=\lim_{s\to 0}\big[J_{\mathrm{small}}^{I_2}(f_{\mathfrak{n},\mathfrak{q}},\textbf{s})+J_{\mathrm{dual}}^{I_2}(f_{\mathfrak{n},\mathfrak{q}},\textbf{s})\big]=0.
\end{align*}
Hence, the meromorphic function $J_{\mathrm{sing}}(f_{\mathfrak{n},\mathfrak{q}},\textbf{s})$ is holomorphic at $\textbf{s}=\textbf{0}$. 

Consider the Taylor expansion 
\begin{equation}\label{eq5.6}
H_{\mathfrak{n}}(s)=a_0+a_1s+O(s^2),
\end{equation}
where $a_0=H_{\mathfrak{n}}(0)$ and $a_1=\frac{dH_{\mathfrak{n}}(s)}{ds}\big|_{s=0}$. Let $c_0$ be the constant term in the Taylor expansion of $\zeta_F(1+s)$, namely, $c_0=\frac{d(s\zeta_F(1+s))}{ds}\big|_{s=0}$. We have
\begin{align*}
&H_{\mathfrak{n}}(s)\zeta_F(1+s)+N(\mathfrak{q})^{-s}H_{\mathfrak{n}}(-s)\zeta_F(1-s)\\
=&2a_0c_0+2a_1\Res_{s=1}\zeta_F(s)+a_0\log N(\mathfrak{q})\cdot\Res_{s=1}\zeta_F(s)+O(s).
\end{align*}
Notice that 
\begin{equation}\label{5.10}
a_0c_0+a_1\Res_{s=1}\zeta_F(s)=\frac{1}{2\pi i}\oint_{\mathcal{C}_{\varepsilon}}\frac{\zeta_F(1+s)H_{\mathfrak{n}}(s)}{s}ds
\end{equation}
and 
\begin{equation}\label{5.11}
a_0c_0+(a_1+a_0\log N(\mathfrak{q}))\Res_{s=1}\zeta_F(s)=\frac{1}{2\pi i}\oint_{\mathcal{C}_{\varepsilon}}\frac{\zeta_F(1+s)N(\mathfrak{q})^sH_{\mathfrak{n}}(s)}{s}ds.
\end{equation}

Therefore, \eqref{5.6} follows from \eqref{5.10} and \eqref{5.11}. 
\end{proof}

Combining Lemma \ref{lem5.1} and Lemma \ref{lem5.2} we obtain the following result. 
\begin{prop}\label{prop5.3}
Let notation be as before. Let $\mathfrak{q}\subseteq\mathcal{O}_F$ be an integral ideal. Let $\textbf{s}=(s/2,s/2)\in \mathbb{C}^2$. Let $\varepsilon>0$, and $\mathcal{C}_{\varepsilon}:=\big\{z\in\mathbb{C}:\ |z|=\varepsilon\big\}$. 
\begin{itemize}
\item The function $J_{\mathrm{sing}}(f_{\mathfrak{n},\mathfrak{q}},\textbf{s})$ is holomorphic at $\textbf{s}=\textbf{0}=(0,0)$, with 
\begin{equation}\label{5.13}
J_{\mathrm{sing}}(f_{\mathfrak{n},\mathfrak{q}},\textbf{0})=\frac{(N(\mathfrak{q})+1)\cdot\delta_{\mathbf{k},\mathfrak{q}}}{2}\cdot\frac{1}{2\pi i}\oint_{\mathcal{C}_{\varepsilon}}\frac{\zeta_F(1+s)H_{\mathfrak{n}}(s)(1+N(\mathfrak{q})^s)}{s}ds,	
\end{equation}
where $\delta_{\mathbf{k},\mathfrak{q}}:=\textbf{1}_{\mathfrak{q}\subsetneq \mathcal{O}_F}+\textbf{1}_{\mathfrak{q}=\mathcal{O}_F  \& \sum_{v\mid\infty}k_v\equiv 0\pmod{4}}$, and 
\begin{align*}
H_{\mathfrak{n}}(s):=2\prod_{v\mid\infty}\frac{2^{k_v}(k_v-1)\Gamma((k_v+s)/2)^2}{2\cdot (2\pi)^{1+s}\Gamma(k_v)}\cdot \frac{\tau(\mathfrak{n})}{D_F^{1/2-s}N(\mathfrak{n})^{(1+s)/2}}.\tag{\ref{eq5.2}}
\end{align*}
\item Explicitly, we have
\begin{align*}
J_{\mathrm{sing}}(f_{\mathfrak{n},\mathfrak{q}},\textbf{0})=& (N(\mathfrak{q})+1)\cdot\delta_{\mathbf{k},\mathfrak{q}}\bigg[H_{\mathfrak{n}}(0)\frac{d(s\zeta_F(1+s))}{ds}\bigg|_{s=0}+\Res_{s=1}\zeta_F(s)\cdot H_{\mathfrak{n}}'(0)\bigg]\\
& +\frac{(N(\mathfrak{q})+1)\cdot \delta_{\mathbf{k},\mathfrak{q}}}{2}\cdot H_{\mathfrak{n}}(0)\log N(\mathfrak{q})\cdot\Res_{s=1}\zeta_F(s),
\end{align*}
where $H_{\mathfrak{n}}'(0):=\frac{dH_{\mathfrak{n}}(s)}{ds}\big|_{s=0}$ is the first derivation of $H_{\mathfrak{n}}(s)$ at $s=0$. 
\end{itemize}
\end{prop}

\section{The Regular Orbital Integrals}\label{sec7}
Recall that the regular orbital integral $J_{\mathrm{reg}}(f_{\mathfrak{n},\mathfrak{q}},\textbf{s})$ is defined by 
\begin{align*}
\sum_{t\in F-\{0,1\}}\int_{\mathbb{A}_F^{\times}}\int_{\mathbb{A}_F^{\times}}f_{\mathfrak{n},\mathfrak{q}}\left(\begin{pmatrix}
x^{-1}& \\
&1
\end{pmatrix}\begin{pmatrix}
1& t\\
1& 1
\end{pmatrix}\begin{pmatrix}
xy& \\
&1
\end{pmatrix}\right)|x|^{s_1+s_2}|y|^{s_2}d^{\times}yd^{\times}x.
\end{align*}

Following the proof of \cite[Theorem 2.1]{RR05} the function $J_{\mathrm{reg}}(f_{\mathfrak{n},\mathfrak{q}},\textbf{s})$ converges absolutely in the region 
\begin{align*}
\begin{cases}
-\min_{v\mid\infty}\{k_v/2\}+1<\Re(s_1+s_2)<\min_{v\mid\infty}\{k_v/2\},\\
-\min_{v\mid\infty}\{k_v/2\}+1<\Re(s_2)<\min_{v\mid\infty}\{k_v/2\}.
\end{cases}
\end{align*}  

Since $\min_{v\mid\infty}\{k_v/2\}\geq 2$, then  $J_{\mathrm{reg}}(f_{\mathfrak{n},\mathfrak{q}},\textbf{0})$ is well defined. Our main result in the section is the following estimate.
\begin{prop}\label{prop6.12}
Let notation be as before. Let $\varepsilon>0$. Then 
\begin{equation}\label{6.49}
J_{\mathrm{reg}}(f_{\mathfrak{n},\mathfrak{q}},\textbf{0})\ll N(\mathfrak{n})^{\frac{1}{2}+\varepsilon}N(\mathfrak{q})^{\varepsilon}\prod_{v\mid\infty}\frac{2^{k_v-1}(k_v-1)}{\pi\sqrt{k_v}}B\left(\frac{k_v}{2},\frac{k_v}{2}\right),
\end{equation}
where the implied constant depends only on $\varepsilon$ and $F$.
\end{prop}

\subsection{Calculation of $J_v(t)$}\label{sec6.1}
Since $J_{\mathrm{reg}}(f_{\mathfrak{n},\mathfrak{q}},\textbf{0})$ converges absolutely, we may decompose it into local integrals:  
\begin{equation}\label{eq1.9}
J_{\mathrm{reg}}(f_{\mathfrak{n},\mathfrak{q}},\textbf{0})=\sum_{t\in F-\{0,1\}}\prod_{v\leq\infty}J_v(t),
\end{equation}
where  
\begin{align*}
J_v(t):=\int_{F_v^{\times}}\int_{F_v^{\times}}f_{v}\left(\begin{pmatrix}
y_v& x_v^{-1}t\\
x_vy_v& 1
\end{pmatrix}\right)d^{\times}y_vd^{\times}x_v.
\end{align*}
\begin{lemma}\label{lemma1.3}
Let $v<\infty$ and $v\nmid \mathfrak{n}\mathfrak{q}$. Then 
\begin{equation}\label{a1.5}
J_v(t)=\Vol(\mathcal{O}_v^{\times})^2\cdot (1-e_v(1-t))\cdot (e_v(t)-e_v(1-t)+1)\cdot \textbf{1}_{e_v(1-t)\leq 0}.
\end{equation}
In particular, $J_v(t)\equiv 1$ if $e_v(t)=e_v(1-t)=0$. 
\end{lemma}
\begin{proof}
Write $e_v(x_v)=r_1$, and  $e_v(y_v)=r_2$. By the definition of $f_v$ in \eqref{t1.8}, $f_{v}\left(\begin{pmatrix}
y_v& x_v^{-1}t\\
x_vy_v& 1
\end{pmatrix}\right)$ is the characteristic function of the number of $r_1, r_2\in \mathbb{Z}$ such that the following constraint  
\begin{equation}\label{1.1.10}
\varpi_v^l\begin{pmatrix}
x_v^{-1}& \\
&1
\end{pmatrix}\begin{pmatrix}
1& t\\
1& 1
\end{pmatrix}\begin{pmatrix}
x_vy_v& \\
&1
\end{pmatrix}=\varpi_v^l\begin{pmatrix}
\varpi_v^{r_2}& \varpi_v^{-r_1}t\\
\varpi_v^{r_1+r_2}& 1
\end{pmatrix}\in K_v
\end{equation}
holds for some $l\in \mathbb{Z}$. Note that \eqref{1.1.10} amounts to  
\begin{equation}\label{1.1.11}
\begin{cases}
2l+r_2+e_v(1-t)=0\\
l+r_2\geq 0,\ \ l\geq 0\\
l-r_1+e_v(t)\geq 0,\ \ l+r_1+r_2\geq 0
\end{cases}
\end{equation}
Here the constraint $2l+r_2+e_v(1-t)=0$ comes from the restriction on the determinants in \eqref{1.1.10}. A further investigation of \eqref{1.1.11} boils down to 
\begin{equation}\label{1.1.12}
\begin{cases}
|r_2|\leq -e_v(1-t),\ \ r_2\equiv -e_v(1-t)\pmod{2}\\
|2r_1+r_2-e_v(t)|\leq e_v(t)-e_v(1-t).
\end{cases}
\end{equation}

As a consequence, 
\begin{align*}
J_v(t)=\Vol(\mathcal{O}_v^{\times})^2\cdot \textbf{1}_{e_v(1-t)\leq 0,\ e_v(t)-e_v(1-t)\geq 0}\sum_{\text{$r_1, r_2$ satisfying \eqref{1.1.12}}}1.
\end{align*}

Therefore, \eqref{a1.5} follows from the fact that the number of $r_1, r_2$ satisfying \eqref{1.1.12} is $(1-e_v(1-t))\cdot (e_v(t)-e_v(1-t)+1)$. 
\end{proof}

\begin{lemma}
Let $v\mid \mathfrak{q}$. Then 
\begin{equation}\label{1.1.13}
J_v(t)=e_v(t)\Vol(\mathcal{O}_v^{\times})^2\cdot \Vol(K_0(\mathfrak{p}_v))^{-1}\cdot \textbf{1}_{e_v(t)\geq 1}.
\end{equation}
\end{lemma}
\begin{proof}
Write $e_v(x_v)=r_1$ and $e_v(y_v)=r_2$. By the definition of $f_v$ in \eqref{t1.7}, $\Vol(K_0(\mathfrak{p}_v))f_{v}\left(\begin{pmatrix}
y_v& x_v^{-1}t\\
x_vy_v& 1
\end{pmatrix}\right)$ is the characteristic function of the number of $r_1, r_2\in \mathbb{Z}$ such that the following constraint  
\begin{equation}\label{1.1.14}
\varpi_v^l\begin{pmatrix}
x^{-1}& \\
&1
\end{pmatrix}\begin{pmatrix}
1& t\\
1& 1
\end{pmatrix}\begin{pmatrix}
xy& \\
&1
\end{pmatrix}=\varpi_v^l\begin{pmatrix}
\varpi_v^{r_2}& \varpi_v^{-r_1}t\\
\varpi_v^{r_1+r_2}& 1
\end{pmatrix}\in K_0(\mathfrak{p}_v)
\end{equation}
holds for some $l\in \mathbb{Z}$. Similar to the analysis in the proof of Lemma \ref{lemma1.3}, the restrictions in \eqref{1.1.14} amount to 
\begin{equation}\label{1.1.15}
\begin{cases}
2l+r_2+e_v(1-t)=0\\
l+r_2=l=0\\
l-r_1+e_v(t)\geq 0,\ \ l+r_1+r_2\geq 1
\end{cases}\ \ \Leftrightarrow\ \ \begin{cases}
l=r_2=e_v(1-t)=0\\
1\leq r_1\leq e_v(t).
\end{cases}
\end{equation}

Consequently, we obtain from (the second system of constraints in) \eqref{1.1.15} that 
\begin{align*}
J_v(t)=\frac{\Vol(\mathcal{O}_v^{\times})^2\cdot \textbf{1}_{e_v(1-t)=0}}{\Vol(K_0(\mathfrak{p}_v))}\sum_{1\leq r_1\leq e_v(t)}1=\frac{e_v(t)}{\Vol(K_0(\mathfrak{p}_v))}\cdot \textbf{1}_{e_v(t)\geq 1},
\end{align*}
which is \eqref{1.1.13}. 
\end{proof}


\begin{lemma}\label{lem1.6}
Let $v\mid \mathfrak{n}$. Denote by $e_v(\mathfrak{n})=r\geq 1$. Then 
\begin{equation}\label{f1.16}
J_v(t)=\frac{\Vol(\mathcal{O}_v^{\times})^2}{q_v^{\frac{r}{2}}}(r+1-e_v(1-t))(r+1-e_v(1-t)+e_v(t))\cdot \textbf{1}_{e_v(t)-e_v(1-t)\geq -r}.
\end{equation}
\end{lemma}
\begin{proof}
Write $e_v(x_v)=r_1$, $e_v(y_v)=r_2$. By the definition of $f_v$ in \eqref{t1.6}, the function $f_{v}\left(\begin{pmatrix}
y_v& x_v^{-1}t\\
x_vy_v& 1
\end{pmatrix}\right)$ is equal to the number of $r_1, r_2\in \mathbb{Z}$ such that the following constraint
\begin{align*}
\varpi_v^l\begin{pmatrix}
y& x^{-1}t\\
xy& 1
\end{pmatrix}=\varpi_v^l\begin{pmatrix}
\varpi_v^{r_2}& \varpi_v^{-r_1}t\\
\varpi_v^{r_1+r_2}& 1
\end{pmatrix}\in \bigsqcup_{\substack{i+j=e_v(\mathfrak{n})\\
i\geq j\geq 0}}K_v\begin{pmatrix}
\varpi_v^i\\
&\varpi_v^j
\end{pmatrix}K_v
\end{align*}
holds for some $l\in \mathbb{Z}$, which amounts to 
\begin{equation}\label{1.17}
\begin{cases}
2l+r_2+e_v(1-t)=r\\
l+r_2\geq 0,\ \ l\geq 0\\
l-r_1+e_v(t)\geq 0,\ \ l+r_1+r_2\geq 0.
\end{cases}
\end{equation}

Relaxing the dependence on $l$, \eqref{1.17} is equivalent to 
\begin{equation}\label{1.18}
\begin{cases}
|r_2|\leq r-e_v(1-t),\ \ r_2\equiv r-e_v(1-t)\pmod{2}\\
|2r_1+r_2-e_v(t)|\leq r-e_v(1-t)+e_v(t).
\end{cases}
\end{equation}

Therefore, we obtain 
\begin{align*}
J_v(t)=q_v^{-\frac{r}{2}}\cdot \Vol(\mathcal{O}_v^{\times})^2\cdot \textbf{1}_{e_v(1-t)\leq r,\ e_v(t)-e_v(1-t)\geq -r}\sum_{\text{$r_1, r_2$ satisfying \eqref{1.18}}}1.
\end{align*}
Therefore, \eqref{f1.16} follows from the fact that the number of $r_1, r_2$ satisfying \eqref{1.18} is $(r-e_v(1-t)+1)\cdot (r-e_v(1-t)+e_v(t)+1)$. 
\end{proof}


Now we consider the archimedean place:
\begin{equation}\label{f1.19}
J_v(t):=\int_{F_v^{\times}}\int_{F_v^{\times}}f_{v}\left(\begin{pmatrix}
y_v& x_v^{-1}t\\
x_vy_v& 1
\end{pmatrix}\right)\cdot \textbf{1}_{y_v(1-t)>0}d^{\times}y_vd^{\times}x_v,\ \ \ v\mid\infty.
\end{equation}

\begin{lemma}\label{lem1.3}
Let $k\in\mathbb{Z}_{>2}$ be even. Let $t\in \mathbb{R}-\{0,1\}$. Let 
\begin{equation}\label{f1.20}
J(t):=\int_{\mathbb{R}^{\times}}\int_{\mathbb{R}^{\times}}\frac{(y(1-t))^{\frac{k}{2}}\textbf{1}_{y(1-t)>0}}{(-x^{-1}t+xy+i(y+1))^{k}}d^{\times}yd^{\times}x.
\end{equation}
We have the following assertions. 
\begin{itemize}
\item The function $J(t)$ is equal to 
\begin{align*}
-\frac{2B(k/2,k/2)}{i^{k}}\cdot\bigg[P_{\frac{k}{2}-1}\left(\frac{1+t}{1-t}\right)\log |t|_v +2\sum_{j=0}^{\floor{\frac{k}{4}}}\frac{k-4j+1}{(2j-1)(\frac{k}{2}-j)}P_{\frac{k}{2}-2j}\left(\frac{1+t}{1-t}\right)\bigg],
\end{align*}
where $B(\cdot, \cdot)$ is the Beta function, and $P_n$ refers to the $n$-th Legendre polynomial. 
\item Suppose that $t>0$. Then 
\begin{align*}
J(t)=\frac{4B(k/2,k/2)}{i^{k}}\cdot Q_{\frac{k}{2}-1}\left(\frac{1+t}{1-t}\right),
\end{align*}
where $Q_{n}$ is the $n$-th Legendre function of the second kind.
\end{itemize}
\end{lemma}
\begin{proof}
We consider the following scenarios. 
\begin{itemize}
\item Suppose $1-t>0$. Since $d^{\times}x=|x|^{-1}dx$ and $d^{\times}y=|y|^{-1}dy$, then 
\begin{equation}\label{f1.6}
J(t)=(1-t)^{\frac{k}{2}}\int_{-\infty}^{\infty}\int_{0}^{\infty}\frac{x^{k}y^{\frac{k}{2}}}{(-t+x^2y+i(xy+x))^{k}}d^{\times}yd^{\times}x=J_+(t)+\overline{J_+(t)},
\end{equation}
where
\begin{align*}
J_+(t):=&(1-t)^{\frac{k}{2}}\int_{0}^{\infty}\int_{0}^{\infty}\frac{x^{k-1}y^{\frac{k}{2}-1}}{(-t+x^2y+i(xy+x))^{k}}dydx.
\end{align*}

By the change of variable $y\mapsto yx^{-1}$, we obtain 
\begin{align*}
J_+(t):=(1-t)^{\frac{k}{2}}\int_{0}^{\infty}\int_{0}^{\infty}\frac{x^{\frac{k}{2}-1}y^{\frac{k}{2}-1}}{(-t+xy+i(y+x))^{k}}dydx.
\end{align*}

Notice that $-t+xy+i(y+x)=1-t+(x+i)(y+i)$. Hence,
\begin{equation}\label{1.5}
J_+(t)=(1-t)^{\frac{k}{2}}\int_{0}^{\infty}\frac{x^{\frac{k}{2}-1}}{(x+i)^k}\int_{0}^{\infty}\frac{y^{\frac{k}{2}-1}}{(y+\frac{1-t}{x+i}+i)^{k}}dydx.
\end{equation}

For $z_0\in \mathbb{C}-\{0\}$, let $\mathcal{C}_R(z_0)$ be the contour defined as in \textsection\ref{sec1.1.1}. 
Let $\mathcal{D}_R(z_0)$ be the closed region with boundary being $\mathcal{C}_R(z_0)$. Since $z_0\neq 0$, the function $h_0(z)=z^{\frac{k}{2}-1}(z+z_0)^{-k}$ is holomorphic in $\mathcal{D}_R(z_0)$ for all $R>0$. As a consequence,  
\begin{equation}\label{f1.7}
\int_{\mathcal{C}_R(z_0)}h_0(z)dz=\sum_{j=1}^3\int_{\mathcal{C}_R^{(j)}(z_0)}h_0(z)dz=0.	
\end{equation}

Notice that $\lim_{R\to+\infty}h_0(z)=0$ uniformly for $z\in \mathcal{C}_R^{(2)}(z_0)$.  Hence,
\begin{equation}\label{f1.8}
\lim_{R\to+\infty}\int_{\mathcal{C}_R^{(2)}(z_0)}h_0(z)dz=0.
\end{equation}

Taking $R\to+\infty$ into \eqref{f1.7}, along with \eqref{f1.8}, we obtain 
\begin{equation}\label{1.6}
\int_{0}^{\infty}\frac{y^{\frac{k}{2}-1}}{(y+z_0)^{k}}dy=z_0^{-\frac{k}{2}}\int_{L}\frac{y^{\frac{k}{2}-1}}{(y+1)^{k}}dy=z_0^{-\frac{k}{2}} B(k/2,k/2),
\end{equation}
where $L$ denotes the half-line defined by the positive real multiples of $z_0$. 

Let $z_0=(1-t)(x+i)^{-1}+i$. Since $t\neq 1$ and $x\neq 0$, then $z_0\neq 0$. It follows from \eqref{1.5} and \eqref{1.6} that 
\begin{equation}\label{1.7'}
J_+(t)=(1-t)^{\frac{k}{2}}B(k/2,k/2)\int_{0}^{\infty}\frac{x^{\frac{k}{2}-1}}{(x+i)^{\frac{k}{2}}\left(ix-t\right)^{\frac{k}{2}}}dx.
\end{equation}

\item Suppose $1-t<0$. Then 
\begin{equation}\label{1.12.}
J(t)=(t-1)^{\frac{k}{2}}\int_{\mathbb{R}^{\times}}\int_0^{\infty}\frac{y^{\frac{k}{2}}}{(x^{-1}t+xy+i(y-1))^{k}}d^{\times}yd^{\times}x=J_-(t)+\overline{J_-(t)},
\end{equation}
where
\begin{align*}
J_-(t):=(t-1)^{\frac{k}{2}}\int_0^{\infty}\int_0^{\infty}\frac{y^{\frac{k}{2}-1}}{(x^{-1}t+xy+i(y-1))^{k}}dyd^{\times}x.
\end{align*}

Changing the variable $y\mapsto yx^{-1}$ leads to 
\begin{align*}
J_-(t)=(t-1)^{\frac{k}{2}}\int_0^{\infty}\frac{x^{\frac{k}{2}-1}}{(x+i)^{k}}\int_0^{\infty}\frac{y^{\frac{k}{2}-1}}{(y+\frac{t-1}{x+i}-i)^{k}}dydx.
\end{align*}

Since $\frac{t-1}{x+i}-i\neq 0$, we deduce from the above expression and \eqref{1.6} that
\begin{equation}\label{f1.12}
J_-(t)=(1-t)^{\frac{k}{2}}B(k/2,k/2)\int_0^{\infty}\frac{x^{\frac{k}{2}-1}}{(x+i)^{\frac{k}{2}}(ix-t)^{\frac{k}{2}}}dx.
\end{equation}
\end{itemize}

Let $\varepsilon>0$. Considering \eqref{1.7'} and \eqref{f1.12}, we define the auxiliary integral 
\begin{align*}
h(z):=(1-z)^{\frac{k}{2}}B(k/2,k/2)\int_{0}^{\infty}\frac{x^{\frac{k}{2}-1}}{(x+i)^{\frac{k}{2}}\left(ix-z\right)^{\frac{k}{2}}}dx.
\end{align*}

Since $k\geq 4$, $h(z)$ converges absolutely in $\{z\in \mathbb{C}:\ \Re(z)>0,\ z\neq 1\}$, defining a holomorphic function therein. Set 
$H(z):=h(z)+\overline{h(\overline{z})},$ where $\Re(z)>0.$

Now we proceed to investigate the meromorphic continuation of $H(z)$ across the vertical line $\Re(z)=0$. Suppose $\Re(z)>0$. Making use of a similar manipulation of contour integral (with $z_0=i$), we derive 
\begin{equation}\label{1.7}
h(z)=\frac{(1-z)^{\frac{k}{2}}B(k/2,k/2)}{i^{k}}\int_{0}^{\infty}\frac{x^{\frac{k}{2}-1}}{(x+1)^{\frac{k}{2}}\left(x+z\right)^{\frac{k}{2}}}dx.
\end{equation}

Put $a=\frac{x}{1+x}\in (0,+\infty)$. Then $x=\frac{a}{1-a}$ and $dx=\frac{da}{(1-a)^2}$. So
\begin{equation}\label{1.8}
\int_{0}^{\infty}\frac{x^{\frac{k}{2}-1}dx}{(x+1)^{\frac{k}{2}}\left(x+z\right)^{\frac{k}{2}}}=\int_0^1\frac{a^{\frac{k}{2}-1}(1-a)^{\frac{k}{2}-1}}{\left(a(1-t)+z\right)^{\frac{k}{2}}}da=\frac{{}_2F_1\left(\frac{k}{2}, \frac{k}{2}; k; 1-z^{-1}\right)}{z^{\frac{k}{2}}B(k/2,k/2)^{-1}},
\end{equation}
where ${}_2F_1\left(\frac{k}{2}, \frac{k}{2}; k; 1-z^{-1}\right)$ is the hypergeometric function. Furthermore, it is well known that  
\begin{equation}\label{1.9}
{}_2F_1\left(\frac{k}{2}, \frac{k}{2}; k; 1-z^{-1}\right)=z^{\frac{k}{2}}{}_2F_1\left(\frac{k}{2}, \frac{k}{2}; k; 1-z\right).
\end{equation}

Since $k$ is even, utilizing  \cite[\textsection 5.4, p.233-p.234]{MOS66} we obtain 
\begin{equation}\label{1.10}
B(k/2,k/2)\cdot {}_2F_1\left(\frac{k}{2},\frac{k}{2};k; 1-z\right)=2(1-z)^{-\frac{k}{2}}\cdot Q_{\frac{k}{2}-1}\left(\frac{1+z}{1-z}\right),
\end{equation}
where $Q_{\frac{k}{2}-1}$ is the Legendre function of the second kind.

Substituting \eqref{1.8} and \eqref{1.9} into \eqref{1.7} yields 
\begin{equation}\label{1.11}
h(z)=\frac{2B(k/2,k/2)}{i^{k}}\cdot Q_{\frac{k}{2}-1}\left(\frac{1+z}{1-z}\right).
\end{equation}

Utilizing the relation between $Q_{\frac{k}{2}-1}$ and Legendre polynomials (cf, p.234 in loc. cit.), we can  explicitly express $Q_{\frac{k}{2}-1}\left(\frac{1+z}{1-z}\right)$ as
\begin{equation}\label{1.20}
-\frac{1}{2}P_{\frac{k}{2}-1}\left(\frac{1+z}{1-z}\right)\cdot\log z -\sum_{j=0}^{\floor{\frac{k}{4}}}\frac{k-4j+1}{(2j-1)(\frac{k}{2}-j)}P_{\frac{k}{2}-2j}\left(\frac{1+z}{1-z}\right),
\end{equation}
where $P_n$ refers to the $n$-th Legendre polynomial. 

We can thus extend $Q_{\frac{k}{2}-1}\left(\frac{1+z}{1-z}\right)$ to a holomorphic function in the region $\{s\in \mathbb{C}:\ \text{$s=1$ or $s\in i[0,\infty)$}\}$. In particular, for all $t\in \mathbb{R}-\{0,1\}$, combining \eqref{1.11} with \eqref{1.20}, the function $H(t)$ boils down to 
\begin{align*}
-\frac{2B(k/2,k/2)}{i^{k}}\cdot\bigg[P_{\frac{k}{2}-1}\left(\frac{1+t}{1-t}\right)\log |t|_v +2\sum_{j=0}^{\floor{\frac{k}{4}}}\frac{k-4j+1}{(2j-1)(\frac{k}{2}-j)}P_{\frac{k}{2}-2j}\left(\frac{1+t}{1-t}\right)\bigg].
\end{align*}

As a consequence of \eqref{f1.6}, \eqref{1.7'}, \eqref{1.12.}, and \eqref{f1.12}, we obtain $J(t)=H(t)$ for all $t\in \mathbb{R}-\{0,1\}$.
Therefore, Lemma \ref{lem1.3} follows. 
\end{proof}

\begin{lemma}\label{lem1.7}
Let $v\mid\infty$. Let $t\in F-\{0,1\}$. Let $t_v$ be the component of $t\in F_v\simeq \mathbb{R}$. We have the following assertions. 
\begin{itemize}
\item Let $J_v(t)$ be the local integral defined by \eqref{f1.19}. Then  
\begin{align*}
J_v(t)=&-\frac{2^{k_v-1}(k_v-1)}{\pi}\cdot B\left(\frac{k_v}{2},\frac{k_v}{2}\right)\cdot P(k_v;t),
\end{align*}
where 
\begin{align*}
P(k_v;t):=P_{\frac{k_v}{2}-1}\left(\frac{1+t}{1-t}\right)\log |t|_v +2\sum_{j=0}^{\floor{k_v/4}}\frac{k_v-4j+1}{(2j-1)(\frac{k_v}{2}-j)}P_{\frac{k_v}{2}-2j}\left(\frac{1+t}{1-t}\right).
\end{align*}
Here $B(\cdot, \cdot)$ is the Beta function, and $P_n$ refers to the $n$-th Legendre polynomial.
\item Suppose $t_v>0$, i.e., $t$ is positive in $F_v\simeq \mathbb{R}$. Then 
\begin{align*}
J_v(t)=&\frac{2^{k_v}(k_v-1)}{\pi}\cdot  B\left(\frac{k_v}{2},\frac{k_v}{2}\right)\cdot Q_{\frac{k}{2}-1}\left(\frac{1+t}{1-t}\right),
\end{align*}
where $Q_{n}$ is the $n$-th Legendre function of the second kind.
\end{itemize}

\end{lemma}
\begin{proof}
Substituting the definition of $f_v$ (cf. \eqref{t1.5}) 
\begin{align*}
f_{v}\left(\begin{pmatrix}
y_v& x_v^{-1}t\\
x_vy_v& 1
\end{pmatrix}\right)=\frac{(2i)^{k_v}(k_v-1)}{4\pi}\cdot \frac{(y_v(1-t))^{\frac{k_v}{2}}}{(-x_v^{-1}t+x_vy_v+i(y_v+1))^{k_v}}
\end{align*} 
into the definition \eqref{f1.19}, we obtain 
\begin{align*}
J_v(t)=\frac{(2i)^{k_v}(k_v-1)}{4\pi}\cdot J(t),
\end{align*}
where $J(t)$ is defined by \eqref{f1.20}. Then Lemma \ref{lem1.7} follows from Lemma \ref{lem1.3}.
\end{proof}

\subsection{Exact Formula for $J_{\mathrm{reg}}(f_{\mathfrak{n},\mathfrak{q}},\textbf{0})$}\label{sec1.3.5}
Combining the local calculations in \textsection\ref{sec6.1}, we will prove a precise formula for the regular orbital integral $J_{\mathrm{reg}}(f_{\mathfrak{n},\mathfrak{q}},\textbf{0})$ in this subsection. Our main result is the following. 
\begin{prop}\label{prop6.6}
Let notation be as before. Then 
\begin{equation}\label{6.28}
J_{\mathrm{reg}}(f_{\mathfrak{n},\mathfrak{q}},\textbf{0})=\frac{V_{\mathfrak{q}}D_F^{-1}}{N(\mathfrak{n})^{\frac{1}{2}}}\prod_{v\mid\infty}\frac{2^{k_v-1}(k_v-1)}{\pi}B\left(\frac{k_v}{2},\frac{k_v}{2}\right)\sum_{\substack{u\in \mathfrak{q}\mathfrak{n}^{-1}\\
u\not\in \{0,-1\}}}\mathcal{P}(u)\prod_{v<\infty}e_{v,\mathfrak{n}}(u),
\end{equation}
where $\mathcal{P}(u):=\prod_{v\mid\infty}\mathcal{P}_v(u)$ and $e_{v,\mathfrak{n}}$ are defined as follows.
\begin{itemize}
\item Let $v\mid\infty$ and $u\in F-\{0,-1\}$. 
\begin{itemize}
\item Suppose $\ |2u+1|_v\leq 1$. Then 
\begin{align*}
\mathcal{P}_v(u):=-\Big[P_{\frac{k_v}{2}-1}(2u+1)\log\Big|\frac{u}{u+1}\Big|_v +2\sum_{j=0}^{\floor{k_v/4}}\frac{k_v-4j+1}{(2j-1)(\frac{k_v}{2}-j)}P_{\frac{k_v}{2}-2j}(2u+1)\Big]
\end{align*}
\item Suppose $\ |2u+1|_v>1$. Then 
\begin{align*}
\mathcal{P}_v(u)=2 Q_{\frac{k_v}{2}-1}(2u+1).
\end{align*}
\end{itemize}
\item Let $v<\infty$. We define 
\begin{equation}\label{6.29}
e_{v,\mathfrak{n}}(u):=\begin{cases}
(e_v(\mathfrak{n})+1+e_v(1+u))(e_v(\mathfrak{n})+1+e_v(u)),\ \ & \text{if $v\nmid \mathfrak{q}$},\\
e_v(u),\ \ & \text{if $v=\mathfrak{q}$}.
\end{cases}
\end{equation}
\end{itemize}
\end{prop}
\begin{proof}
For $t\in F-\{0,1\}$, we write $u=t(1-t)^{-1}\in F-\{0,-1\}$. From the nonarchimedean analysis in Lemmas \ref{lemma1.3}--\ref{lem1.6}, we have, for $v<\infty$, 
\begin{equation}\label{1.35}
\prod_{v<\infty}J_v(t)=  N(\mathfrak{n})^{-\frac{1}{2}}\Vol(K_0(\mathfrak{q}))^{-1}\textbf{1}_u\cdot \prod_{v<\infty}e_{v,\mathfrak{n}},
\end{equation}
where $\textbf{1}_u$ is the characteristic function of the following constraints:
\begin{equation}\label{1.37}
\begin{cases}
e_v(u)\geq 0,\ \ & \text{$v\nmid \mathfrak{n}\mathfrak{q}$}\\
e_v(u)\geq 1,\ \ & \text{$v=\mathfrak{q}$}\\
e_v(u)\geq -e_v(\mathfrak{n}),\ \ & \text{$v\mid \mathfrak{n}$}.
\end{cases} 
\end{equation}

Note that \eqref{1.37} amounts to $u\in \mathfrak{q}\mathfrak{n}^{-1}$. Moreover, $t>0$ if and only if $|2u+1|=|(t+1)/(t-1)|>1$. Therefore, the formula \eqref{6.28} follows from \eqref{1.35} and Lemma \ref{lem1.7}.
\end{proof}

\begin{remark}\label{divisor function remark}
When $F = \mathbb{Q}$ and $\mathfrak{q} = \mathcal{O}_F$, the factor $\prod_{v < \infty} e_{v,\mathfrak{n}}(u)$ simplifies to a shifted convolution of divisor functions. This finding aligns with the convolution formula proved by Kuznetsov; see \cite[Theorem 17]{IS97} and \cite[Theorem 4.2]{BF21}. However, in this paper, we utilize slightly different archimedean weights $\mathcal{P}(u)$. This discrepancy arises from our use of the period integral representation for \textit{complete} $L$-functions, as opposed to the Dirichlet series representation for \textit{finite} $L $-functions employed in \cite{IS97}, \cite{BF21}, and other prior studies.   
\end{remark}

\subsection{Majorization of  $J_{\mathrm{reg}}(f_{\mathfrak{n},\mathfrak{q}},\textbf{0})$}\label{sec7.3}
\subsubsection{Properties of Legendre Functions}
Let $k\in \mathbb{Z}$. Let $Q_k$ be the Legendre function of the second kind. Let $x\in \mathbb{R}-[-1,1]$. By \cite[p.234]{MOS66}, we have the integral representation 
\begin{equation}\label{6.32}
Q_k(x)=2^{-k-1}\int_{-1}^1\frac{(1-t^2)^k}{(x-t)^{k+1}}dt.
\end{equation}

Let $P_k$ be the Legendre polynomial, $k\geq 1$. According to \cite[\textsection 5.4.4, p.237]{MOS66},
\begin{equation}\label{6.33}
|P_k(x)|<\sqrt{2\pi^{-1}}k^{-1/2}(1-x^2)^{-1/4},\ \ x\in [-1,1].
\end{equation} 

\subsubsection{Analysis of $\mathcal{P}_v(u)$}

Let $u\in \mathfrak{q}\mathfrak{n}^{-1}-\{0,-1\}$. Let $\mathcal{P}_v(u)$ be defined as in Proposition \ref{prop6.6}. We have the upper bounds for $|\mathcal{P}_v(u)|$ as follows. 
 \begin{lemma}\label{lem6.8}
Let notation be as before. Let $v\mid\infty$. Suppose $|2u+1|_v\leq 1$. Then 
\begin{equation}\label{6.34}
|\mathcal{P}_v(u)|\ll \frac{1}{k_v^{1/2}|u(u+1)|_v^{1/4+\varepsilon}},
\end{equation}
where the implied constant depends only on $\varepsilon$.	
\end{lemma}
\begin{proof}
Utilizing \eqref{6.33} we obtain 
\begin{align*}
|\mathcal{P}_v(u)|\leq & \bigg|P_{\frac{k_v}{2}-1}(2u+1)\log\Big|\frac{u}{u+1}\Big|_v +2\sum_{j=0}^{\floor{k_v/4}}\frac{k_v-4j+1}{(2j-1)(\frac{k_v}{2}-j)}P_{\frac{k_v}{2}-2j}(2u+1)\bigg|\\
\ll & \frac{k_v^{-1/2}|\log(|u|_v^{-1}-1
)|}{|u(u+1)|_v^{1/4}}+\sum_{j=0}^{\floor{k_v/4}}\frac{(k_v-4j+1)^{1/2}}{|(2j-1)(k_v/2-j)||u(u+1)|_v^{1/4}}+\frac{1}{k_v^2},
\end{align*}
where the term $k_v^{-2}$ measures the contribution from $j=\floor{k_v/4}$ when $k_v\in 4\mathbb{Z}$, along with the fact that $P_0(x)\equiv 1$. Estimate the sum over $j$ we obtain 
\begin{equation}\label{6.34.}
|\mathcal{P}_v(u)|\ll \frac{1+|\log(|u|_v^{-1}-1
)|}{k_v^{1/2}|u(u+1)|_v^{1/4}},\ \ \ |2u+1|_v\leq 1,
\end{equation}
where the implied constant is absolute. Let $\varepsilon>0$. Then for $|2u+1|_v\leq 1$, we have 
\begin{equation}\label{equa6.35}
|\log(|u|_v^{-1}-1
)|\ll_{\varepsilon} |u(u+1)|_v^{-\varepsilon},
\end{equation} 
where the implied constant depends only on $\varepsilon
$. Substituting \eqref{equa6.35} into \eqref{6.34.} leads to the desired estimate \eqref{6.34}.
\end{proof}

\begin{lemma}\label{lem6.9}
Let notation be as before. Let $v\mid\infty$ and $m_v\in \mathbb{Z}_{\geq 2}$. Suppose $2^{m_v-1}<|2u+1|_v\leq 2^{m_v}$. Then 
\begin{equation}\label{eq6.38}
|\mathcal{P}_v(u)|\ll \frac{2^{-k_v/2}k_v^{-1/2}}{|u(u+1)|_v^{1/2}(2^{m_v-1}-1)^{k_v/2-1}},
\end{equation}
where the implied constant is absolute.
\end{lemma}
\begin{proof}
Taking advantage of the integral representation  \eqref{6.32},
\begin{align*}
\mathcal{P}_v(u)=2 Q_{\frac{k_v}{2}-1}(2u+1)=2^{1-k_v/2}\int_{-1}^1\frac{(1-t^2)^{k_v/2-1}}{(2u+1-t)^{k_v/2}}dt.
\end{align*}

Now we consider the scenario that $|2u+1|_v> 1$. Then  there exists a unique integer $m_v\in \mathbb{Z}_{\geq 1}$ such that $2^{m_v-1}<|2u+1|_v\leq 2^{m_v}$. 

\begin{itemize}
\item If $2u+1$ is positive in $F_v$, then $2\leq 2^{m_v-1}<|2u+1|_v\leq 2^{m_v}$. Thus 
\begin{align*}
|\mathcal{P}_v(u)|=2^{1-k_v/2}\int_{-1}^1\frac{(1-t^2)^{k_v/2-1}}{(2u+1-t)^{k_v/2}}dt\leq \frac{2^{1-k_v/2}}{(2^{m_v-1}-1)^{k_v/2}}\int_{-1}^1(1-t^2)^{k_v/2-1}dt.
\end{align*}

By the change of variable $t\mapsto \sqrt{t}$, we derive
\begin{align*}
\int_{-1}^1(1-t^2)^{k_v/2-1}dt=\int_0^1(1-t)^{k_v/2-1}t^{-1/2}dt=B(1/2,k_v/2)=\frac{\sqrt{\pi}\cdot\Gamma(k_v/2)}{\Gamma(k_v/2+1/2)}.
\end{align*}
In conjunction with Stirling formula, we obtain  
\begin{equation}\label{6.36}
|\mathcal{P}_v(u)|\leq \frac{2^{1-k_v/2}\sqrt{\pi}}{(2^{m_v-1}-1)^{k_v/2}}\cdot \frac{\Gamma(k_v/2)}{\Gamma(k_v/2+1/2)}\ll \frac{2^{-k_v/2}}{(2^{m_v-1}-1)^{k_v/2}}\cdot k_v^{-1/2},
\end{equation}
where the implied constant is absolute.
\item If $2u+1$ is negative in $F_v$, then $-2^{m_v}\leq 2u+1\leq -2^{m_v-1}\leq -2$. Thus 
\begin{align*}
|\mathcal{P}_v(u)|=2^{1-\frac{k_v}{2}}\int_{-1}^1\frac{(1-t^2)^{\frac{k_v}{2}-1}}{(-2(u+1)+1+t)^{\frac{k_v}{2}}}dt\leq \frac{2^{1-\frac{k_v}{2}}}{(2^{m_v-1}-1)^{\frac{k_v}{2}}}\int_{-1}^1(1-t^2)^{\frac{k_v}{2}-1}dt,
\end{align*}
which yields the same bound as \eqref{6.36}. 

Consider the following cases:
\begin{itemize}
\item If $m_v=2$, then it follows from $2^{m_v-1}<|2u+1|_v\leq 2^{m_v}$ that $3/4\leq |u(u+1)|_v\leq 15/4$. In this case \eqref{6.36} boils down to 
\begin{equation}\label{f6.38}
|\mathcal{P}_v(u)|\ll \frac{2^{-k_v/2}}{|u(u+1)|_v^{1/2}(2^{m_v-1}-1)^{k_v/2-1}}\cdot k_v^{-1/2},
\end{equation}
where the implied constant is absolute.

\item If $m_v>2$, then the constraint $2^{m_v-1}<|2u+1|_v\leq 2^{m_v}$ implies that $|u(u+1)|_v\geq 15/4$, and $|u(u+1)|_v^{1/2}\leq 2^{m_v-1}$. So \eqref{6.36} becomes 
\begin{equation}\label{f6.39}
|\mathcal{P}_v(u)|\ll \frac{2^{-\frac{k_v}{2}}k_v^{-\frac{1}{2}}}{(|u(u+1)|_v-1)^{\frac{1}{2}}(2^{m_v-1}-1)^{\frac{k_v}{2}-1}}\ll \frac{2^{-\frac{k_v}{2}}k_v^{-\frac{1}{2}}|u(u+1)|_v^{-\frac{1}{2}}}{(2^{m_v-1}-1)^{\frac{k_v}{2}-1}},
\end{equation}
where the implied constant is absolute. 
\end{itemize}
\end{itemize}

Combining \eqref{f6.38} with \eqref{f6.39} we then conclude \eqref{eq6.38}.
\end{proof}

\begin{lemma}\label{6.10}
Let notation be as before. Let $v\mid\infty$. Suppose $1<|2u+1|_v\leq 2$. Then 
\begin{equation}\label{eq6.36}
|\mathcal{P}_v(u)|\ll |u(u+1)|_v^{-1/2}k_v^{-1/2},
\end{equation}
where the implied constant is absolute.
\end{lemma}
\begin{proof}
We consider the following cases according to the sign of $2u+1$.
\begin{itemize}
\item Suppose $2u+1$ is positive in $F_v$. It follows from $1<|2u+1|_v\leq 2$ that $0<u<1/2$ in $F_v$. Utilizing \eqref{6.32}, we obtain 
\begin{equation}\label{6.41}
|\mathcal{P}_v(u)|=2^{1-\frac{k_v}{2}}\int_{-1}^1\frac{(1-t^2)^{\frac{k_v}{2}-1}}{(2u+1-t)^{\frac{k_v}{2}}}dt\leq \frac{2^{1-\frac{k_v}{2}}}{u^{1/2}}\int_{-1}^1\frac{(1-t^2)^{\frac{k_v}{2}-1}}{(1-t)^{\frac{k_v}{2}-1/2}}dt.
\end{equation}

Changing the variable $1+t\mapsto t-1$, we derive 
\begin{equation}\label{6.42}
2^{1-\frac{k_v}{2}}\int_{-1}^1\frac{(1+t)^{\frac{k_v}{2}-1}}{(1-t)^{1/2}}dt=2^{1-\frac{k_v}{2}}\int_{0}^2\frac{t^{\frac{k_v}{2}-1}}{(2-t)^{1/2}}dt\ll \int_{0}^1\frac{t^{\frac{k_v}{2}-1}}{(1-t)^{1/2}}dt.
\end{equation}

Gathering \eqref{6.41} with \eqref{6.42}, along with the asymptotic behavior of Beta functions, we obtain 
\begin{equation}\label{6.43}
|\mathcal{P}_v(u)|\ll u^{-1/2}k_v^{-1/2},
\end{equation}
where the implied constant is absolute. 

\item Suppose $2u+1$ is negative in $F_v$. It follows from $1<|2u+1|_v\leq 2$ that $-3/2<u<-1$ in $F_v$. Utilizing \eqref{6.32}, we obtain 
\begin{align*}
|\mathcal{P}_v(u)|=2^{1-\frac{k_v}{2}}\int_{-1}^1\frac{(1-t^2)^{\frac{k_v}{2}-1}}{(-2(u+1)+1+t)^{\frac{k_v}{2}}}dt\leq \frac{2^{1-\frac{k_v}{2}}}{|u+1|_v^{1/2}}\int_{-1}^1\frac{(1-t^2)^{\frac{k_v}{2}-1}}{(1+t)^{\frac{k_v}{2}-1/2}}dt,
\end{align*}
which is $\ll |u+1|_v^{-1}k_v^{-1}$, with the implied constant being absolute.
\end{itemize}

Combining the above discussions we the conclude \eqref{eq6.36}.
\end{proof}

\subsubsection{Majorization of  $J_{\mathrm{reg}}(f_{\mathfrak{n},\mathfrak{q}},\textbf{0})$}\label{sec7.3.3}
Let $\textbf{m}=(m_v)_{v\mid\infty}\in \mathbb{Z}_{\geq 0}^{d_F}$. For $v\mid\infty$, we define the set of algebraic numbers
\begin{equation}\label{6.36.}
\mathfrak{S}_{\textbf{m}}:=\big\{u=(u_v)_{v\leq\infty}\in \mathfrak{q}\mathfrak{n}^{-1}-\{0,-1\}:\ u_v\in I_{m_v},\ v\mid\infty\big\},
\end{equation}
where  
\begin{align*}
\begin{cases}
I_{m_v}:=\big\{u\in \mathfrak{q}\mathfrak{n}^{-1}-\{0,-1\}:\ |2u+1|_v\leq 1\big\},\ \ &\text{if $m_v=0$,}\\
I_{m_v}:=\big\{u\in \mathfrak{q}\mathfrak{n}^{-1}-\{0,-1\}:\ 2^{m_v-1}<|2u+1|_v\leq 2^{m_v}\big\},\ \ &\text{if $m_v\geq 1$.}
\end{cases}
\end{align*}

By definition, we have the partition
\begin{equation}\label{6.37}
\mathfrak{q}\mathfrak{n}^{-1}-\{0,-1\}=\bigsqcup_{\textbf{m}=(m_v)_{v\mid\infty}\in \mathbb{Z}_{\geq 0}^{d_F}}\mathfrak{S}_{\textbf{m}}.
\end{equation}

\begin{lemma}\label{lem6.11}
Let $u\in \mathfrak{S}_{\textbf{m}}$, where $\textbf{m}=(m_v)_{v\mid\infty}\in \mathbb{Z}_{\geq 0}^{d_F}$. Let $m\in \mathbb{Z}_{\geq 1}$. Then 
\begin{equation}\label{6.40}
\sum_{\substack{u\in\mathfrak{S}_{\textbf{m}}\\
N(u)=N(\mathfrak{q})N(\mathfrak{n})^{-1}m}}1\ll \left(\frac{N(\mathfrak{n})}{mN(\mathfrak{q})}\right)^{\varepsilon}\prod_{v\mid\infty}(m_v+1),
\end{equation}
where the implied constant depends only on $\varepsilon$ and $F$.
\end{lemma}
\begin{proof}
Write $\mathfrak{a}=\mathfrak{q}^{-1}\mathfrak{n}$. Since $u\in \mathfrak{S}_{\textbf{m}}$, then $(u)\mathfrak{a}$ is an integral ideal with norm $m$. Note that there are $O(m^{\varepsilon})$ integral ideals $\{\mathfrak{b}_j\}$ of norm $m$, where the implied constant is absolute. Then 
\begin{equation}\label{6.38}
\sum_{\substack{u\in\mathfrak{S}_{\textbf{m}}\\
N(u)=N(\mathfrak{q})N(\mathfrak{n})^{-1}m}}1=\sum_j\sum_{\substack{u\in\mathfrak{S}_{\textbf{m}}\\
(u)\mathfrak{a}=\mathfrak{b}_j}}1.
\end{equation}
So $\mathfrak{a}^{-1}\mathfrak{b}_j=(u)$ is a principal ideal. Write $\mathfrak{a}^{-1}\mathfrak{b}_j=(b_j)$ for some $b_j\in F^{\times}$. Then 
\begin{equation}\label{6.39}
\sum_{\substack{u\in\mathfrak{S}_{\textbf{m}}\\
(u)\mathfrak{a}=\mathfrak{b}_j}}1=\sum_{\substack{u\in\mathfrak{S}_{\textbf{m}}\\
u\in b_j\mathcal{O}_F^{\times}}}1=\sum_{\substack{\gamma
\in \mathcal{O}_F^{\times}\\
\gamma\in b_j^{-1}\mathfrak{S}_{\textbf{m}}}}1\ll \prod_{v\mid\infty}\bigg|\log (|b_j|_v^{-1}2^{m_v+1})\bigg|,
\end{equation}
where the last inequality is a consequence of the Dirichlet unit theorem and a volume calculation under the logarithmic map, and the implied constant depends on the base field $F$.

Therefore, \eqref{6.40} follows from \eqref{6.38}, \eqref{6.39}, and the fact that $N(b_j)=N(u)=N(\mathfrak{q})N(\mathfrak{n})^{-1}m$. 
\end{proof}

\begin{proof}[Proof of Proposition \ref{prop6.12}]
As a consequence of the decomposition \eqref{6.37}, we have
\begin{equation}\label{6.50}
\mathcal{P}_{\mathfrak{q},\mathfrak{n}}:=\sum_{\substack{u\in \mathfrak{q}\mathfrak{n}^{-1}\\
u\not\in \{0,-1\}}}\mathcal{P}(u)\prod_{v<\infty}e_{v,\mathfrak{n}}(u)=\sum_{\textbf{m}=(m_v)_{v\mid\infty}\in \mathbb{Z}_{\geq 0}^{d_F}}\sum_{\substack{u\in \mathfrak{S}_{\textbf{m}}}}\mathcal{P}(u)\prod_{v<\infty}e_{v,\mathfrak{n}}(u),
\end{equation}
where $\mathcal{P}(u)$ and $e_{v,\mathfrak{n}}(u)$ are defined as in Proposition \ref{prop6.6}. 

Since $N(1+u)\leq N(u)$ amounts to $N(-(1+u))=N(1+u)\leq N(-(1+u)+1)$. Hence, upon the changing of variable $u\mapsto -(1+u)$, which is invariant in the domain $F-\{0,-1\}$, we derive from \eqref{6.50} that
\begin{equation}\label{6.51}
\big|\mathcal{P}_{\mathfrak{q},\mathfrak{n}}\big|\leq \sum_{\textbf{m}=(m_v)_{v\mid\infty}\in \mathbb{Z}_{\geq 0}^{d_F}}\sum_{\substack{u\in \mathfrak{S}_{\textbf{m}}\\ N(u)\leq N(1+u)}}\mathcal{P}(u)\prod_{v<\infty}e_{v,\mathfrak{n}}(u).
\end{equation}

Let $\textbf{m}=(m_v)_{v\mid\infty}\in \mathbb{Z}_{\geq 0}^{d_F}$, and  $u\in\mathfrak{S}_{\textbf{m}}$. Then $
N(u)=N(\mathfrak{q})N(\mathfrak{n})^{-1}m$ for some positive integer $m$. By Lemmas \ref{lem6.8}, \ref{lem6.9}, and \ref{6.10}, we obtain 
\begin{align*}
|\mathcal{P}(u)|\ll \prod_{\substack{v\mid\infty\\ m_v=0}}\frac{k_v^{-1/2}}{|u(u+1)|_v^{\frac{1}{4}+\varepsilon}}\cdot \prod_{\substack{v\mid\infty\\ m_v=1}}\frac{k_v^{-1/2}}{|u(u+1)|_v^{\frac{1}{2}}}\prod_{\substack{v\mid\infty\\ m_v\geq 2}}\frac{k_v^{-1/2}}{|u(u+1)|_v^{\frac{1}{2}}(2^{m_v}-2)^{\frac{k_v}{2}-1}},
\end{align*}
where the implied constant is absolute. Notice that 
\begin{align*}
\prod_{\substack{v\mid\infty,\ m_v=0}}|u(u+1)|_v^{\frac{1}{4}-\varepsilon}\ll 1.
\end{align*}
Substituting this into the above bound for $|\mathcal{P}(u)|$ leads to
\begin{equation}\label{6.52}
|\mathcal{P}(u)|\ll \frac{1}{N(u(u+1))^{1/2}}\prod_{v\mid\infty}k_v^{-1/2}\prod_{\substack{v\mid\infty\\ m_v\geq 2}}\frac{1}{(2^{m_v}-2)^{k_v/2-1}}
\end{equation}
for $u\in\mathfrak{S}_{\textbf{m}}$. Moreover, according to the definition \eqref{6.29}, we have, for $u\in\mathfrak{S}_{\textbf{m}}$, and for $\varepsilon>0$, that 
\begin{equation}\label{6.54}
|e_{v,\mathfrak{n}}(u)|\ll N((u)\mathfrak{n})^{\varepsilon}\ll N(\mathfrak{n})^{\varepsilon}N(u)^{\varepsilon},
\end{equation}
where the implied constant depends at most on $\varepsilon$ and $F$. 

Substituting \eqref{6.52} and \eqref{6.54} into \eqref{6.51} we obtain 
\begin{equation}\label{6.55}
\big|\mathcal{P}_{\mathfrak{q},\mathfrak{n}}\big|\ll \sum_{\textbf{m}=(m_v)_{v\mid\infty}\in \mathbb{Z}_{\geq 0}^{d_F}}\sum_{u\in \mathfrak{S}_{\textbf{m}}}\frac{N(\mathfrak{n})^{\varepsilon}}{N(u)^{1-\varepsilon}}\prod_{v\mid\infty}k_v^{-\frac{1}{2}}\prod_{\substack{v\mid\infty\\ m_v\geq 2}}\frac{1}{2^{m_v}},
\end{equation}
where we have made use of the inequality that $(2^{m_v}-2)^{k_v/2-1}\geq 2^{m_v-1}$, which follows from the assumption that $k_v\geq 4$ for each $v\mid\infty$. 

Let $u\in\mathfrak{S}_{\textbf{m}}$, we write $N(u)=N(\mathfrak{q}\mathfrak{n}^{-1})m$ for a unique $m\in \mathbb{Z}_{\geq 1}$. Then 
\begin{align*}
N(\mathfrak{q}\mathfrak{n}^{-1})m=\prod_{v\mid\infty}|u|_v\ll \prod_{v\mid\infty,\ m_v\geq 2}(2^{m_v-1}+1/2)\leq \prod_{v\mid\infty,\ m_v\geq 2}2^{m_v},
\end{align*}
where the implied constant is absolute. Hence, 
\begin{equation}\label{6.56}
\prod_{v\mid\infty,\ m_v\geq 2}2^{3\varepsilon m_v}\gg  N(\mathfrak{q}\mathfrak{n}^{-1})^{3\varepsilon}m^{3\varepsilon},
\end{equation}
where the implied constant depends on $d_F$, the degree of $F/\mathbb{Q}$. So \eqref{6.55} implies 
\begin{equation}\label{6.57}
\big|\mathcal{P}_{\mathfrak{q},\mathfrak{n}}\big|\ll \frac{N(\mathfrak{n})^{1+4\varepsilon}}{N(\mathfrak{q})^{1+3\varepsilon}}\sum_{m\geq 1}\sum_{\textbf{m}}\sum_{\substack{u\in \mathfrak{S}_{\textbf{m}}\\ N(u)=N(\mathfrak{q}\mathfrak{n}^{-1})m}}\frac{1}{m^{1+2\varepsilon}}\prod_{v\mid\infty}k_v^{-\frac{1}{2}}\prod_{\substack{v\mid\infty\\ m_v\geq 2}}\frac{1}{2^{(1-3\varepsilon)m_v}},
\end{equation}
where $\textbf{m}=(m_v)_{v\mid\infty}\in \mathbb{Z}_{\geq 0}^{d_F}$. 

Utilizing Lemma \ref{lem6.11} to bound the inner sum in \eqref{6.57}, we deduce 
\begin{equation}\label{6.59}
\big|\mathcal{P}_{\mathfrak{q},\mathfrak{n}}\big|\ll \frac{N(\mathfrak{n})^{1+4\varepsilon}}{N(\mathfrak{q})^{1+3\varepsilon}}\prod_{v\mid\infty}k_v^{-\frac{1}{2}}\sum_{m\geq 1}\frac{1}{m^{1+2\varepsilon}}\prod_{\substack{v\mid\infty}}\bigg[\sum_{m_v\geq 0}\frac{m_v+1}{2^{(1-3\varepsilon)m_v}}\bigg]\ll \frac{N(\mathfrak{n})^{1+4\varepsilon}}{N(\mathfrak{q})^{1+3\varepsilon}}\prod_{v\mid\infty}k_v^{-\frac{1}{2}},
\end{equation}
where the implied constant depends at most on $\varepsilon$ and $F$.  

Therefore, the estimate \eqref{6.49} follows from \eqref{6.59} and the expression \eqref{6.28} in Proposition \ref{prop6.6}.
\end{proof}

\section{The Mollified Second Moment}\label{sec8}

\subsection{The $\lambda_{\pi}(\mathfrak{n})$-weighted Second Moment}\label{sec8.1}
\begin{thmx}\label{thm7.1}
Let notation be as before. Let $\mathbf{k}=(k_v)_{v\mid\infty}\in \mathbb{Z}_{>2}^{d_F}$, where $k_v$ is even, $v\mid\infty$. Let $\mathfrak{q}$ be either $\mathcal{O}_F$ or a prime ideal. Let $\mathfrak{n}\subseteq \mathcal{O}_F$ be an integral ideal with $(\mathfrak{n},\mathfrak{q})=1$. Let $\chi_{\mathfrak{q}}$ be the nontrivial unramified quadratic character of $F_{\mathfrak{q}}^{\times}$ if $\mathfrak{q}\subsetneq\mathcal{O}_F$. Let $\varepsilon>0$, and $\mathcal{C}_{\varepsilon}:=\big\{z\in\mathbb{C}:\ |z|=\varepsilon\big\}$. Then 
\begin{align*}
&\frac{1}{\zeta_{\mathfrak{q}}(2)^2}\sum_{\substack{\pi\in \mathcal{F}(\mathbf{k},\mathfrak{q})}}\frac{\lambda_{\pi}(\mathfrak{n})L(1/2,\pi)^2}{L^{(\mathfrak{q})}(1,\pi,\Ad)}+\frac{2V_{\mathfrak{q}}\textbf{1}_{\mathfrak{q}\subsetneq \mathcal{O}_F}}{N(\mathfrak{q})}\sum_{\substack{\pi\in \mathcal{F}(\mathbf{k},\mathcal{O}_F)}}\frac{\lambda_{\pi}(\mathfrak{n})L_{\pi_{\mathfrak{q}}}L(1/2,\pi)^2}{L(1,\pi,\Ad)}\\
=&\frac{(N(\mathfrak{q})+1)\cdot \delta_{\mathbf{k},\mathfrak{q}}}{2\pi i}\oint_{\mathcal{C}_{\varepsilon}}\frac{\zeta_F(1+s)G_{\mathfrak{n},\mathfrak{q}}(s)}{s}ds+\frac{2D_FV_{\mathfrak{q}}}{N(\mathfrak{n})^{\frac{1}{2}}}\prod_{v\mid\infty}\frac{k_v-1}{2\pi}\sum_{\substack{u\in \mathfrak{q}\mathfrak{n}^{-1}\\
u\not\in \{0,-1\}}}\mathcal{P}(u)\prod_{v<\infty}e_{v,\mathfrak{n}}(u),
\end{align*}
where we define $\zeta_{\mathfrak{q}}(2)=1$ and $L^{(\mathfrak{q})}(1,\pi,\Ad)=L(1,\pi,\Ad)$ if $\mathfrak{q}=\mathcal{O}_F$, $V_{\mathfrak{q}}$ is defined by \eqref{1.1}, $\delta_{\mathbf{k},\mathfrak{q}}=\textbf{1}_{\mathfrak{q}\subsetneq \mathcal{O}_F}+\textbf{1}_{\mathfrak{q}=\mathcal{O}_F  \& \sum_{v\mid\infty}k_v\equiv 0\pmod{4}}$,
$L_{\pi_{\mathfrak{q}}}:=L_{\mathfrak{q}}(1/2,\pi_{\mathfrak{q}}\times\chi_{\mathfrak{q}})$, and 
\begin{equation}\label{G}
G_{\mathfrak{n},\mathfrak{q}}(s):=2(1+N(\mathfrak{q})^s)D_F^{\frac{3}{2}+s}\prod_{v\mid\infty}\frac{(k_v-1)\Gamma((k_v+s)/2)^2}{2\cdot (2\pi)^{1+s}\pi\Gamma(k_v/2)^2}\cdot \frac{\tau(\mathfrak{n})}{N(\mathfrak{n})^{(1+s)/2}}.
\end{equation}	
\end{thmx}
\begin{proof}
Utilizing the regularized relative trace formula we obtain 
\begin{align*}
J_{\mathrm{Spec}}(f_{\mathfrak{n},\mathfrak{q}},\textbf{0})=J_{\mathrm{sing}}(f_{\mathfrak{n},\mathfrak{q}},\textbf{0})+J_{\mathrm{reg}}(f_{\mathfrak{n},\mathfrak{q}},\textbf{0}).\tag{\ref{1.13}}
\end{align*}

By Theorem \ref{spec} in \textsection\ref{sec2.3}, we obtain 
\begin{align*}
J_{\mathrm{Spec}}(f_{\mathfrak{n},\mathfrak{q}},\textbf{0})=J_{\mathrm{Spec}}^{\mathrm{new}}(f_{\mathfrak{n},\mathfrak{q}},\textbf{0})+J_{\mathrm{Spec}}^{\mathrm{old}}(f_{\mathfrak{n},\mathfrak{q}},\textbf{0}),\tag{\ref{f2.15}}
\end{align*} 
where  
\begin{align*}
J_{\mathrm{Spec}}^{\mathrm{new}}(f_{\mathfrak{n},\mathfrak{q}},\textbf{0})=&\frac{1}{2D_F^2}\prod_{v\mid\infty}\frac{2^{k_v}\pi\Gamma(k_v/2)^2}{\Gamma(k_v)}\sum_{\substack{\pi\in \mathcal{F}(\mathbf{k},\mathfrak{q})}}\frac{\lambda_{\pi}(\mathfrak{n})L(1/2,\pi)^2}{\zeta_{\mathfrak{q}}(2)^2L^{(\mathfrak{q})}(1,\pi,\Ad)},\\
J_{\mathrm{Spec}}^{\mathrm{old}}(f_{\mathfrak{n},\mathfrak{q}},\textbf{0})=&\frac{1}{2D_F^2}\prod_{v\mid\infty}\frac{2^{k_v}\pi \Gamma(k_v/2)^2}{\Gamma(k_v)}\sum_{\substack{\pi\in \mathcal{F}(\mathbf{k},\mathcal{O}_F)}}\frac{\lambda_{\pi}(\mathfrak{n})C_{\pi_{\mathfrak{q}}}(\textbf{0})L(1/2,\pi)^2}{L(1,\pi,\Ad)}.
\end{align*}
Here $C_{\pi_{\mathfrak{q}}}(\textbf{s})$ is defined as in \eqref{equ2.18} in Lemma \ref{lem2.4}. Hence, 
\begin{align*}
C_{\pi_{\mathfrak{q}}}(\textbf{0})=&1+L_{\mathfrak{q}}(1/2,\pi_{\mathfrak{q}})L_{\mathfrak{q}}(1/2,\pi_{\mathfrak{q}}\times\chi_{\mathfrak{q}})\prod_{j=1}^2(1+N(\mathfrak{q})^{-1}-\lambda_{\pi}(\mathfrak{q})N(\mathfrak{q})^{-1/2})\\
=&1+\frac{L_{\mathfrak{q}}(1/2,\pi_{\mathfrak{q}}\times\chi_{\mathfrak{q}})}{L_{\mathfrak{q}}(1/2,\pi_{\mathfrak{q}})}=1+\frac{1+N(\mathfrak{q})^{-1}-\lambda_{\pi}(\mathfrak{q})N(\mathfrak{q})^{-1/2}}{1+N(\mathfrak{q})^{-1}+\lambda_{\pi}(\mathfrak{q})N(\mathfrak{q})^{-1/2}}.
\end{align*}

Therefore, we derive that
\begin{equation}\label{equ7.1}
C_{\pi_{\mathfrak{q}}}(\textbf{0})=2(1+N(\mathfrak{q})^{-1})L_{\mathfrak{q}}(1/2,\pi_{\mathfrak{q}}\times\chi_{\mathfrak{q}})=2V_{\mathfrak{q}}N(\mathfrak{q})^{-1}L_{\mathfrak{q}}(1/2,\pi_{\mathfrak{q}}\times\chi_{\mathfrak{q}}).	
\end{equation}

Recall the description of in Proposition \ref{prop5.3} in \textsection\ref{sec5}:
\begin{equation}\label{c7.2}
J_{\mathrm{sing}}(f_{\mathfrak{n},\mathfrak{q}},\textbf{0})=\frac{N(\mathfrak{q})+1}{2}\cdot\frac{1}{2\pi i}\oint_{\mathcal{C}_{\varepsilon}}\frac{\zeta_F(1+s)H_{\mathfrak{n}}(s)(1+N(\mathfrak{q})^s)}{s}ds,	
\end{equation}
where $H_{\mathfrak{n}}(s)$ is defined as in \eqref{eq5.2}, namely,
\begin{align*}
H_{\mathfrak{n}}(s):=2\prod_{v\mid\infty}\frac{2^{k_v}(k_v-1)\Gamma((k_v+s)/2)^2}{2\cdot (2\pi)^{1+s}\Gamma(k_v)}\cdot \frac{\tau(\mathfrak{n})}{D_F^{1/2-s}N(\mathfrak{n})^{(1+s)/2}}.
\end{align*}

By Proposition \ref{prop6.6} in \textsection\ref{sec1.3.5}, $J_{\mathrm{reg}}(f_{\mathfrak{n},\mathfrak{q}},\textbf{0})$ is equal to 
\begin{align*}
\frac{V_{\mathfrak{q}}D_F^{-1}}{N(\mathfrak{n})^{\frac{1}{2}}}\prod_{v\mid\infty}\frac{2^{k_v-1}(k_v-1)}{\pi}B\left(\frac{k_v}{2},\frac{k_v}{2}\right)\sum_{\substack{u\in \mathfrak{q}\mathfrak{n}^{-1}\\
u\not\in \{0,-1\}}}\mathcal{P}(u)\prod_{v<\infty}e_{v,\mathfrak{n}}(u),\tag{\ref{6.28}}
\end{align*}
where $\mathcal{P}(u)$ are defined via Legendre functions (or polynomials), and $e_{v,\mathfrak{n}}(u)$ is the evaluation function (cf. Proposition \ref{prop6.6}).

Therefore, Theorem \ref{thm7.1} follows from substituting \eqref{f2.15}, \eqref{equ7.1}, \eqref{c7.2}, and \eqref{6.28} into the relative trace formula \eqref{1.13}.
\end{proof}

\begin{cor}\label{cor7.2}
Let notation be as before. Let $\mathbf{k}=(k_v)_{v\mid\infty}\in \mathbb{Z}_{>2}^{d_F}$, where $k_v$ is even, $v\mid\infty$. Let $\mathfrak{q}$ be either $\mathcal{O}_F$ or a prime ideal. Let $\mathfrak{n}\subseteq \mathcal{O}_F$ be an integral ideal with $(\mathfrak{n},\mathfrak{q})=1$. Let $\chi_{\mathfrak{q}}$ be the nontrivial unramified quadratic character of $F_{\mathfrak{q}}^{\times}$ if $\mathfrak{q}\subsetneq\mathcal{O}_F$. We have
\begin{align*}
&\frac{1}{\zeta_{\mathfrak{q}}(2)^2}\sum_{\substack{\pi\in \mathcal{F}(\mathbf{k},\mathfrak{q})}}\frac{\lambda_{\pi}(\mathfrak{n})L(1/2,\pi)^2}{L^{(\mathfrak{q})}(1,\pi,\Ad)}+\frac{2V_{\mathfrak{q}}\textbf{1}_{\mathfrak{q}\subsetneq \mathcal{O}_F}}{N(\mathfrak{q})}\sum_{\substack{\pi\in \mathcal{F}(\mathbf{k},\mathcal{O}_F)}}\frac{\lambda_{\pi}(\mathfrak{n})L_{\pi_{\mathfrak{q}}}L(1/2,\pi)^2}{L(1,\pi,\Ad)}\\
=&\frac{(N(\mathfrak{q})+1)\cdot \delta_{\mathbf{k},\mathfrak{q}}}{2\pi i}\oint_{\mathcal{C}_{\varepsilon}}\frac{\zeta_F(1+s)G_{\mathfrak{n},\mathfrak{q}}(s)}{s}ds+O(N(\mathfrak{n})^{1/2+\varepsilon}N(\mathfrak{q})^{\varepsilon}\|\mathbf{k}\|^{1/2+\varepsilon}),
\end{align*}
where $\|\mathbf{k}\|:=\prod_{v\mid\infty}k_v$, and the implied constant depends only on $\varepsilon$ and $F$. 
\end{cor}
\begin{proof}
Corollary \ref{cor7.2} follows by a similar argument to the proof of Theorem \ref{thm7.1}, with the replacement of \eqref{6.28} therein by Proposition \ref{prop6.12} in \textsection\ref{sec1.3.5}.
\end{proof}

\subsection{Contribution From Old Forms}\label{sec8.2}
Let $\mathfrak{q}$ be a prime ideal. Let $\mathfrak{n}\subseteq \mathcal{O}_F$ be an integral ideal with $(\mathfrak{n},\mathfrak{q})=1$. Define 
\begin{equation}\label{y7.4}
J_{\mathrm{Spec}}^{\mathrm{old}}(\mathfrak{n}):=\frac{2V_{\mathfrak{q}}}{N(\mathfrak{q})}\sum_{\substack{\pi\in \mathcal{F}(\mathbf{k},\mathcal{O}_F)}}\frac{\lambda_{\pi}(\mathfrak{n})L_{\pi_{\mathfrak{q}}}L(1/2,\pi)^2}{L(1,\pi,\Ad)}.
\end{equation}
where  $L_{\pi_{\mathfrak{q}}}:=L_{\mathfrak{q}}(1/2,\pi_{\mathfrak{q}}\times\chi_{\mathfrak{q}})$. Here $\chi_{\mathfrak{q}}$ is the nontrivial unramified quadratic character of $F_{\mathfrak{q}}^{\times}$.
\begin{lemma}\label{lem7.3}
Let notation be as before. Let $\mathfrak{q}$ be a prime ideal. Let $J_{\mathrm{Spec}}^{\mathrm{old}}(\mathfrak{n})$ be defined as in \eqref{y7.4}. Let $0<\varepsilon<10^{-3}$. Then 
\begin{align*}
J_{\mathrm{Spec}}^{\mathrm{old}}(\mathfrak{n})=\frac{2V_{\mathfrak{q}}}{N(\mathfrak{q})}\cdot\frac{2\delta_{\mathbf{k}}}{2\pi i}\oint_{\mathcal{C}_{\varepsilon}}\frac{\zeta_F(1+s)G_{\mathfrak{n},\mathcal{O}_F}(s)}{s(1+N(\mathfrak{q})^{-1-s/2})^2}ds+O(N(\mathfrak{n})^{1/2+\varepsilon}N(\mathfrak{q})^{\varepsilon}\|\mathbf{k}\|^{1/2+\varepsilon}),
\end{align*}
where $\delta_{\mathbf{k}}:=\textbf{1}_{\sum_{v\mid\infty}k_v\equiv 0\pmod{4}}$, and the implied constant depends only on $\varepsilon$ and $F$. 
\end{lemma}
\begin{proof}
Let $m_0>1$ be a sufficient large integer. Hence, 
\begin{equation}\label{c7.3}
J_{\mathrm{Spec}}^{\mathrm{old}}(\mathfrak{n})=J_{\mathrm{Spec}}(\mathfrak{n},m_0)^++J_{\mathrm{Spec}}(\mathfrak{n},m_0)^-,
\end{equation}
where 
\begin{align*}
J_{\mathrm{Spec}}(\mathfrak{n},m_0)^+:=&\frac{2V_{\mathfrak{q}}}{N(\mathfrak{q})}\sum_{m>m_0}\frac{(-1)^m}{N(\mathfrak{q})^{m/2}}\sum_{\substack{\pi\in \mathcal{F}(\mathbf{k},\mathcal{O}_F)}}\frac{\lambda_{\pi}(\mathfrak{n}\mathfrak{q}^m)L(1/2,\pi)^2}{L(1,\pi,\Ad)},\\
J_{\mathrm{Spec}}(\mathfrak{n},m_0)^-:=&\frac{2V_{\mathfrak{q}}}{N(\mathfrak{q})}\sum_{0\leq m\leq m_0}\frac{(-1)^m}{N(\mathfrak{q})^{m/2}}\sum_{\substack{\pi\in \mathcal{F}(\mathbf{k},\mathcal{O}_F)}}\frac{\lambda_{\pi}(\mathfrak{n}\mathfrak{q}^m)L(1/2,\pi)^2}{L(1,\pi,\Ad)}.
\end{align*}

By Corollary \ref{cor7.2} with $\mathfrak{q}=\mathcal{O}_F$, and $\mathfrak{n}$ replaced by $\mathfrak{n}\mathfrak{q}^m$, we obtain 
\begin{equation}\label{fc7.3}
\sum_{\substack{\pi\in \mathcal{F}(\mathbf{k},\mathcal{O}_F)}}\frac{\lambda_{\pi}(\mathfrak{n}\mathfrak{q}^m)L(1/2,\pi)^2}{L(1,\pi,\Ad)}
=\frac{2\delta_{\mathbf{k}}}{2\pi i}\oint_{\mathcal{C}_{\varepsilon}}\frac{\zeta_F(1+s)G_{\mathfrak{n}\mathfrak{q}^m,\mathcal{O}_F}(s)}{s}ds+E(m),
\end{equation}
where 
\begin{equation}\label{c7.5}
E(m)\ll N(\mathfrak{n}\mathfrak{q}^m)^{1/2+\varepsilon}\|\mathbf{k}\|^{1/2+\varepsilon},	
\end{equation}
with the implied constant depending only on $F$ and $\varepsilon$. 

By \eqref{fc7.3} and \eqref{c7.5}, we obtain 
\begin{equation}\label{x7.6}
J_{\mathrm{Spec}}(\mathfrak{n},m_0)^-=\mathcal{M}(m_0)+O(m_0N(\mathfrak{n})^{1/2+\varepsilon}N(\mathfrak{q})^{m_0\varepsilon}\|\mathbf{k}\|^{1/2+\varepsilon}),
\end{equation}
where the implied constant depends only on $\varepsilon$ and $F$, and 
\begin{equation}\label{x7.8}
\mathcal{M}(m_0):=\frac{2V_{\mathfrak{q}}}{N(\mathfrak{q})}\sum_{0\leq m\leq m_0}\frac{(-1)^m}{N(\mathfrak{q})^{m/2}}\cdot \frac{2\delta_{\mathbf{k}}}{2\pi i}\oint_{\mathcal{C}_{\varepsilon}}\frac{\zeta_F(1+s)G_{\mathfrak{n}\mathfrak{q}^m,\mathcal{O}_F}(s)}{s}ds.
\end{equation}

By the definition of $G_{\mathfrak{n}\mathfrak{q}^m,\mathcal{O}_F}$ in \eqref{G}, and the fact that $(\mathfrak{n},\mathfrak{q})=1$, we have
\begin{align*}
G_{\mathfrak{n}\mathfrak{q}^m,\mathcal{O}_F}(s)=G_{\mathfrak{n},\mathcal{O}_F}(s)\cdot \frac{\tau(\mathfrak{q}^m)}{N(\mathfrak{q})^{(1+s)m/2}}.
\end{align*}
Plugging this into \eqref{x7.8}, together with $\tau(\mathfrak{q}^m)=m+1$, we deduce  
\begin{equation}\label{x7.9.}
\mathcal{M}(m_0)=\frac{2V_{\mathfrak{q}}}{N(\mathfrak{q})}\cdot\frac{2\delta_{\mathbf{k}}}{2\pi i}\oint_{\mathcal{C}_{\varepsilon}}\sum_{0\leq m\leq m_0} \frac{(-1)^m(m+1)}{N(\mathfrak{q})^{(1+s/2)m}}\cdot \frac{\zeta_F(1+s)G_{\mathfrak{n},\mathcal{O}_F}(s)}{s}ds.
\end{equation}

For $
\Re(s)\geq -\varepsilon$, we have 
\begin{equation}\label{x7.10.}
\sum_{0\leq m\leq m_0} \frac{(-1)^m(m+1)}{N(\mathfrak{q})^{(1+s/2)m}}=\frac{1}{(1+N(\mathfrak{q})^{-1-s/2})^2}+O(m_0N(\mathfrak{q})^{-(1-\varepsilon/2)m_0}).
\end{equation}

Substituting \eqref{x7.10.} into \eqref{x7.9.}, $\mathcal{M}(m_0)$ boils down to 
\begin{equation}\label{n7.11}
\frac{2V_{\mathfrak{q}}}{N(\mathfrak{q})}\cdot\frac{2\delta_{\mathbf{k}}}{2\pi i}\oint_{\mathcal{C}_{\varepsilon}}\frac{\zeta_F(1+s)G_{\mathfrak{n},\mathcal{O}_F}(s)}{s(1+N(\mathfrak{q})^{-1-s/2})^2}ds+O(m_0N(\mathfrak{q})^{-(1-\varepsilon/2)m_0}\|\mathbf{k}\|^{1+\varepsilon}).	
\end{equation}

On the other hand, by $|\lambda_{\pi}(\mathfrak{n}\mathfrak{q}^m)|\ll N(\mathfrak{n}\mathfrak{q}^m)^{\vartheta}$, $0\leq \vartheta\leq 7/64$, we derive 
\begin{equation}\label{x7.8.}
J_{\mathrm{Spec}}(\mathfrak{n},m_0)^+\ll \frac{N(\mathfrak{n})^{\vartheta}}{N(\mathfrak{q})^{(1/2-\vartheta)m_0}}\sum_{\substack{\pi\in \mathcal{F}(\mathbf{k},\mathcal{O}_F)}}\frac{L(1/2,\pi)^2}{L(1,\pi,\Ad)}.
\end{equation}

Making use of Corollary \ref{cor7.2} with $\mathfrak{q}=\mathfrak{n}=\mathcal{O}_F$,  we obtain 
\begin{align*}
\sum_{\substack{\pi\in \mathcal{F}(\mathbf{k},\mathcal{O}_F)}}\frac{L(1/2,\pi)^2}{L(1,\pi,\Ad)}=\frac{2 \delta_{\mathbf{k}}}{2\pi i}\oint_{\mathcal{C}_{\varepsilon}}\frac{\zeta_F(1+s)G_{\mathcal{O}_F}(s)}{s}ds+O(\|\mathbf{k}\|^{\frac{1}{2}+\varepsilon})\ll \|\mathbf{k}\|^{1+\varepsilon},
\end{align*}
Substituting this into \eqref{x7.8.} yields
\begin{equation}\label{x7.9}
J_{\mathrm{Spec}}(\mathfrak{n},m_0)^+\ll \frac{N(\mathfrak{n})^{\vartheta}\|\mathbf{k}\|^{1+\varepsilon}}{N(\mathfrak{q})^{(1/2-\vartheta)m_0}},
\end{equation}
where the implied constant depends only on $F$. 

Combining \eqref{c7.3}, \eqref{x7.6}, \eqref{n7.11}, and \eqref{x7.9}, we conclude that 
\begin{align*}
 J_{\mathrm{Spec}}^{\mathrm{old}}(\mathfrak{n})=&\frac{2V_{\mathfrak{q}}}{N(\mathfrak{q})}\cdot\frac{2\delta_{\mathbf{k}}}{2\pi i}\oint_{\mathcal{C}_{\varepsilon}}\frac{\zeta_F(1+s)G_{\mathfrak{n},\mathcal{O}_F}(s)}{s(1+N(\mathfrak{q})^{-1-s/2})^2}ds+O(N(\mathfrak{n})^{1/2+\varepsilon}N(\mathfrak{q})^{m_0\varepsilon}\|\mathbf{k}\|^{1/2+\varepsilon})\\
 &+O(m_0N(\mathfrak{q})^{-(1-\varepsilon/2)(m_0+1)}\|\mathbf{k}\|^{1+\varepsilon})+O(N(\mathfrak{n})^{\vartheta}\|\mathbf{k}\|^{1+\varepsilon}N(\mathfrak{q})^{-(1/2-\vartheta)m_0}).
 \end{align*}	
 
Consequently, Lemma \ref{lem7.3} follows from taking $m_0=100(1+(\log N(\mathfrak{q}))^{-1}\log \|\mathbf{k}\|)$, i.e., $N(\mathfrak{q})^{m_0}=(\|\mathbf{k}\|N(\mathfrak{q}))^{100}$, into the above expression.
\end{proof}

\subsection{The Mollified Relative Trace Formula}\label{sec8.3}
Let $\rho$ be a multiplicative arithmetic function. Suppose $\rho(\mathfrak{p})\ll 1$ for all prime ideals $\mathfrak{p}$, with the implied constant being absolute. Let $s\in\mathbb{C}$. Define
\begin{equation}\label{eq7.1}
L(s,\rho):=\prod_{\mathfrak{p}}L_{\mathfrak{p}}(s,\rho),\ \ L_{\mathfrak{p}}(s,\rho):=(1-\rho(\mathfrak{p})N(\mathfrak{p})^{-s})^{-1}.
\end{equation}
Then $L(s,\rho)$ converges absolutely in $\Re(s)\gg 1$, namely, when $\Re(s)$ is sufficiently large. In particular, when $\rho\equiv\textbf{1}$, we have $L(s,\rho)=\zeta_F(s)$. 

One special choice of $\rho$ is given by \eqref{rho} as defined in \textsection\ref{sec1.1.6}.

Let $M_{\xi,\rho}(\pi)$ be the mollifier defined as in \textsection\ref{sec1.1.6}, i.e., 
\begin{align*}
M_{\xi,\rho}(\pi)= \frac{1}{\log \xi} \sum_{\substack{\mathfrak{n} \subseteq \mathcal{O}_F \\ (\mathfrak{n},\mathfrak{q}) = 1}} \frac{\lambda_\pi(\mathfrak{n})\mu_F(\mathfrak{n})\rho(\mathfrak{n})}{\sqrt{N(\mathfrak{n})}}\cdot  \frac{1}{2\pi i}\int_{(2)} \frac{\xi^s}{N(\mathfrak{n})^s} \frac{ds}{s^3}.\tag{\ref{M}}
\end{align*}

\begin{defn}\label{defn7.3}
Let notation be as before. Define the mollified second moment by 
\begin{align*}
J_{\mathrm{Spec}}^{\heartsuit,\mathrm{new}}(\xi,\rho):=&\frac{1}{\zeta_{\mathfrak{q}}(2)^2}\sum_{\substack{\pi\in \mathcal{F}(\mathbf{k},\mathfrak{q})}}\frac{L(1/2,\pi)^2M_{\xi,\rho}(\pi)^2}{L^{(\mathfrak{q})}(1,\pi,\Ad)},\\
J_{\mathrm{Spec}}^{\heartsuit,\mathrm{old}}(\xi,\rho):=&\frac{2V_{\mathfrak{q}}\textbf{1}_{\mathfrak{q}\subsetneq \mathcal{O}_F}}{N(\mathfrak{q})}\sum_{\substack{\pi\in \mathcal{F}(\mathbf{k},\mathcal{O}_F)}}\frac{L_{\pi_{\mathfrak{q}}}L(1/2,\pi)^2M_{\xi,\rho}(\pi)^2}{L(1,\pi,\Ad)}.
\end{align*}
\end{defn}

\begin{defn}\label{defn7.4}
Let $0<\varepsilon<10^{-3}$. Let $\rho$ be a  absolutely bounded multiplicative arithmetic function. Let $G(s)$ be a holomorphic function in $\Re(s)>-1$. We define the integral $\mathbb{L}(G,\rho,\mathfrak{q})$ by 
\begin{align*}
-\frac{1}{4\pi^2(\log \xi)^2}\int_{(2)} \int_{(2)}\frac{\xi^{s_1}\xi^{s_2}}{s_1^3s_2^3}\frac{1}{2\pi i}\oint_{\mathcal{C}_{\varepsilon}}\frac{G(s)\zeta_F(1+s)L(s,s_1,s_2;\rho,\mathfrak{q})}{s}dsds_1ds_2,
\end{align*}
where $\mathcal{C}_{\varepsilon}=\{z\in \mathbb{C}:\ |z|=\varepsilon\}$, and     
\begin{align*}
L(s,s_1,s_2;\rho,\mathfrak{q}):=\sum_{\substack{\mathfrak{n}_1 \subset \mathcal{O}_F \\ (\mathfrak{n}_1,\mathfrak{q}) = 1}} \sum_{\substack{\mathfrak{n}_2 \subset \mathcal{O}_F \\ (\mathfrak{n}_2,\mathfrak{q}) = 1}} \frac{\mu_F(\mathfrak{n}_1)\rho(\mathfrak{n}_1)\mu_F(\mathfrak{n}_2)\rho(\mathfrak{n}_2)}{N(\mathfrak{n}_1)^{1/2+s_1}N(\mathfrak{n}_2)^{1/2+s_2}}\sum_{\substack{\mathfrak{m}\mid\mathfrak{c}}}\frac{\tau(\mathfrak{n}_1\mathfrak{n}_2\mathfrak{m}^{-2})}{N(\mathfrak{n}_1\mathfrak{n}_2\mathfrak{m}^{-2})^{\frac{1+s}{2}}}.
\end{align*}
Here $\mathfrak{c}:=\gcd(\mathfrak{n}_1,\mathfrak{n}_2)$.  
\end{defn}

We will show that $L(s,s_1,s_2;\rho,\mathfrak{q})$ admits a holomorphic continuation in $\Re(s_1)> -2\varepsilon$ and $\Re(s_2)> -2\varepsilon$. Hence, $\mathbb{L}(G,\rho,\mathfrak{q})$ is well defined. 

\begin{prop}\label{prop7.6}
Let notation be as before. For an integral ideal $\mathfrak{a}\subseteq\mathcal{O}_F$, we define the meromorphic function 
\begin{equation}\label{eq7.18}
G_{\mathfrak{a}}(s):= (1+N(\mathfrak{a})^s)D_F^{s}\prod_{v\mid\infty}\frac{\Gamma((k_v+s)/2)^2}{(2\pi)^{s}\Gamma(k_v/2)^2},\ \ s\in \mathbb{C}.
\end{equation}
Then we have the following 
\begin{equation}\label{sing7.17}
J_{\mathrm{Spec}}^{\heartsuit,\mathrm{new}}(\xi,\rho)=J_{\mathrm{Geom}}^{\heartsuit}(\xi,\rho)+O(\xi^{2+\varepsilon}N(\mathfrak{q})^{\varepsilon}\|\mathbf{k}\|^{1/2+\varepsilon}),
\end{equation}
where the implied constant in $O(\xi^{2+\varepsilon}N(\mathfrak{q})^{\varepsilon}\|\mathbf{k}\|^{1/2+\varepsilon})$ depends only on $\varepsilon$, $\rho$,  and $F$, and the main term $J_{\mathrm{Geom}}^{\heartsuit}(\xi,\rho)$ is defined by 
\begin{equation}\label{maingeom}
2D_F^{\frac{3}{2}}\prod_{v\mid\infty}\frac{k_v-1}{4\pi^2}\bigg[(N(\mathfrak{q})+1)\delta_{\mathbf{k},\mathfrak{q}}\mathbb{L}(G_{\mathfrak{q}},\rho,\mathfrak{q})-\frac{4\zeta_{\mathfrak{q}}(1)\delta_{\mathbf{k}}\textbf{1}_{\mathfrak{q}\subsetneq \mathcal{O}_F}}{\zeta_{\mathfrak{q}}(2)} \mathbb{L}(G_{\mathfrak{q}}^{\mathrm{old}},\rho,\mathfrak{q})\bigg],
\end{equation}
with $G_{\mathfrak{q}}^{\mathrm{old}}(s):=G_{\mathcal{O}_F}(s)(1+N(\mathfrak{q})^{-1-s/2})^{-2}$. 
\end{prop}
\begin{proof}
Let $\mathfrak{n}\subseteq \mathcal{O}_F$ be an integral ideal with $(\mathfrak{n},\mathfrak{q})=1$. Define 
\begin{align*}
J_{\mathrm{reg}}(\mathfrak{n}):=&\sum_{\substack{\pi\in \mathcal{F}(\mathbf{k},\mathfrak{q})}}\frac{\lambda_{\pi}(\mathfrak{n})L(1/2,\pi)^2}{\zeta_{\mathfrak{q}}(2)^2L^{(\mathfrak{q})}(1,\pi,\Ad)}+\frac{2V_{\mathfrak{q}}\textbf{1}_{\mathfrak{q}\subsetneq \mathcal{O}_F}}{N(\mathfrak{q})}\sum_{\substack{\pi\in \mathcal{F}(\mathbf{k},\mathcal{O}_F)}}\frac{\lambda_{\pi}(\mathfrak{n})L_{\pi_{\mathfrak{q}}}L(1/2,\pi)^2}{L(1,\pi,\Ad)}\\
&-\frac{(N(\mathfrak{q})+1)\cdot \delta_{\mathbf{k},\mathfrak{q}}}{2\pi i}\oint_{\mathcal{C}_{\varepsilon}}\frac{\zeta_F(1+s)G_{\mathfrak{n},\mathfrak{q}}(s)}{s}ds.
\end{align*}

By Corollary \ref{cor7.2} we have
\begin{equation}\label{c7.6}
J_{\mathrm{reg}}(\mathfrak{n}) \ll N(\mathfrak{n})^{1/2+\varepsilon}N(\mathfrak{q})^{\varepsilon}\|\mathbf{k}\|^{1/2+\varepsilon},
\end{equation}
where the implied constant depends only on $\varepsilon$ and $F$.

Let $\mathfrak{n}_1$ and $\mathfrak{n}_2$ be integral ideals. Taking advantage of Hecke relations 
\begin{align*}
\lambda_{\pi}(\mathfrak{n}_1)\lambda_{\pi}(\mathfrak{n}_2)=\sum_{\substack{\mathfrak{m}\mid\mathfrak{c}}}\lambda_\pi(\mathfrak{n}_1\mathfrak{n}_2\mathfrak{m}^{-2}),
\end{align*}
where $\mathfrak{c}=\mathrm{gcd}(\mathfrak{n}_1,\mathfrak{n}_2)$, we obtain 
\begin{align*}
M_{\xi,\rho}(\pi)^2=&-\frac{1}{4\pi^2(\log \xi)^2}\sum_{\substack{\mathfrak{n}_1 \subset \mathcal{O}_F \\ (\mathfrak{n}_1,\mathfrak{q}) = 1}} \sum_{\substack{\mathfrak{n}_2 \subset \mathcal{O}_F \\ (\mathfrak{n}_2,\mathfrak{q}) = 1}} \frac{\mu_F(\mathfrak{n}_1)\mu_F(\mathfrak{n}_2)\rho(\mathfrak{n}_1)\rho(\mathfrak{n}_2)}{N(\mathfrak{n}_1)^{1/2}N(\mathfrak{n}_2)^{1/2}}\\
&\int_{(2)} \frac{\xi^{s_1}}{N(\mathfrak{n}_1)^{s_1}} \frac{ds_1}{s_1^3}\int_{(2)} \frac{\xi^{s_2}}{N(\mathfrak{n}_2)^{s_2}} \frac{ds_2}{s_2^3}\sum_{\substack{\mathfrak{m}\mid\mathfrak{c}}}\lambda_\pi(\mathfrak{n}_1\mathfrak{n}_2\mathfrak{m}^{-2})
\end{align*}

By Corollary \ref{cor7.2}, along with the Definition \ref{defn7.4}, we obtain 
\begin{equation}\label{sing7.7}
J_{\mathrm{Spec}}^{\heartsuit,\mathrm{new}}(\xi,\rho)+J_{\mathrm{Spec}}^{\heartsuit,\mathrm{old}}(\xi,\rho)=J_{\mathrm{sing}}^{\heartsuit}(\xi,\rho)+J_{\mathrm{reg}}^{\heartsuit}(\xi,\rho),	
\end{equation}
where 
\begin{align*}
J_{\mathrm{sing}}^{\heartsuit}(\xi,\rho):=&2(N(\mathfrak{q})+1)\delta_{\mathbf{k},\mathfrak{q}}D_F^{\frac{3}{2}}\prod_{v\mid\infty}\frac{k_v-1}{4\pi^2}\cdot \mathbb{L}(G_{\mathfrak{q}},\rho,\mathfrak{q}),\\
J_{\mathrm{reg}}^{\heartsuit}(\xi,\rho):=&-\frac{1}{4\pi^2(\log \xi)^2}\sum_{\substack{\mathfrak{n}_1 \subset \mathcal{O}_F \\ (\mathfrak{n}_1,\mathfrak{q}) = 1}} \sum_{\substack{\mathfrak{n}_2 \subset \mathcal{O}_F \\ (\mathfrak{n}_2,\mathfrak{q}) = 1}} \frac{\mu_F(\mathfrak{n}_1)\mu_F(\mathfrak{n}_2)\rho(\mathfrak{n}_1)\rho(\mathfrak{n}_2)}{N(\mathfrak{n}_1)^{1/2}N(\mathfrak{n}_2)^{1/2}}\\
&\int_{(2)} \frac{\xi^{s_1}}{N(\mathfrak{n}_1)^{s_1}} \frac{ds_1}{s_1^3}\int_{(2)} \frac{\xi^{s_2}}{N(\mathfrak{n}_2)^{s_2}} \frac{ds_2}{s_2^3}\sum_{\substack{\mathfrak{m}\mid\mathrm{gcd}(\mathfrak{n}_1,\mathfrak{n}_2)}}J_{\mathrm{reg}}(\mathfrak{n}_1\mathfrak{n}_2\mathfrak{m}^{-2}).
\end{align*}

Shifting contour we obtain by Cauchy theorem we obtain 
\begin{equation}\label{c7.8}
\frac{1}{2\pi i}\int_{(2)} \frac{\xi^{s_1}}{N(\mathfrak{n}_1)^{s_1}}\frac{ds_1}{s_1^3}=\frac{(\log \xi N(\mathfrak{n}_1)^{-1})^2}{2} \cdot \textbf{1}_{N(\mathfrak{n}_1)\leq \xi}.
\end{equation}

Substituting \eqref{c7.8} into the definition of $J_{\mathrm{reg}}^{\heartsuit}(\xi,\rho)$ leads to 
\begin{align*}
J_{\mathrm{reg}}^{\heartsuit}(\xi,\rho)=&\frac{1}{(\log \xi)^2} \sum_{\substack{\mathfrak{n}_1 \subseteq \mathcal{O}_F \\ (\mathfrak{n}_1,\mathfrak{q}) = 1\\ N(\mathfrak{n}_1)\leq \xi}}\sum_{\substack{\mathfrak{n} _2\subseteq \mathcal{O}_F \\ (\mathfrak{n}_2,\mathfrak{q}) = 1\\ N(\mathfrak{n}_2)\leq \xi}}  \frac{\mu_F(\mathfrak{n}_1)\rho(\mathfrak{n}_1)\mu_F(\mathfrak{n}_2)\rho(\mathfrak{n}_2)}{\sqrt{N(\mathfrak{n}_1)N(\mathfrak{n}_2)}}\\
&\cdot \frac{(\log \xi N(\mathfrak{n}_1)^{-1})^2(\log \xi N(\mathfrak{n}_2)^{-1})^2}{4}\sum_{\substack{\mathfrak{m}\mid\mathrm{gcd}(\mathfrak{n}_1,\mathfrak{n}_2)}}J_{\mathrm{reg}}(\mathfrak{n}_1\mathfrak{n}_2\mathfrak{m}^{-2}).
\end{align*}

Taking advantage of the estimate \eqref{c7.6} into the above expression we derive 
\begin{align*}
J_{\mathrm{reg}}^{\heartsuit}(\xi,\rho)\ll & \xi^{\varepsilon} \sum_{\substack{\mathfrak{n}_1 \subseteq \mathcal{O}_F \\ (\mathfrak{n}_1,\mathfrak{q}) = 1\\ N(\mathfrak{n}_1)\leq \xi}}\sum_{\substack{\mathfrak{n} _2\subseteq \mathcal{O}_F \\ (\mathfrak{n}_2,\mathfrak{q}) = 1\\ N(\mathfrak{n}_2)\leq \xi}}  N(\mathfrak{n}_1)^{\varepsilon}N(\mathfrak{n}_2)^{\varepsilon}\sum_{\substack{\mathfrak{m}\mid\mathrm{gcd}(\mathfrak{n}_1,\mathfrak{n}_2)}}N(\mathfrak{m})^{-1}N(\mathfrak{q})^{\varepsilon}\|\mathbf{k}\|^{\frac{1}{2}+\varepsilon}\\
\ll &\xi^{3\varepsilon}N(\mathfrak{q})^{\varepsilon}\|\mathbf{k}\|^{1/2+\varepsilon}\sum_{\substack{\mathfrak{m} \subseteq \mathcal{O}_F \\ (\mathfrak{m},\mathfrak{q}) = 1\\ N(\mathfrak{m})\leq \xi}}\sum_{\substack{\mathfrak{n}_1 \subseteq \mathcal{O}_F \\ (\mathfrak{n}_1,\mathfrak{q}) = 1\\ N(\mathfrak{n}_1)\leq \xi/N(\mathfrak{m})}}\sum_{\substack{\mathfrak{n} _2\subseteq \mathcal{O}_F \\ (\mathfrak{n}_2,\mathfrak{q}) = 1\\ N(\mathfrak{n}_2)\leq \xi/N(\mathfrak{m})}}  N(\mathfrak{m})^{-1}\\
\ll &\xi^{4\varepsilon}N(\mathfrak{q})^{\varepsilon}\|\mathbf{k}\|^{1/2+\varepsilon}\sum_{\substack{\mathfrak{m} \subseteq \mathcal{O}_F \\ (\mathfrak{m},\mathfrak{q}) = 1\\ N(\mathfrak{m})\leq \xi}}\frac{\xi^2}{N(\mathfrak{m})^3}\ll \xi^{2+5\varepsilon}N(\mathfrak{q})^{\varepsilon}\|\mathbf{k}\|^{1/2+\varepsilon}.
\end{align*}

Substituting this into \eqref{sing7.7} yields 
\begin{equation}\label{sing7.6}
J_{\mathrm{Spec}}^{\heartsuit,\mathrm{new}}(\xi,\rho)+J_{\mathrm{Spec}}^{\heartsuit,\mathrm{old}}(\xi,\rho)=J_{\mathrm{sing}}^{\heartsuit}(\xi,\rho)+O(\xi^{2+\varepsilon}N(\mathfrak{q})^{\varepsilon}\|\mathbf{k}\|^{1/2+\varepsilon}),
\end{equation}
where the implied constant depends only on $\varepsilon$, $\rho$, $G$, and $F$. 

By Definition \ref{defn7.3}  we have
\begin{align*}
J_{\mathrm{Spec}}^{\heartsuit,\mathrm{old}}(\xi,\rho)=&-\frac{\textbf{1}_{\mathfrak{q}\subsetneq \mathcal{O}_F}}{4\pi^2(\log \xi)^2}\sum_{\substack{\mathfrak{n}_1 \subset \mathcal{O}_F \\ (\mathfrak{n}_1,\mathfrak{q}) = 1}} \sum_{\substack{\mathfrak{n}_2 \subset \mathcal{O}_F \\ (\mathfrak{n}_2,\mathfrak{q}) = 1}} \frac{\mu_F(\mathfrak{n}_1)\mu_F(\mathfrak{n}_2)\rho(\mathfrak{n}_1)\rho(\mathfrak{n}_2)}{N(\mathfrak{n}_1)^{1/2}N(\mathfrak{n}_2)^{1/2}}\\
&\int_{(2)} \frac{\xi^{s_1}}{N(\mathfrak{n}_1)^{s_1}} \frac{ds_1}{s_1^3}\int_{(2)} \frac{\xi^{s_2}}{N(\mathfrak{n}_2)^{s_2}} \frac{ds_2}{s_2^3}\sum_{\substack{\mathfrak{m}\mid\mathrm{gcd}(\mathfrak{n}_1,\mathfrak{n}_2)}}J_{\mathrm{Spec}}^{\mathrm{old}}(\mathfrak{n}_1\mathfrak{n}_2\mathfrak{m}^{-2})
\end{align*}
where $J_{\mathrm{Spec}}^{\mathrm{old}}(\cdot)$ is defined as in \eqref{y7.4}. Taking advantage of Lemma \ref{lem7.3}, along with the formula \eqref{c7.8}, we derive
\begin{equation}\label{f7.22}
J_{\mathrm{Spec}}^{\heartsuit,\mathrm{old}}(\xi,\rho)=J_{\mathrm{main}}^{\heartsuit,\mathrm{old}}(\xi,\rho)+J_{\mathrm{error}}^{\heartsuit,\mathrm{old}}(\xi,\rho),
\end{equation}
where $J_{\mathrm{main}}^{\heartsuit,\mathrm{old}}(\xi,\rho)$ is defined by  
\begin{align*}
J_{\mathrm{main}}^{\heartsuit,\mathrm{old}}(\xi,\rho):=&\frac{-\textbf{1}_{\mathfrak{q}\subsetneq \mathcal{O}_F}}{4\pi^2(\log \xi)^2}\sum_{\substack{\mathfrak{n}_1 \subset \mathcal{O}_F \\ (\mathfrak{n}_1,\mathfrak{q}) = 1}} \sum_{\substack{\mathfrak{n}_2 \subset \mathcal{O}_F \\ (\mathfrak{n}_2,\mathfrak{q}) = 1}} \frac{\mu_F(\mathfrak{n}_1)\mu_F(\mathfrak{n}_2)\rho(\mathfrak{n}_1)\rho(\mathfrak{n}_2)}{N(\mathfrak{n}_1)^{1/2}N(\mathfrak{n}_2)^{1/2}}\int_{(2)} \frac{\xi^{s_1}}{N(\mathfrak{n}_1)^{s_1}} \frac{ds_1}{s_1^3}\\
&\int_{(2)} \frac{\xi^{s_2}ds_2}{N(\mathfrak{n}_2)^{s_2}s_2^3} \sum_{\substack{\mathfrak{m}\mid\mathrm{gcd}(\mathfrak{n}_1,\mathfrak{n}_2)}}\frac{2V_{\mathfrak{q}}}{N(\mathfrak{q})}\cdot\frac{2\delta_{\mathbf{k}}}{2\pi i}\oint_{\mathcal{C}_{\varepsilon}}\frac{\zeta_F(1+s)G_{\mathfrak{n}_1\mathfrak{n}_2\mathfrak{m}^{-2},\mathcal{O}_F}(s)}{s(1+N(\mathfrak{q})^{-1-s/2})^2}ds,
\end{align*}
with $\delta_{\mathbf{k}}:=\textbf{1}_{\sum_{v\mid\infty}k_v\equiv 0\pmod{4}}$, and 
\begin{align*}
J_{\mathrm{error}}^{\heartsuit,\mathrm{old}}(\xi,\rho):=&\frac{N(\mathfrak{q})^{\varepsilon}\|\mathbf{k}\|^{1/2+\varepsilon}\cdot \textbf{1}_{\mathfrak{q}\subsetneq \mathcal{O}_F}}{4(\log \xi)^2} \sum_{\substack{\mathfrak{n}_1 \subseteq \mathcal{O}_F \\ (\mathfrak{n}_1,\mathfrak{q}) = 1\\ N(\mathfrak{n}_1)\leq \xi}}\sum_{\substack{\mathfrak{n} _2\subseteq \mathcal{O}_F \\ (\mathfrak{n}_2,\mathfrak{q}) = 1\\ N(\mathfrak{n}_2)\leq \xi}}  \frac{\mu_F(\mathfrak{n}_1)^2|\rho(\mathfrak{n}_1)|\mu_F(\mathfrak{n}_2)^2|\rho(\mathfrak{n}_2)|}{\sqrt{N(\mathfrak{n}_1)N(\mathfrak{n}_2)}}\\
&\cdot \frac{(\log \xi N(\mathfrak{n}_1)^{-1})^2(\log \xi N(\mathfrak{n}_2)^{-1})^2}{4}\sum_{\substack{\mathfrak{m}\mid\mathrm{gcd}(\mathfrak{n}_1,\mathfrak{n}_2)}}N(\mathfrak{n}_1\mathfrak{n}_2\mathfrak{m}^{-2})^{1/2+\varepsilon}.
\end{align*}

Similar to the estimate of $J_{\mathrm{reg}}^{\heartsuit}(\xi,\rho)$ we deduce that 
\begin{equation}\label{eq7.23}
J_{\mathrm{error}}^{\heartsuit,\mathrm{old}}(\xi,\rho)\ll \xi^{2+5\varepsilon}N(\mathfrak{q})^{\varepsilon}\|\mathbf{k}\|^{1/2+\varepsilon}.
\end{equation}

According to the definition of $G_{\mathfrak{n},\mathfrak{q}}(s)$ in \eqref{G}, we have
\begin{equation}\label{eq7.25}
G_{\mathfrak{n}_1\mathfrak{n}_2\mathfrak{m}^{-2},\mathcal{O}_F}(s)=D_F^{\frac{3}{2}}\prod_{v\mid\infty}\frac{k_v-1}{2\pi^2}\cdot G_{\mathcal{O}_F}(s)\cdot \frac{\tau(\mathfrak{n}_1\mathfrak{n}_2\mathfrak{m}^{-2})}{N(\mathfrak{n}_1\mathfrak{n}_2\mathfrak{m}^{-2})^{(1+s)/2}},
\end{equation}
where $G_{\mathcal{O}_F}(s)$ is defined by \eqref{eq7.18} with $\mathfrak{a}=\mathcal{O}_F$.

Substituting \eqref{eq7.25} in the definition of $J_{\mathrm{main}}^{\heartsuit,\mathrm{old}}(\xi,\rho)$ yields 
\begin{equation}\label{eq7.26}
J_{\mathrm{main}}^{\heartsuit,\mathrm{old}}(\xi,\rho)=\frac{4\delta_{\mathbf{k}}V_{\mathfrak{q}}\textbf{1}_{\mathfrak{q}\subsetneq \mathcal{O}_F}}{N(\mathfrak{q})}D_F^{\frac{3}{2}}\prod_{v\mid\infty}\frac{k_v-1}{2\pi^2}\cdot \mathbb{L}(G_{\mathfrak{q}}^{\mathrm{old}},\rho,\mathfrak{q}),
\end{equation}
where $G_{\mathfrak{q}}^{\mathrm{old}}(s):=G_{\mathcal{O}_F}(s)(1+N(\mathfrak{q})^{-1-s/2})^{-2}$. 

Therefore, \eqref{sing7.17} follows from \eqref{sing7.6}, \eqref{f7.22}, \eqref{eq7.23}, and \eqref{eq7.26}, along with the fact that $V_{\mathfrak{q}}N(\mathfrak{q})^{-1}=\zeta_{\mathfrak{q}}(1)\zeta_{\mathfrak{q}}(2)^{-1}$ for $\mathfrak{q}\neq\mathcal{O}_F$. 
\end{proof}

\subsection{The Mollified Singular Orbital Integral} 
\begin{lemma}\label{lem7.1}
Let notation be as before. Let $L(s,s_1,s_2;\rho,\mathfrak{q})$ be defined as in Definition \ref{defn7.4}. We have the following assertions.
\begin{itemize}
\item The series $L(s,s_1,s_2;\rho,\mathfrak{q})$ converges absolutely in the region 
\begin{equation}\label{equation7.1}
\begin{cases}
2\Re(s_1)+\Re(s)>0,\ \ 2\Re(s_1)+\Re(s)>0\\
\Re(s_1)+\Re(s_2)>0.
\end{cases}
\end{equation}
\item In the region \eqref{equation7.1}, we have 
\begin{equation}\label{7.9}
L(s,s_1,s_2;\rho,\mathfrak{q})=\prod_{\mathfrak{p}\nmid\mathfrak{q}}G_{\mathfrak{p}}(s,s_1,s_2),
\end{equation}
where the function $G_{\mathfrak{p}}(s,s_1,s_2)$ is defined by 
\begin{equation}\label{eq7.3}
1-\frac{2\rho(\mathfrak{p})}{N(\mathfrak{p})^{1+s_1+\frac{s}{2}}}-\frac{2\rho(\mathfrak{p})}{N(\mathfrak{p})^{1+s_2+\frac{s}{2}}}+\frac{\rho(\mathfrak{p})^2}{N(\mathfrak{p})^{1+s_1+s_2}}+\frac{3\rho(\mathfrak{p})^2}{N(\mathfrak{p})^{2+s_1+s_2+s}}.
\end{equation}
\item Suppose $\rho(\mathfrak{p})\ll 1$ for all prime ideals $\mathfrak{p}$, with the implied constant being absolute. The function $L(s,s_1,s_2;\rho,\mathfrak{q})$ admits a meromorphic continuation to the region 
\begin{equation}\label{eq7.4}
\begin{cases}
2\Re(s_1)+\Re(s)>-1,\ \ 2\Re(s_2)+\Re(s)>-1\\
\Re(s_1)+\Re(s_2)+\Re(s)>-1\\
2\Re(s_1)+\Re(s_2)+3\Re(s)/2>-2\\
\Re(s_1)+2\Re(s_2)+3\Re(s)/2>-2.
\end{cases}
\end{equation}
Moreover, in the region \eqref{eq7.4} we have 
\begin{equation}\label{eq7.5}
L(s,s_1,s_2;\rho,\mathfrak{q})=\frac{L(1+s_1+s_2,\rho)E(s,s_1,s_2;\rho,\mathfrak{q})}{L(1+s_1+s/2,\rho)^2L(1+s_2+s/2,\rho)^2},
\end{equation}
where $E(s,s_1,s_2;\rho,\mathfrak{q})$ is defined by  
\begin{equation}\label{eq7.6}
\frac{L_{\mathfrak{q}}(1+s_1+s/2,\rho)^2L_{\mathfrak{q}}(1+s_2+s/2,\rho)^2}{L_{\mathfrak{q}}(1+s_1+s_2,\rho)}\prod_{\mathfrak{p}\nmid\mathfrak{q}}\frac{L_{\mathfrak{p}}(1+s_1+s/2,\rho)^2L_{\mathfrak{p}}(1+s_2+s/2,\rho)^2}{L_{\mathfrak{p}}(1+s_1+s_2,\rho)G_{\mathfrak{p}}(s,s_1,s_2)^{-1}}.
\end{equation}
\end{itemize}
\end{lemma}
\begin{proof}
Let $\mathfrak{c}\subseteq \mathcal{O}_F$ be the greatest common divisor of $\mathfrak{n}_1$ and $\mathfrak{n}_2$. Write $\mathfrak{n}_j=\mathfrak{m}_j\mathfrak{c}$, $j=1,2$. Swapping sums we can rewrite $L(s,s_1,s_2;\rho,\mathfrak{q})$ as 
\begin{align*}
\sum_{\substack{\mathfrak{c} \subset \mathcal{O}_F \\ (\mathfrak{c},\mathfrak{q}) = 1}} \sum_{\substack{\mathfrak{m}_1 \subset \mathcal{O}_F \\ (\mathfrak{m}_1,\mathfrak{c}\mathfrak{q}) = 1}} \sum_{\substack{\mathfrak{m}_2 \subset \mathcal{O}_F \\ (\mathfrak{m}_2,\mathfrak{m}_1\mathfrak{c}\mathfrak{q}) = 1}} \frac{\mu_F(\mathfrak{m}_1)\mu_F(\mathfrak{m}_2)\mu_F(\mathfrak{c})^2\rho(\mathfrak{m}_1)\rho(\mathfrak{m}_2)\rho(\mathfrak{c})^{2}}{N(\mathfrak{m}_1)^{\frac{1}{2}+s_1}N(\mathfrak{m}_2)^{\frac{1}{2}+s_2}N(\mathfrak{c})^{1+s_1+s_2}}\sum_{\substack{\mathfrak{m}\mid\mathfrak{c}}}\frac{\tau(\mathfrak{m}_1\mathfrak{m}_2\mathfrak{m}^2)}{N(\mathfrak{m}_1\mathfrak{m}_2\mathfrak{m}^2)^{\frac{1+s}{2}}}.
\end{align*}
where 
\begin{equation}\label{f7.1}
\lambda_s(\mathfrak{c}):=\sum_{\substack{\mathfrak{m}\mid\mathfrak{c}}}\frac{\tau(\mathfrak{m}^2)}{N(\mathfrak{m})^{1+s}}=\prod_{\mathfrak{p}\mid\mathfrak{c}}\left(1+\tau(\mathfrak{p}^2)N(\mathfrak{p})^{-1-s}\right)=\prod_{\mathfrak{p}\mid\mathfrak{c}}\left(1+3N(\mathfrak{p})^{-1-s}\right).
\end{equation}
In conjunction with the relation 
\begin{align*}
\textbf{1}_{(\mathfrak{m}_1,\mathfrak{m}_2)=1}=\sum_{\mathfrak{a}\mid (\mathfrak{m}_1,\mathfrak{m}_2)}\mu_F(\mathfrak{a}),
\end{align*}
we obtain the following expression:  
\begin{align*}
L(s,s_1,s_2;\rho,\mathfrak{q})=\sum_{\substack{\mathfrak{c} \subset \mathcal{O}_F \\ (\mathfrak{c},\mathfrak{q}) = 1}} &\frac{\lambda_s(\mathfrak{c})\mu_F(\mathfrak{c})^2\rho(\mathfrak{c})^{2}}{N(\mathfrak{c})^{1+s_1+s_2}}\sum_{\substack{\mathfrak{m}_1 \subset \mathcal{O}_F \\ (\mathfrak{m}_1,\mathfrak{c}\mathfrak{q}) = 1}} \frac{\tau(\mathfrak{m}_1)\mu_F(\mathfrak{m}_1)\rho(\mathfrak{m}_1)}{N(\mathfrak{m}_1)^{1+s_1+\frac{s}{2}}}\\
&\sum_{\substack{\mathfrak{m}_2 \subset \mathcal{O}_F \\ (\mathfrak{m}_2,\mathfrak{c}\mathfrak{q}) = 1}} \frac{\tau(\mathfrak{m}_2)\mu_F(\mathfrak{m}_2)\rho(\mathfrak{m}_2)}{N(\mathfrak{m}_2)^{1+s_2+\frac{s}{2}}}\sum_{\mathfrak{a}\mid (\mathfrak{m}_1,\mathfrak{m}_2)}\mu_F(\mathfrak{a}).
\end{align*}

Hence, after the changing of variables $\mathfrak{m}_1\mapsto \mathfrak{m}_1\mathfrak{a}$ and $\mathfrak{m}_2\mapsto \mathfrak{m}_2\mathfrak{a}$, we derive  
\begin{align*}
L(s,s_1,s_2;\rho,\mathfrak{q})=&\sum_{\substack{\mathfrak{c} \subset \mathcal{O}_F \\ (\mathfrak{c},\mathfrak{q}) = 1}} \frac{\lambda_s(\mathfrak{c})\mu_F(\mathfrak{c})^2\rho(\mathfrak{c})^{2}}{N(\mathfrak{c})^{1+s_1+s_2}}\sum_{(\mathfrak{a},\mathfrak{c}\mathfrak{q}) = 1}\frac{\mu_F(\mathfrak{a})\tau(\mathfrak{a})^2\rho(\mathfrak{a})^{2}}{N(\mathfrak{a})^{2+s_1+s_2+s}}\\
&\sum_{\substack{\mathfrak{m}_1\subset \mathcal{O}_F \\ (\mathfrak{m}_1,\mathfrak{a}\mathfrak{c}\mathfrak{q}) = 1}}\sum_{\substack{\mathfrak{m}_2 \subset \mathcal{O}_F \\ 
(\mathfrak{m}_2,\mathfrak{a}\mathfrak{c}\mathfrak{q}) = 1}} \frac{\tau(\mathfrak{m}_1)\tau(\mathfrak{m}_2)\mu_F(\mathfrak{m}_1)\mu_F(\mathfrak{m}_2)\rho(\mathfrak{m}_1)\rho(\mathfrak{m}_2)}{N(\mathfrak{m}_1)^{1+s_1+\frac{s}{2}}N(\mathfrak{m}_2)^{1+s_2+\frac{s}{2}}}.
\end{align*}

Notice that 
\begin{align*}
\sum_{\substack{\mathfrak{m}_2 \subset \mathcal{O}_F \\ (\mathfrak{m}_2,\mathfrak{a}\mathfrak{c}\mathfrak{q}) = 1}} \frac{\tau(\mathfrak{m}_2)\mu_F(\mathfrak{m}_2)\rho(\mathfrak{m}_2)}{N(\mathfrak{m}_2)^{1+s_2+\frac{s}{2}}}=\prod_{\mathfrak{p}\nmid \mathfrak{a}\mathfrak{c}\mathfrak{q}}\left(1+\frac{\tau(\mathfrak{p})\mu_F(\mathfrak{p})\rho(\mathfrak{p})}{N(\mathfrak{p})^{1+s_2+\frac{s}{2}}}\right),
\end{align*}
which is equal to $\prod_{\mathfrak{p}\nmid \mathfrak{a}\mathfrak{c}\mathfrak{q}}\left(1-2\rho(\mathfrak{p})N(\mathfrak{p})^{-1-s_2-\frac{s}{2}}\right).$

Therefore, the function $L(s,s_1,s_2;\rho,\mathfrak{q})$ boils down to  
\begin{equation}\label{7.1}
\sum_{\substack{\mathfrak{c} \subset \mathcal{O}_F \\ (\mathfrak{c},\mathfrak{q}) = 1}} \frac{\lambda_s(\mathfrak{c})\mu_F(\mathfrak{c})^2\rho(\mathfrak{c})^{2}}{N(\mathfrak{c})^{1+s_1+s_2}}\sum_{(\mathfrak{a},\mathfrak{c}\mathfrak{q}) = 1}\frac{\mu_F(\mathfrak{a})\tau(\mathfrak{a})^2\rho(\mathfrak{a})^{2}}{N(\mathfrak{a})^{2+s_1+s_2+s}}\prod_{\mathfrak{p}\nmid \mathfrak{a}\mathfrak{c}\mathfrak{q}}\nu_{s,s_1,s_2}(\mathfrak{p})^{-1},
\end{equation}
where for an integral ideal $\mathfrak{b}$, the arithmetic function $\nu_{s,s_1,s_2}(\mathfrak{b})$ is defined by   
\begin{equation}\label{f7.3}
\nu_{s,s_1,s_2}(\mathfrak{b})=\prod_{\mathfrak{p}\mid \mathfrak{b}}\left(1-2\rho(\mathfrak{p})N(\mathfrak{p})^{-1-s_1-\frac{s}{2}}\right)^{-1}\left(1-2\rho(\mathfrak{p})N(\mathfrak{p})^{-1-s_2-\frac{s}{2}}\right)^{-1}.
\end{equation}

Write $D_1(s,s_1,s_2):=\prod_{\mathfrak{p}}\nu_{s,s_1,s_2}(\mathfrak{p})^{-1}$. Then \eqref{7.1} becomes
\begin{align*}
L(s,s_1,s_2;\rho,\mathfrak{q})=&D_1(s,s_1,s_2)\nu_{s,s_1,s_2}(\mathfrak{q})\sum_{\substack{\mathfrak{c} \subset \mathcal{O}_F \\ (\mathfrak{c},\mathfrak{q}) = 1}} \frac{\lambda_s(\mathfrak{c})\mu_F(\mathfrak{c})^2\rho(\mathfrak{c})^2\nu_{s,s_1,s_2}(\mathfrak{c})}{N(\mathfrak{c})^{1+s_1+s_2}}\\
&\qquad \qquad \sum_{(\mathfrak{a},\mathfrak{c}\mathfrak{q}) = 1}\frac{\mu_F(\mathfrak{a})\tau(\mathfrak{a})^2\rho(\mathfrak{a})^2\nu_{s,s_1,s_2}(\mathfrak{a})}{N(\mathfrak{a})^{2+s_1+s_2+s}}.
\end{align*}

Summing over integral ideals $\mathfrak{a}\subseteq \mathcal{O}_F$ with $(\mathfrak{a},\mathfrak{c}\mathfrak{q}) = 1$ in the above equality, we can rewrite the function $L(s,s_1,s_2;\rho,\mathfrak{q})D_1(s,s_1,s_2)^{-1}\nu_{s,s_1,s_2}(\mathfrak{q})^{-1}$ as
\begin{equation}\label{7.3}
D_2(s,s_1,s_2)\sum_{\substack{\mathfrak{c} \subset \mathcal{O}_F \\ (\mathfrak{c},\mathfrak{q}) = 1}} \frac{\lambda_s(\mathfrak{c})\mu_F(\mathfrak{c})^2\rho(\mathfrak{c})^2\nu_{s,s_1,s_2}(\mathfrak{c})}{N(\mathfrak{c})^{1+s_1+s_2}}\prod_{\mathfrak{p}\mid\mathfrak{c}\mathfrak{q}}\left(1-\frac{4\rho(\mathfrak{p})^2\nu_{s,s_1,s_2}(\mathfrak{p})}{N(\mathfrak{p})^{2+s_1+s_2+s}}\right)^{-1},
\end{equation}
where 
\begin{align*}
D_2(s,s_1,s_2):=\prod_{\mathfrak{p}}\left(1-\frac{4\rho(\mathfrak{p})^2\nu_{s,s_1,s_2}(\mathfrak{p})}{N(\mathfrak{p})^{2+s_1+s_2+s}}\right).
\end{align*}

By a straightforward calculation, we have 
\begin{align*}
&\sum_{\substack{\mathfrak{c} \subset \mathcal{O}_F \\ (\mathfrak{c},\mathfrak{q}) = 1}} \frac{\lambda_s(\mathfrak{c})\mu_F(\mathfrak{c})^2\rho(\mathfrak{c})^2\nu_{s,s_1,s_2}(\mathfrak{c})}{N(\mathfrak{c})^{1+s_1+s_2}}\prod_{\mathfrak{p}\mid\mathfrak{c}\mathfrak{q}}\left(1-\frac{4\rho(\mathfrak{p})^2\nu_{s,s_1,s_2}(\mathfrak{p})}{N(\mathfrak{p})^{2+s_1+s_2+s}}\right)^{-1}\\
=&\prod_{\mathfrak{p}\mid \mathfrak{q}}\left(1-\frac{4\rho(\mathfrak{p})^2\nu_{s,s_1,s_2}(\mathfrak{p})}{N(\mathfrak{p})^{2+s_1+s_2+s}}\right)^{-1}\prod_{\mathfrak{p}\nmid\mathfrak{q}}\left(1+\frac{\lambda_s(\mathfrak{p})\rho(\mathfrak{p})^2\nu_{s,s_1,s_2}(\mathfrak{p})N(\mathfrak{p})^{-1-s_1-s_2}}{1-4\rho(\mathfrak{p})^2\nu_{s,s_1,s_2}(\mathfrak{p})N(\mathfrak{p})^{-2-s_1-s_2-s}}\right).
\end{align*}

Substituting this into \eqref{7.3} we derive that 
\begin{equation}\label{f7.6}
L(s,s_1,s_2;\rho,\mathfrak{q})=D_1(s,s_1,s_2)D_2(s,s_1,s_2)D_3(s,s_1,s_2)\prod_{\mathfrak{p}\mid \mathfrak{q}}G_{\mathfrak{p}}^*(s,s_1,s_2),
\end{equation}
where 
\begin{align*}
D_3(s,s_1,s_2):=&\prod_{\mathfrak{p}}\left(1+\frac{\lambda_s(\mathfrak{p})\rho(\mathfrak{p})^2\nu_{s,s_1,s_2}(\mathfrak{p})}{N(\mathfrak{p})^{1+s_1+s_2}}\left(1-\frac{4\rho(\mathfrak{p})^2\nu_{s,s_1,s_2}(\mathfrak{p})}{N(\mathfrak{p})^{2+s_1+s_2+s}}\right)^{-1}\right),\\
G_{\mathfrak{p}}^*(s,s_1,s_2):=&\left(\nu_{s,s_1,s_2}(\mathfrak{p})^{-1}-\frac{4\rho(\mathfrak{p})^2}{N(\mathfrak{p})^{2+s_1+s_2+s}}+\frac{\lambda_s(\mathfrak{p})\rho(\mathfrak{p})^2}{N(\mathfrak{p})^{1+s_1+s_2}}\right)^{-1}.
\end{align*}

According to the definition of $\lambda_s(\mathfrak{p})$ in \eqref{f7.1} 
and $\nu_{s,s_1,s_2}(\mathfrak{p})$ in \eqref{f7.3}, the function $G_{\mathfrak{p}}^*(s,s_1,s_2)$ is equal to  
\begin{equation}\label{7.6}
\left(1-\frac{2\rho(\mathfrak{p})}{N(\mathfrak{p})^{1+s_1+\frac{s}{2}}}-\frac{2\rho(\mathfrak{p})}{N(\mathfrak{p})^{1+s_2+\frac{s}{2}}}+\frac{\rho(\mathfrak{p})^2}{N(\mathfrak{p})^{1+s_1+s_2}}+\frac{3\rho(\mathfrak{p})^2}{N(\mathfrak{p})^{2+s_1+s_2+s}}\right)^{-1}.
\end{equation}

Moreover, the function $D_1(s,s_1,s_2)D_2(s,s_1,s_2)D_3(s,s_1,s_2)$ is equal to  
\begin{equation}\label{7.7}
\prod_{\mathfrak{p}}\nu_{s,s_1,s_2}(\mathfrak{p})^{-1}\left(1-\frac{4\rho(\mathfrak{p})^2\nu_{s,s_1,s_2}(\mathfrak{p})}{N(\mathfrak{p})^{2+s_1+s_2+s}}+\frac{\lambda_s(\mathfrak{p})\rho(\mathfrak{p})^2\nu_{s,s_1,s_2}(\mathfrak{p})}{N(\mathfrak{p})^{1+s_1+s_2}}\right)=\prod_{\mathfrak{p}}\frac{1}{G_{\mathfrak{p}}^*(s,s_1,s_2)}.
\end{equation}

Notice that $G_{\mathfrak{p}}^*(s,s_1,s_2)^{-1}=G_{\mathfrak{p}}(s,s_1,s_2)$, which is defined by \eqref{eq7.3}. Substituting \eqref{7.6} and \eqref{7.7} into \eqref{f7.6}, we obtain the expression \eqref{7.9}, implying that $L(s,s_1,s_2;\rho,\mathfrak{q})$ converges absolutely in the region \eqref{equation7.1}.  

Let $(s,s_1,s_2)\in \mathbb{C}^3$ satisfy \eqref{equation7.1}. Then 
\begin{align*}
E(s,s_1,s_2;\rho,\mathfrak{q}):=\frac{L(s,s_1,s_2;\rho,\mathfrak{q})\zeta_F(1+s_1+s/2)^2\zeta_F(1+s_2+s/2)^2}{\zeta_F(1+s_1+s_2)} 
\end{align*}
is equal to \eqref{eq7.6}. Moreover, 
\begin{equation}\label{7.15}
\frac{\zeta_{\mathfrak{p}}(1+s_1+s/2)^2\zeta_{\mathfrak{p}}(1+s_2+s/2)^2G_{\mathfrak{p}}(s,s_1,s_2)}{\zeta_{\mathfrak{p}}(1+s_1+s_2)}=1+O(M(s,s_1,s_2)),
\end{equation}
where the implied constant is absolute, and 
\begin{align*}
M(s,s_1,s_2):=&\frac{1}{N(\mathfrak{p})^{2+\Re(s_1+s_2+s)}}+\frac{1}{N(\mathfrak{p})^{2+2\Re(s_1)+\Re(s)}}+\frac{1}{N(\mathfrak{p})^{2+2\Re(s_2)+\Re(s)}}\\
&+\frac{1}{N(\mathfrak{p})^{3+2\Re(s_1)+\Re(s_2)+\frac{3\Re(s)}{2}}}+\frac{1}{N(\mathfrak{p})^{3+\Re(s_1)+2\Re(s_2)+\frac{3\Re(s)}{2}}}.
\end{align*}

As a consequence, the product of \eqref{7.15} over $\mathfrak{p}\nmid\mathfrak{q}$ converges absolutely in the region \eqref{eq7.4}, from which we conclude \eqref{eq7.5}.
\end{proof}

\begin{lemma}
Let notation be as in Lemma \ref{lem7.1}. 
\begin{itemize}
\item We have
\begin{equation}\label{equation7.15}
E(0,0,0;\rho,\mathfrak{q})=L_{\mathfrak{q}}(1,\rho)^3\prod_{\mathfrak{p}\nmid\mathfrak{q}}\bigg[1-\frac{\rho(\mathfrak{p})N(\mathfrak{p})^{-1}(1-\rho(\mathfrak{p})-\rho(\mathfrak{p})^2N(\mathfrak{p})^{-2})}{(1-\rho(\mathfrak{p})N(\mathfrak{p})^{-1})^{3}}\bigg].	
\end{equation} 
\item Suppose $\rho$ is defined by \eqref{rho}. We have
\begin{equation}\label{a7.16}
E(0,0,0;\rho,\mathfrak{q})=\begin{cases}
\zeta_{\mathfrak{q}}(1)^3\zeta_{\mathfrak{q}}(2)^{-2}\zeta(2)^{-1},\ \ &\text{if $\mathfrak{q}\subsetneq \mathcal{O}_F$,}\\
\zeta(2)^{-1},\ \ &\text{if $\mathfrak{q}=\mathcal{O}_F$}.
\end{cases}
\end{equation}
\item Suppose that 
\begin{equation}\label{b7.17}
\sum_{\mathfrak{p}}\frac{|1-\rho(\mathfrak{p})|}{N(\mathfrak{q})}<\infty.
\end{equation}
Then we have 
\begin{equation}\label{a7.18}
R(\rho):=\lim_{s\to 0}\frac{L(1+s,\rho)}{\zeta_F(1+s)}=\prod_{\mathfrak{p}}\bigg[1-\frac{(1-\rho(\mathfrak{p}))N(\mathfrak{p})^{-1}}{1-\rho(\mathfrak{p})N(\mathfrak{p})^{-1}}\bigg].	
\end{equation}

\item Assume \eqref{b7.17}. Then $E(0,0,0;\rho,\mathfrak{q})R(\rho)^{-3}$ is equal to  
\begin{equation}\label{a7.19}
L_{\mathfrak{q}}(1,\rho)^2\zeta_{\mathfrak{q}}(1)\prod_{\mathfrak{p}}\frac{1-(4\rho(\mathfrak{p})-\rho(\mathfrak{p})^2)N(\mathfrak{p})^{-1}+3\rho(\mathfrak{p})^2N(\mathfrak{p})^{-2}}{(1-N(\mathfrak{p})^{-1})^3}
\end{equation}
if $\mathfrak{q}\subsetneq\mathcal{O}_F$, and for $\mathfrak{q}=\mathcal{O}_F$, 
\begin{equation}\label{b7.19}
\frac{E(0,0,0;\rho,\mathfrak{q})}{R(\rho)^3}=\prod_{\mathfrak{p}}\frac{1-(4\rho(\mathfrak{p})-\rho(\mathfrak{p})^2)N(\mathfrak{p})^{-1}+3\rho(\mathfrak{p})^2N(\mathfrak{p})^{-2}}{(1-N(\mathfrak{p})^{-1})^3}.
\end{equation}

\item In particular, taking $\rho$ as defined by \eqref{rho}, \eqref{a7.19} simples to 
\begin{equation}\label{a7.20}
\frac{E(0,0,0;\rho,\mathfrak{q})}{R(\rho)^3}=\zeta_{\mathfrak{q}}(1)L_{\mathfrak{q}}(1,\rho)^2\prod_{\mathfrak{p}}\frac{1}{(1-N(\mathfrak{p})^{-2})^2}=\frac{\zeta_{\mathfrak{q}}(1)^3\zeta_F(2)^2}{\zeta_{\mathfrak{q}}(2)^2}
\end{equation}
if $\mathfrak{q}\subsetneq\mathcal{O}_F$, and for $\mathfrak{q}=\mathcal{O}_F$, 
\begin{equation}\label{a7.22}
\frac{E(0,0,0;\rho,\mathfrak{q})}{R(\rho)^3}=\zeta_F(2)^2.
\end{equation}
\end{itemize}
\end{lemma}
\begin{proof}
Recall the definition in \eqref{eq7.6} of  $E(s,s_1,s_2;\rho,\mathfrak{q})$:
\begin{align*}
\frac{L_{\mathfrak{q}}(1+s_1+s/2,\rho)^2L_{\mathfrak{q}}(1+s_2+s/2,\rho)^2}{L_{\mathfrak{q}}(1+s_1+s_2,\rho)}\prod_{\mathfrak{p}\nmid\mathfrak{q}}\frac{L_{\mathfrak{p}}(1+s_1+s/2,\rho)^2L_{\mathfrak{p}}(1+s_2+s/2,\rho)^2}{L_{\mathfrak{p}}(1+s_1+s_2,\rho)G_{\mathfrak{p}}(s,s_1,s_2)^{-1}}.
\end{align*}

Take $s_1=s_2=s/2$, we obtain 
\begin{equation}\label{fc7.16}
E(s,s/2,s/2;\rho,\mathfrak{q})=L_{\mathfrak{q}}(1+s,\rho)^3\prod_{\mathfrak{p}\nmid\mathfrak{q}}\frac{G_{\mathfrak{p}}(s,s/2,s/2)}{(1-\rho(\mathfrak{p})N(\mathfrak{p})^{-1-s})^{3}}.
\end{equation}

By \eqref{eq7.3} we obtain 
\begin{equation}\label{fc7.17}
G_{\mathfrak{p}}(s,s/2,s/2)=1-(4\rho(\mathfrak{p})-\rho(\mathfrak{p})^2)N(\mathfrak{p})^{-1-s}+3\rho(\mathfrak{p})^2N(\mathfrak{p})^{-2-2s}.
\end{equation}

Combining \eqref{fc7.16} and \eqref{fc7.17} leads to  
\begin{align*}
E(s,s/2,s/2;\rho,\mathfrak{q})=L_{\mathfrak{q}}(1,\rho)^3\prod_{\mathfrak{p}\nmid\mathfrak{q}}\frac{1-(4\rho(\mathfrak{p})-\rho(\mathfrak{p})^2)N(\mathfrak{p})^{-1-s}+3\rho(\mathfrak{p})^2N(\mathfrak{p})^{-2-2s}}{(1-\rho(\mathfrak{p})N(\mathfrak{p})^{-1-s})^{3}}.
\end{align*}

Taking the limit $s\to 0$, we obtain 
\begin{align*}
\lim_{s\to 0}E(s,s/2,s/2;\rho,\mathfrak{q})=L_{\mathfrak{q}}(1,\rho)^3\prod_{\mathfrak{p}\nmid\mathfrak{q}}\frac{1-(4\rho(\mathfrak{p})-\rho(\mathfrak{p})^2)N(\mathfrak{p})^{-1}+3\rho(\mathfrak{p})^2N(\mathfrak{p})^{-2}}{(1-\rho(\mathfrak{p})N(\mathfrak{p})^{-1})^{3}},
\end{align*}
which simplifies to \eqref{equation7.15}.

Taking $\rho(\mathfrak{p})=(1+N(\mathfrak{p})^{-1})^{-1}$, the formula \eqref{equation7.15} simplifies to 
\begin{align*}
E(0,0,0;\rho,\mathfrak{q})=(1-\rho(\mathfrak{q})N(\mathfrak{q})^{-1})^{-3}\prod_{\mathfrak{p}\nmid\mathfrak{q}}(1-N(\mathfrak{p})^{-2})=\frac{(1+N(\mathfrak{q})^{-1})^2}{(1-N(\mathfrak{q})^{-1})\zeta_F(2)}
\end{align*}
if $\mathfrak{q}\neq\mathcal{O}_F$, and $E(0,0,0;\rho,\mathfrak{q})=\zeta_F(2)^{-1}$ if $\mathfrak{q}=\mathcal{O}_F$, which yields \eqref{a7.16}. 

The formula \eqref{a7.18} follows formally from \eqref{equation7.15}. As a consequence of \eqref{b7.17}, the infinite product in \eqref{a7.18} converges absolutely. In particular, $R(\rho)$ is well defined. Moreover, \eqref{a7.19} and \eqref{b7.19} follow from a direct calculation utilizing \eqref{equation7.15}, \eqref{a7.16} and \eqref{a7.18}, along with \eqref{b7.17}. As a special case, \eqref{a7.20} and \eqref{a7.22} holds from a direct calculation.
\end{proof}

\begin{prop}\label{prop7.3}
Let notation be as before. Let $G(s)$ be a holomorphic function in $\Re(s)>-1$. Let $0<\varepsilon<10^{-3}$. Suppose 
\begin{equation}\label{7.25}
\sum_{\mathfrak{p}}\frac{|1-\rho(\mathfrak{p})|}{N(\mathfrak{q})^{1-\varepsilon}}<\infty,\ \ |\rho(\mathfrak{p})|<N(\mathfrak{p})^{1-\varepsilon}\ \ \text{for all prime $\mathfrak{p}$}.
\end{equation}
 
\begin{itemize}
\item $\mathbb{L}(G,\rho,\mathfrak{q})$ converges, and  
\begin{equation}\label{main7.15}
\mathbb{L}(G,\rho,\mathfrak{q})=\frac{E(0,0,0;\rho,\mathfrak{q})}{R(\rho)^3(\Res_{s=1}\zeta_F(s))^2}\cdot\bigg[\frac{G'(0)}{\log\xi}+G(0)\bigg]+O(G(0)(\log\xi)^{-1}),
\end{equation}
where the implied constant depends only on $F$ and $\rho$. Here $R(\rho)$ is defined by \eqref{a7.18}.
\item Taking $\rho$ as in \eqref{rho}. Then 
\begin{equation}\label{equ7.28}
\mathbb{L}(G,\rho,\mathfrak{q})=\frac{C_{\mathfrak{q}}\cdot \zeta_F(2)^2}{(\Res_{s=1}\zeta_F(s))^2}\cdot\bigg[\frac{G'(0)}{\log\xi}+G(0)\bigg]+O(G(0)(\log\xi)^{-1}),
\end{equation}
where $C_{\mathfrak{q}}:=\zeta_{\mathfrak{q}}(1)^3\zeta_{\mathfrak{q}}(2)^{-2}\textbf{1}_{\mathfrak{q}\subsetneq \mathcal{O}_F}+\textbf{1}_{\mathfrak{q}=\mathcal{O}_F}$, and the implied constant depends only on $F$.
\end{itemize}
\end{prop}
\begin{proof}
By definition of $L(1+s,\rho)$ in \eqref{eq7.1} and the assumption \eqref{7.25} we have 
\begin{equation}\label{eq7.28}
\frac{L(1+s,\rho)}{\zeta_F(1+s)}=\prod_{\mathfrak{p}}\bigg[1-\frac{(1-\rho(\mathfrak{p}))N(\mathfrak{p})^{-1-s}}{1-\rho(\mathfrak{p})N(\mathfrak{p})^{-1-s}}\bigg],
\end{equation}
implying that $L(1+s,\rho)/\zeta_F(1+s)$ converges absolutely in $\Re(s)>-\varepsilon$. Consequently, $L(1+s,\rho)$ is meromorphic in $\Re(s)>-\varepsilon$ and is holomorphic at $s\neq 0$. Moreover, $L(1+s,\rho)$ has a simple pole at $s=0$, with 
\begin{equation}\label{7.28}
\underset{s=1}{\Res}\ L(s,\rho)=R(\rho)\cdot \underset{s=1}{\Res}\ \zeta_F(s),
\end{equation}
where $R(\rho)$ is defined as in \eqref{a7.18}. Define  
\begin{equation}\label{7.29}
\mathcal{L}_{\varepsilon}:=\big\{z=it:\ \text{$t>\varepsilon$ or $t<-\varepsilon$}\big\}\bigcup \big\{z=\varepsilon e^{i\theta}:\ \pi/2\leq \theta\leq 3\pi/2\big\}.
\end{equation}

Let $\Re(s_1)>-10^{-1}$ and $\Re(s_2)>-10^{-1}$. Denote by 
\begin{align*}
\mathcal{I}(s_1,s_2):=\frac{1}{2\pi i}\oint_{\mathcal{C}_{\varepsilon}}\frac{G(s)\zeta_F(1+s)L(s,s_1,s_2;\rho,\mathfrak{q})}{s}ds.
\end{align*}
Utilizing the formula \eqref{eq7.5} in Lemma \ref{lem7.1}:
\begin{align*}
L(s,s_1,s_2;\rho,\mathfrak{q})=\frac{L(1+s_1+s_2,\rho)E(s,s_1,s_2;\rho,\mathfrak{q})}{L(1+s_1+s/2,\rho)^2L(1+s_2+s/2,\rho)^2},
\end{align*}
we can rewrite $\mathcal{I}(s_1,s_2)$ as 
\begin{align*}
\frac{1}{2\pi i}\oint_{\mathcal{C}_{\varepsilon}}\frac{s\zeta_F(1+s)G(s)E(s,s_1,s_2;\rho,\mathfrak{q})}{s^2}\cdot \frac{L(1+s_1+s_2,\rho)}{L(1+s_1+s/2,\rho)^2L(1+s_2+s/2,\rho)^2}ds.
\end{align*}

By Cauchy integral, we derive  
\begin{align*}
\mathcal{I}(s_1,s_2)=&\frac{d(s\zeta_F(1+s)G(s))}{ds}\bigg|_{s=0}\cdot \frac{E(0,s_1,s_2;\rho,\mathfrak{q})L(1+s_1+s_2,\rho)}{L(1+s_1,\rho)^2L(1+s_2,\rho)^2}\\
&+\underset{s=1}{\Res}\ \zeta_F(s) G(0)\cdot\frac{\partial E(s,s_1,s_2;\rho,\mathfrak{q})}{\partial s}\bigg|_{s=0}\cdot \frac{L(1+s_1+s_2,\rho)}{L(1+s_1,\rho)^2L(1+s_2,\rho)^2}\\
&-\underset{s=1}{\Res}\ \zeta_F(s) G(0)E(0,s_1,s_2;\rho,\mathfrak{q})\cdot \frac{L(1+s_1+s_2,\rho)L'(1+s_1,\rho)}{L(1+s_1,\rho)^3L(1+s_2,\rho)^2}\\
&-\underset{s=1}{\Res}\ \zeta_F(s) G(0)E(0,s_1,s_2;\rho,\mathfrak{q})\cdot \frac{L(1+s_1+s_2,\rho)L'(1+s_2,\rho)}{L(1+s_1,\rho)^2L(1+s_2,\rho)^3}.
\end{align*}

Substituting this into the definition of $\mathbb{L}(G,\rho,\mathfrak{q})$ yields 
\begin{equation}\label{7.30}
\mathbb{L}(G,\rho,\mathfrak{q})=\mathbb{L}_1+\mathbb{L}_2+\mathbb{L}_3+\mathbb{L}_4,	
\end{equation}
where $\mathbb{L}_1$ is defined by 
\begin{align*}
-\frac{d(s\zeta_F(1+s)G(s))}{4\pi^2(\log \xi)^2ds}\bigg|_{s=0}\int_{(2)} \int_{(2)}\frac{\xi^{s_1}\xi^{s_2}E(0,s_1,s_2;\rho,\mathfrak{q})L(1+s_1+s_2,\rho)}{s_1^3s_2^3L(1+s_1,\rho)^2L(1+s_2,\rho)^2}
ds_1ds_2,
\end{align*}
and $\mathbb{L}_2$ is defined by 
\begin{align*}
-\underset{s=1}{\Res}\ \zeta_F(s) G(0)\int_{(2)} \int_{(2)}\frac{\partial E(s,s_1,s_2;\rho,\mathfrak{q})}{4\pi^2(\log \xi)^2\partial s}\bigg|_{s=0} \frac{s_1^{-3}\xi^{s_1}\xi^{s_2}L(1+s_1+s_2,\rho)}{s_2^3L(1+s_1,\rho)^2L(1+s_2,\rho)^2}
ds_1ds_2,
\end{align*}
and $\mathbb{L}_3$ is defined by 
\begin{align*}
\frac{\Res_{s=1}\zeta_F(s) G(0)}{4\pi^2(\log \xi)^2}\int_{(2)} \int_{(2)}\frac{\xi^{s_1}\xi^{s_2}E(0,s_1,s_2;\rho,\mathfrak{q})L(1+s_1+s_2,\rho)L'(1+s_1,\rho)}{s_1^3s_2^3L(1+s_1,\rho)^3L(1+s_2,\rho)^2}
ds_1ds_2,
\end{align*}
and $\mathbb{L}_4$ is defined by 
\begin{align*}
\frac{\Res_{s=1}\zeta_F(s) G(0)}{4\pi^2(\log \xi)^2}\int_{(2)} \int_{(2)}\frac{\xi^{s_1}\xi^{s_2}E(0,s_1,s_2;\rho,\mathfrak{q})L(1+s_1+s_2,\rho)L'(1+s_2,\rho)}{s_1^3s_2^3L(1+s_1,\rho)^2L(1+s_2,\rho)^3}
ds_1ds_2.
\end{align*}

Now we proceed to compute or estimate each $\mathbb{L}_j$, $1\leq j\leq 4$. 
\begin{itemize}
\item We consider the asymptotic behavior of $\mathbb{L}_1$. Shifting contour, the integral 
\begin{align*}
-\frac{1}{4\pi^2}\int_{(2)} \int_{(2)}\frac{\xi^{s_1}\xi^{s_2}}{s_1^3s_2^3}\cdot \frac{E(0,s_1,s_2;\rho,\mathfrak{q})L(1+s_1+s_2,\rho)}{L(1+s_1,\rho)^2L(1+s_2,\rho)^2}
ds_1ds_2
\end{align*}
is equal to 
\begin{align*}
&\frac{1}{(\Res_{s=1}L(s,\rho))^2}\cdot \frac{d(E(0,0,s_2;\rho,\mathfrak{q})(s_2L(1+s_2,\rho))^{-1}\xi^{s_2})}{d s_2}\bigg|_{s_2=0}
\\
&+\frac{1}{2\pi i}\int_{\mathcal{C}} \frac{E(0,0,s_2;\rho,\mathfrak{q})}{(\Res_{s=1}L(s,\rho))^2L(1+s_2,\rho)}\cdot \frac{\xi^{s_2}}{s_2^3}
ds_2\\
&+\frac{1}{2\pi i}\int_{\mathcal{C}}\frac{E(0,s_1,0;\rho,\mathfrak{q})}{(\Res_{s=1}L(s,\rho))^2L(1+s_1,\rho)}\cdot \frac{\xi^{s_1}}{s_1^3}ds_1\\
&-\frac{1}{4\pi^2}\int_{\mathcal{C}}\int_{\mathcal{C}} \frac{\xi^{s_1}\xi^{s_2}}{s_1^3s_2^3}\cdot \frac{E(0,s_1,s_2;\rho,\mathfrak{q})L(1+s_1+s_2,\rho)}{L(1+s_1,\rho)^2\L(1+s_2,\rho)^2}
ds_2ds_1.
\end{align*}

Notice that the first term 
\begin{align*}
\frac{1}{(\Res_{s=1}L(s,\rho))^2}\cdot \frac{d(E(0,0,s_2;\rho,\mathfrak{q})(s_2L(1+s_2,\rho))^{-1}\xi^{s_2})}{d s_2}\bigg|_{s_2=0}
\end{align*}
is equal to 
\begin{align*}
&\frac{1}{(\Res_{s=1}L(s,\rho))^3}\cdot \frac{d(E(0,0,s_2;\rho,\mathfrak{q}))}{d s_2}\bigg|_{s_2=0}+\frac{E(0,0,0;\rho,\mathfrak{q})\log\xi}{(\Res_{s=1}L(s,\rho))^3}\\
&+\frac{E(0,0,0;\rho,\mathfrak{q})}{(\Res_{s=1}L(s,\rho))^2}\cdot \frac{d((s_2L(1+s_2,\rho))^{-1})}{d s_2}\bigg|_{s_2=0}.
\end{align*}

In conjunction with the well known bound 
\begin{equation}\label{eq7.32}
\frac{\zeta_F'(s)}{\zeta_F(s)}\ll \log(\Im(s)),\ \ \zeta_F(s)^{-1}\ll \log(\Im(s))
\end{equation}
on the contour $\mathcal{L}_{\varepsilon}$, and the absolute convergence of $L(1+s,\rho)\zeta_F(1+s)^{-1}$ (cf. \eqref{eq7.28}) in the region $\Re(s)>-\varepsilon$, we derive that 
\begin{equation}\label{7.32}
\mathbb{L}_1=\frac{d(s\zeta_F(1+s)G(s))}{(\log \xi)^2ds}\bigg|_{s=0}\bigg[\frac{E(0,0,0;\rho,\mathfrak{q})\log\xi}{(\Res_{s=1}L(s,\rho))^3}+O(1)\bigg],
\end{equation}
where the implied constant in $O(1)$ depends only on $F$ and $\rho$.  

\item We estimate the term $\mathbb{L}_2$ as follows. The inner integral 
\begin{align*}
-\frac{1}{4\pi^2}\int_{(2)} \int_{(2)}\frac{\xi^{s_1}\xi^{s_2}}{s_1^3s_2^3}\cdot \frac{\partial E(s,s_1,s_2;\rho,\mathfrak{q})}{\partial s}\bigg|_{s=0}\cdot \frac{L(1+s_1+s_2,\rho)}{L(1+s_1,\rho)^2L(1+s_2,\rho)^2}
ds_1ds_2
\end{align*}
in the definition of $\mathbb{L}_2$ is equal to 
\begin{align*}
& \frac{1}{(\Res_{s=1}L(s,\rho))^2}\cdot \frac{d}{ds_2}\bigg[\frac{\partial E(s,0,s_2;\rho,\mathfrak{q})}{\partial s}\bigg|_{s=0}\cdot \xi^{s_2}(s_2L(1+s_2,\rho))^{-1}\bigg]\bigg|_{s_2=0}
\\
& \frac{1}{2\pi i}\int_{\mathcal{C}}\frac{1}{(\Res_{s=1}L(s,\rho))^2}\cdot \frac{\partial E(s,0,s_2;\rho,\mathfrak{q})}{\partial s}\bigg|_{s=0}\cdot \frac{\xi^{s_2}}{s_2^3L(1+s_2,\rho)}
ds_2\\
&+\frac{1}{2\pi i}\int_{\mathcal{C}}\frac{1}{(\Res_{s=1}L(s,\rho))^2}\cdot \frac{\partial E(s,s_1,0;\rho,\mathfrak{q})}{\partial s}\bigg|_{s=0}\cdot \frac{\xi^{s_1}}{s_1^3L(1+s_1,\rho)}
ds_1\\
&-\frac{1}{4\pi^2}\int_{\mathcal{C}} \int_{\mathcal{C}}\frac{\xi^{s_1}\xi^{s_2}}{s_1^3s_2^3}\cdot \frac{\partial E(s,s_1,s_2;\rho,\mathfrak{q})}{\partial s}\bigg|_{s=0}\cdot \frac{L(1+s_1+s_2,\rho)}{L(1+s_1,\rho)^2L(1+s_2,\rho)^2}
ds_1ds_2.
\end{align*}

The first term in the above expression is 
\begin{align*}
&\frac{1}{(\Res_{s=1}L(s,\rho))^2}\cdot \frac{d}{ds_2}\bigg[\frac{\partial E(s,0,s_2;\rho,\mathfrak{q})}{\partial s}\bigg|_{s=0}\cdot \xi^{s_2}(s_2L(1+s_2,\rho))^{-1}\bigg]\bigg|_{s_2=0}\\
=&\frac{1}{(\Res_{s=1}L(s,\rho))^3}\cdot \frac{\partial^2 E(s,0,s_2;\rho,\mathfrak{q})}{\partial s_2\partial s}\bigg|_{\substack{s=0\\
s_2=0}}+\frac{\log\xi}{(\Res_{s=1}L(s,\rho))^3}\cdot \frac{\partial E(s,0,0;\rho,\mathfrak{q})}{\partial s}\\
&+\frac{1}{(\Res_{s=1}L(s,\rho))^2}\cdot \frac{\partial E(s,0,0;\rho,\mathfrak{q})}{\partial s}\bigg|_{s=0}\cdot \frac{d(s_2L(1+s_2,\rho))^{-1}}{ds_2}\bigg|_{s_2=0}.
\end{align*}

Substituting the above calculation into the definition of $\mathbb{L}_2$, along with the estimate \eqref{eq7.32} we derive 
\begin{equation}\label{7.34}
\mathbb{L}_2=\frac{\Res_{s=1}\zeta_F(s) G(0)}{(\log \xi)^2}\cdot\bigg[\frac{\log\xi}{(\Res_{s=1}L(s,\rho))^3}\cdot \frac{\partial E(s,0,0;\rho,\mathfrak{q})}{\partial s}+O(1)\bigg].
\end{equation}

\item The term $\mathbb{L}_3$ can be computed similarly by shifting contours. By Cauchy formula, the inner integral 
\begin{align*}
-\frac{1}{4\pi^2}\int_{(2)} \int_{(2)}\frac{\xi^{s_1}\xi^{s_2}}{s_1^3s_2^3}\cdot \frac{E(0,s_1,s_2;\rho,\mathfrak{q})L(1+s_1+s_2,\rho)L'(1+s_1,\rho)}{L(1+s_1,\rho)^3L(1+s_2,\rho)^2}
ds_1ds_2
\end{align*}
in the definition of $\mathbb{L}_3$ is equal to 
\begin{align*}
&\frac{1}{(\Res_{s=1}L(s,\rho))^2}\cdot \frac{1}{2}\frac{d^2}{ds_1^2}\bigg[\frac{E(0,s_1,0;\rho,\mathfrak{q})L'(1+s_1,\rho)\xi^{s_1}}{L(1+s_1,\rho)^2}\bigg]\bigg|_{s_1=0}\\
&+\frac{1}{(\Res_{s=1}L(s,\rho))^2}\cdot \frac{1}{2\pi i}\int_{\mathcal{C}} \frac{E(0,s_1,0;\rho,\mathfrak{q})L'(1+s_1,\rho)}{L(1+s_1,\rho)^2}
\frac{\xi^{s_1}}{s_1^3}ds_1\\
&+\frac{1}{(\Res_{s=1}L(s,\rho))^2}\cdot \frac{1}{2\pi i}\int_{\mathcal{C}} \frac{E(0,0,s_2;\rho,\mathfrak{q})L'(1+s_2,\rho)}{L(1+s_2,\rho)^2}
\frac{\xi^{s_2}}{s_2^3}ds_2\\
&-\frac{1}{4\pi^2} \int_{\mathcal{C}}\int_{\mathcal{C}}\frac{\xi^{s_1}\xi^{s_2}}{s_1^3s_2^3}\cdot \frac{E(0,s_1,s_2;\rho,\mathfrak{q})L(1+s_1+s_2,\rho)L'(1+s_1,\rho)}{L(1+s_1,\rho)^3L(1+s_2,\rho)^2}
ds_1ds_2.
\end{align*}

The first term in the above expression is 
\begin{align*}
&-\frac{1}{2(\Res_{s=1}L(s,\rho))^2}\cdot \frac{d^2}{ds_1^2}\bigg[\frac{E(0,s_1,0;\rho,\mathfrak{q})L'(1+s_1,\rho)\xi^{s_1}}{L(1+s_1,\rho)^2}\bigg]\bigg|_{s_1=0}\\
=&-\frac{(\log\xi)^2}{2\cdot (\Res_{s=1}L(s,\rho))^2}\cdot \lim_{s_1=0}\frac{E(0,0,0;\rho,\mathfrak{q})L'(1+s_1,\rho)}{L(1+s_1,\rho)^2}+O((\log\xi))\\
=&\frac{(\log\xi)^2E(0,0,0;\rho,\mathfrak{q})}{2\cdot  (\Res_{s=1}L(s,\rho))^3}+O((\log\xi)).
\end{align*}

Substituting these calculations into the definition of $\mathbb{L}_3$ yields 
\begin{equation}\label{7.35}
\mathbb{L}_3=\frac{\Res_{s=1}\zeta_F(s)G(0)E(0,0,0;\rho,\mathfrak{q})}{2\cdot  (\Res_{s=1}L(s,\rho))^3}+O(G(0)(\log\xi)^{-1}).
\end{equation}

\item By the symmetry between $\mathbb{L}_3$ and $\mathbb{L}_4$, we obtain 
\begin{equation}\label{7.36}
\mathbb{L}_4=\frac{\Res_{s=1}\zeta_F(s)G(0)E(0,0,0;\rho,\mathfrak{q})}{2\cdot  (\Res_{s=1}L(s,\rho))^3}+O(G(0)(\log\xi)^{-1}).
\end{equation}
\end{itemize}

Therefore, combining the estimates \eqref{7.32}, \eqref{7.34}, \eqref{7.35} and \eqref{7.36}, we deduce that $\mathbb{L}(G,\rho,\mathfrak{q})$ is equal to 
\begin{align*}
\frac{\Res_{s=1}\zeta_F(s)E(0,0,0;\rho,\mathfrak{q})G'(0)}{\log\xi\cdot (\Res_{s=1}L(s,\rho))^3}+\frac{\Res_{s=1}\zeta_F(s)G(0)E(0,0,0;\rho,\mathfrak{q})}{2\cdot  (\Res_{s=1}L(s,\rho))^3}+O\left(\frac{G(0)}{(\log\xi)}\right).	
\end{align*}

Therefore, \eqref{main7.15}  follows from the above expression and  the fact that
\begin{align*}
\Res_{s=1}L(s,\rho)=R(\rho)\Res_{s=1}\zeta_F(s).
\end{align*} 

Substituting \eqref{a7.20} and \eqref{a7.22} into \eqref{main7.15}, we derive \eqref{equ7.28}.
\end{proof}

\subsection{The Mollified Second Moment: New Forms}
\begin{thmx}\label{thmC}
Let notation be as before. Let $\rho$ be defined by \eqref{rho}. Let $\xi>1$ and $0<\varepsilon<10^{-3}$. Let $\mathfrak{q}\subseteq \mathcal{O}_F$ be an integral ideal. Let $\zeta_{\mathfrak{q}}(s)=(1-N(\mathfrak{q})^{-s})^{-1}$ if $\mathfrak{q}$ is a prime ideal, and $\zeta_{\mathfrak{q}}(s)\equiv 1$ if $\mathfrak{q}=\mathcal{O}_F$. Then 
\begin{equation}\label{eq7.65}
\sum_{\substack{\pi\in \mathcal{F}(\mathbf{k},\mathfrak{q})}}\frac{L(1/2,\pi)^2M_{\xi,\rho}(\pi)^2}{L(1,\pi,\Ad)}=\frac{4\zeta_F(2)^2D_F^{\frac{3}{2}}}{(\Res_{s=1}\zeta_F(s))^2}\prod_{v\mid\infty}\frac{k_v-1}{4\pi^2}\cdot \mathcal{M}_{\mathfrak{q},\mathbf{k}}+\mathcal{E}_{\mathfrak{q},\mathbf{k}},
\end{equation}
where 
\begin{align*}
\mathcal{M}_{\mathfrak{q},\mathbf{k}}:=&c_{\mathfrak{q}}\cdot (N(\mathfrak{q})+1)\cdot\bigg[\frac{\log N(\mathfrak{q})^{1/2}\|\mathbf{k}\|}{\log\xi}+1\bigg]- \frac{4\zeta_{\mathfrak{q}}(2)\cdot\delta_{\mathbf{k}}\cdot\textbf{1}_{\mathfrak{q}\subsetneq \mathcal{O}_F}}{1+N(\mathfrak{q})^{-1}}\bigg[\frac{\log \|\mathbf{k}\|}{\log\xi}+1\bigg],
\end{align*}
with $\delta_{\mathbf{k}}:=\textbf{1}_{\sum_{v\mid\infty}k_v\equiv 0\pmod{4}}$, $c_{\mathfrak{q}}:=\zeta_{\mathfrak{q}}(1)^3\zeta_{\mathfrak{q}}(2)^{-1}\textbf{1}_{\mathfrak{q}\subsetneq \mathcal{O}_F}+\delta_{\mathbf{k}}\textbf{1}_{\mathfrak{q}=\mathcal{O}_F}$, and 
\begin{align*}
\mathcal{E}_{\mathfrak{q},\mathbf{k}}\ll (\log\xi)^{-1}N(\mathfrak{q})\|\mathbf{k}\|+\xi^{2+\varepsilon}N(\mathfrak{q})^{\varepsilon}\|\mathbf{k}\|^{1/2+\varepsilon},
\end{align*}
with the implied constant depending only on $\varepsilon$ and $F$. 
\end{thmx}
\begin{proof}
Consider the following two scenarios according to $\mathfrak{q}=\mathcal{O}_F$ and $\mathfrak{q}\subsetneq \mathcal{O}_F$.
\begin{itemize}
\item Suppose $\mathfrak{q}=\mathcal{O}_F$. Then by \eqref{sing7.17} and \eqref{maingeom} in Proposition \ref{prop7.6}, we obtain  
\begin{equation}\label{eq7.66}
\sum_{\substack{\pi\in \mathcal{F}(\mathbf{k},\mathcal{O}_F)}}\frac{L(1/2,\pi)^2M_{\xi,\rho}(\pi)^2}{L(1,\pi,\Ad)}=2\delta_{\mathbf{k}}D_F^{\frac{3}{2}}\prod_{v\mid\infty}\frac{k_v-1}{2\pi^2}\cdot \mathbb{L}+O(\xi^{2+\varepsilon}\|\mathbf{k}\|^{1/2+\varepsilon}),
\end{equation}
where $\mathbb{L}:=\mathbb{L}(G_{\mathcal{O}_F},\rho,\mathcal{O}_F)$ is defined as in Definition \ref{defn7.4}, and the implied constant depends only on $\varepsilon$ and $F$.  By Proposition \ref{prop7.3}, 
\begin{align*}
\mathbb{L}(G_{\mathcal{O}_F},\rho,\mathcal{O}_F)=\frac{ \zeta_F(2)^2}{(\Res_{s=1}\zeta_F(s))^2}\cdot\bigg[\frac{G_{\mathcal{O}_F}'(0)}{\log\xi}+G_{\mathcal{O}_F}(0)\bigg]+O(G_{\mathcal{O}_F}(0)(\log\xi)^{-1}).
\end{align*}

Recall the definition, for an integral ideal $\mathfrak{a}\subseteq\mathcal{O}_F$, 
\begin{align*}
G_{\mathfrak{a}}(s):= (1+N(\mathfrak{a})^s)D_F^{s}\prod_{v\mid\infty}\frac{\Gamma((k_v+s)/2)^2}{(2\pi)^{s}\Gamma(k_v/2)^2}.\tag{\ref{eq7.18}}
\end{align*}
Hence, $G_{\mathfrak{a}}(0)=2$. Moreover, 
taking the derivative in \eqref{eq7.18}, along with the asymptotic behavior of digamma function
\begin{align*}
\Gamma(z)'/\Gamma(z)=\log z-(2z)^{-1}-(12z^2)^{-1}+O(z^{-6}),
\end{align*} 
we obtain 
\begin{equation}\label{7.66}
G_{\mathfrak{a}}'(s)=\log (N(\mathfrak{a}) D_F^2)+4\frac{d}{ds}\prod_{v\mid\infty}\frac{\Gamma((k_v+s)/2)^2}{2\cdot (2\pi)^{s}\Gamma(k_v/2)^2}\bigg|_{s=0}=\log N(\mathfrak{a})\|\mathbf{k}\|^2+O(1).
\end{equation}

Therefore, it follows from \eqref{eq7.66} and \eqref{7.66} that
\begin{equation}\label{7.64}
\sum_{\substack{\pi\in \mathcal{F}(\mathbf{k},\mathcal{O}_F)}}\frac{L(1/2,\pi)^2M_{\xi,\rho}(\pi)^2}{L(1,\pi,\Ad)}=\frac{8\delta_{\mathbf{k}}D_F^{\frac{3}{2}}\zeta_F(2)^2}{(\Res_{s=1}\zeta_F(s))^2}\prod_{v\mid\infty}\frac{k_v-1}{4\pi^2}\cdot\bigg[\frac{\log \|\mathbf{k}\|}{\log\xi}+1\bigg]+\mathcal{E}_1,
\end{equation}
where 
\begin{equation}\label{equ7.68}
\mathcal{E}_1\ll (\log\xi)^{-1}N(\mathfrak{q})\|\mathbf{k}\|+\xi^{2+\varepsilon}\|\mathbf{k}\|^{1/2+\varepsilon}	
\end{equation}
with the implied constant depending only on $\varepsilon$ and $F$.

\item Suppose $\mathfrak{q}\subsetneq \mathcal{O}_F$. By \eqref{sing7.17} and \eqref{maingeom} in Proposition \ref{prop7.6}, 
\begin{equation}\label{7.68}
\sum_{\substack{\pi\in \mathcal{F}(\mathbf{k},\mathfrak{q})}}\frac{L(1/2,\pi)^2M_{\xi,\rho}(\pi)^2}{\zeta_{\mathfrak{q}}(2)L^{(\mathfrak{q})}(1,\pi,\Ad)}=2\mathcal{M}'D_F^{\frac{3}{2}}\prod_{v\mid\infty}\frac{k_v-1}{4\pi^2}+O(\xi^{2+\varepsilon}N(\mathfrak{q})^{\varepsilon}\|\mathbf{k}\|^{\frac{1}{2}+\varepsilon}),
\end{equation}
where 
\begin{equation}\label{eq7.69}
\mathcal{M}':=\zeta_{\mathfrak{q}}(2)\cdot \bigg[(N(\mathfrak{q})+1)\mathbb{L}(G_{\mathfrak{q}},\rho,\mathfrak{q})-\frac{4\zeta_{\mathfrak{q}}(1)\delta_{\mathbf{k}}}{\zeta_{\mathfrak{q}}(2)}\cdot \mathbb{L}(G_{\mathfrak{q}}^{\mathrm{old}},\rho,\mathfrak{q})\bigg].
\end{equation}

For $\pi=\otimes_{v\leq\infty}\pi_v\in\mathcal{F}(\mathbf{k},\mathfrak{q})$, the local representation $\pi_{\mathfrak{q}}$ is of the form $\mathrm{St}\otimes \chi$, where $\mathrm{St}$ is the Steinberg representation, and $\chi$ is a unramified quadratic character. The local $L$-factor is 
\begin{align*}
L_{\mathfrak{q}}(1,\pi,\Ad)=\frac{L(1,\pi_{\mathfrak{q}}\times\widetilde{\pi}_{\mathfrak{q}})}{\zeta_{\mathfrak{q}}(1)}=\zeta_{\mathfrak{q}}(2).
\end{align*}
Therefore, the formula \eqref{7.68} amounts to  
\begin{equation}\label{eq7.70}
\sum_{\substack{\pi\in \mathcal{F}(\mathbf{k},\mathfrak{q})}}\frac{L(1/2,\pi)^2M_{\xi,\rho}(\pi)^2}{L(1,\pi,\Ad)}=2\mathcal{M}'D_F^{\frac{3}{2}}\prod_{v\mid\infty}\frac{k_v-1}{4\pi^2}+O(\xi^{2+\varepsilon}N(\mathfrak{q})^{\varepsilon}\|\mathbf{k}\|^{\frac{1}{2}+\varepsilon}).
\end{equation}

Making use of Proposition \ref{prop7.3}, we have
\begin{align*}
&\mathbb{L}(G_{\mathfrak{q}},\rho,\mathfrak{q})=\frac{\zeta_{\mathfrak{q}}(1)^3\cdot \zeta_F(2)^2}{\zeta_{\mathfrak{q}}(2)^{2}(\Res_{s=1}\zeta_F(s))^2}\cdot\bigg[\frac{G_{\mathfrak{q}}'(0)}{\log\xi}+G_{\mathfrak{q}}(0)\bigg]+O(G_{\mathfrak{q}}(0)(\log\xi)^{-1}),\\
&\mathbb{L}(G_{\mathfrak{q}}^{\mathrm{old}},\rho,\mathfrak{q})=\frac{\zeta_F(2)^2}{(\Res_{s=1}\zeta_F(s))^2}\bigg[\frac{1}{\log\xi}\frac{dG_{\mathfrak{q}}^{\mathrm{old}}(s)}{ds}\bigg|_{s=0}+G_{\mathfrak{q}}^{\mathrm{old}}(0)\bigg]+O\left(\frac{G_{\mathfrak{q}}^{\mathrm{old}}(0)}{\log\xi}\right),
\end{align*}
where $G_{\mathfrak{q}}^{\mathrm{old}}(s):=G_{\mathcal{O}_F}(s)(1+N(\mathfrak{q})^{-1-s/2})^{-2}$. Utilizing \eqref{7.66},
\begin{equation}\label{7.69}
\mathbb{L}(G_{\mathfrak{q}},\rho,\mathfrak{q})=\frac{2\zeta_{\mathfrak{q}}(1)^3\cdot \zeta_F(2)^2}{\zeta_{\mathfrak{q}}(2)^{2}(\Res_{s=1}\zeta_F(s))^2}\cdot\bigg[\frac{\log N(\mathfrak{q})^{\frac{1}{2}}\|\mathbf{k}\|}{\log\xi}+1\bigg]+O\left(\frac{1}{\log\xi}\right),
\end{equation}
and $\mathbb{L}(G_{\mathfrak{q}}^{\mathrm{old}},\rho,\mathfrak{q})$ boils down to  
\begin{equation}\label{7.70}
\frac{2\zeta_F(2)^2(\Res_{s=1}\zeta_F(s))^{-2}}{(1+N(\mathfrak{q})^{-1})^2}\bigg[\frac{N(\mathfrak{q})^{-1}\log N(\mathfrak{q})}{(1+N(\mathfrak{q})^{-1})\log\xi}+\frac{\log \|\mathbf{k}\|}{\log\xi}+1\bigg]+O\left(\frac{1}{\log\xi}\right).
\end{equation}

Combining \eqref{7.69} with \eqref{7.70} into \eqref{eq7.69}, we obtain 
\begin{align*}
\mathcal{M}'=&\frac{\zeta_{\mathfrak{q}}(1)^3}{\zeta_{\mathfrak{q}}(2)}\cdot\frac{2(N(\mathfrak{q})+1)\cdot  \zeta_F(2)^2}{(\Res_{s=1}\zeta_F(s))^2}\cdot\bigg[\frac{\log N(\mathfrak{q})^{1/2}\|\mathbf{k}\|}{\log\xi}+1\bigg]+O\left(\frac{N(\mathfrak{q})}{\log\xi}\right)\\
&- \frac{8\zeta_{\mathfrak{q}}(2)\cdot\delta_{\mathbf{k}}\zeta_F(2)^2}{(1+N(\mathfrak{q})^{-1})(\Res_{s=1}\zeta_F(s))^{2}}\bigg[\frac{\log N(\mathfrak{q})}{(1+N(\mathfrak{q}))\log\xi}+\frac{\log \|\mathbf{k}\|}{\log\xi}+1\bigg].
\end{align*}
\end{itemize}

Therefore, \eqref{eq7.65} follows from \eqref{7.64}, \eqref{equ7.68}, \eqref{eq7.70}, and the above asymptotic behavior of $\mathcal{M}'$. 
\end{proof}

\section{The First Moment via a Relative Trace Formula}\label{sec9}
Let $f_{\mathfrak{n},\mathfrak{q}}:=\otimes_{v\leq \infty}f_v\in L^1(\overline{G}(\mathbb{A}_F))$ be the test function defined as in \textsection\ref{sec1.2}. Let $y=(y_v)\in \mathbb{A}_F^{\times}$ with $y_v=1$ if $v\mid\infty$, and $y_v=\varpi_v^{-d_v}$, where $d_v$ is the valuation of the local different (cf. \textsection\ref{sec2.1.2}). 

Let $\Re(s)\gg 1$. Consider the function 
\begin{equation}\label{9.1}
I(f_{\mathfrak{n},\mathfrak{q}},s):=\int_{F^{\times}\backslash\mathbb{A}_F^{\times}}\int_{F\backslash\mathbb{A}_F}\K\left(\begin{pmatrix}
a\\
& 1
\end{pmatrix},\begin{pmatrix}
1& b\\
& 1
\end{pmatrix}\begin{pmatrix}
y\\
&1
\end{pmatrix}\right)\psi(b)|a|^{s}dbd^{\times}a.
\end{equation}

We introduce the translation by $\diag(y,1)$ in \eqref{9.1} to be consistent with with the normalization \eqref{eq3.10}. By \eqref{eq1.6} and the rapid decay of cusp forms, the function $I(f_{\mathfrak{n},\mathfrak{q}},s)$ converges absolutely for all $s\in \mathbb{C}$. 

Similar to the second moment case, we denote $I_{\mathrm{Spec}}(f_{\mathfrak{n},\mathfrak{q}},s)$ for $I(f_{\mathfrak{n},\mathfrak{q}},s)$ when we substitute the spectral decomposition  \eqref{eq1.6} for $\K(\cdot,\cdot)$ into \eqref{9.1} referring to it as the spectral side. Likewise, we denote  
$I_{\mathrm{Geom}}(f_{\mathfrak{n},\mathfrak{q}},s)$ for $I(f_{\mathfrak{n},\mathfrak{q}},s)$ when we substitute the geometric expansion  \eqref{f1.5} for $\K(\cdot,\cdot)$ into \eqref{9.1}, referring to it as the geometric side. By \eqref{f1.5} and \eqref{eq1.6} we obtain the relative trace formula 
\begin{align*}
I_{\mathrm{Spec}}(f_{\mathfrak{n},\mathfrak{q}},s)=I_{\mathrm{Geom}}(f_{\mathfrak{n},\mathfrak{q}},s).
\end{align*}

\subsection{The Spectral Side}\label{sec9.1}
Let $\Re(s)\gg 1$.  Substituting the spectral decomposition \eqref{eq1.6} into \eqref{9.1} yields
\begin{align*}
I_{\mathrm{Spec}}(f_{\mathfrak{n},\mathfrak{q}},s)=\sum_{\pi\in \Pi_{\mathbf{k}}(\mathfrak{q})}\sum_{\phi\in\mathfrak{B}_{\pi}}\mathcal{P}(s,\pi(f_{\mathfrak{n},\mathfrak{q}})\phi)\overline{W_{\phi}(\diag(y,1))},
\end{align*}
where $\mathcal{P}(s,\pi(f_{\mathfrak{n},\mathfrak{q}})\phi)$ is defined by \eqref{eq2.2}, and 
\begin{align*}
W_{\phi}(\diag(y,1)):=\int_{F\backslash\mathbb{A}_F}\phi\left(\begin{pmatrix}
1& b\\
& 1
\end{pmatrix}\begin{pmatrix}
y\\
&1
\end{pmatrix}\right)\overline{\psi}(b)db
\end{align*}
is the Whittaker function of $\phi$ relative to the additive character $\psi$.

Parallel to Theorems \ref{thm3.5} and \ref{spec}, we obtain the following.  
\begin{prop}\label{prop9.1}
Let notation be as before. Then $I_{\mathrm{Spec}}(f_{\mathfrak{n},\mathfrak{q}},s)$ admits a meromorphic continuation to $s\in \mathbb{C}$. Moreover, 
\begin{itemize}
\item If $\mathfrak{q}=\mathcal{O}_F$, then 
\begin{equation}\label{equ9.2}
I_{\mathrm{Spec}}(f_{\mathfrak{n},\mathfrak{q}},s)=\prod_{v\mid\infty}\frac{2^{k_v}(2\pi)^{\frac{k_v}{2}-s}\Gamma(k_v/2+s)}{\pi^{-1}e^{2\pi}\Gamma(k_v)}\sum_{\substack{\pi\in \mathcal{F}(\mathbf{k},\mathcal{O}_F)}}\frac{\lambda_{\pi}(\mathfrak{n})L(s+1/2,\pi)}{2L(1,\pi,\Ad)D_F^{3/2-s}}.
\end{equation}

\item If $\mathfrak{q}\subsetneq \mathcal{O}_F$, then
\begin{equation}\label{eq9.2}
I_{\mathrm{Spec}}(f_{\mathfrak{n},\mathfrak{q}},s)=I_{\mathrm{Spec}}^{\mathrm{new}}(f_{\mathfrak{n},\mathfrak{q}},s)+I_{\mathrm{Spec}}^{\mathrm{old}}(f_{\mathfrak{n},\mathfrak{q}},s),
\end{equation}
where
\begin{align*}
I_{\mathrm{Spec}}^{\mathrm{new}}(f_{\mathfrak{n},\mathfrak{q}},s):=&\prod_{v\mid\infty}\frac{2^{k_v}(2\pi)^{\frac{k_v}{2}-s}\Gamma(k_v/2+s)}{\pi^{-1}e^{2\pi}\Gamma(k_v)}\sum_{\substack{\pi\in \mathcal{F}(\mathbf{k},\mathfrak{q})}}\frac{\lambda_{\pi}(\mathfrak{n})L(s+1/2,\pi)}{2\zeta_{\mathfrak{q}}(2)^2L^{(\mathfrak{q})}(1,\pi,\Ad)D_F^{3/2-s}},\\
I_{\mathrm{Spec}}^{\mathrm{old}}(f_{\mathfrak{n},\mathfrak{q}},s):=&\prod_{v\mid\infty}\frac{2^{k_v}(2\pi)^{\frac{k_v}{2}-s} \Gamma(k_v/2+s)}{\pi^{-1}e^{2\pi}\Gamma(k_v)}\sum_{\substack{\pi\in \mathcal{F}(\mathbf{k},\mathcal{O}_F)}}\frac{\lambda_{\pi}(\mathfrak{n})C_{\pi_{\mathfrak{q}}}(s)L(s+1/2,\pi)}{2L(1,\pi,\Ad)D_F^{3/2-s}}.
\end{align*}
Here $C_{\pi_{\mathfrak{q}}}(s)$ is defined by 
\begin{align*}
1-\frac{\lambda_{\pi}(\mathfrak{q})((1+N(\mathfrak{q})^{-1})N(\mathfrak{q})^{-s}-\lambda_{\pi}(\mathfrak{q})N(\mathfrak{q})^{-1/2})L_{\mathfrak{q}}(1/2,\pi_{\mathfrak{q}})L_{\mathfrak{q}}(1/2,\pi_{\mathfrak{q}}\times\chi_{\mathfrak{q}})}{N(\mathfrak{q})^{1/2}}.
\end{align*}
\item When $s=0$, we have $C_{\pi_{\mathfrak{q}}}(0)=(1+N(\mathfrak{q})^{-1})L_{\mathfrak{q}}(1/2,\pi_{\mathfrak{q}}\times\chi_{\mathfrak{q}})$.
\end{itemize}
\end{prop}
\begin{proof}
When $\mathfrak{q}=\mathcal{O}_F$, following the arguments in \textsection\ref{sec3.1}, we derive \eqref{equ9.2}. Notice that the extra factor $e^{2\pi}$ follows from the explicit calculation of Whittaker functions in \eqref{2.4} and \eqref{2.5} (cf. the proof of Lemma \ref{lem2.1}), and we also make use of the fact that $W_{v}(\diag(y,1))=1$ for $v<\infty$, according to \eqref{eq3.10}. 

Now we suppose $\mathfrak{q}\subsetneq \mathcal{O}_F$. Parallel to \eqref{f2.15} we obtain \eqref{eq9.2}, where 
\begin{align*}
I_{\mathrm{Spec}}^{\mathrm{new}}(f_{\mathfrak{n},\mathfrak{q}},s):=&\sum_{\substack{\pi\in \mathcal{F}(\mathbf{k},\mathfrak{q}),\ \phi\in \mathcal{B}_{\pi}^{\mathrm{new}}}}\frac{\mathcal{P}(s,\pi(f_{\mathfrak{n},\mathfrak{q}})\phi)\overline{W_{\phi}(\diag(y,1))}}{\langle\phi,\phi\rangle},\\
I_{\mathrm{Spec}}^{\mathrm{old}}(f_{\mathfrak{n},\mathfrak{q}},s):=&\sum_{\substack{\pi\in \mathcal{F}(\mathbf{k},\mathcal{O}_F),\ \phi\in \mathfrak{B}_{\pi}^{K_{\mathfrak{q}}[1]}}}\frac{\mathcal{P}(s,\pi(f_{\mathfrak{n},\mathfrak{q}})\phi)\overline{W_{\phi}(\diag(y,1))}}{\langle\phi,\phi\rangle}.
\end{align*}

Let $\pi=\otimes_{v\leq\infty}\pi_v\in \mathcal{F}(\mathbf{k},\mathcal{O}_F)$. Let $\phi^{\circ}$ be the new form in $\pi$ such that $\langle \phi^{\circ},\phi^{\circ}\rangle=1$. Let $\alpha_{\pi_{\mathfrak{q}}}$ and $\beta_{\pi_{\mathfrak{q}}}$ be the coefficients in Lemma \ref{lem2.2}. Then   
\begin{align*}
\mathfrak{B}_{\pi}^{K_{\mathfrak{q}}[1]}=\mathrm{Span}\{\phi^{\circ}, \alpha_{\pi_{\mathfrak{q}}}\phi^{\circ}+\beta_{\pi_{\mathfrak{q}}}\pi_{\mathfrak{q}}(\diag(1,\varpi_{\mathfrak{q}}))\phi^{\circ}\}.\tag{\ref{2.17}}
\end{align*}

Let $\phi\in \mathfrak{B}_{\pi}^{K_{\mathfrak{q}}[1]}$. Parallel to \eqref{2.14} we have $\pi(f_{\mathfrak{n},\mathfrak{q}})\phi=\lambda_{\pi}(\mathfrak{n})\phi$. Hence, for $\Re(s_1)\gg 1$ and $\Re(s_2)\gg 1$, we deduce
\begin{equation}\label{9.2}
I_{\mathrm{Spec}}^{\mathrm{old}}(f_{\mathfrak{n},\mathfrak{q}},s)=\sum_{\substack{\pi\in \mathcal{F}(\mathbf{k},\mathcal{O}_F),\ \phi\in \mathfrak{B}_{\pi}^{K_{\mathfrak{q}}[1]}}}\frac{\lambda_{\pi}(\mathfrak{n})\cdot \mathcal{P}(s,\phi)\overline{W_{\phi}(\diag(y,1))}}{\langle\phi,\phi\rangle}.
\end{equation}

By a change of variable, $\mathcal{P}(s,\pi_{\mathfrak{q}}(\diag(1,\varpi_{\mathfrak{q}}))\phi^{\circ})=N(\mathfrak{q})^{-s}\mathcal{P}(s,\phi^{\circ})$. Moreover, $W_{\pi_{\mathfrak{q}}(\diag(1,\varpi_{\mathfrak{q}}))\phi^{\circ}}(I_2)=0$. Therefore, for $\phi=\alpha_{\pi_{\mathfrak{q}}}\phi^{\circ}+\beta_{\pi_{\mathfrak{q}}}\pi_{\mathfrak{q}}(\diag(1,\varpi_{\mathfrak{q}}))\phi^{\circ}$, 
\begin{align*}
\mathcal{P}(s,\phi)\overline{W_{\phi}(\diag(y,1))}= \alpha_{\pi_{\mathfrak{q}}}(\alpha_{\pi_{\mathfrak{q}}}+\beta_{\pi_{\mathfrak{q}}}N(\mathfrak{q})^{-s})\cdot \mathcal{P}(s,\phi^{\circ})\overline{W_{\phi^{\circ}}(I_2)}.
\end{align*}

Substituting this into \eqref{9.2}, we can thus rewrite $I_{\mathrm{Spec}}^{\mathrm{old}}(f_{\mathfrak{n},\mathfrak{q}},s)$ as  
\begin{equation}\label{9.3}
\sum_{\substack{\pi\in \mathcal{F}(\mathbf{k},\mathcal{O}_F),\ \phi^{\circ}\in \mathfrak{B}_{\pi}^{\mathrm{new}}}}\bigg[1+\alpha_{\pi_{\mathfrak{q}}}(\alpha_{\pi_{\mathfrak{q}}}+\beta_{\pi_{\mathfrak{q}}}N(\mathfrak{q})^{-s})\bigg]\cdot \frac{\lambda_{\pi}(\mathfrak{n})\cdot \mathcal{P}(s,\phi^{\circ})\overline{W_{\phi^{\circ}}(I_2)}}{\langle\phi^{\circ},\phi^{\circ}\rangle}.
\end{equation}

By \eqref{eq2.11} in Lemma \ref{lem2.2}, $1+\alpha_{\pi_{\mathfrak{q}}}(\alpha_{\pi_{\mathfrak{q}}}+\beta_{\pi_{\mathfrak{q}}}N(\mathfrak{q})^{-s})$ boils down to 
\begin{align*}
&1+(\gamma^2-\gamma  N(\mathfrak{q})^{-s})\zeta_{\mathfrak{q}}(1)^2\zeta_{\mathfrak{q}}(2)^{-2}L_{\mathfrak{q}}(1/2,\pi_{\mathfrak{q}})L_{\mathfrak{q}}(1/2,\pi_{\mathfrak{q}}\times\chi_{\mathfrak{q}}),
\end{align*}
which, after a straightforward calculation, is equal to $C_{\pi_{\mathfrak{q}}}(s)$. 

Substituting $s=0$ into the definition of $C_{\pi_{\mathfrak{q}}}(s)$ yields 
\begin{align*}
C_{\pi_{\mathfrak{q}}}(0)=1-\lambda_{\pi}(\mathfrak{q})N(\mathfrak{q})^{-1/2}L_{\mathfrak{q}}(1/2,\pi_{\mathfrak{q}}\times\chi_{\mathfrak{q}})=(1+N(\mathfrak{q})^{-1})L_{\mathfrak{q}}(1/2,\pi_{\mathfrak{q}}\times\chi_{\mathfrak{q}}).
\end{align*}

Consequently, Proposition \ref{prop9.1} follows. 
\end{proof}

\subsection{The Geometric Side}\label{sec9.2}
 Substituting Bruhat decomposition 
\begin{align*}
\mathrm{GL}_2(F)=B(F)\bigsqcup T(F)\begin{pmatrix}
& -1\\
1
\end{pmatrix}N(F)\bigsqcup \bigsqcup _{t\in F^{\times}}T(F)\begin{pmatrix}
-t&-1\\
1&
\end{pmatrix}
N(F)
\end{align*}
into the definition of $I(f_{\mathfrak{n},\mathfrak{q}},s)$, we obtain, parallel to \eqref{fc2.14}, the decomposition 
\begin{align*}
I_{\mathrm{Geom}}(f_{\mathfrak{n},\mathfrak{q}},s)=\sum_{\delta\in \{I_2,w\}}I_{\mathrm{small}}^{\delta}(f_{\mathfrak{n},\mathfrak{q}},s)+I_{\mathrm{reg}}(f_{\mathfrak{n},\mathfrak{q}},s),\ \ \Re(s)\gg 1,
\end{align*}
where 
\begin{align*}
&I_{\mathrm{small}}^{I_2}(f_{\mathfrak{n},\mathfrak{q}},s):=\int_{\mathbb{A}_F^{\times}}\int_{\mathbb{A}_F}f_{\mathfrak{n},\mathfrak{q}}\left(\begin{pmatrix}
a\\
& 1
\end{pmatrix}^{-1}\begin{pmatrix}
y& b\\
& 1
\end{pmatrix}\right)\psi(b)|a|^{s}dbd^{\times}a,\\
&I_{\mathrm{small}}^{w}(f_{\mathfrak{n},\mathfrak{q}},s):=\int_{\mathbb{A}_F^{\times}}\int_{\mathbb{A}_F}f_{\mathfrak{n},\mathfrak{q}}\left(\begin{pmatrix}
a\\
& 1
\end{pmatrix}^{-1}\begin{pmatrix}
& -1\\
1
\end{pmatrix} \begin{pmatrix}
y& b\\
& 1
\end{pmatrix}\right)\psi(b)|a|^{s}dbd^{\times}a,\\
&I_{\mathrm{reg}}(f_{\mathfrak{n},\mathfrak{q}},s):=\sum_{t\in F^{\times}}\int_{\mathbb{A}_F^{\times}}\int_{\mathbb{A}_F}\left(\begin{pmatrix}
a\\
& 1
\end{pmatrix}^{-1}\begin{pmatrix}
-t& -1\\
1&
\end{pmatrix}\begin{pmatrix}
y& b\\
& 1
\end{pmatrix}\right)\psi(b)|a|^{s}dbd^{\times}a.
\end{align*}

This is a generalization of \cite{KL10} from $F=\mathbb{Q}$ to totally real fields. 

\subsection{The Orbital Integral $I_{\mathrm{small}}^{I_2}(f_{\mathfrak{n},\mathfrak{q}},s)$}
By definition we have
\begin{align*}
I_{\mathrm{small}}^{I_2}(f_{\mathfrak{n},\mathfrak{q}},s)=\prod_vI_{\mathrm{small},v}^{I_2}(s),
\end{align*}
where 
\begin{equation}\label{eq9.6}
I_{\mathrm{small},v}^{I_2}(s)=\int_{F_v^{\times}}\int_{F_v}f_v\left(\begin{pmatrix}
a_v^{-1}y_v& a_v^{-1}b_v\\
& 1
\end{pmatrix}\right)\psi_v(b_v)|a_v|_v^{s}db_vd^{\times}a_v.
\end{equation}

\subsubsection{Archimedean integrals}
Let $v\mid\infty$. By definition, and a change of variable $b_v\mapsto -b_v$, we obtain 
\begin{align*}
I_{\mathrm{small},v}^{I_2}(s)=\frac{(k_v-1)\cdot (2i)^{k_v}}{4\pi}\cdot \int_{0}^{\infty}\int_{F_v}\frac{a_v^{k_v/2}e^{-2\pi i b_v}}{(b_v+i(a_v+1))^{k_v}}|a_v|_v^{s}db_vd^{\times}a_v.
\end{align*}

Utilizing Lemma \ref{lem0.4} we then derive 
\begin{equation}\label{9.6}
I_{\mathrm{small},v}^{I_2}(s)=\frac{(k_v-1)\cdot 2^{k_v}(2\pi)^{k_v/2-s}}{4\pi e^{2\pi}\Gamma(k_v)}\cdot \Gamma(k_v/2+s).
\end{equation}

\subsubsection{Non-Archimedean integrals}
Let $v<\infty$. Let $r=\max\{e_v(\mathfrak{n}),0\}$. By definition of $f_{\mathfrak{n},\mathfrak{q}}$ in \eqref{sec1.2}, 
\begin{align*}
f_v\left(\begin{pmatrix}
a_v^{-1}y_v& a_v^{-1}b_v\\
& 1
\end{pmatrix}\right)\equiv 0
\end{align*} 
unless there exists some $l\in \mathbb{Z}$ such that 
\begin{equation}\label{9.7}
\varpi_v^l\begin{pmatrix}
a_v^{-1}\varpi_v^{-d_v}& a_v^{-1}b_v\\
& 1
\end{pmatrix}\in \bigsqcup_{\substack{i+j=r\\
i\geq j\geq 0}}K_v[e_v(\mathfrak{q})]\begin{pmatrix}
\varpi_v^i\\
&\varpi_v^j
\end{pmatrix}K_v[e_v(\mathfrak{q})].
\end{equation}

We consider the following scenarios according to the range of $r$.
\begin{itemize}
\item Suppose $r=0$. Then \eqref{9.7} boils down to 
\begin{align*}
\begin{cases}
2l-e_v(a_v)-d_v=0\\
l-e_v(a_v)-d_v\geq 0,\ \ l\geq 0\\
l-e_v(a_v)+e_v(b_v)\geq 0
\end{cases}\ \ \Leftrightarrow\ \ \begin{cases}
e_v(a_v)=-d_v,\ \ l=0\\
e_v(b_v)\geq -d_v.
\end{cases}
\end{align*} 

Substituting this into the definition of $I_{\mathrm{small},v}^{I_2}(s)$ yields 
\begin{equation}\label{9.8}
I_{\mathrm{small},v}^{I_2}(s)=V_{\mathfrak{q}}\cdot q_v^{d_v(1+s)}\cdot \Vol(\mathcal{O}_v)\Vol(\mathcal{O}_v^{\times})=V_{\mathfrak{q}}\cdot q_v^{d_vs}.
\end{equation}

\item Suppose $r\geq 1$, namely, $v\mid\mathfrak{n}$. Then $r=e_v(\mathfrak{n})$, $e_v(\mathfrak{q})=d_v=0$. Therefore, the constraint  \eqref{9.7} amounts to 
\begin{align*}
\begin{cases}
2l-e_v(a_v)=r\\
l-e_v(a_v)\geq 0,\ \ l\geq 0,\ \ l-e_v(a_v)+e_v(b_v)\geq 0.
\end{cases}
\end{align*}

Plugging these constraints into the definition \eqref{eq9.6}, we obtain 
\begin{align*}
I_{\mathrm{small},v}^{I_2}(s)=q_v^{-\frac{e_v(\mathfrak{n})}{2}}\sum_{l=0}^r\int_{\varpi_v^{2l-r}\mathcal{O}_v^{\times}}\int_{F_v}\psi_v(b_v)\textbf{1}_{l-e_v(a_v)+e_v(b_v)\geq 0}db_v|a_v|_v^{s}d^{\times}a_v.
\end{align*}

Notice that 
\begin{align*}
\int_{F_v}\psi_v(b_v)\textbf{1}_{l-e_v(a_v)+e_v(b_v)\geq 0}db_v=\textbf{1}_{e_v(a_v)-l\geq 0}.
\end{align*}
Plugging this into the expression for $I_{\mathrm{small},v}^{I_2}(s)$ yields 
\begin{equation}\label{9.9}
I_{\mathrm{small},v}^{I_2}(s)=q_v^{-\frac{e_v(\mathfrak{n})}{2}}\int_{\varpi_v^{r}\mathcal{O}_v^{\times}}|a_v|_v^{s}d^{\times}a_v=q_v^{-(1/2+s)e_v(\mathfrak{n})}.
\end{equation}
\end{itemize}

\subsubsection{The global integral}
By combining \eqref{9.6}, \eqref{9.8} and \eqref{9.9} we arrive at the following result. 
\begin{prop}\label{prop9.2}
Let notation be as before. 
\begin{itemize}
\item $I_{\mathrm{small}}^{I_2}(f_{\mathfrak{n},\mathfrak{q}},s)$ converges absolutely in $\Re(s)>-2^{-1}\max_{v\mid\infty}\{k_v\}.$
\item The function $I_{\mathrm{small}}^{I_2}(f_{\mathfrak{n},\mathfrak{q}},s)$ admits a meromorphic continuation to $s\in \mathbb{C}$ explicitly given by 
\begin{align*}
I_{\mathrm{small}}^{I_2}(f_{\mathfrak{n},\mathfrak{q}},s)=\frac{V_{\mathfrak{q}}\cdot D_F^s}{N(\mathfrak{n})^{1/2+s}}\prod_{v\mid\infty}\frac{(k_v-1)\cdot 2^{k_v}(2\pi)^{k_v/2-s}\Gamma(k_v/2+s)}{4\pi e^{2\pi}\Gamma(k_v)}.
\end{align*}	
\end{itemize}
\end{prop}

\subsection{The Orbital Integral $I_{\mathrm{small}}^{w}(f_{\mathfrak{n},\mathfrak{q}},s)$}
By the definition in \textsection\ref{sec9.2}, we can write 
\begin{align*}
I_{\mathrm{small}}^{w}(f_{\mathfrak{n},\mathfrak{q}},s)=\prod_{v}I_{\mathrm{small},v}^{w}(s),
\end{align*}
where 
\begin{equation}\label{9.11}
I_{\mathrm{small},v}^{w}(s):=\int_{F_v^{\times}}\int_{F_v}f_v\left(\begin{pmatrix}
& -a_v^{-1}\\
y_v& b_v
\end{pmatrix}\right)\psi_v(b_v)|a_v|_v^{s}db_vd^{\times}a_v.
\end{equation}

\subsubsection{Archimedean integrals}
Let $v\mid\infty$. By the definition \eqref{9.11}, and a change of variable $b_v\mapsto -b_v$, we obtain 
\begin{align*}
I_{\mathrm{small},v}^{w}(s)=\frac{(k_v-1)\cdot 2^{k_v}}{4\pi}\cdot \int_{0}^{\infty}\int_{F_v}\frac{a_v^{-k_v/2}e^{-2\pi i b_v}}{(b_v+i(a_v^{-1}+1))^{k_v}}|a_v|_v^{s}db_vd^{\times}a_v.
\end{align*}

Utilizing Lemma \ref{lem0.4} we then derive 
\begin{equation}\label{9.10}
I_{\mathrm{small},v}^{I_2}(s)=\frac{i^{k_v}\cdot (k_v-1)\cdot 2^{k_v}(2\pi)^{k_v/2+s}}{4\pi e^{2\pi}\Gamma(k_v)}\cdot \Gamma(k_v/2-s).
\end{equation}

\subsubsection{Non-Archimedean integrals}
Let $v<\infty$. Let $r=\max\{e_v(\mathfrak{n}),0\}$. By definition of $f_{\mathfrak{n},\mathfrak{q}}$ in \eqref{sec1.2}, 
\begin{align*}
f_v\left(\begin{pmatrix}
& -a_v^{-1}\\
y_v& b_v
\end{pmatrix}\right)\equiv 0
\end{align*} 
unless there exists some $l\in \mathbb{Z}$ such that 
\begin{equation}\label{9.13}
\varpi_v^l\begin{pmatrix}
& -a_v^{-1}\\
\varpi_v^{-d_v}& b_v
\end{pmatrix}\in \bigsqcup_{\substack{i+j=r\\
i\geq j\geq 0}}K_v[e_v(\mathfrak{q})]\begin{pmatrix}
\varpi_v^i\\
&\varpi_v^j
\end{pmatrix}K_v[e_v(\mathfrak{q})].
\end{equation}

We consider the following scenarios according to the range of $r$.
\begin{itemize}
\item Suppose $r=0$. Then \eqref{9.13} boils down to
\begin{equation}\label{eq9.14}
\begin{cases}
2l-e_v(a_v)-d_v=0\\
l-e_v(a_v)\geq 0,\ \ l-d_v\geq e_v(\mathfrak{q})\\
l+e_v(b_v)\geq 0.
\end{cases}
\end{equation}

Notice that the constraints \eqref{eq9.14} is empty if $e_v(\mathfrak{q})\geq 1$ for some $v<\infty$, namely, $\mathfrak{q}\subsetneq \mathcal{O}_F$. If $\mathfrak{q}= \mathcal{O}_F$, then \eqref{eq9.14} simplifies to  
\begin{align*}
\begin{cases}
e_v(a_v)=d_v=l\\
e_v(b_v)\geq -d_v.
\end{cases}
\end{align*}

Substituting this into the definition of $I_{\mathrm{small},v}^{I_2}(s)$ yields 
\begin{equation}\label{9.14}
I_{\mathrm{small},v}^{I_2}(s)=\textbf{1}_{\mathfrak{q}=\mathcal{O}_F}\cdot q_v^{d_v(1-s)}\cdot \Vol(\mathcal{O}_v)\Vol(\mathcal{O}_v^{\times})=\textbf{1}_{\mathfrak{q}=\mathcal{O}_F}\cdot  q_v^{-d_vs}.
\end{equation}

\item Suppose $r\geq 1$, namely, $v\mid\mathfrak{n}$. Then $r=e_v(\mathfrak{n})$ and $e_v(\mathfrak{q})=0$ and thus \eqref{9.13} amounts to 
\begin{align*}
\begin{cases}
2l-e_v(a_v)=r\\
l-e_v(a_v)\geq 0,\ \ l\geq 0,\ \ l+e_v(b_v)\geq 0.
\end{cases}
\end{align*}

Plugging these constraints into the definition \eqref{9.11}, we obtain 
\begin{align*}
I_{\mathrm{small},v}^{w}(s)=q_v^{-\frac{e_v(\mathfrak{n})}{2}}\sum_{l=0}^r\int_{\varpi_v^{2l-r}\mathcal{O}_v^{\times}}\int_{F_v}\psi_v(b_v)\textbf{1}_{l+e_v(b_v)\geq 0}db_v|a_v|_v^{s}d^{\times}a_v.
\end{align*}

Notice that 
\begin{align*}
\int_{F_v}\psi_v(b_v)\textbf{1}_{l+e_v(b_v)\geq 0}db_v=\textbf{1}_{l\leq  0}.
\end{align*}
Plugging this into the expression for $I_{\mathrm{small},v}^{I_2}(s)$ yields 
\begin{equation}\label{9.15}
I_{\mathrm{small},v}^{w}(s)=q_v^{-\frac{e_v(\mathfrak{n})}{2}}\int_{\varpi_v^{-e_v(\mathfrak{n})}\mathcal{O}_v^{\times}}|a_v|_v^{s}d^{\times}a_v=q_v^{-(1/2-s)e_v(\mathfrak{n})}.
\end{equation}
\end{itemize}

\subsubsection{The global integral}
By combining \eqref{9.10}, \eqref{9.14} and \eqref{9.15} we arrive at the following result. 
\begin{prop}\label{prop9.3}
Let notation be as before. 
\begin{itemize}
\item $I_{\mathrm{small}}^{w}(f_{\mathfrak{n},\mathfrak{q}},s)$ converges absolutely in $\Re(s)<2^{-1}\min_{v\mid\infty}\{k_v\}.$
\item The function $I_{\mathrm{small}}^{w}(f_{\mathfrak{n},\mathfrak{q}},s)$ admits a meromorphic continuation to $s\in \mathbb{C}$ explicitly given by 
\begin{align*}
I_{\mathrm{small}}^{w}(f_{\mathfrak{n},\mathfrak{q}},s)=\frac{\textbf{1}_{\mathfrak{q}=\mathcal{O}_F}\cdot  q_v^{-d_vs}}{N(\mathfrak{n})^{1/2-s}}\prod_{v\mid\infty}\frac{i^{k_v}\cdot (k_v-1)\cdot 2^{k_v}(2\pi)^{k_v/2+s}\Gamma(k_v/2-s)}{4\pi e^{2\pi}\Gamma(k_v)}.
\end{align*}
\end{itemize}
\end{prop}

\begin{remark}
Notice that, when $\mathfrak{q}=\mathcal{O}_F$, we have, by Propositions \ref{prop9.2} and \ref{prop9.3}, 
\begin{align*}
I_{\mathrm{small}}^{I_2}(f_{\mathfrak{n},\mathcal{O}_F},s)=i^{\sum_{v\mid\infty}k_v}\cdot I_{\mathrm{small}}^{w}(f_{\mathfrak{n},\mathcal{O}_F},-s).
\end{align*}	
Hence, $I_{\mathrm{small}}^{I_2}(f_{\mathfrak{n},\mathcal{O}_F},0)+I_{\mathrm{small}}^{w}(f_{\mathfrak{n},\mathcal{O}_F},0)=0$ unless $\sum_{v\mid\infty}k_v\equiv 0\pmod{4}$.
\end{remark}

\subsection{The Regular Orbital Integral}\label{sec9.5}
Recall the regular orbital integral (cf. \textsection\ref{sec9.2}):
\begin{align*}
I_{\mathrm{reg}}(f_{\mathfrak{n},\mathfrak{q}},s):=\sum_{t\in F^{\times}}\prod_{v}I_v(t,s),
\end{align*} 
where 
\begin{equation}\label{eq9.16}
I_v(t,s):=\int_{F_v^{\times}}\int_{F_v}f_v\left(\begin{pmatrix}
-a_v^{-1}y_vt& -a_v^{-1}(b_vt+1)\\
y_v& b_v
\end{pmatrix}\right)\psi_v(b_v)|a_v|_v^{s}db_vd^{\times}a_v.
\end{equation}
For simplicity, we denote by $I_v(t)=I_v(t,0)$. We will study $I_v(t)$ at each local place $v$ in the following subsections.

\subsubsection{Calculation of $I_v(t)$ at $v<\infty$ and $v\nmid\mathfrak{n}\mathfrak{q}$}
Let $v\nmid \mathfrak{n}\mathfrak{q}$ be a non-Archimedean place. By definition, 
\begin{align*}
f_v\left(\begin{pmatrix}
-a_v^{-1}y_vt& -a_v^{-1}(b_vt+1)\\
y_v& b_v
\end{pmatrix}\right)=0
\end{align*} 
unless there exists some $l\in \mathbb{Z}$ such that 
\begin{align*}
\varpi_v^l\begin{pmatrix}
-a_v^{-1}\varpi_v^{-d_v}t& -a_v^{-1}(b_vt+1)\\
\varpi_v^{-d_v}& b_v
\end{pmatrix}\in K_v,
\end{align*}
which is equivalent to 
\begin{equation}\label{9.16}
\begin{cases}
2l-e_v(a_v)=d_v\\
l-e_v(a_v)+e_v(t)\geq d_v\\
l\geq d_v\\
l+e_v(b_v)\geq 0\\
l-e_v(a_v)+e_v(b_vt+1)\geq 0
\end{cases}\ \ \Leftrightarrow\ \ \begin{cases}
2l-e_v(a_v)=d_v\\
e_v(t)\geq l \geq d_v\\
l+e_v(b_v)\geq 0\\
e_v(b_vt+1)\geq l-d_v.
\end{cases}
\end{equation}

Substituting this into \eqref{eq9.16} leads to the simplification 
\begin{equation}\label{9.18}
I_v(t,s)=\Vol(\mathcal{O}_v^{\times})\cdot \sum_{l=d_v}^{e_v(t)}q_v^{-(2l-d_v)s}\int_{\varpi_v^{-l}\mathcal{O}_v}\psi_v(b_v)\textbf{1}_{e_v(b_vt+1)\geq l-d_v}db_v.	
\end{equation}

Notice that for $l-d_v\geq 1$, we have $e_v(b_vt+1)\geq l-d_v$ if and only if $b_vt+1=\beta_v\varpi_v^{l-d_v}$ for some $\beta_v\in \mathcal{O}_v$, i.e., $b_v=-t^{-1}+\beta_vt^{-1}\varpi_v^{l-d_v}$, which implies that $e_v(b_v)+e_v(t)=0$. Along with the constraints $e_v(t)\geq l $ and $
l+e_v(b_v)\geq 0$ in \eqref{9.16}, we derive $e_v(b_v)=-l$ and $e_v(t)=l\geq d_v+1$. Hence, it follows from \eqref{9.18} that 
\begin{align*}
\frac{I_v(t,s)}{\Vol(\mathcal{O}_v^{\times})}=&q_v^{-d_vs-d_v}\textbf{1}_{e_v(t)\geq d_v}+\textbf{1}_{\substack{l=e_v(t)\\
l\geq d_v+1}}q_v^{-(2l-d_v)s}\int_{\varpi_v^{-l}\mathcal{O}_v^{\times}}\psi_v(b_v)\textbf{1}_{e_v(b_vt+1)\geq l-d_v}db_v.
\end{align*}

After a straightforward calculation, we derive 
\begin{equation}\label{eq9.20}
I_v(t,s)=q_v^{-2d_v}\cdot \bigg[q_v^{-d_vs}\textbf{1}_{e_v(t)\geq d_v}+\psi_v(-t^{-1})q_v^{d_v-(2e_v(t)-d_v)s}\cdot \textbf{1}_{e_v(t)\geq d_v+1}\bigg].
\end{equation}

In particular, when $e_v(t)=0$, then $I_v(t,s)=q_v^{-2d_v}\cdot q_v^{-d_vs}$, which is equal to $1$ if $d_v=0$, i.e., $v$ is not a ramified place. Hence, we conclude that $I_v(t,s)\equiv 1$ for all but finitely many places $v$'s. 

\subsubsection{Calculation of $I_v(t)$ at $v\mid \mathfrak{q}$}
Let $v\mid \mathfrak{q}$. Then $d_v=0$, implying that $y_v=1$. By definition of $f_v$ in \eqref{t1.7}, we have  
\begin{align*}
f_v\left(\begin{pmatrix}
-a_v^{-1}y_vt& -a_v^{-1}(b_vt+1)\\
y_v& b_v
\end{pmatrix}\right)=0
\end{align*} 
unless there exists some $l\in \mathbb{Z}$ such that 
\begin{align*}
\varpi_v^l\begin{pmatrix}
-a_v^{-1}t& -a_v^{-1}(b_vt+1)\\
1& b_v
\end{pmatrix}\in K_v[1],
\end{align*}
which is equivalent to 
\begin{align*}
\begin{cases}
2l-e_v(a_v)=0,\ \ l\geq 1\\ 
l-e_v(a_v)+e_v(t)=0\\
l+e_v(b_v)=0\\
l-e_v(a_v)+e_v(b_vt+1)\geq 0
\end{cases}\ \ \Leftrightarrow\ \ \begin{cases}
e_v(a_v)=2l\\
e_v(t)=l\geq 1\\
e_v(b_vt+1)\geq l\geq 1.
\end{cases}
\end{align*}

From $e_v(b_vt+1)\geq l\geq 1$ we may write $b_v=-t^{-1}+\beta_vt^{-1}\varpi_v^{l}$, where $\beta_v\in \mathcal{O}_v$. Substituting this parametrization into \eqref{eq9.16} leads to 
\begin{equation}\label{9.21}
I_v(t,s)=\frac{V_{\mathfrak{q}}\textbf{1}_{e_v(t)\geq 1}}{q_v^{2e_v(t)s}}\int_{\mathcal{O}_v}\overline{\psi}_v(t^{-1}-\beta_vt^{-1}\varpi_v^{l})d\beta_v=\overline{\psi}_v(t^{-1})|t|_v^{2s}V_{\mathfrak{q}} \textbf{1}_{e_v(t)\geq 1}.
\end{equation}

\subsubsection{Calculation of $I_v(t)$ at $v\mid \mathfrak{n}$}
Let $v\mid \mathfrak{n}$. Then $d_v=0$, implying that $y_v=1$. Write $r=e_v(\mathfrak{n})$. By definition of $f_v$ in \eqref{t1.7}, we have  
\begin{align*}
f_v\left(\begin{pmatrix}
-a_v^{-1}y_vt& -a_v^{-1}(b_vt+1)\\
y_v& b_v
\end{pmatrix}\right)=0
\end{align*} 
unless there exists some $l\in \mathbb{Z}$ such that 
\begin{equation}\label{eq9.22}
\varpi_v^l\begin{pmatrix}
-a_v^{-1}t& -a_v^{-1}(b_vt+1)\\
1& b_v
\end{pmatrix}\in \bigsqcup_{\substack{i+j=r\\
i\geq j\geq 0}}K_v\begin{pmatrix}
\varpi_v^i\\
&\varpi_v^j
\end{pmatrix}K_v.
\end{equation}

Notice that \eqref{eq9.22} is equivalent to 
\begin{align*}
\varpi_v^l\begin{pmatrix}
t& b_vt+1\\
a_v& a_vb_v
\end{pmatrix}\in \bigsqcup_{\substack{i+j=r\\
i\geq j\geq 0}}K_v\begin{pmatrix}
\varpi_v^i\\
&\varpi_v^j
\end{pmatrix}K_v,
\end{align*}
which is equivalent to 
\begin{equation}\label{9.22}
\begin{cases}
2l+e_v(a_v)=r\geq 1\\
l+e_v(t)\geq 0\\
l+e_v(a_v)\geq 0\\ 
l+e_v(a_v)+e_v(b_v)\geq 0\\
l+e_v(b_vt+1)\geq 0
\end{cases}\ \ \Leftrightarrow\ \ \begin{cases}
e_v(a_v)=r-2l\\
-e_v(t)\leq l\leq r\\
e_v(b_v)\geq l-r\\
e_v(b_vt+1)\geq -l.
\end{cases}
\end{equation}

Substituting \eqref{9.22} into the definition \eqref{eq9.16} leads to 
\begin{equation}\label{eq9.23}
I_v(t,s)=q_v^{-r/2}\cdot \textbf{1}_{e_v(t)\geq -r}\sum_{l=-e_v(t)}^{r}q_v^{(2l-r)s}\int_{\mathfrak{p}_v^{l-r}}\psi_v(b_v)\textbf{1}_{e_v(b_vt+1)\geq -l}db_v.
\end{equation}

\begin{itemize}
\item Suppose $l\leq -1$. We can write $b_v=-t^{-1}+\beta_vt^{-1}\varpi_v^{-l}$, where $\beta_v\in \mathcal{O}_v$. In this case we have $e_v(b_v)=e_v(-t^{-1}+\beta_vt^{-1}\varpi_v^{-l})=-e_v(t)\geq l-r$, i.e., $e_v(t)\leq r-l$. In this case, we have
\begin{align*}
\textbf{1}_{l\leq -1}\int_{\mathfrak{p}_v^{l-r}}\psi_v(b_v)\textbf{1}_{e_v(b_vt+1)\geq -l}db_v=\psi_v(-t^{-1})q_v^{e_v(t)+l}\textbf{1}_{l\leq -1}\cdot \textbf{1}_{\substack{e_v(t)\leq r-l\\ -l-e_v(t)\geq 0}}.
\end{align*}
In conjunction with the constraint $-e_v(t)\leq l$ in \eqref{9.22} we obtain
\begin{equation}\label{9.24}
\textbf{1}_{l\leq -1}\int_{\mathfrak{p}_v^{l-r}}\psi_v(b_v)\textbf{1}_{e_v(b_vt+1)\geq -l}db_v=\psi_v(-t^{-1})\textbf{1}_{l\leq -1}\cdot\textbf{1}_{\substack{l=-e_v(t)}}.
\end{equation}

\item Suppose $l\geq 0$. Then $e_v(b_vt)\geq -l$. Along with $e_v(b_v)\geq l-r$ in \eqref{9.22} we derive that $e_v(b_v)\geq \max\{l-r,-l-e_v(t)\}$. As a consequence, we have, when $l\geq 0$, that
\begin{equation}\label{9.26}
\int_{\mathfrak{p}_v^{l-r}}\psi_v(b_v)\textbf{1}_{e_v(b_vt+1)\geq -l}db_v=\textbf{1}_{\max\{l-r,-l-e_v(t)\}\geq 0}\cdot q_v^{-\max\{l-r,-l-e_v(t)\}}.
\end{equation}

Combining the condition $\max\{l-r,-l-e_v(t)\}\geq 0$ with \eqref{9.22} yields 
\begin{equation}\label{9.27}
\begin{cases}
e_v(a_v)=r-2l\\
-e_v(t)\leq l\leq r\\
l\geq 0\\
\max\{l-r,-l-e_v(t)\}\geq 0
\end{cases}\ \ \Leftrightarrow\ \ \begin{cases}
e_v(a_v)=r-2l\\
r\geq -e_v(t)\\
l=r\ \ \text{or}\ \ 0\leq l=-e_v(t).
\end{cases}
\end{equation}

\end{itemize}

Substituting \eqref{9.24}, \eqref{9.26} and \eqref{9.27} into  \eqref{eq9.23} we obtain 
\begin{align*}
I_v(t,s)=&q_v^{-r/2}q_v^{-(2e_v(t)+r)s}\psi_v(-t^{-1})\cdot \textbf{1}_{e_v(t)\geq 1}\\
&+q_v^{-r/2}\cdot \textbf{1}_{e_v(t)\geq -r}\sum_{l=0}^{r}q_v^{(2l-r)s}\cdot \textbf{1}_{l=r\ \text{or}\ l=-e_v(t)},
\end{align*}
which simplifies to 
\begin{equation}\label{9.28}
I_v(t,s)=q_v^{-r/2}q_v^{-(2e_v(t)+r)s}\psi_v(-t^{-1})\cdot \textbf{1}_{e_v(t)>-r}+q_v^{-r/2}q_v^{rs}\cdot \textbf{1}_{e_v(t)\geq -r}.
\end{equation}

\subsubsection{Calculation of $I_v(t)$ at $v\mid \infty$}
Let $v\mid\infty$. Recall the definition 
\begin{align*}
I_v(t):=\int_{F_v^{\times}}\int_{F_v}f_v\left(\begin{pmatrix}
-a_v^{-1}y_vt& -a_v^{-1}(b_vt+1)\\
y_v& b_v
\end{pmatrix}\right)\psi_v(b_v)|a_v|_v^{s}db_vd^{\times}a_v.
\end{align*}

Substituting the definition of $f_v$ (cf. \eqref{t1.5} in \textsection\ref{sec1.2}) into the above definition yields 
\begin{align*}
I_v(t)=\frac{(k_v-1)\cdot (2i)^{k_v}}{4\pi}\cdot \int_{F_v^{\times}}\int_{F_v}\frac{a_v^{s-k_v/2}\psi_v(b_v)\textbf{1}_{a_v>0}}{(1+a_v^{-1}(b_vt+1)+i(-a_v^{-1}t+b_v))^{k_v}}
db_vd^{\times}a_v.
\end{align*}

By a straightforward simplification, we derive 
\begin{align*}
I_v(t)=&\frac{(k_v-1)\cdot 2^{k_v}}{4\pi}\cdot \int_{0}^{\infty}\frac{a_v^{s-k_v/2}}{(1-ia_v^{-1}t)^{k_v}}\int_{-\infty}^{\infty}\frac{e^{-2\pi i b_v}}{(b_v+\frac{ia_v^{-1}}{1-ia_v^{-1}t}+i)^{k_v}}
db_vd^{\times}a_v.
\end{align*}

Shifting the contour and utilizing Cauchy's formula, we obtain 
\begin{align*}
I_v(t)=\frac{(k_v-1)}{4\pi e^{2\pi}}\cdot \frac{(4\pi i)^{k_v}}{\Gamma(k_v)}
\cdot \int_{0}^{\infty}\frac{a_v^{s+\frac{k_v}{2}-1}}{(a_v-it)^{k_v}}e^{-\frac{2\pi }{a_v-it}}da_v.
\end{align*}

Notice that 
\begin{equation}\label{eq9.29}
\overline{I_v(t)}=I_v(-t),\ \ \text{if $t<0$ in $F_v$}.
\end{equation}

As a consequence, we may assume $t>0$ in $F_v$. Moreover, by Lemma \ref{lem0.4},
\begin{align*}
\int_{0}^{\infty}\frac{a_v^{s+\frac{k_v}{2}-1}}{(a_v-i)^{k_v}}e^{-\frac{2\pi t^{-1}}{a_v-i}}da_v=\frac{i^{k_v}e^{-2\pi i t^{-1}}}{e^{i\pi (s+k_v/2)/2}}\int_0^1e^{2\pi i t^{-1}a}a^{s+\frac{k_v}{2}-1}(1-a)^{\frac{k_v}{2}-s-1}da.
\end{align*}

As a result, when $t_v>0$, the function $I_v(t,s)$ is equal to 
\begin{equation}\label{9.30}
\frac{(k_v-1)}{4\pi e^{2\pi}}\cdot \frac{(4\pi )^{k_v}|t|_v^{s-\frac{k_v}{2}}\Gamma(k_v/2)^2}{\Gamma(k_v)^2}
\cdot \frac{e^{-2\pi i t^{-1}}}{e^{i\pi (s+k_v/2)/2}}\cdot  {}_{1}F_{1}(k_v/2+s; k_v; 2\pi i t^{-1}),
\end{equation}
where 
\begin{equation}\label{eq9.31}
{}_{1}F_{1}(b; c; z):=\frac{\Gamma(b)}{\Gamma(a)\Gamma(b-a)}\cdot\int_0^1e^{az}a^{b-1}(1-a)^{c-b-1}da
\end{equation}
is the confluent hypergeometric function.

As a consequence, taking $s=0$ in \eqref{9.30} leads to 
\begin{equation}\label{9.31}
I_v(t)=\frac{(k_v-1)}{4\pi e^{2\pi}}\cdot \frac{(-i)^{\frac{k_v}{2}}(4\pi )^{k_v}|t|_v^{-\frac{k_v}{2}}}{\Gamma(k_v)}
\cdot \frac{\Gamma(k_v/2)^2}{\Gamma(k_v)}\cdot  {}_{1}F_{1}(k_v/2; k_v; -2\pi i t^{-1}).
\end{equation}

\subsubsection{The global integral}
By \eqref{eq9.20}, \eqref{9.21}, \eqref{9.28}, and \eqref{9.31}, in conjunction with the fact that $\prod_{v<\infty}I_{v}(t,s)\equiv 0$ unless $t\in \mathfrak{q}\mathfrak{n}^{-1}-\{0\}$, we conclude the following consequence.
\begin{prop}\label{prop9.5}
Let notation be as before. 
\begin{itemize}
\item The function $I_{\mathrm{reg}}(f_{\mathfrak{n},\mathfrak{q}},s)$ converges absolutely in the strip  
\begin{align*}
-2^{-1}\min_{v\mid\infty}\{k_v\}<\Re(s)<2^{-1}\min_{v\mid\infty}\{k_v\}.
\end{align*}
\item The function $I_{\mathrm{reg}}(f_{\mathfrak{n},\mathfrak{q}},s)$ admits a meromorphic continuation to $s\in \mathbb{C}$, given explicitly by 
\begin{align*}
I_{\mathrm{reg}}(f_{\mathfrak{n},\mathfrak{q}},s)=\sum_{t\in \mathfrak{q}\mathfrak{n}^{-1}-\{0\}}I_{\infty}(t,s)I_{\mathrm{fin}}(t,s),
\end{align*}
where $I_{\infty}(t,s)$ is defined by 
\begin{align*}
\prod_{v\mid\infty}\frac{(k_v-1)}{4\pi e^{2\pi}}\cdot \frac{(-i)^{\frac{k_v}{2}}(4\pi )^{k_v}|t|_v^{-\frac{k_v}{2}}}{\Gamma(k_v)}
\cdot \frac{\Gamma(k_v/2)^2}{\Gamma(k_v)}\cdot  {}_{1}F_{1}(k_v/2; k_v; -2\pi i t^{-1}),
\end{align*}
and $I_{\mathrm{fin}}(t,s)$ is defined by 
\begin{align*}
&\frac{V_{\mathfrak{q}}}{N(\mathfrak{n})^{\frac{1}{2}+s}}\prod_{v\nmid \mathfrak{n}\mathfrak{q}}q_v^{-2d_v} \bigg[q_v^{-d_vs}\textbf{1}_{e_v(t)\geq d_v}+\overline{\psi}_v(t^{-1})q_v^{d_v(1+s)}|t|_v^{2s}\textbf{1}_{e_v(t)\geq d_v+1}\bigg]\\
&\prod_{v\mid\mathfrak{q}}\overline{\psi}_v(t^{-1})|t|_v^{2s} \textbf{1}_{e_v(t)\geq 1}\prod_{v\mid\mathfrak{n}}\bigg[\overline{\psi}_v(t^{-1})|t|_v^{2s}\textbf{1}_{e_v(t)>-e_v(\mathfrak{n})}+q_v^{2e_v(\mathfrak{n})s}\textbf{1}_{e_v(t)\geq -e_v(\mathfrak{n})}\bigg].
\end{align*}
\item Moreover, at $s=0$, we have
\begin{equation}\label{9.32}
I_{\mathrm{reg}}(f_{\mathfrak{n},\mathfrak{q}},0)=\sum_{t\in \mathfrak{q}\mathfrak{n}^{-1}-\{0\}}I_{\infty}(t)I_{\mathrm{fin}}(t),
\end{equation}
where 
\begin{equation}\label{9.33}
I_{\infty}(t):=\prod_{v\mid\infty}\frac{(k_v-1)}{4\pi e^{2\pi}i^{\frac{k_v}{2}}}\cdot \frac{(4\pi )^{k_v}\Gamma(k_v/2)^2}{|t|_v^{\frac{k_v}{2}}\Gamma(k_v)^2}
\cdot {}_{1}F_{1}(k_v/2; k_v; -2\pi i t^{-1}),
\end{equation}
and 
\begin{align}
I_{\mathrm{fin}}(t):=&\frac{V_{\mathfrak{q}}}{N(\mathfrak{n})^{\frac{1}{2}}D_F^2}\prod_{v\nmid \mathfrak{n}\mathfrak{q}}\bigg[\textbf{1}_{e_v(t)\geq d_v}+\overline{\psi}_v(t^{-1})q_v^{d_v}\textbf{1}_{e_v(t)\geq d_v+1}\bigg]\nonumber \\
&\quad \prod_{v\mid\mathfrak{q}}\overline{\psi}_v(t^{-1})\textbf{1}_{e_v(t)\geq 1}\prod_{v\mid\mathfrak{n}}\bigg[\overline{\psi}_v(t^{-1})\textbf{1}_{e_v(t)>-e_v(\mathfrak{n})}+\textbf{1}_{e_v(t)\geq -e_v(\mathfrak{n})}\bigg].\label{9.34}
\end{align}
\end{itemize}
\end{prop}

\subsection{The $\lambda_{\pi}(\mathfrak{n})$-weighted First Moment}\label{sec9.6}
Combining Propositions \ref{prop9.1}, \ref{prop9.2}, \ref{prop9.3}, and \ref{prop9.5}, we derive the following calculation of the first moment. 
\begin{thmx}\label{thmD}
Let notation be as before. Let $\mathbf{k}=(k_v)_{v\mid\infty}\in \mathbb{Z}_{>2}^{d_F}$, where $k_v$ is even, $v\mid\infty$. Let $\mathfrak{q}$ be either $\mathcal{O}_F$ or a prime ideal. Let $\mathfrak{n}\subseteq \mathcal{O}_F$ be an integral ideal with $(\mathfrak{n},\mathfrak{q})=1$. Let $\chi_{\mathfrak{q}}$ be the nontrivial unramified quadratic character of $F_{\mathfrak{q}}^{\times}$ if $\mathfrak{q}\subsetneq\mathcal{O}_F$. Let $V_{\mathfrak{q}}$ be defined as in \eqref{1.1}, and $\delta_{\mathbf{k},\mathfrak{q}}=\textbf{1}_{\mathfrak{q}\subsetneq \mathcal{O}_F}+\textbf{1}_{\mathfrak{q}=\mathcal{O}_F  \& \sum_{v\mid\infty}k_v\equiv 0\pmod{4}}$. Then 
\begin{align*}
&\frac{1}{\zeta_{\mathfrak{q}}(2)^2}\sum_{\substack{\pi\in \mathcal{F}(\mathbf{k},\mathfrak{q})}}\frac{\lambda_{\pi}(\mathfrak{n})L(1/2,\pi)}{L^{(\mathfrak{q})}(1,\pi,\Ad)}+\frac{V_{\mathfrak{q}}\textbf{1}_{\mathfrak{q}\subsetneq \mathcal{O}_F}}{N(\mathfrak{q})}\sum_{\substack{\pi\in \mathcal{F}(\mathbf{k},\mathcal{O}_F)}}\frac{\lambda_{\pi}(\mathfrak{n})L_{\pi_{\mathfrak{q}}}L(1/2,\pi)}{L(1,\pi,\Ad)}\\
=&\frac{2\delta_{\mathbf{k},\mathfrak{q}}(N(\mathfrak{q})+1)\cdot D_F^{\frac{3}{2}}}{N(\mathfrak{n})^{1/2}}\prod_{v\mid\infty}\frac{(k_v-1)}{4\pi ^2}+\mathcal{E},
\end{align*}
where we define $\zeta_{\mathfrak{q}}(2)=1$ and $L^{(\mathfrak{q})}(1,\pi,\Ad)=L(1,\pi,\Ad)$ if $\mathfrak{q}=\mathcal{O}_F$,
$L_{\pi_{\mathfrak{q}}}:=L_{\mathfrak{q}}(1/2,\pi_{\mathfrak{q}}\times\chi_{\mathfrak{q}})$, $I_{\mathrm{fin}}(t)$ is defined by \eqref{9.34}, and the term $\mathcal{E}$ is defined by 
\begin{align*}
\mathcal{E}:=D_F^{\frac{3}{2}}\sum_{t\in \mathfrak{q}\mathfrak{n}^{-1}-\{0\}}I_{\mathrm{fin}}(t)\prod_{v\mid\infty}\frac{(k_v-1)}{4\pi^2i^{\frac{k_v}{2}}}\cdot \frac{(2\pi )^{\frac{k_v}{2}}\Gamma(k_v/2)}{|t|_v^{\frac{k_v}{2}}\Gamma(k_v)}
\cdot {}_{1}F_{1}(k_v/2; k_v; -2\pi i t^{-1}),
\end{align*}
with ${}_{1}F_{1}(k_v/2; k_v; 2\pi i t^{-1})$ being the confluent hypergeometric function (cf. \eqref{eq9.31}).
\end{thmx}

Theorem \ref{thmD} generalizes \cite[Theorem 1.1]{KL10} in the case of Hilbert modular forms, though without the character twist. Obtaining a good bound for the tail $\mathcal{E}$ is crucial. We will employ a different strategy than that of \textit{loc. cit.}, as discussed in Remark \ref{rmk9.8} below. The key result is the following estimate. 
\begin{lemma}\label{lem9.8}
Let notation be as before. Suppose $k_v>2$ is even  for all $k_v\mid\infty$. Let $\varepsilon>0$. Then 
\begin{equation}\label{equa9.36}
\mathcal{E}\ll \|\mathbf{k}\|^{\varepsilon}N(\mathfrak{q})^{\varepsilon}N(\mathfrak{n})^{\frac{1}{2}+\varepsilon},
\end{equation}	
where the implied constant depends on $F$ and $\varepsilon$. 
\end{lemma}
\begin{proof}
Taking advantage of \cite[\textsection 6.5, p.276]{MOS66}, along with a meromorphic continuation, we obtain 
\begin{align*}
{}_{1}F_{1}(k_v/2; k_v; -2\pi i t^{-1})=\frac{\Gamma(k_v)(-2\pi i t^{-1})^{\frac{1-k_v}{2}}}{\Gamma(k_v/2)e^{2\pi i t^{-1}}}\int_0^{\infty}\frac{J_{k_v-1}(2\sqrt{-2\pi i t^{-1}u})}{e^{u}u^{\frac{1}{2}}}du,
\end{align*}
where we fix a branch for the square-root $\sqrt{-2\pi i t^{-1}u}$. Consequently,  
\begin{align*}
\frac{(2\pi )^{\frac{k_v}{2}}\Gamma(k_v/2){}_{1}F_{1}(k_v/2; k_v; -2\pi i t^{-1})}{i^{\frac{k_v}{2}}|t|_v^{\frac{k_v}{2}}\Gamma(k_v)}
=\frac{(-2\pi i t^{-1})^{\frac{1}{2}}}{e^{2\pi i t^{-1}}}\int_0^{\infty}\frac{J_{k_v-1}(2\sqrt{-2\pi i t^{-1}u})}{e^{u}u^{\frac{1}{2}}}du.
\end{align*}

Therefore, $\mathcal{E}$ boils down to  
\begin{equation}\label{equ9.36}
D_F^{\frac{3}{2}}\sum_{t\in \mathfrak{q}\mathfrak{n}^{-1}-\{0\}}I_{\mathrm{fin}}(t)\prod_{v\mid\infty}\frac{(k_v-1)(-2\pi i t^{-1})^{\frac{1}{2}}}{4\pi^2e^{2\pi i t^{-1}}}\int_0^{\infty}\frac{J_{k_v-1}(2\sqrt{-2\pi i t^{-1}u})}{e^{u}u^{\frac{1}{2}}}du.
\end{equation}

Let $0<\varepsilon<10^{-3}$. Making use of the Mellin-Barnes integral representation
\begin{align*}
J_{k_v-1}(z)=\frac{1}{2\pi i}\int_{(-2-2\varepsilon)}2^sz^{-s-1}\Gamma\left(\frac{k_v+s}{2}\right)\Gamma\left(\frac{k_v-s}{2}\right)^{-1}ds,
\end{align*}
in conjunction with the Stirling formula, we obtain  
\begin{align*}
J_{k_v-1}(z)\ll |z|^{1+2\varepsilon}_vk_v^{-1+2\varepsilon},
\end{align*}
where the implied constant depends only on $\varepsilon$. In particular, 
\begin{equation}\label{9.36}
J_{k_v-1}(2\sqrt{-2\pi i t^{-1}u})\ll |t|_v^{-\frac{1}{2}-\varepsilon}u^{\frac{1}{2}+\varepsilon}k_v^{-1+2\varepsilon}.
\end{equation}

Together with the triangle inequality, we deduce from \eqref{9.36} that 
\begin{align*}
\int_0^{\infty}\frac{J_{k_v-1}(2\sqrt{-2\pi i t^{-1}u})}{e^{u}u^{\frac{1}{2}}}du\ll |t|_v^{-\frac{1}{2}-\varepsilon}k_v^{-1+2\varepsilon}\int_0^{\infty}e^{-u}u^{\varepsilon}du\ll |t|_v^{-\frac{1}{2}-\varepsilon}k_v^{-1+2\varepsilon}.
\end{align*}

Substituting this into \eqref{equ9.36} yields 
\begin{equation}\label{9.38}
|\mathcal{E}|\ll \|\mathbf{k}\|^{\varepsilon}\sum_{t\in \mathfrak{q}\mathfrak{n}^{-1}-\{0\}}N(t)^{-1-\varepsilon}\cdot |I_{\mathrm{fin}}(t)|. 
\end{equation}

According to the definition of $I_{\mathrm{fin}}(t)$ in \eqref{9.34}, we have
\begin{equation}\label{9.39}
|I_{\mathrm{fin}}(t)|\ll \frac{V_{\mathfrak{q}}}{N(\mathfrak{n})^{\frac{1}{2}}}\prod_{v\nmid \mathfrak{n}\mathfrak{q}}\textbf{1}_{e_v(t)\geq d_v}\prod_{v\mid\mathfrak{n}}2\cdot \textbf{1}_{e_v(t)\geq -e_v(\mathfrak{n})}\ll 2^{\omega(\mathfrak{n})}V_{\mathfrak{q}}N(\mathfrak{n})^{-1/2},
\end{equation}
where the implied constant depends on $F$. Here $\omega(\mathfrak{n})$ refers to the number of prime factors of $\mathfrak{n}$ in its primary decomposition.

Therefore, \eqref{equa9.36} follows from \eqref{9.38} and \eqref{9.39}, using the arguments in the proof of Proposition \ref{prop6.12} (cf. \textsection\ref{sec7.3.3}).
\end{proof}

\begin{remark}\label{rmk9.8}
Applying the bound $|{}_{1}F_{1}(k_v/2; k_v; 2\pi i t^{-1})|\leq 1$ as in \cite{KL10} yields 
\begin{equation}\label{eq9.36}
|\mathcal{E}|\leq C\cdot \sum_{t\in \mathfrak{q}\mathfrak{n}^{-1}-\{0\}}\prod_{v\mid\infty}|t|_v^{-\frac{k_v}{2}}\prod_{v\nmid \mathfrak{n}\mathfrak{q}}\bigg[\textbf{1}_{e_v(t)\geq d_v}+q_v^{d_v}\textbf{1}_{e_v(t)\geq d_v+1}\bigg],
\end{equation}
where 
\begin{align*}
C:=\frac{2^{\omega(\mathfrak{n})}V_{\mathfrak{q}}}{N(\mathfrak{n})^{\frac{1}{2}}D_F^{\frac{1}{2}}}\prod_{v\mid\infty}\frac{(k_v-1)(2\pi )^{\frac{k_v}{2}}\Gamma(k_v/2)}{4\pi^2\Gamma(k_v)}.
\end{align*} 

However, handling the term $\prod_{v \mid \infty} |t|_v^{-\frac{k_v}{2}}$ on the right-hand side of \eqref{eq9.36} becomes challenging in the non-parallel weight case, i.e., when the $k_v$'s are not all equal.
\end{remark}

Combining Theorem \ref{thmD} and Lemma \ref{lem9.8}, we conclude the following. 
\begin{cor}\label{cor9.9}
Let notation be as in Theorem \ref{thmD}. Let $0<\varepsilon<10^{-2}$.  Then 
\begin{align*}
&\frac{1}{\zeta_{\mathfrak{q}}(2)^2}\sum_{\substack{\pi\in \mathcal{F}(\mathbf{k},\mathfrak{q})}}\frac{\lambda_{\pi}(\mathfrak{n})L(1/2,\pi)}{L^{(\mathfrak{q})}(1,\pi,\Ad)}+\frac{V_{\mathfrak{q}}\textbf{1}_{\mathfrak{q}\subsetneq \mathcal{O}_F}}{N(\mathfrak{q})}\sum_{\substack{\pi\in \mathcal{F}(\mathbf{k},\mathcal{O}_F)}}\frac{\lambda_{\pi}(\mathfrak{n})L_{\pi_{\mathfrak{q}}}L(1/2,\pi)}{L(1,\pi,\Ad)}\\
=&\frac{2\delta_{\mathbf{k},\mathfrak{q}}(N(\mathfrak{q})+1)\cdot D_F^{\frac{3}{2}}}{N(\mathfrak{n})^{1/2}}\prod_{v\mid\infty}\frac{(k_v-1)}{4\pi ^2}+O(\|\mathbf{k}\|^{\varepsilon}N(\mathfrak{q})^{\varepsilon}N(\mathfrak{n})^{\frac{1}{2}+\varepsilon}),
\end{align*}
where the implied constant depends on $F$ and $\varepsilon$. 
\end{cor}

\subsection{Contribution From Old Forms}\label{sec9.7}
Let $\mathfrak{q}$ be a prime ideal. Let $\mathfrak{n}\subseteq \mathcal{O}_F$ be an integral ideal with $(\mathfrak{n},\mathfrak{q})=1$. Parallel to \eqref{y7.4} we define 
\begin{equation}\label{9.42}
I_{\mathrm{Spec}}^{\mathrm{old}}(\mathfrak{n}):=\frac{V_{\mathfrak{q}}}{N(\mathfrak{q})}\sum_{\substack{\pi\in \mathcal{F}(\mathbf{k},\mathcal{O}_F)}}\frac{\lambda_{\pi}(\mathfrak{n})L_{\pi_{\mathfrak{q}}}L(1/2,\pi)^2}{L(1,\pi,\Ad)},
\end{equation}
where  $L_{\pi_{\mathfrak{q}}}:=L_{\mathfrak{q}}(1/2,\pi_{\mathfrak{q}}\times\chi_{\mathfrak{q}})$. Here $\chi_{\mathfrak{q}}$ is the nontrivial unramified quadratic character of $F_{\mathfrak{q}}^{\times}$.
\begin{lemma}\label{lem9.10}
Let notation be as before. Let $\mathfrak{q}$ be a prime ideal. Let $I_{\mathrm{Spec}}^{\mathrm{old}}(\mathfrak{n})$ be defined as in \eqref{9.42}. Let $0<\varepsilon<10^{-3}$. Then 
\begin{equation}\label{eq9.43}
I_{\mathrm{Spec}}^{\mathrm{old}}(\mathfrak{n})=\frac{4\delta_{\mathbf{k}}\cdot D_F^{\frac{3}{2}}}{N(\mathfrak{n})^{1/2}}\prod_{v\mid\infty}\frac{(k_v-1)}{4\pi ^2}+O(N(\mathfrak{n})^{1/2+\varepsilon}N(\mathfrak{q})^{\varepsilon}\|\mathbf{k}\|^{\varepsilon}),
\end{equation}
where $\delta_{\mathbf{k}}:=\textbf{1}_{\sum_{v\mid\infty}k_v\equiv 0\pmod{4}}$, and the implied constant depends only on $\varepsilon$ and $F$. 
\end{lemma}
\begin{proof}
Let $m_0>1$ be a sufficient large integer. Hence, 
\begin{equation}\label{9.43}
I_{\mathrm{Spec}}^{\mathrm{old}}(\mathfrak{n})=I_{\mathrm{Spec}}(\mathfrak{n},m_0)^++I_{\mathrm{Spec}}(\mathfrak{n},m_0)^-,
\end{equation}
where 
\begin{align*}
I_{\mathrm{Spec}}(\mathfrak{n},m_0)^+:=&\frac{V_{\mathfrak{q}}}{N(\mathfrak{q})}\sum_{m>m_0}\frac{(-1)^m}{N(\mathfrak{q})^{m/2}}\sum_{\substack{\pi\in \mathcal{F}(\mathbf{k},\mathcal{O}_F)}}\frac{\lambda_{\pi}(\mathfrak{n}\mathfrak{q}^m)L(1/2,\pi)}{L(1,\pi,\Ad)},\\
I_{\mathrm{Spec}}(\mathfrak{n},m_0)^-:=&\frac{V_{\mathfrak{q}}}{N(\mathfrak{q})}\sum_{0\leq m\leq m_0}\frac{(-1)^m}{N(\mathfrak{q})^{m/2}}\sum_{\substack{\pi\in \mathcal{F}(\mathbf{k},\mathcal{O}_F)}}\frac{\lambda_{\pi}(\mathfrak{n}\mathfrak{q}^m)L(1/2,\pi)}{L(1,\pi,\Ad)}.
\end{align*}

By Corollary \ref{cor9.9} with $\mathfrak{q}=\mathcal{O}_F$, and $\mathfrak{n}$ replaced by $\mathfrak{n}\mathfrak{q}^m$, we obtain 
\begin{equation}\label{9.44}
\sum_{\substack{\pi\in \mathcal{F}(\mathbf{k},\mathcal{O}_F)}}\frac{\lambda_{\pi}(\mathfrak{n}\mathfrak{q}^m)L(1/2,\pi)}{L(1,\pi,\Ad)}
=\frac{4\delta_{\mathbf{k},\mathfrak{q}}\cdot D_F^{\frac{3}{2}}}{N(\mathfrak{n}\mathfrak{q}^m)^{1/2}}\prod_{v\mid\infty}\frac{(k_v-1)}{4\pi ^2}+E(m),
\end{equation}
where $\delta_{\mathbf{k},\mathfrak{q}}:=\textbf{1}_{\mathfrak{q}\subsetneq \mathcal{O}_F}+\textbf{1}_{\mathfrak{q}=\mathcal{O}_F  \& \sum_{v\mid\infty}k_v\equiv 0\pmod{4}}$, and 
\begin{equation}\label{9.45}
E(m)\ll \|\mathbf{k}\|^{\varepsilon}N(\mathfrak{q})^{\varepsilon}N(\mathfrak{n}\mathfrak{q}^m)^{\frac{1}{2}+\varepsilon},	
\end{equation}
with the implied constant depending only on $F$ and $\varepsilon$. 

By \eqref{9.44} and \eqref{9.45}, we obtain 
\begin{equation}\label{9.46}
I_{\mathrm{Spec}}(\mathfrak{n},m_0)^-=\mathcal{M}(m_0)+O(m_0N(\mathfrak{n})^{1/2+\varepsilon}N(\mathfrak{q})^{m_0\varepsilon}\|\mathbf{k}\|^{-1+\varepsilon}),
\end{equation}
where the implied constant depends only on $\varepsilon$ and $F$, and 
\begin{align*}
\mathcal{M}(m_0):=\frac{V_{\mathfrak{q}}}{N(\mathfrak{q})}\sum_{0\leq m\leq m_0}\frac{(-1)^m}{N(\mathfrak{q})^{m}}\cdot \frac{4\delta_{\mathbf{k},\mathfrak{q}}\cdot D_F^{\frac{3}{2}}}{N(\mathfrak{n})^{1/2}}\prod_{v\mid\infty}\frac{(k_v-1)}{4\pi ^2}.
\end{align*}

Computing the inner sum over $m$ we deduce 
\begin{equation}\label{9.47}
\mathcal{M}(m_0)=\frac{4\delta_{\mathbf{k},\mathfrak{q}}\cdot D_F^{\frac{3}{2}}}{N(\mathfrak{n})^{1/2}}\prod_{v\mid\infty}\frac{(k_v-1)}{4\pi ^2}\cdot (1+O(N(\mathfrak{q})^{-m_0+1})).
\end{equation}

On the other hand, by $|\lambda_{\pi}(\mathfrak{n}\mathfrak{q}^m)|\ll N(\mathfrak{n}\mathfrak{q}^m)^{\vartheta}$, $0\leq \vartheta\leq 7/64$, we derive 
\begin{equation}\label{9.48}
I_{\mathrm{Spec}}(\mathfrak{n},m_0)^+\ll \frac{N(\mathfrak{n})^{\vartheta}}{N(\mathfrak{q})^{(1/2-\vartheta)m_0}}\sum_{\substack{\pi\in \mathcal{F}(\mathbf{k},\mathcal{O}_F)}}\frac{L(1/2,\pi)}{L(1,\pi,\Ad)}.
\end{equation}

Making use of Corollary \ref{cor9.9} with $\mathfrak{q}=\mathfrak{n}=\mathcal{O}_F$,  we obtain 
\begin{align*}
\sum_{\substack{\pi\in \mathcal{F}(\mathbf{k},\mathcal{O}_F)}}\frac{L(1/2,\pi)}{L(1,\pi,\Ad)}\ll N(\mathfrak{n})^{1/2}\|\mathbf{k}\|^{1+\varepsilon}+\|\mathbf{k}\|^{\varepsilon}N(\mathfrak{q})^{\varepsilon}N(\mathfrak{n})^{\frac{1}{2}+\varepsilon}.
\end{align*}
Substituting this into \eqref{9.48} yields
\begin{equation}\label{9.49}
J_{\mathrm{Spec}}(\mathfrak{n},m_0)^+\ll \frac{N(\mathfrak{n})^{1/2}\|\mathbf{k}\|^{1+\varepsilon}+\|\mathbf{k}\|^{\varepsilon}N(\mathfrak{q})^{\varepsilon}N(\mathfrak{n})^{\frac{1}{2}+\varepsilon}}{N(\mathfrak{q})^{(1/2-\vartheta)m_0}},
\end{equation}
where the implied constant depends only on $F$. 

Take $m_0=100(1+(\log N(\mathfrak{q}))^{-1}\log \|\mathbf{k}\|)$, i.e., $N(\mathfrak{q})^{m_0}=(\|\mathbf{k}\|N(\mathfrak{q}))^{100}$. Combining \eqref{9.43}, \eqref{9.46}, \eqref{9.47}, and \eqref{9.49}, we conclude \eqref{eq9.43}. 
\end{proof}

\subsection{The Mollified First Moment: New Forms}\label{sec9.8}
Let $\rho$ be a multiplicative arithmetic function defined as in \textsection\ref{sec8.3}. Suppose $\rho(\mathfrak{p})\ll 1$ for all prime ideals $\mathfrak{p}$, with the implied constant being absolute. For $\xi>1$, we denote by $M_{\xi,\rho}(\pi)$ the mollifier defined as in \textsection\ref{sec1.1.6}, i.e., 
\begin{align*}
M_{\xi,\rho}(\pi)= \frac{1}{\log \xi} \sum_{\substack{\mathfrak{n} \subseteq \mathcal{O}_F \\ (\mathfrak{n},\mathfrak{q}) = 1}} \frac{\lambda_\pi(\mathfrak{n})\mu_F(\mathfrak{n})\rho(\mathfrak{n})}{\sqrt{N(\mathfrak{n})}}\cdot  \frac{1}{2\pi i}\int_{(2)} \frac{\xi^s}{N(\mathfrak{n})^s} \frac{ds}{s^3}.\tag{\ref{M}}
\end{align*}

\begin{defn}\label{defn9.11}
Let notation be as before. Define the mollified first moment by 
\begin{align*}
I_{\mathrm{Spec}}^{\heartsuit,\mathrm{new}}(\xi,\rho):=&\frac{1}{\zeta_{\mathfrak{q}}(2)^2}\sum_{\substack{\pi\in \mathcal{F}(\mathbf{k},\mathfrak{q})}}\frac{L(1/2,\pi)M_{\xi,\rho}(\pi)}{L^{(\mathfrak{q})}(1,\pi,\Ad)},\\
I_{\mathrm{Spec}}^{\heartsuit,\mathrm{old}}(\xi,\rho):=&\frac{V_{\mathfrak{q}}\textbf{1}_{\mathfrak{q}\subsetneq \mathcal{O}_F}}{N(\mathfrak{q})}\sum_{\substack{\pi\in \mathcal{F}(\mathbf{k},\mathcal{O}_F)}}\frac{L_{\pi_{\mathfrak{q}}}L(1/2,\pi)M_{\xi,\rho}(\pi)}{L(1,\pi,\Ad)}.
\end{align*}
\end{defn}

\begin{lemma}\label{lem9.12}
Let notation be as before. Let $L_{\mathfrak{q}}(1,\rho)\equiv 1$ if $\mathfrak{q}=\mathcal{O}_F$, and $L_{\mathfrak{q}}(1,\rho)=(1-\rho(\mathfrak{q})N(\mathfrak{q})^{-1})^{-1}$ if $\mathfrak{q}\subsetneq \mathcal{O}_F$.  
\begin{itemize}
\item We have the following asymptotic expansion:
\begin{align*}
I_{\mathrm{Spec}}^{\heartsuit,\mathrm{old}}(\xi,\rho)=&\frac{4\delta_{\mathbf{k}}\cdot D_F^{\frac{3}{2}}\cdot \textbf{1}_{\mathfrak{q}\subsetneq \mathcal{O}_F}}{\log \xi}\prod_{v\mid\infty}\frac{(k_v-1)}{4\pi ^2}\cdot\bigg[\frac{L_{\mathfrak{q}}(1,\rho)\cdot \log \xi}{R(\rho)\cdot \Res_{s=1}\zeta_F(s)}+O(1)\bigg]\\
&+O(N(\mathfrak{q})^{\varepsilon}\|\mathbf{k}\|^{\varepsilon}\xi^{1+2\varepsilon}\textbf{1}_{\mathfrak{q}\subsetneq \mathcal{O}_F}),
\end{align*}
where the implied constants depend only on $F$ and $\varepsilon$. 

\item Suppose $\rho(\mathfrak{n}):=\textbf{1}_{\mathfrak{n}=\mathcal{O}_F}+\textbf{1}_{\mathfrak{n}\subsetneq \mathcal{O}_F}\prod_{\substack{\mathfrak{p}\mid\mathfrak{n}\\ \text{$\mathfrak{p}$ prime}}}(1+N(\mathfrak{p})^{-1})^{-1}$ as defined in  \eqref{rho}. Then 
\begin{align*}
I_{\mathrm{Spec}}^{\heartsuit,\mathrm{old}}(\xi,\rho)=&\frac{4\zeta_F(2)\delta_{\mathbf{k}}\cdot L_{\mathfrak{q}}(1,\rho)\cdot D_F^{\frac{3}{2}}\cdot \textbf{1}_{\mathfrak{q}\subsetneq \mathcal{O}_F}}{\Res_{s=1}\zeta_F(s)}\prod_{v\mid\infty}\frac{(k_v-1)}{4\pi ^2}\cdot (1+O((\log\xi)^{-1}))\\
&+O(N(\mathfrak{q})^{\varepsilon}\|\mathbf{k}\|^{\varepsilon}\xi^{1+2\varepsilon}\textbf{1}_{\mathfrak{q}\subsetneq \mathcal{O}_F}),
\end{align*}
where the implied constants depend only on $F$ and $\varepsilon$. 
\end{itemize}
\end{lemma}
\begin{proof}
Let $0<\varepsilon<10^{-3}$.  Define 
\begin{align*}
I_{\mathrm{Spec},1}^{\heartsuit,\mathrm{old}}(\xi,\rho):=&\frac{4\delta_{\mathbf{k}}\cdot D_F^{\frac{3}{2}}\cdot \textbf{1}_{\mathfrak{q}\subsetneq \mathcal{O}_F}}{\log \xi}\prod_{v\mid\infty}\frac{(k_v-1)}{4\pi ^2}\cdot \frac{1}{2\pi i}\int_{(2)}\xi^s \sum_{\substack{\mathfrak{n} \subseteq \mathcal{O}_F \\ (\mathfrak{n},\mathfrak{q}) = 1}} \frac{\mu_F(\mathfrak{n})\rho(\mathfrak{n})}{N(\mathfrak{n})^{1+s}}\frac{ds}{s^3},\\
I_{\mathrm{Spec},2}^{\heartsuit,\mathrm{old}}(\xi,\rho):=&\frac{N(\mathfrak{q})^{\varepsilon}\|\mathbf{k}\|^{\varepsilon}\textbf{1}_{\mathfrak{q}\subsetneq \mathcal{O}_F}}{\log \xi}\sum_{\substack{\mathfrak{n} \subseteq \mathcal{O}_F \\ (\mathfrak{n},\mathfrak{q}) = 1\\ N(\mathfrak{q})\leq \xi}} \frac{|\mu_F(\mathfrak{n})\rho(\mathfrak{n})|}{\sqrt{N(\mathfrak{n})}}\cdot  \frac{(\log \xi N(\mathfrak{n}_1)^{-1})^2}{2}\cdot N(\mathfrak{n})^{1/2+\varepsilon}.
\end{align*}

By \eqref{c7.8} and Lemma \ref{lem9.10}, we have
\begin{equation}\label{9.52}
\big|I_{\mathrm{Spec}}^{\heartsuit,\mathrm{old}}(\xi,\rho)-I_{\mathrm{Spec},1}^{\heartsuit,\mathrm{old}}(\xi,\rho)\big|\ll I_{\mathrm{Spec},2}^{\heartsuit,\mathrm{old}}(\xi,\rho),
\end{equation}
where the implied constant depends only on $F$ and $\varepsilon$. Since $|\rho(\mathfrak{n})|\ll 1$, then 
\begin{equation}\label{eq9.52}
I_{\mathrm{Spec},2}^{\heartsuit,\mathrm{old}}(\xi,\rho)\ll \frac{(N(\mathfrak{q})\|\mathbf{k}\|)^{\varepsilon}\textbf{1}_{\mathfrak{q}\subsetneq \mathcal{O}_F}}{(\log \xi)^{-1}}\sum_{\substack{\mathfrak{n} \subseteq \mathcal{O}_F \\ (\mathfrak{n},\mathfrak{q}) = 1\\ N(\mathfrak{q})\leq \xi}} N(\mathfrak{n})^{\varepsilon}\ll (N(\mathfrak{q})\|\mathbf{k}\|)^{\varepsilon}\xi^{1+2\varepsilon}\textbf{1}_{\mathfrak{q}\subsetneq \mathcal{O}_F},
\end{equation}
where the implied constant depends only on $F$ and $\varepsilon$. 

Let $L(s,\rho)$ be defined as in \eqref{eq7.1}. For $\Re(s)>0$, we have 
\begin{equation}\label{9.54}
\sum_{\substack{\mathfrak{n} \subseteq \mathcal{O}_F \\ (\mathfrak{n},\mathfrak{q}) = 1}} \frac{\mu_F(\mathfrak{n})\rho(\mathfrak{n})}{N(\mathfrak{n})^{1+s}}=(1-\rho(\mathfrak{q})N(\mathfrak{q})^{-1-s})^{-1}L(1+s,\rho)^{-1}.
\end{equation}

Recall that (cf. \eqref{eq7.28}) $R(s,\rho):=L(1+s,\rho)/\zeta_F(1+s)$ converges absolutely in $\Re(s)>-\varepsilon$. Therefore, by \eqref{9.54},  
\begin{align*}
\frac{1}{2\pi i}\int_{(2)}\xi^s \sum_{\substack{\mathfrak{n} \subseteq \mathcal{O}_F \\ (\mathfrak{n},\mathfrak{q}) = 1}} \frac{\mu_F(\mathfrak{n})\rho(\mathfrak{n})}{N(\mathfrak{n})^{1+s}}\frac{ds}{s^3}=\frac{1}{2\pi i}\int_{(2)} \frac{(1-\rho(\mathfrak{q})N(\mathfrak{q})^{-1-s})^{-1}\cdot\xi^s}{ s\zeta_F(1+s)R(s,\rho)}\frac{ds}{s^2},
\end{align*}
which, by Cauchy formula, is equal to 
\begin{align*}
\frac{d}{ds}\frac{(1-\rho(\mathfrak{q})N(\mathfrak{q})^{-1-s})^{-1}\cdot\xi^s}{ s\zeta_F(1+s)R(s,\rho)}\Big|_{s=0}+\frac{1}{2\pi i}\int_{(-\varepsilon/10)} \frac{(1-\rho(\mathfrak{q})N(\mathfrak{q})^{-1-s})^{-1}\cdot\xi^s}{ s\zeta_F(1+s)R(s,\rho)}\frac{ds}{s^2}.
\end{align*}

Combining the above calculations we derive that 
\begin{equation}\label{9.55}
\frac{1}{2\pi i}\int_{(2)}\xi^s \sum_{\substack{\mathfrak{n} \subseteq \mathcal{O}_F \\ (\mathfrak{n},\mathfrak{q}) = 1}} \frac{\mu_F(\mathfrak{n})\rho(\mathfrak{n})}{N(\mathfrak{n})^{1+s}}\frac{ds}{s^3}=\frac{(1-\rho(\mathfrak{q})N(\mathfrak{q})^{-1})^{-1}\cdot \log \xi}{R(0,\rho)\cdot \Res_{s=1}\zeta_F(s)}+O(1),
\end{equation}
where the implied constant depends on $F$ and $\varepsilon$.

Then the first part of Lemma \ref{lem9.12} follows from \eqref{9.52}, \eqref{eq9.52}, and \eqref{9.55}. Moreover, a straightforward calculation leads to 
\begin{equation}\label{eq9.55}
R(0,\rho)=\lim_{s\to 0}\frac{L(1+s,\rho)}{\zeta_F(1+s)}=\prod_{\mathfrak{p}}\bigg[1-N(\mathfrak{p})^{-2}\bigg]=\zeta_F(2)^{-1}
\end{equation}
if $\rho(\mathfrak{q})=(1+N(\mathfrak{p})^{-1})^{-1}$. As a conseqeunce, the second part of Lemma \ref{lem9.12} follows from \eqref{eq9.55}. 
\end{proof}

\begin{thmx}\label{prop9.13}
Let notation be as before. Let $L_{\mathfrak{q}}(1,\rho)\equiv 1$ if $\mathfrak{q}=\mathcal{O}_F$, and $L_{\mathfrak{q}}(1,\rho)=(1-\rho(\mathfrak{q})N(\mathfrak{q})^{-1})^{-1}$ if $\mathfrak{q}\subsetneq \mathcal{O}_F$. Then 
\begin{itemize}
\item The spectral side $I_{\mathrm{Spec}}^{\heartsuit,\mathrm{new}}(\xi,\rho)$ is equal to 
\begin{align*}
&\frac{2\delta_{\mathbf{k},\mathfrak{q}}(N(\mathfrak{q})+1)-4\delta_{\mathbf{k}} \textbf{1}_{\mathfrak{q}\subsetneq \mathcal{O}_F}}{D_F^{-3/2}\log \xi}\prod_{v\mid\infty}\frac{(k_v-1)}{4\pi ^2}\bigg[\frac{L_{\mathfrak{q}}(1,\rho)\log \xi}{R(\rho)\cdot \Res_{s=1}\zeta_F(s)}+O(1)\bigg]\\
&+O(N(\mathfrak{q})^{\varepsilon}\|\mathbf{k}\|^{\varepsilon}\xi^{1+2\varepsilon}),
\end{align*}
where $\delta_{\mathbf{k},\mathfrak{q}}:=\textbf{1}_{\mathfrak{q}\subsetneq \mathcal{O}_F}+\textbf{1}_{\mathfrak{q}=\mathcal{O}_F  \& \sum_{v\mid\infty}k_v\equiv 0\pmod{4}}$, and the implied constants depend only on $F$ and $\varepsilon$. 

\item Let $\rho(\mathfrak{n}):=\textbf{1}_{\mathfrak{n}=\mathcal{O}_F}+\textbf{1}_{\mathfrak{n}\subsetneq \mathcal{O}_F}\prod_{\substack{\mathfrak{p}\mid\mathfrak{n}\\ \text{$\mathfrak{p}$ prime}}}(1+N(\mathfrak{p})^{-1})^{-1}$ be the arithmetic function defined in  \eqref{rho}. Then  
\begin{align*}
I_{\mathrm{Spec}}^{\heartsuit,\mathrm{new}}(\xi,\rho)=&\frac{(2\delta_{\mathbf{k},\mathfrak{q}}(N(\mathfrak{q})+1)-4\delta_{\mathbf{k}}\textbf{1}_{\mathfrak{q}\subsetneq \mathcal{O}_F})\zeta_F(2)\cdot L_{\mathfrak{q}}(1,\rho)}{D_F^{-3/2} \cdot  \Res_{s=1}\zeta_F(s)}\prod_{v\mid\infty}\frac{(k_v-1)}{4\pi ^2}\\
&\qquad  \qquad +O((\log\xi)^{-1}N(\mathfrak{q})\|\mathbf{k}\|)+O(N(\mathfrak{q})^{\varepsilon}\|\mathbf{k}\|^{\varepsilon}\xi^{1+2\varepsilon}),
\end{align*}
where the implied constants depend only on $F$ and $\varepsilon$. 
\end{itemize}
\end{thmx}
\begin{proof}
By Corollary \ref{cor9.9},
\begin{equation}\label{9.56}
|I_{\mathrm{Spec}}^{\heartsuit,\mathrm{new}}(\xi,\rho)+I_{\mathrm{Spec}}^{\heartsuit,\mathrm{old}}(\xi,\rho)-\mathcal{I}_1|\ll \mathcal{I}_2,
\end{equation}
where 
\begin{align*}
\mathcal{I}_1:=\frac{2\delta_{\mathbf{k},\mathfrak{q}}(N(\mathfrak{q})+1)\cdot D_F^{\frac{3}{2}}}{\log \xi}\prod_{v\mid\infty}\frac{(k_v-1)}{4\pi ^2}\cdot \frac{1}{2\pi i}\int_{(2)}\xi^s \sum_{\substack{\mathfrak{n} \subseteq \mathcal{O}_F \\ (\mathfrak{n},\mathfrak{q}) = 1}} \frac{\mu_F(\mathfrak{n})\rho(\mathfrak{n})}{N(\mathfrak{n})^{1+s}}\frac{ds}{s^3},
\end{align*}
and 
\begin{align*}
\mathcal{I}_2:=\frac{N(\mathfrak{q})^{\varepsilon}\|\mathbf{k}\|^{\varepsilon}}{\log \xi}\sum_{\substack{\mathfrak{n} \subseteq \mathcal{O}_F \\ (\mathfrak{n},\mathfrak{q}) = 1\\ N(\mathfrak{q})\leq \xi}} \frac{|\mu_F(\mathfrak{n})\rho(\mathfrak{n})|}{\sqrt{N(\mathfrak{n})}}\cdot  \frac{(\log \xi N(\mathfrak{n}_1)^{-1})^2}{2}\cdot N(\mathfrak{n})^{1/2+\varepsilon}.
\end{align*}

Similar to \eqref{eq9.52} we have
\begin{equation}\label{9.57}
\mathcal{I}_2\ll \frac{(N(\mathfrak{q})\|\mathbf{k}\|)^{\varepsilon}}{(\log \xi)^{-1}}\sum_{\substack{\mathfrak{n} \subseteq \mathcal{O}_F \\ (\mathfrak{n},\mathfrak{q}) = 1\\ N(\mathfrak{q})\leq \xi}} N(\mathfrak{n})^{\varepsilon}\ll (N(\mathfrak{q})\|\mathbf{k}\|)^{\varepsilon}\xi^{1+2\varepsilon},
\end{equation}
where the implied constant depends only on $F$ and $\varepsilon$. 

Similar to the calculation of $I_{\mathrm{Spec},1}^{\heartsuit,\mathrm{old}}(\xi,\rho)$ in Lemma \ref{lem9.12}, we have
\begin{align*}
\mathcal{I}_1=&\frac{2\delta_{\mathbf{k},\mathfrak{q}}(N(\mathfrak{q})+1)\cdot D_F^{\frac{3}{2}}}{\log \xi}\prod_{v\mid\infty}\frac{(k_v-1)}{4\pi ^2}\cdot\bigg[\frac{(1-\rho(\mathfrak{q})N(\mathfrak{q})^{-1})^{-1}\cdot \log \xi}{R(\rho)\cdot \Res_{s=1}\zeta_F(s)}+O(1)\bigg]\\
&+O(N(\mathfrak{q})^{\varepsilon}\|\mathbf{k}\|^{\varepsilon}\xi^{1+2\varepsilon}).
\end{align*}

Therefore, Theorem \ref{prop9.13} follows from Lemma \ref{lem9.12}, \eqref{9.56}, \eqref{9.57}, the above asymptotic behavior of $\mathcal{I}_1$. 
\end{proof}

\section{Uniform non-vanishing in Harmonic Average}\label{sec10}
\begin{lemma}\label{lem10.1}
Let notation be as before. Then 
\begin{equation}\label{10.1}
\sum_{\substack{\pi\in \mathcal{F}(\mathbf{k},\mathfrak{q})}}\frac{1}{L(1,\pi,\Ad)}\sim 2(V_{\mathfrak{q}}-\textbf{1}_{\mathfrak{q}\subsetneq \mathcal{O}_F})\zeta_{\mathfrak{q}}(2)\cdot D_F^{\frac{3}{2}}\cdot \prod_{v\mid\infty}\frac{(k_v-1)}{4\pi ^2}.
\end{equation}	
\end{lemma}
\begin{proof}
Utilizing Petersson formula we have 
\begin{align*}
&\frac{1}{\zeta_{\mathfrak{q}}(2)^2}\sum_{\substack{\pi\in \mathcal{F}(\mathbf{k},\mathfrak{q})}}\frac{1}{L^{(\mathfrak{q})}(1,\pi,\Ad)}+\sum_{\substack{\pi\in \mathcal{F}(\mathbf{k},\mathcal{O}_F)}}\frac{\textbf{1}_{\mathfrak{q}\subsetneq \mathcal{O}_F}}{L(1,\pi,\Ad)}\sim 2V_{\mathfrak{q}}\cdot D_F^{\frac{3}{2}}\cdot \prod_{v\mid\infty}\frac{(k_v-1)}{4\pi ^2}.
\end{align*}

Taking $\mathfrak{q}=\mathcal{O}_F$ in the above formula yields 
\begin{align*}
\sum_{\substack{\pi\in \mathcal{F}(\mathbf{k},\mathcal{O}_F)}}\frac{1}{L(1,\pi,\Ad)}\sim 2 D_F^{\frac{3}{2}}\cdot \prod_{v\mid\infty}\frac{(k_v-1)}{4\pi ^2}.
\end{align*}

Therefore, \eqref{10.1} follows from the fact that $L_v(s,\pi_v,\Ad)=\zeta_v(s+1)$ at $v\mid\mathfrak{q}$. 
\end{proof}

\begin{thm}\label{thm10.2}
Let notation be as before. Let $0<\varepsilon<10^{-3}$, $A>(\log\xi)^{3/2+\varepsilon}$, and $(\log N(\mathfrak{q})\|\mathbf{k}\|)^{\varepsilon}\leq \xi\leq N(\mathfrak{q})^{1/2-\varepsilon}\|\mathbf{k}\|^{1/4-\varepsilon}$. Then 
\begin{align*}
\sum_{\substack{\pi\in \mathcal{F}(\mathbf{k},\mathfrak{q})\\
L(1/2,\pi)> A^{-1}}}\frac{1}{L(1,\pi,\Ad)}\geq (1-\varepsilon)\Big[\frac{\log\xi\cdot \textbf{1}_{\mathfrak{q}=\mathcal{O}_F}}{\log \xi\|\mathbf{k}\|}+\mathcal{R}_{\mathfrak{q},\mathbf{k}}(\xi)\Big]\cdot \sum_{\substack{\pi\in \mathcal{F}(\mathbf{k},\mathfrak{q})}}\frac{1}{L(1,\pi,\Ad)},
\end{align*}	
where the function $\mathcal{R}_{\mathfrak{q},\mathbf{k}}(\xi)$, for $\mathfrak{q}\subsetneq \mathcal{O}_F$, is defined by  
\begin{align*}
\frac{(1-N(\mathfrak{q})^{-2})^3\log\xi\cdot \textbf{1}_{\mathfrak{q}\subsetneq\mathcal{O}_F}}{(1+N(\mathfrak{q})^{-1})^4\log N(\mathfrak{q})+2\cdot (1+10N(\mathfrak{q})^{-2}+4N(\mathfrak{q})^{-3}+N(\mathfrak{q})^{-4})\cdot\log \xi\|\mathbf{k}\|}.
\end{align*}
\end{thm}
\begin{proof}
Let $\rho(\mathfrak{n}):=\textbf{1}_{\mathfrak{n}=\mathcal{O}_F}+\textbf{1}_{\mathfrak{n}\subsetneq \mathcal{O}_F}\prod_{\substack{\mathfrak{p}\mid\mathfrak{n}\\ \text{$\mathfrak{p}$ prime}}}(1+N(\mathfrak{p})^{-1})^{-1}$ be defined as in  \eqref{rho}. Take $\xi=N(\mathfrak{q})^{1/2-\varepsilon}\|\mathbf{k}\|^{1/4-\varepsilon}$. By \eqref{eq7.65} in Theorem \ref{thmC} we obtain 
\begin{equation}\label{10.2}
\mathcal{S}_2:=\sum_{\substack{\pi\in \mathcal{F}(\mathbf{k},\mathfrak{q})}}\frac{L(1/2,\pi)^2M_{\xi,\rho}(\pi)^2}{L(1,\pi,\Ad)}\sim \frac{\zeta_F(2)^2D_F^{\frac{3}{2}}}{(\Res_{s=1}\zeta_F(s))^2}\prod_{v\mid\infty}\frac{k_v-1}{4\pi^2}\cdot \mathcal{M}_{\mathfrak{q},\mathbf{k}}^{(2)},	
\end{equation}
where 
\begin{align*}
\mathcal{M}_{\mathfrak{q},\mathbf{k}}^{(2)}:=&4c_{\mathfrak{q}}(N(\mathfrak{q})+1)\cdot\bigg[\frac{\log N(\mathfrak{q})^{1/2}\|\mathbf{k}\|}{\log\xi}+1\bigg]- \frac{16\zeta_{\mathfrak{q}}(2)\delta_{\mathbf{k}}\cdot\textbf{1}_{\mathfrak{q}\subsetneq \mathcal{O}_F}}{1+N(\mathfrak{q})^{-1}}\bigg[\frac{\log \|\mathbf{k}\|}{\log\xi}+1\bigg].
\end{align*}

By Theorem \ref{prop9.13} we obtain 
\begin{equation}\label{10.3}
\mathcal{S}_1:=\sum_{\substack{\pi\in \mathcal{F}(\mathbf{k},\mathfrak{q})}}\frac{L(1/2,\pi)M_{\xi,\rho}(\pi)}{L(1,\pi,\Ad)}\sim \frac{\zeta_F(2)D_F^{\frac{3}{2}}}{\Res_{s=1}\zeta_F(s)}\prod_{v\mid\infty}\frac{k_v-1}{4\pi ^2}\cdot \mathcal{M}_{\mathfrak{q},\mathbf{k}}^{(1)},
\end{equation}
where 
\begin{align*}
\mathcal{M}_{\mathfrak{q},\mathbf{k}}^{(1)}:= \zeta_{\mathfrak{q}}(2)\cdot(2\delta_{\mathbf{k},\mathfrak{q}}(N(\mathfrak{q})+1)-4\delta_{\mathbf{k}}\textbf{1}_{\mathfrak{q}\subsetneq \mathcal{O}_F})\cdot L_{\mathfrak{q}}(1,\rho),
\end{align*}
with $\delta_{\mathbf{k},\mathfrak{q}}:=\textbf{1}_{\mathfrak{q}\subsetneq \mathcal{O}_F}+\textbf{1}_{\mathfrak{q}=\mathcal{O}_F  \& \sum_{v\mid\infty}k_v\equiv 0\pmod{4}}$. Notice that $L_{\mathfrak{q}}(1,\rho)\equiv 1$ if $\mathfrak{q}=\mathcal{O}_F$, and $L_{\mathfrak{q}}(1,\rho)=(1-\rho(\mathfrak{q})N(\mathfrak{q})^{-1})^{-1}=1+N(\mathfrak{q})^{-1}$ if $\mathfrak{q}\subsetneq \mathcal{O}_F$.

Similar to the argument in \cite[Lemma 8.1]{BF21} we have 
\begin{equation}\label{10.4}
\sum_{\substack{\pi\in \mathcal{F}(\mathbf{k},\mathfrak{q})}}\frac{|M_{\xi,\rho}(\pi)|^2}{L^{(\mathfrak{q})}(1,\pi,\Ad)}\ll V_{\mathfrak{q}}\|\mathbf{k}\|D_F^{\frac{3}{2}}(\log \xi)^3,
\end{equation}
where the implied constant depends on $F$ and $\varepsilon$. By Cauchy inequality and \eqref{10.4}, 
\begin{align*}
\sum_{\substack{\pi\in \mathcal{F}(\mathbf{k},\mathfrak{q})\\ L(1/2,\pi)< A^{-1}}}\frac{L(1/2,\pi)M_{\xi,\rho}(\pi)}{L^{(\mathfrak{q})}(1,\pi,\Ad)}\ll \Bigg[\sum_{\substack{\pi\in \mathcal{F}(\mathbf{k},\mathfrak{q})\\ L(1/2,\pi)< A^{-1}}}\frac{|L(1/2,\pi)|^2}{L^{(\mathfrak{q})}(1,\pi,\Ad)}\Bigg]^{\frac{1}{2}}(V_{\mathfrak{q}}\|\mathbf{k}\|D_F^{\frac{3}{2}}(\log \xi)^3)^{\frac{1}{2}},
\end{align*}
from which we deduce that 
\begin{align*}
\sum_{\substack{\pi\in \mathcal{F}(\mathbf{k},\mathfrak{q})\\ L(1/2,\pi)< A^{-1}}}\frac{L(1/2,\pi)M_{\xi,\rho}(\pi)}{L^{(\mathfrak{q})}(1,\pi,\Ad)}\ll A^{-1}\Bigg[\sum_{\substack{\pi\in \mathcal{F}(\mathbf{k},\mathfrak{q})}}\frac{1}{L^{(\mathfrak{q})}(1,\pi,\Ad)}\Bigg]^{\frac{1}{2}}(V_{\mathfrak{q}}\|\mathbf{k}\|D_F^{\frac{3}{2}}(\log \xi)^3)^{\frac{1}{2}}.
\end{align*}

It then follows from Lemma \ref{lem10.1} that 
\begin{align*}
\mathcal{S}_1':=\sum_{\substack{\pi\in \mathcal{F}(\mathbf{k},\mathfrak{q})\\ L(1/2,\pi)< A^{-1}}}\frac{L(1/2,\pi)M_{\xi,\rho}(\pi)}{L^{(\mathfrak{q})}(1,\pi,\Ad)}\ll A^{-1}\cdot V_{\mathfrak{q}}\|\mathbf{k}\|D_F^{\frac{3}{2}}\cdot (\log \xi)^{\frac{3}{2}}.
\end{align*}

Taking advantage of Cauchy inequality we derive 
\begin{align*}
\sum_{\substack{\pi\in \mathcal{F}(\mathbf{k},\mathfrak{q})}}\frac{\textbf{1}_{L(1/2,\pi)> A^{-1}}}{L(1,\pi,\Ad)}\geq \frac{(\mathcal{S}_1-\mathcal{S}_1')^2}{\mathcal{S}_2}.
\end{align*}
In conjunction with \eqref{10.2}, \eqref{10.3}, the above inequality boils down to 
\begin{equation}\label{10.5}
\sum_{\substack{\pi\in \mathcal{F}(\mathbf{k},\mathfrak{q})}}\frac{\textbf{1}_{L(1/2,\pi)> A^{-1}}}{L(1,\pi,\Ad)}\geq  D_F^{\frac{3}{2}}\cdot \prod_{v\mid\infty}\frac{k_v-1}{4\pi ^2}\cdot \frac{(\mathcal{M}_{\mathfrak{q},\mathbf{k}}^{(1)})^2}{\mathcal{M}_{\mathfrak{q},\mathbf{k}}^{(2)}}\cdot (1+O(A^{-1}(\log \xi)^{\frac{3}{2}})).
\end{equation}

\begin{itemize}
\item Suppose $\mathfrak{q}=\mathcal{O}_F$ and $\sum_{v\mid\infty}k_v\equiv 0\pmod{4}$. Then 
\begin{align*}
\frac{\mathcal{S}_1^2}{\mathcal{S}_2}=\frac{\left(\frac{\zeta_F(2)D_F^{\frac{3}{2}}}{\Res_{s=1}\zeta_F(s)}\prod_{v\mid\infty}\frac{k_v-1}{4\pi ^2}\cdot \mathcal{M}_{\mathfrak{q},\mathbf{k}}^{(1)}\right)^2}{\frac{\zeta_F(2)^2D_F^{\frac{3}{2}}}{(\Res_{s=1}\zeta_F(s))^2}\prod_{v\mid\infty}\frac{k_v-1}{4\pi^2}\cdot \mathcal{M}_{\mathfrak{q},\mathbf{k}}^{(2)}}=D_F^{\frac{3}{2}}\cdot \prod_{v\mid\infty}\frac{k_v-1}{4\pi ^2}\cdot \frac{(\mathcal{M}_{\mathfrak{q},\mathbf{k}}^{(1)})^2}{\mathcal{M}_{\mathfrak{q},\mathbf{k}}^{(2)}}.
\end{align*} 

Since $\mathfrak{q}=\mathcal{O}_F$, then 
\begin{equation}\label{10.6}
\frac{(\mathcal{M}_{\mathfrak{q},\mathbf{k}}^{(1)})^2}{\mathcal{M}_{\mathfrak{q},\mathbf{k}}^{(2)}}=\frac{4(N(\mathfrak{q})+1)^2}{4c_{\mathfrak{q}}\cdot (N(\mathfrak{q})+1)\cdot\Big[\frac{\log N(\mathfrak{q})^{1/2}\|\mathbf{k}\|}{\log\xi}+1\Big]}=\frac{2}{\frac{\log \|\mathbf{k}\|}{\log\xi}+1}.
\end{equation}

Combining \eqref{10.5}, \eqref{10.6}, and Lemma \ref{lem10.1}, we derive that 
\begin{equation}\label{10.7}
\frac{\sum_{\substack{\pi\in \mathcal{F}(\mathbf{k},\mathfrak{q})}}L(1,\pi,\Ad)^{-1}\textbf{1}_{L(1/2,\pi)> A^{-1}}}{\sum_{\substack{\pi\in \mathcal{F}(\mathbf{k},\mathfrak{q})}}L(1,\pi,\Ad)^{-1}}\geq \frac{\log\xi}{\log \xi\|\mathbf{k}\|}\cdot (1+O(A^{-1}(\log \xi)^{\frac{1}{2}})), 
\end{equation}
where the implied constant depends on $\varepsilon$ and $F$. 

\item Suppose $\mathfrak{q}\subsetneq \mathcal{O}_F$. Denote by 
\begin{align*}
\mathcal{R}:=\frac{(\mathcal{M}_{\mathfrak{q},\mathbf{k}}^{(1)})^2}{2(V_{\mathfrak{q}}-\textbf{1}_{\mathfrak{q}\subsetneq \mathcal{O}_F})\zeta_{\mathfrak{q}}(2)\cdot \mathcal{M}_{\mathfrak{q},\mathbf{k}}^{(2)}}.
\end{align*}
By a straightforward calculation we obtain 
\begin{align*}
\mathcal{R}=\frac{1}{2\zeta_{\mathfrak{q}}(1)^3(1+N(\mathfrak{q})^{-1})\Big[\frac{\log N(\mathfrak{q})^{1/2}\|\mathbf{k}\|}{\log\xi}+1\Big]- \frac{8\zeta_{\mathfrak{q}}(2)^2N(\mathfrak{q})^{-1}}{1+N(\mathfrak{q})^{-1}}\Big[\frac{\log \|\mathbf{k}\|}{\log\xi}+1\Big]}.
\end{align*}
 
After a simplification, the above term reduces to 
\begin{align*}
\mathcal{R}=\frac{(1-N(\mathfrak{q})^{-1})^2}{\frac{1+N(\mathfrak{q})^{-1}}{1-N(\mathfrak{q})^{-1}}\cdot\frac{\log N(\mathfrak{q})}{\log\xi}+2\cdot \frac{1+10N(\mathfrak{q})^{-2}+4N(\mathfrak{q})^{-3}+N(\mathfrak{q})^{-4}}{(1-N(\mathfrak{q})^{-1})(1+N(\mathfrak{q})^{-1})^3}\cdot \Big[\frac{\log \|\mathbf{k}\|}{\log\xi}+1\Big]}.
\end{align*}

In conjunction with Lemma \ref{lem10.1}, we derive that 
\begin{align*}
\frac{\sum_{\substack{\pi\in \mathcal{F}(\mathbf{k},\mathfrak{q})}}L(1,\pi,\Ad)^{-1}\textbf{1}_{L(1/2,\pi)> A^{-1}}}{\sum_{\substack{\pi\in \mathcal{F}(\mathbf{k},\mathfrak{q})}}L(1,\pi,\Ad)^{-1}}\geq (1-\varepsilon)\mathcal{R},
\end{align*}
which yields Theorem \ref{thm10.2} in the case that $\mathfrak{q}\subsetneq \mathcal{O}_F$.  
\end{itemize}

Therefore, Theorem \ref{thm10.2} holds. 
\end{proof}

\begin{cor}\label{cor10.3}
Let notation be as before. Let $0<\varepsilon<10^{-3}$, and $A>(\log\xi)^{3/2+\varepsilon}$. As $N(\mathfrak{q})\|\mathbf{k}\|\to\infty$, we have  
\begin{align*}
\frac{\sum_{\substack{\pi\in \mathcal{F}(\mathbf{k},\mathfrak{q})}}L(1,\pi,\Ad)^{-1}\textbf{1}_{L(1/2,\pi)> A^{-1}}}{\sum_{\substack{\pi\in \mathcal{F}(\mathbf{k},\mathfrak{q})}}L(1,\pi,\Ad)^{-1}}\geq \begin{cases}
\frac{1}{5}-\varepsilon,\ \ \text{if $\mathfrak{q}=\mathcal{O}_F$}\\
\frac{1}{4}-\varepsilon,\ \ \text{if $\mathbf{k}$ is fixed},
\end{cases}
\end{align*}
where the implied constant depends only on $F$ and $\varepsilon$. Moreover, 
\begin{align*}
\frac{\sum_{\substack{\pi\in \mathcal{F}(\mathbf{k},\mathfrak{q})}}L(1,\pi,\Ad)^{-1}\textbf{1}_{L(1/2,\pi)> A^{-1}}}{\sum_{\substack{\pi\in \mathcal{F}(\mathbf{k},\mathfrak{q})}}L(1,\pi,\Ad)^{-1}}\geq\frac{(1-\varepsilon)(1-N(\mathfrak{q})^{-2})^3}{10 (1+10N(\mathfrak{q})^{-2}+4N(\mathfrak{q})^{-3}+N(\mathfrak{q})^{-4})}
\end{align*}
if $\mathfrak{q}\subsetneq \mathcal{O}_F$ is fixed and $\|\mathbf{k}\|\to\infty$. 
\end{cor}

\section{Uniform non-vanishing in Natural Average}\label{sec11}
In this section we aim to remove the harmonic weight $L(1,\pi,\Ad)$ in Theorem \ref{thm10.2}, proving uniform non-vanishing of central $L$-values in the natural density.

In this section, we will follow the method in \cite{KowalskiMichel1999} to remove the harmonic weight in previous sections.

\subsection{Some Preparations} For $\pi\in \cF(\mathbf{k},\fq)$, by \cite{KowalskiMichel1999}, we have the decomposition
\[L(1,\pi,\Ad)=\omega_{\pi}(x)+\omega_{\pi}(x,y)+O((\|\mathbf{k}\|N(\fq))^{-2+\epsilon}),\]
for $1\leq x<y$. Here:
\[\omega_{\pi}(x,y)=\sum_{\substack{x<N(\fl\ff^2)\leq y\\(\ff,\fq)=1}}\frac{\lambda_{\pi}(\fl^2)}{N(\fl\ff^2)}\]
and $\omega_{\pi}(x)=\omega_{\pi}(0,x),$ provided $y\geq (\|\mathbf{k}\|N(\fq)))^{10}.$

Let $\alpha_{\pi}$ be a complex number. Then we can write
\[\sum_{\pi\in\cF(\mathbf{k},\fq)}\alpha_{\pi}=\sum_{\pi\in\cF(\mathbf{k},\fq)}\frac{\alpha_{\pi}L(1,\pi,\Ad)}{L(1,\pi,\Ad)}=\sum_{\pi\in\cF(\mathbf{k},\fq)}\frac{\omega_{\pi}(x)\alpha_{\pi}}{L(1,\pi,\Ad)}+\sum_{\pi\in\cF(\mathbf{k},\fq)}\frac{\omega_{\pi}(x,y)\alpha_{\pi}}{L(1,\pi,\Ad)}+O(1).\]

By the method in \cite[Section~3]{KowalskiMichel1999}, if we can show
\begin{align}\label{conditions to estimate the middle sum}
\begin{split}
\sum_{\pi\in\cF(\mathbf{k},\fq)}\frac{|\alpha_{\pi}|}{L(1,\pi,\Ad)}&\ll (\|\mathbf{k}\|N(\fq))\log^A(\|\mathbf{k}\|N(\fq)) \\
\max_{\pi\in\cF(\mathbf{k},\fq)}\alpha_{\pi}&\ll (\|\mathbf{k}\|N(\fq))^{1-\delta}  \\
\end{split}
\end{align}
for arbitrary $A>0$ and some $\delta>0,$ then 
\begin{equation}\label{the estimate for the middle sum in adjoint}
\sum_{\pi\in\cF(\mathbf{k},\fq)}\frac{\omega_{\pi}(x,y)\alpha_{\pi}}{L(1,\pi,\Ad)}\ll (\|\mathbf{k}\|N(\fq))^{1-\delta'}.
\end{equation}
For sufficiently small $\epsilon$, we take $x=(\|\mathbf{k}\|N(\fq))^{\epsilon}.$ Then $\delta'$ is dependent on $\delta$ and $\epsilon$. Therefore, one has
\[\sum_{\pi\in\cF(\mathbf{k},\fq)}\alpha_{\pi}=\sum_{\pi\in\cF(\mathbf{k},\fq)}\frac{\omega_{\pi}(x)\alpha_{\pi}}{L(1,\pi,\Ad)}+O\left(\frac{\|\mathbf{k}\|N(\fq)}{\log(\|\mathbf{k}\|N(\fq))}\right).\]

\begin{remark}
    To prove \eqref{the estimate for the middle sum in adjoint}, the key ingredient is a large sieve inequality for the adjoint $L$-function. This is known for the $F=\bQ$ case (see \cite{Luo1999} and \cite{DukeKowalski2000}). For the general totally real number field case, the method still works or one can follow \cite{ThornerZaman2021} to consider the adjoint $L$-functions.
\end{remark}

We prove the following lemma, which is an analogue of Lemma \ref{lem10.1}.
\begin{lemma}\label{dimension formula}
    For $\mathbf{k}\geq\textbf{4}$ and $\fq$ a prime ideal, one has
    \[\#\cF(\mathbf{k},\fq)\sim 2N(\fq)D_F^{3/2}\zeta_F(2)\prod_{v|\infty}\frac{(k_v-1)}{4\pi^2}\]
\end{lemma}

\begin{proof}
    First, we have
    \begin{flalign*}
\#\cF(\mathbf{k},\fq)&=\sum_{\pi\in\cF(\mathbf{k},\fq)}1=\sum_{\pi\in\cF(\mathbf{k},\fq)}\frac{\omega_{\pi}(x)}{L(1,\pi,\Ad)}+O\left(\frac{\|\mathbf{k}\|N(\fq)}{\log(\|\mathbf{k}\|N(\fq))}\right)\\
&=\sum_{\substack{1\leq N(\fl\ff^2)\leq (\|\mathbf{k}\|N(\fq))^{\epsilon}\\(\ff,\fq)=1}}\frac{1}{N(\fl\ff^2)}\sum_{\pi\in\cF(\mathbf{k},\fq)}\frac{\lambda_{\pi}(\fl^2)}{L(1,\pi,\Ad)}+O\left(\frac{\|\mathbf{k}\|N(\fq)}{\log(\|\mathbf{k}\|N(\fq))}\right)
\end{flalign*}
    since $\alpha_{\pi}=1$ obviously satisfies the conditions \ref{conditions to estimate the middle sum}.
    
    Utilizing Petersson formula, we have, for $N(\fl)\ll (\|\mathbf{k}\|N(\fq))^{\epsilon}.$
    \[\frac{1}{\zeta_{\fq}(2)}\sum_{\pi\in\cF(\mathbf{k},\fq)}\frac{\lambda_{\pi}(\fl^2)}{L(1,\pi,\Ad)}+\sum_{\pi\in\cF(\mathbf{k},\cO_F)}\frac{\lambda_{\pi}(\fl^2)\bm{1}_{\fq\subsetneq\cO_F}}{L(1,\pi,\Ad)}\sim2V_{\fq}D_F^{3/2}\delta_{\fl^2=\cO_F}\prod_{v|\infty}\frac{(k_v-1)}{4\pi^2}.\]
(Notice that $L_v(1,\pi,\Ad)=\zeta_v(2)$ when $v|\fq.$) Then we sum over $1\leq N(\fl\ff^2)\ll (\|\mathbf{k}\|N(\fq))^{\epsilon}$ with $(\ff,\fq)=1$ and insert it into the expression for $\#\cF(\mathbf{k},\fq)$, one obtains:
\[\frac{1}{\zeta_{\fq}(2)}\sum_{\pi\in\cF(\mathbf{k},\fq)}\frac{\omega_{\pi}(x)}{L(1,\pi,\Ad)}+\bm{1}_{\fq\subsetneq\cO_F}\sum_{\pi\in\cF(\mathbf{k},\cO_F)}\frac{\omega_{\pi}(x)}{L(1,\pi,\Ad)}\sim2V_{\fq}D_F^{3/2}\zeta_F^{(\fq)}(2)\prod_{v|\infty}\frac{(k_v-1)}{4\pi^2}.\]

Next we investigate the sum over full level case:
\[\sum_{\pi\in\cF(\mathbf{k},\cO_F)}\frac{\omega_{\pi}(x)}{L(1,\pi,\Ad)}=\sum_{\pi\in\cF(\mathbf{k},\cO_F)}\frac{\omega_{\pi}(x')}{L(1,\pi,\Ad)}+\sum_{\pi\in\cF(\mathbf{k},\cO_F)}\frac{\omega_{\pi}(x',x)}{L(1,\pi,\Ad)}\]
where $x'=\|\mathbf{k}\|^{\epsilon}$. For $\omega(x',x)$ part, we again apply the method in \cite[Section 3]{KowalskiMichel1999} and this will only contribute the error term. Then a similar argument will show:
\[\sum_{\pi\in\cF(\mathbf{k},\cO_F)}\frac{\omega_{\pi}(x')}{L(1,\pi,\Ad)}\sim 2D_F^{3/2}\zeta_F^{(\fq)}(2)\prod_{v|\infty}\frac{(k_v-1)}{4\pi^2}.\]
Therefore, we showed:
\[\sum_{\pi\in\cF(\mathbf{k},\fq)}\frac{\omega_{\pi}(x)}{L(1,\pi,\Ad)}\sim2(V_{\fq}-\bm{1}_{\fq\subsetneq\cO_F})D_F^{3/2}\zeta_F(2)\prod_{v|\infty}\frac{(k_v-1)}{4\pi^2}=2N(\fq)D_F^{3/2}\zeta_F(2)\prod_{v|\infty}\frac{(k_v-1)}{4\pi^2}\]
since $V_{\fq}-\bm{1}_{\fq\subsetneq\cO_F}=N(\fq)$. This finishes the proof.
\end{proof}

Next, we remove the harmonic weight in the mollified first and second moment.

Taking $\alpha_{\pi}$ to be $M_{\xi,\rho}(\pi)^2L(1/2,\pi)^2$ or $M_{\xi,\rho}(\pi)L(1/2,\pi),$ we again write:
\[\sum_{\pi\in\cF(\mathbf{k},\fq)}\alpha_{\pi}=\sum_{\pi\in\cF(\mathbf{k},\fq)}\frac{\alpha_{\pi}L(1,\pi,\Ad)}{L(1,\pi,\Ad)}=\sum_{\pi\in\cF(\mathbf{k},\fq)}\frac{\omega_{\pi}(x)\alpha_{\pi}}{L(1,\pi,\Ad)}+\sum_{\pi\in\cF(\mathbf{k},\fq)}\frac{\omega_{\pi}(x,y)\alpha_{\pi}}{L(1,\pi,\Ad)}+O(1).\]

Then Theorem \ref{thm11.1} and the method in \cite[Section 3]{KowalskiMichel1999} will show:

\[\sum_{\pi\in\cF(\mathbf{k},\fq)}\alpha_{\pi}=\sum_{\pi\in\cF(\mathbf{k},\fq)}\frac{\omega_{\pi}(x)\alpha_{\pi}}{L(1,\pi,\Ad)}+O\left(\frac{\|\mathbf{k}\|N(\fq)}{\log(\|\mathbf{k}\|N(\fq))}\right).\]

In the left of this section, we will prove Theorem \ref{thm11.1} and then sketch the proof for the uniform non-vanishing result in the natural weight.

\subsection{Amplification of the Mollified Second Moment}\label{sec11.2}

Let $\rho$ be a multiplicative arithmetic function. Suppose $\rho(\mathfrak{p})\ll 1$ for all prime ideals $\mathfrak{p}$, with the implied constant being absolute. Let $\xi>1$ be a parameter to be determined. For $\pi\in \Pi_{\mathbf{k}}(\mathfrak{q})$, let $M_{\xi,\rho}(\pi)$ be the mollifier defined by \eqref{M}. 

A crucial condition to remove the harmonic weight $L(1,\pi,\Ad)$ (cf. \cite[\textsection 3.3]{KowalskiMichel1999}) is the bound
\begin{equation}\label{11.1}
M_{\xi,\rho}(\pi)L(1/2,\pi)\ll (N(\mathfrak{q})\|\mathbf{k}\|)^{\frac{1-\delta}{2}}
\end{equation}
for some $\delta>0$. For $\xi\leq N(\mathfrak{q})^{1/2-\varepsilon}\|\mathbf{k}\|^{1/4-\varepsilon}$, we have, by \eqref{c7.8}, that  
\begin{align*}
M_{\xi,\rho}(\pi)=\sum_{\substack{\mathfrak{n} \subseteq \mathcal{O}_F \\ (\mathfrak{n},\mathfrak{q}) = 1}} \frac{\lambda_\pi(\mathfrak{n})\mu_F(\mathfrak{n})\rho(\mathfrak{n})}{\sqrt{N(\mathfrak{n})}}\cdot \frac{(\log \xi N(\mathfrak{n})^{-1})^2}{2\log \xi} \cdot \textbf{1}_{N(\mathfrak{n})\leq \xi}\ll N(\mathfrak{q})^{\frac{1-\varepsilon}{4}}\|\mathbf{k}\|^{\frac{1-\varepsilon}{8}}.
\end{align*}

As a result, we obtain \eqref{11.1} if an explicit subconvexity bound is available:
\begin{equation}\label{11.2}
L(1/2,\pi)\ll N(\mathfrak{q})^{\frac{1}{4}+\eta}\|\mathbf{k}\|^{\frac{3}{8}+\eta}
\end{equation} 
for any tiny $\eta>0$. In particular, the estimate \eqref{11.2} with $\eta=\varepsilon/10$ implies \eqref{11.1} with $\delta=\varepsilon/10$. However, the bound in \eqref{11.2} seems beyond the reach of current results. 

In this section we will prove \eqref{11.1} utilizing the amplification method. The main result is the following.
\begin{thm}\label{thm11.1}
Let notation be as before. Let $\pi\in \Pi_{\mathbf{k}}(\mathfrak{q})$, and $M_{\xi,\rho}(\pi)$ be the mollifier defined by \eqref{M}. Let $0<\varepsilon<10^{-3}$ and  $1\leq \xi<N(\mathfrak{q})^{\frac{1}{2}}\|\mathbf{k}\|^{\frac{1}{4}}$. Then  
\begin{equation}\label{11.3}
M_{\xi,\rho}(\pi)L(1/2,\pi)\ll (\xi N(\mathfrak{q})\|\mathbf{k}\|)^{\varepsilon}\cdot \xi^{\frac{1}{3}}N(\mathfrak{q})^{\frac{1}{3}}\|\mathbf{k}\|^{\frac{5}{12}},
\end{equation}
where the implied constant depends only on $F$ and $\varepsilon$. In particular, for $0<\delta<1/4$, and $\xi\leq N(\mathfrak{q})^{1/2-\delta}\|\mathbf{k}\|^{1/4-\delta}$, we have 
\begin{equation}\label{11.4}
M_{\xi,\rho}(\pi)L(1/2,\pi)\ll (N(\mathfrak{q})\|\mathbf{k}\|)^{\frac{1}{2}-\frac{\delta}{3}+\varepsilon},	
\end{equation}
where the implied constant depends only on $F$ and $\varepsilon$.
\end{thm}
\begin{remark}
By taking $\xi=1$ in \eqref{11.3} we obtain the subconvexity $L(1/2,\pi)\ll ( N(\mathfrak{q})\|\mathbf{k}\|)^{\varepsilon}\cdot N(\mathfrak{q})^{\frac{1}{3}}\|\mathbf{k}\|^{\frac{5}{12}}$ in the weight aspect. 
\end{remark}

\subsubsection{Weighted mollified second moment}\label{sec11.1.1}
Let $\mathfrak{l}\subseteq\mathcal{O}_F$ be an integral ideal such that $\mathfrak{q}\nmid\mathfrak{l}$. Define 
\begin{align*}
\mathcal{J}_{\xi,\rho}^{\mathrm{new}}(\mathfrak{l}):=&\frac{1}{\zeta_{\mathfrak{q}}(2)}\sum_{\substack{\pi\in \mathcal{F}(\mathbf{k},\mathfrak{q})}}\frac{\lambda_\pi(\mathfrak{l})L(1/2,\pi)^2M_{\xi,\rho}(\pi)^2}{L(1,\pi,\Ad)},\\
\mathcal{J}_{\xi,\rho}^{\mathrm{old}}(\mathfrak{l}):=&\frac{2V_{\mathfrak{q}}\textbf{1}_{\mathfrak{q}\subsetneq \mathcal{O}_F}}{N(\mathfrak{q})}\sum_{\substack{\pi\in \mathcal{F}(\mathbf{k},\mathcal{O}_F)}}\frac{\lambda_{\pi}(\mathfrak{l})L_{\pi_{\mathfrak{q}}}L(1/2,\pi)^2M_{\xi,\rho}(\pi)^2}{L(1,\pi,\Ad)}.
\end{align*}

\begin{prop}\label{prop11.2}
Let notation be as before. Then 
\begin{equation}\label{11.11}
\big|\mathcal{J}_{\xi,\rho}^{\mathrm{new}}(\mathfrak{l})+\mathcal{J}_{\xi,\rho}^{\mathrm{old}}(\mathfrak{l})\big|\ll \xi^{10\varepsilon}(N(\mathfrak{q}\mathfrak{l})\|\mathbf{k}\|)^{\varepsilon}\cdot \Big[N(\mathfrak{q})\|\mathbf{k}\|N(\mathfrak{l})^{-\frac{1}{2}}+\xi^{2}\|\mathbf{k}\|^{\frac{1}{2}}N(\mathfrak{l})^{\frac{1}{2}}\Big],
\end{equation}
where the implied constant depending on $F$ and $\varepsilon$. 
\end{prop}
\begin{proof}
By \eqref{c7.8} and the Hecke relation
\begin{align*}
\lambda_\pi(\mathfrak{l})\lambda_\pi(\mathfrak{n}_1\mathfrak{n}_2\mathfrak{m}^{-2})=\sum_{\substack{\mathfrak{c}\mid\gcd(\mathfrak{n}_1\mathfrak{n}_2\mathfrak{m}^{-2},\mathfrak{l})}}\lambda_\pi(\mathfrak{l}\mathfrak{n}_1\mathfrak{n}_2\mathfrak{m}^{-2}\mathfrak{c}^{-2}),
\end{align*}
we can rewrite $\mathcal{J}_{\xi,\rho}^{\mathrm{new}}(\mathfrak{l})+\mathcal{J}_{\xi,\rho}^{\mathrm{old}}(\mathfrak{l})$ as 
\begin{align*}
\mathcal{J}_{\xi,\rho}^{\mathrm{new}}(\mathfrak{l})+\mathcal{J}_{\xi,\rho}^{\mathrm{old}}(\mathfrak{l})=&\sum_{\substack{\mathfrak{n}_1 \subseteq \mathcal{O}_F \\ (\mathfrak{n}_1,\mathfrak{q}) = 1\\ N(\mathfrak{n}_1)\leq \xi}}\sum_{\substack{\mathfrak{n} _2\subseteq \mathcal{O}_F \\ (\mathfrak{n}_2,\mathfrak{q}) = 1\\ N(\mathfrak{n}_2)\leq \xi}}  \frac{\mu_F(\mathfrak{n}_1)\rho(\mathfrak{n}_1)\mu_F(\mathfrak{n}_2)\rho(\mathfrak{n}_2)}{\sqrt{N(\mathfrak{n}_1)N(\mathfrak{n}_2)}}\cdot \frac{(\log \xi N(\mathfrak{n}_1)^{-1})^2}{4(\log \xi)^2}\\
&\quad\cdot (\log \xi N(\mathfrak{n}_2)^{-1})^2\sum_{\substack{\mathfrak{m}\mid\gcd(\mathfrak{n}_1,\mathfrak{n}_2)\\ \mathfrak{c}\mid\gcd(\mathfrak{n}_1\mathfrak{n}_2\mathfrak{m}^{-2},\mathfrak{l})}}J_{\mathrm{Spec}}(f_{\mathfrak{l}\mathfrak{n}_1\mathfrak{n}_2\mathfrak{m}^{-2}\mathfrak{c}^{-2},\mathfrak{q}},\textbf{s}).
\end{align*}

Let $\mathfrak{n}=\mathfrak{l}\mathfrak{n}_1\mathfrak{n}_2\mathfrak{m}^{-2}\mathfrak{c}^{-2}$. Taking advantage of Corollary \ref{cor7.2}, i.e.,  
\begin{align*}
&\frac{1}{\zeta_{\mathfrak{q}}(2)}\sum_{\substack{\pi\in \mathcal{F}(\mathbf{k},\mathfrak{q})}}\frac{\lambda_{\pi}(\mathfrak{n})L(1/2,\pi)^2}{L(1,\pi,\Ad)}+\frac{2V_{\mathfrak{q}}\textbf{1}_{\mathfrak{q}\subsetneq \mathcal{O}_F}}{N(\mathfrak{q})}\sum_{\substack{\pi\in \mathcal{F}(\mathbf{k},\mathcal{O}_F)}}\frac{\lambda_{\pi}(\mathfrak{n})L_{\pi_{\mathfrak{q}}}L(1/2,\pi)^2}{L(1,\pi,\Ad)}\\
=&\frac{(N(\mathfrak{q})+1)\cdot \delta_{\mathbf{k},\mathfrak{q}}}{2\pi i}\oint_{\mathcal{C}_{\varepsilon}}\frac{\zeta_F(1+s)G_{\mathfrak{n},\mathfrak{q}}(s)}{s}ds+O(N(\mathfrak{n})^{1/2+\varepsilon}N(\mathfrak{q})^{\varepsilon}\|\mathbf{k}\|^{1/2+\varepsilon}),
\end{align*}
we derive that 
\begin{equation}\label{eq11.4}
\big|\mathcal{J}_{\xi,\rho}^{\mathrm{new}}(\mathfrak{l})+\mathcal{J}_{\xi,\rho}^{\mathrm{old}}(\mathfrak{l})-\mathcal{J}_{\xi,\rho}^{\mathrm{main}}(\mathfrak{l})\big|\ll \mathcal{J}_{\xi,\rho}^{\mathrm{tail}}(\mathfrak{l}),
\end{equation}
where the implied constant depends on $\varepsilon$ and $F$. Here $\mathcal{J}_{\xi,\rho}^{\mathrm{main}}(\mathfrak{l})$ is defined by 
\begin{align*}
&-\frac{1}{4\pi^2(\log \xi)^2}\sum_{\substack{\mathfrak{n}_1 \subset \mathcal{O}_F \\ (\mathfrak{n}_1,\mathfrak{q}) = 1}} \sum_{\substack{\mathfrak{n}_2 \subset \mathcal{O}_F \\ (\mathfrak{n}_2,\mathfrak{q}) = 1}} \frac{\mu_F(\mathfrak{n}_1)\mu_F(\mathfrak{n}_2)\rho(\mathfrak{n}_1)\rho(\mathfrak{n}_2)}{N(\mathfrak{n}_1)^{1/2}N(\mathfrak{n}_2)^{1/2}}\int_{(2)} \frac{\xi^{s_1}}{N(\mathfrak{n}_1)^{s_1}} \frac{ds_1}{s_1^3}\\
&\int_{(2)} \frac{\xi^{s_2}ds_2}{N(\mathfrak{n}_2)^{s_2}s_2^3}\sum_{\substack{\mathfrak{m}\mid\gcd(\mathfrak{n}_1,\mathfrak{n}_2)\\ \mathfrak{c}\mid\gcd(\mathfrak{n}_1\mathfrak{n}_2\mathfrak{m}^{-2},\mathfrak{l})}}\frac{(N(\mathfrak{q})+1) \delta_{\mathbf{k},\mathfrak{q}}}{2\pi i}\oint_{\mathcal{C}_{\varepsilon}}\frac{\zeta_F(1+s)G_{\mathfrak{l}\mathfrak{n}_1\mathfrak{n}_2\mathfrak{m}^{-2}\mathfrak{c}^{-2},\mathfrak{q}}(s)}{s}ds,
\end{align*}
and $\mathcal{J}_{\xi,\rho}^{\mathrm{tail}}(\mathfrak{l})$ is defined by 
\begin{align*}
\sum_{\substack{\mathfrak{n}_1 \subseteq \mathcal{O}_F \\ (\mathfrak{n}_1,\mathfrak{q}) = 1\\ N(\mathfrak{n}_1)\leq \xi}}\sum_{\substack{\mathfrak{n} _2\subseteq \mathcal{O}_F \\ (\mathfrak{n}_2,\mathfrak{q}) = 1\\ N(\mathfrak{n}_2)\leq \xi}}  \frac{1}{\sqrt{N(\mathfrak{n}_1)N(\mathfrak{n}_2)}}\sum_{\substack{\mathfrak{m}\mid\gcd(\mathfrak{n}_1,\mathfrak{n}_2)\\ \mathfrak{c}\mid\gcd(\mathfrak{n}_1\mathfrak{n}_2\mathfrak{m}^{-2},\mathfrak{l})}}N(\mathfrak{l}\mathfrak{n}_1\mathfrak{n}_2\mathfrak{m}^{-2}\mathfrak{c}^{-2})^{\frac{1}{2}+\varepsilon}N(\mathfrak{q})^{\varepsilon}\|\mathbf{k}\|^{\frac{1}{2}+\varepsilon},
\end{align*}
where the implied constant depends on $\varepsilon$ and $F$. By a direct estimate, $\mathcal{J}_{\xi,\rho}^{\mathrm{tail}}(\mathfrak{l})$ is 
\begin{align*}
\ll \xi^{2\varepsilon}N(\mathfrak{q})^{\varepsilon}(\|\mathbf{k}\|N(\mathfrak{l}))^{\frac{1}{2}+\varepsilon}\sum_{\substack{\mathfrak{n}_1 \subseteq \mathcal{O}_F \\ (\mathfrak{n}_1,\mathfrak{q}) = 1\\ N(\mathfrak{n}_1)\leq \xi}}\sum_{\substack{\mathfrak{n} _2\subseteq \mathcal{O}_F \\ (\mathfrak{n}_2,\mathfrak{q}) = 1\\ N(\mathfrak{n}_2)\leq \xi}}\sum_{\substack{\mathfrak{m}\mid\gcd(\mathfrak{n}_1,\mathfrak{n}_2)\\ \mathfrak{c}\mid\gcd(\mathfrak{n}_1\mathfrak{n}_2\mathfrak{m}^{-2},\mathfrak{l})}}\frac{1}{N(\mathfrak{m})^{1+2\varepsilon}N(\mathfrak{c})^{1+2\varepsilon}},
\end{align*}
where the implied constant depends on $F$ and $\varepsilon$. As a consequence, 
\begin{equation}\label{eq11.5}
\mathcal{J}_{\xi,\rho}^{\mathrm{tail}}(\mathfrak{l})\ll \xi^{2+2\varepsilon}N(\mathfrak{q})^{\varepsilon}\|\mathbf{k}\|^{\frac{1}{2}+\varepsilon}N(\mathfrak{l})^{\frac{1}{2}+\varepsilon}.
\end{equation}

According to the definition \eqref{G} we have
\begin{align*}
G_{\mathfrak{l}\mathfrak{n}_1\mathfrak{n}_2\mathfrak{m}^{-2}\mathfrak{c}^{-2},\mathfrak{q}}(s)=\frac{\tau(\mathfrak{l}\mathfrak{n}_1\mathfrak{n}_2\mathfrak{m}^{-2}\mathfrak{c}^{-2})}{N(\mathfrak{l}\mathfrak{n}_1\mathfrak{n}_2\mathfrak{m}^{-2}\mathfrak{c}^{-2})^{(1+s)/2}}\cdot G_{\mathcal{O}_F,\mathfrak{q}}(s).
\end{align*}

Substituting this into the definition of $\mathcal{J}_{\xi,\rho}^{\mathrm{main}}(\mathfrak{l})$, we can rewrite $\mathcal{J}_{\xi,\rho}^{\mathrm{main}}(\mathfrak{l})$ as 
\begin{align*}
-\frac{(N(\mathfrak{q})+1)\delta_{\mathbf{k},\mathfrak{q}}}{4\pi^2(\log \xi)^2}\int_{(2\varepsilon)} \frac{\xi^{s_1}}{s_1^3} \int_{(2\varepsilon)} \frac{\xi^{s_1}}{s_2^3} \frac{1}{2\pi i}\oint_{\mathcal{C}_{\varepsilon}}\frac{G_{\mathcal{O}_F,\mathfrak{q}}(s)L(s,s_1,s_2;\rho,\mathfrak{q},\mathfrak{l})}{s\zeta_F(1+s)^{-1}}dsds_1ds_2,
\end{align*}
where $L(s,s_1,s_2;\rho,\mathfrak{q},\mathfrak{l})$ is defined by 
\begin{align*}
\sum_{\substack{\mathfrak{n}_1 \subset \mathcal{O}_F \\ (\mathfrak{n}_1,\mathfrak{q}) = 1}} \sum_{\substack{\mathfrak{n}_2 \subset \mathcal{O}_F \\ (\mathfrak{n}_2,\mathfrak{q}) = 1}} \frac{\mu_F(\mathfrak{n}_1)\mu_F(\mathfrak{n}_2)\rho(\mathfrak{n}_1)\rho(\mathfrak{n}_2)}{N(\mathfrak{n}_1)^{1/2+s_1}N(\mathfrak{n}_2)^{1/2+s_2}}\sum_{\substack{\mathfrak{m}\mid\gcd(\mathfrak{n}_1,\mathfrak{n}_2)\\ \mathfrak{c}\mid\gcd(\mathfrak{n}_1\mathfrak{n}_2\mathfrak{m}^{-2},\mathfrak{l})}}\frac{\tau(\mathfrak{l}\mathfrak{n}_1\mathfrak{n}_2\mathfrak{m}^{-2}\mathfrak{c}^{-2})}{N(\mathfrak{l}\mathfrak{n}_1\mathfrak{n}_2\mathfrak{m}^{-2}\mathfrak{c}^{-2})^{(1+s)/2}}.
\end{align*}

Applying the triangle inequality we obtain 
\begin{align*}
L(s,s_1,s_2;\rho,\mathfrak{q},\mathfrak{l})\ll \sum_{\substack{\mathfrak{n}_1,\mathfrak{n}_2 \subset \mathcal{O}_F \\ (\mathfrak{n}_1\mathfrak{n}_2,\mathfrak{q}) = 1}}  \frac{1}{N(\mathfrak{n}_1\mathfrak{n}_2)^{\frac{1}{2}+2\varepsilon}}\sum_{\substack{\mathfrak{m}\mid\gcd(\mathfrak{n}_1,\mathfrak{n}_2)\\ \mathfrak{c}\mid\gcd(\mathfrak{n}_1\mathfrak{n}_2\mathfrak{m}^{-2},\mathfrak{l})}}\frac{1}{N(\mathfrak{l}\mathfrak{n}_1\mathfrak{n}_2\mathfrak{m}^{-2}\mathfrak{c}^{-2})^{\frac{1}{2}-\varepsilon}},
\end{align*}
where the right hand side boils down to 
\begin{align*}
\sum_{\substack{\mathfrak{a} \subset \mathcal{O}_F \\ (\mathfrak{a},\mathfrak{q}) = 1}}\sum_{\substack{\mathfrak{n}_1,\mathfrak{n}_2 \subset \mathcal{O}_F \\ (\mathfrak{n}_1\mathfrak{n}_2,\mathfrak{q}) = 1\\
(\mathfrak{n}_1,\mathfrak{n}_2) = 1}} \frac{1}{N(\mathfrak{a}^2\mathfrak{n}_1\mathfrak{n}_2)^{1/2+2\varepsilon}}\sum_{\substack{\mathfrak{m}\mid\mathfrak{a}\\ \mathfrak{c}\mid\gcd(\mathfrak{n}_1\mathfrak{n}_2\mathfrak{a}^2\mathfrak{m}^{-2},\mathfrak{l})}}\frac{1}{N(\mathfrak{l}\mathfrak{n}_1\mathfrak{n}_2\mathfrak{a}^2\mathfrak{m}^{-2}\mathfrak{c}^{-2})^{1/2-\varepsilon}}. 
\end{align*}

By changing the variable $\mathfrak{m}\mapsto \mathfrak{a}\mathfrak{m}^{-1}$ and $\mathfrak{c}\mapsto \gcd(\mathfrak{l},\mathfrak{n}_1\mathfrak{n}_2\mathfrak{m}^{2})$, we can thus bound $L(s,s_1,s_2;\rho,\mathfrak{q},\mathfrak{l})$ by 
\begin{align*}
\sum_{\substack{\mathfrak{a} \subset \mathcal{O}_F \\ (\mathfrak{a},\mathfrak{q}) = 1}}\sum_{\substack{\mathfrak{n}_1,\mathfrak{n}_2 \subset \mathcal{O}_F \\ (\mathfrak{n}_1\mathfrak{n}_2,\mathfrak{q}) = 1\\
(\mathfrak{n}_1,\mathfrak{n}_2) = 1}} \frac{1}{N(\mathfrak{a}^2\mathfrak{n}_1\mathfrak{n}_2)^{\frac{1}{2}+2\varepsilon}}\sum_{\substack{\mathfrak{m}\mid\mathfrak{a}}}\frac{N(\gcd(\mathfrak{l},\mathfrak{n}_1\mathfrak{n}_2\mathfrak{m}^{2}))^{1-\varepsilon}}{N(\mathfrak{l}\mathfrak{n}_1\mathfrak{n}_2\mathfrak{m}^{2})^{1/2-\varepsilon}}\sum_{\mathfrak{c}\mid \gcd(\mathfrak{l},\mathfrak{n}_1\mathfrak{n}_2\mathfrak{m}^{2})}\frac{1}{N(\mathfrak{c})^{1-2\varepsilon}}.
\end{align*}

Notice that $\sum_{\mathfrak{c}\mid \gcd(\mathfrak{l},\mathfrak{n}_1\mathfrak{n}_2\mathfrak{m}^{2})}N(\mathfrak{c})^{-1+2\varepsilon}\ll \gcd(\mathfrak{l},\mathfrak{n}_1\mathfrak{n}_2\mathfrak{m}^{2})^{\varepsilon}$. Hence 
\begin{equation}\label{eq11.6}
L(s,s_1,s_2;\rho,\mathfrak{q},\mathfrak{l})\ll \mathcal{I}(\mathfrak{l}),	
\end{equation}
where 
\begin{align*}
\mathcal{I}(\mathfrak{l}):=\sum_{\substack{\mathfrak{a} \subset \mathcal{O}_F \\ (\mathfrak{a},\mathfrak{q}) = 1}}\sum_{\substack{\mathfrak{n}_1,\mathfrak{n}_2 \subset \mathcal{O}_F \\ (\mathfrak{n}_1\mathfrak{n}_2,\mathfrak{q}) = 1\\
(\mathfrak{n}_1,\mathfrak{n}_2) = 1}}\frac{1}{N(\mathfrak{a}^2\mathfrak{n}_1\mathfrak{n}_2)^{1/2+2\varepsilon}}\sum_{\substack{\mathfrak{m}\mid\mathfrak{a}}}\frac{N(\gcd(\mathfrak{l},\mathfrak{n}_1\mathfrak{n}_2\mathfrak{m}^{2}))^{1-\varepsilon}}{N(\mathfrak{l}\mathfrak{n}_1\mathfrak{n}_2\mathfrak{m}^{2})^{1/2-\varepsilon}}.
\end{align*}

Writing $\mathfrak{a}=\mathfrak{m}\mathfrak{b}$ we obtain 
\begin{align*}
\mathcal{I}(\mathfrak{l})=\sum_{\substack{\mathfrak{b} \subset \mathcal{O}_F \\ (\mathfrak{b},\mathfrak{q}) = 1}}\sum_{\substack{\mathfrak{m} \subset \mathcal{O}_F \\ (\mathfrak{m},\mathfrak{q}) = 1}}\sum_{\substack{\mathfrak{n}_1,\mathfrak{n}_2 \subset \mathcal{O}_F \\ (\mathfrak{n}_1\mathfrak{n}_2,\mathfrak{q}) = 1\\
(\mathfrak{n}_1,\mathfrak{n}_2) = 1}}\frac{1}{N(\mathfrak{b}^2\mathfrak{m}^2\mathfrak{n}_1\mathfrak{n}_2)^{1/2+2\varepsilon}}\cdot\frac{N(\gcd(\mathfrak{l},\mathfrak{n}_1\mathfrak{n}_2\mathfrak{m}^{2}))^{1-\varepsilon}}{N(\mathfrak{l}\mathfrak{n}_1\mathfrak{n}_2\mathfrak{m}^{2})^{1/2-\varepsilon}}.
\end{align*}

Summing over $\mathfrak{b}$, and write $\mathfrak{m}_1=\mathfrak{m}\mathfrak{n}_1$, and $\mathfrak{m}_2=\mathfrak{m}\mathfrak{n}_2$, we derive  
\begin{align*}
\mathcal{I}(\mathfrak{l})\ll \frac{\zeta_F(1+4\varepsilon)}{N(\mathfrak{l})^{1/2-\varepsilon}}\sum_{\substack{\mathfrak{m}_1,\mathfrak{m}_2 \subset \mathcal{O}_F \\ (\mathfrak{m}_1\mathfrak{m}_2,\mathfrak{q}) = 1}}\frac{N(\gcd(\mathfrak{l},\mathfrak{m}_1\mathfrak{m}_2))^{1-\varepsilon}}{N(\mathfrak{m}_1\mathfrak{m}_2)^{1+\varepsilon}}.
\end{align*}

As a consequence, we obtain 
\begin{align*}
\mathcal{I}(\mathfrak{l})\ll\frac{\zeta_F(1+4\varepsilon)}{N(\mathfrak{l})^{1/2-\varepsilon}}\sum_{\mathfrak{b}\mid\mathfrak{l}}\sum_{\substack{\mathfrak{n} \subset \mathcal{O}_F \\ 
(\mathfrak{n},\mathfrak{q}) = 1\\
(\mathfrak{n},\mathfrak{b}^{-1}\mathfrak{l}) = 1}}\frac{N(\mathfrak{b})^{1-\varepsilon}}{N(\mathfrak{n}\mathfrak{b})^{1+\varepsilon}} \sum_{\substack{\mathfrak{m}_1,\mathfrak{m}_2 \subset \mathcal{O}_F \\ \mathfrak{n}=\mathfrak{m}_1\mathfrak{m}_2\\
(\mathfrak{m}_1\mathfrak{m}_2,\mathfrak{q}) = 1}}1\ll \frac{\zeta_F(1+\varepsilon)^3}{N(\mathfrak{l})^{1/2-\varepsilon}}\sum_{\mathfrak{b}\mid\mathfrak{l}}1.
\end{align*}

In conjunction with 
\eqref{eq11.6}, we conclude that  
\begin{equation}\label{11.9}
L(s,s_1,s_2;\rho,\mathfrak{q},\mathfrak{l})\ll N(\mathfrak{l})^{-\frac{1}{2}+2\varepsilon},
\end{equation}
where the implied constant depends on $F$ and $\varepsilon$.

For $s\in \mathcal{C}_{\varepsilon}$, we have $|\zeta_F(1+s)|\ll \varepsilon^{-1}$, with the implied constant depending on $F$. In conjunction with \eqref{11.9} we derive 
\begin{equation}\label{11.10}
\mathcal{J}_{\xi,\rho}^{\mathrm{main}}(\mathfrak{l})\ll \xi^{10\varepsilon}N(\mathfrak{q})^{1+\varepsilon}\|\mathbf{k}\|^{1+\varepsilon}N(\mathfrak{l})^{-\frac{1}{2}+\varepsilon}.
\end{equation}

Therefore, \eqref{11.11} follows from substituting \eqref{eq11.5} and \eqref{11.10} into \eqref{eq11.4}.
\end{proof}

\subsubsection{Proof of Theorem \ref{thm11.1}}
Let $\pi\in \Pi_{\mathbf{k}}(\mathfrak{q})$. For a prime $\mathfrak{p}\neq \mathfrak{q}$, by Hecke relation $\lambda_{\pi}(\mathfrak{l})^2=\lambda_{\pi}(\mathfrak{l}^2)+1$, we have
\begin{align*}
\min\{|\lambda_{\pi}(\mathfrak{l})|,|\lambda_{\pi}(\mathfrak{l}^2)|\}\geq 1/2.
\end{align*}
Hence, there exists some $r_{\mathfrak{p}}\in \{1,2\}$ such that $N(\mathfrak{p}^{r_{\mathfrak{p}}})\geq 1/2$. 

Let $L>1$ and $\mathcal{L}:=\{\mathfrak{l}=\mathfrak{p}^{r_{\mathfrak{p}}}:\ L<N(\mathfrak{p})\leq 2L,\ \ \mathfrak{p}\neq\mathfrak{q},\ N(\mathfrak{p}^{r_{\mathfrak{p}}})\geq 1/2,\ r_{\mathfrak{p}}\in \{1,2\}\}$. For $\mathfrak{l}\in \mathcal{L}$, let $\alpha_{\mathfrak{l}}:=\lambda_{\pi}(\mathfrak{l})/|\lambda_{\pi}(\mathfrak{l})|$. Consider 
\begin{align*}
\mathcal{J}_{\xi,\rho}^{\mathrm{new}}:=&\frac{1}{\zeta_{\mathfrak{q}}(2)}\sum_{\substack{\sigma\in \mathcal{F}(\mathbf{k},\mathfrak{q})}}\frac{L(1/2,\sigma)^2M_{\xi,\rho}(\sigma)^2}{L(1,\sigma,\Ad)}\cdot \Big|\sum_{\mathfrak{l}\in \mathcal{L}}\alpha_{\mathfrak{l}}\lambda_{\sigma}(\mathfrak{l})\Big|^2,\\
\mathcal{J}_{\xi,\rho}^{\mathrm{old}}:=&\frac{2V_{\mathfrak{q}}\textbf{1}_{\mathfrak{q}\subsetneq \mathcal{O}_F}}{N(\mathfrak{q})}\sum_{\substack{\sigma\in \mathcal{F}(\mathbf{k},\mathcal{O}_F)}}\frac{L_{\sigma_{\mathfrak{q}}}L(1/2,\sigma)^2M_{\xi,\rho}(\sigma)^2}{L(1,\sigma,\Ad)}\cdot \Big|\sum_{\mathfrak{l}\in \mathcal{L}}\alpha_{\mathfrak{l}}\lambda_{\sigma}(\mathfrak{l})\Big|^2.
\end{align*}

By definition, for $\sigma=\pi$, we have
\begin{align*}
\Big|\sum_{\mathfrak{l}\in \mathcal{L}}\alpha_{\mathfrak{l}}\lambda_{\sigma}(\mathfrak{l})\Big|^2=\Big|\sum_{\mathfrak{l}\in \mathcal{L}}|\lambda_{\pi}(\mathfrak{l})|\Big|^2\gg (\#\mathcal{L})^2\gg \frac{L^2}{(\log L)^2}.
\end{align*}
Therefore, dropping all but $\sigma=\pi$ yields 
\begin{equation}\label{eq11.10}
\mathcal{J}_{\xi,\rho}^{\mathrm{new}}+\mathcal{J}_{\xi,\rho}^{\mathrm{old}}\gg \frac{L(1/2,\pi)^2M_{\xi,\rho}(\pi)^2}{L(1,\pi,\Ad)}\cdot \frac{L^2}{(\log L)^2},
\end{equation}
where the implied constant depends on $F$ and $\varepsilon$. 

On the other hand, squaring the sum over $\mathfrak{l}=\mathfrak{p}^{r_{\mathfrak{p}}}\in \mathcal{L}$ yields
\begin{align*}
\mathcal{J}_{\xi,\rho}^{\mathrm{new}}+\mathcal{J}_{\xi,\rho}^{\mathrm{old}}=&\sum_{\substack{\mathfrak{l}_1,\mathfrak{l}_2\in\mathcal{L}\\
\gcd(\mathfrak{l}_1,\mathfrak{l}_2)=\mathcal{O}_F}}\alpha_{\mathfrak{l}_1}\overline{\alpha_{\mathfrak{l}_2}}\cdot (\mathcal{J}_{\xi,\rho}^{\mathrm{new}}(\mathfrak{l}_1\mathfrak{l}_2)+\mathcal{J}_{\xi,\rho}^{\mathrm{old}}(\mathfrak{l}_1\mathfrak{l}_2))\\
&+\sum_{\substack{\mathfrak{l}\in\mathcal{L}}}|\alpha_{\mathfrak{l}}|^2\cdot (\mathcal{J}_{\xi,\rho}^{\mathrm{new}}(\mathfrak{l}^2)+\mathcal{J}_{\xi,\rho}^{\mathrm{old}}(\mathfrak{l}^2)+\mathcal{J}_{\xi,\rho}^{\mathrm{new}}(\mathcal{O}_F)+\mathcal{J}_{\xi,\rho}^{\mathrm{old}}(\mathcal{O}_F)).
\end{align*}

Taking advantage of Proposition \ref{prop11.2} and $|\alpha_{\mathfrak{l}}|\leq 1$, we obtain 
\begin{align*}
\mathcal{J}_{\xi,\rho}^{\mathrm{new}}+\mathcal{J}_{\xi,\rho}^{\mathrm{old}}\ll & \xi^{\varepsilon}(N(\mathfrak{q})\|\mathbf{k}\|L)^{\varepsilon}\sum_{\substack{\mathfrak{l}_1,\mathfrak{l}_2\in\mathcal{L}\\
\gcd(\mathfrak{l}_1,\mathfrak{l}_2)=\mathcal{O}_F}}\Big[\frac{N(\mathfrak{q})\|\mathbf{k}\|}{N(\mathfrak{l}_1\mathfrak{l}_2)^{\frac{1}{2}}}+\xi^{2}\|\mathbf{k}\|^{\frac{1}{2}}N(\mathfrak{l}_1\mathfrak{l}_2)^{\frac{1}{2}}\Big]\\
&+\xi^{\varepsilon}(N(\mathfrak{q})\|\mathbf{k}\|L)^{\varepsilon}\cdot \sum_{\substack{\mathfrak{l}\in\mathcal{L}}}\Big[N(\mathfrak{q})\|\mathbf{k}\|N(\mathfrak{l}^2)^{-\frac{1}{2}}+\xi^{2}\|\mathbf{k}\|^{\frac{1}{2}}N(\mathfrak{l}^2)^{\frac{1}{2}}\Big]\\
&+\xi^{\varepsilon}(N(\mathfrak{q})\|\mathbf{k}\|L)^{\varepsilon}\cdot \sum_{\substack{\mathfrak{l}\in\mathcal{L}}}\Big[N(\mathfrak{q})\|\mathbf{k}\|+\xi^{2}\|\mathbf{k}\|^{\frac{1}{2}}\Big],
\end{align*}
where the implied constant depends on $F$ and $\varepsilon$. Consequently,
\begin{equation}\label{eq11.11}
\mathcal{J}_{\xi,\rho}^{\mathrm{new}}+\mathcal{J}_{\xi,\rho}^{\mathrm{old}}\ll\xi^{\varepsilon}(N(\mathfrak{q})\|\mathbf{k}\|L)^{\varepsilon}\cdot \Big[N(\mathfrak{q})\|\mathbf{k}\|L+\xi^{2}\|\mathbf{k}\|^{\frac{1}{2}}L^{4}\Big].
\end{equation}

Combining \eqref{eq11.10} and \eqref{eq11.11} we obtain 
\begin{equation}\label{11.12}
\frac{L(1/2,\pi)^2M_{\xi,\rho}(\pi)^2}{L(1,\pi,\Ad)}\ll \xi^{\varepsilon}(N(\mathfrak{q})\|\mathbf{k}\|L)^{\varepsilon}\cdot \Big[N(\mathfrak{q})\|\mathbf{k}\|L^{-1}+\xi^{2}\|\mathbf{k}\|^{\frac{1}{2}}L^2\Big].
\end{equation}

Taking $N(\mathfrak{q})\|\mathbf{k}\|L^{-1}=\xi^{2}\|\mathbf{k}\|^{\frac{1}{2}}L^2$, i.e., $L=\xi^{-2/3}N(\mathfrak{q})^{1/3}\|\mathbf{k}\|^{1/6}$, into \eqref{11.12}, along with $L(1,\pi,\Ad)\ll N(\mathfrak{q})^{\varepsilon}\|\mathbf{k}\|^{\varepsilon}$, we obtain \eqref{11.3}.

\subsection{The Final Steps} In this part, we will remove the harmonic weight and establish the Main Theorems. Using the same mollifier as that in \cite{KowalskiMichel1999}, we can obtain:
\begin{equation}
   \sum_{\pi\in \mathcal{F}(\mathbf{k},\mathfrak{q})}L(1/2,\pi)M_{\xi}(\pi)\sim\frac{\zeta_F(2)^2 D_F^{3/2}}{  \Res_{s=1}\zeta_F(s)}\prod_{v\mid\infty}\frac{(k_v-1)}{4\pi ^2}\cdot\mathcal{M}_{\fq,\mathbf{k}}^{(1)},
\end{equation}
and
\begin{equation}
\sum_{\pi\in \mathcal{F}(\mathbf{k},\mathfrak{q})}L(1/2,\pi)^2M_{\xi}(\pi)^2\sim\frac{\zeta_F(2)^3 D_F^{3/2}}{(\Res_{s=1}\zeta_F(s))^2}\prod_{v\mid\infty}\frac{k_v-1}{4\pi^2}\cdot \mathcal{M}_{\mathfrak{q},\mathbf{k}}^{(2)}.
\end{equation}

Let $A>(\log\xi)^{\frac{3}{2}+\epsilon}.$ By Cauchy's inequality, one has:
\[\sum_{\substack{\pi\in \mathcal{F}(\mathbf{k},\mathfrak{q})\\ L(1/2,\pi)<A^{-1}}}L(1/2,\pi)M_{\xi}(\pi)\ll \frac{(\#\cF(\mathbf{k},\fq))^{1/2}}{A}\cdot\left(\sum_{\substack{\pi\in \mathcal{F}(\mathbf{k},\mathfrak{q})}}\frac{M_{\xi}(\pi)^2L^{(\fq)}(1,\pi,\Ad)}{L^{(\mathfrak{q})}(1,\pi,\Ad)}\right).\]
Similar to the proof of \eqref{10.4}, we have:
\[\sum_{\substack{\pi\in \mathcal{F}(\mathbf{k},\mathfrak{q})}}\frac{M_{\xi}(\pi)^2L^{(\fq)}(1,\pi,\Ad)}{L^{(\mathfrak{q})}(1,\pi,\Ad)}\ll V_{\fq}\|\mathbf{k}\|D_F^{3/2}(\log\xi)^{3}.\]
Here is the idea for the proof: we first truncate $L^{(\mathfrak{q})}(1,\pi,\Ad)$ in the numerator by $\omega_{\pi}(x).$ This works since $M_{\xi}(\pi)^{2}$ satisfies the conditions \ref{conditions to estimate the middle sum}. The we apply Petersson and use the fact that $M_{\xi}(\pi)$ is summed  over squarefree ideals.

The proof of Theorem \ref{main theorem} can be established by following an argument similar to that in Theorem \ref{thm10.2}. By following a similar argument as presented in Section \ref{sec10}, we establish Corollary \ref{uniform nonvanishing corollary} and Corollary \ref{corollary in the level aspect}.

\appendix

\section{The Classical Method vs. The Relative Trace Formula}\label{classical vs RTF}


In this appendix we compare the Relative Trace Formula (\textbf{RTF}) used in this paper with the classical methods that were used before by several authors, and briefly explain why we choose to employ \textbf{RTF} in this problem. The basic strategies of both methods are the same. We always start with an exact formula of the second moment of central $L$-values weighted by the $\fn^{th}$-Hecke eigenvalue (commonly referred as Kuznetsov's formula, see, for example, \cite[Theorem 4.2]{BF21}), and then sum over $\fn$ on a short range to obtain the desired mollified second moment.


We briefly explain how the classical method works in the basic case $F = \bQ$. The desired weighted second moment is
$$\sum_{\pi \in \cF(k,q)} \omega_\pi \lambda_\pi(n)L(1/2,\pi)^2,$$
where $\omega_\pi$ is the harmonic weight, and $\lambda_\pi(n)$ is the $n^{th}$ Hecke eigenvalue of $\pi$. By opening $L(1/2,\pi)^2$ using the approximate functional equation, and applying the classical Petersson trace formula, we get
$$\sum_{\pi}\omega_\pi\lambda_{\pi}(n)L(1/2,\pi)^2\approx \mbox{diagonal term}+\sum_{m\geq1}\frac{V(m)d(m)}{m^{1/2}}\sum_{c>0,q|c}\frac{S(m,n;c)}{c}J_{k-1}\left(\frac{4\pi\sqrt{mn}}{c}\right).$$
Here $V$ is some smooth function, $d(m)$ is the divisor function, $S(m,n;c)$ is the usual Kloosterman sum, and $J_{k-1}$ is the $J$-Bessel function of order $k-1$. 


It can be shown that the diagonal term contributes to the main term. To treat the off-diagonal term, we further open $V(m)$, $S(m,n;c)$ by their definitions, and open $J_{k-1}$ using the Mellin-Barnes integral representation. Then the summation over $m$ yields the following series
\begin{equation}\label{key series in classical case}
    E\left(s,\frac{a}{c}\right) = \sum_{m=1}^\infty \frac{e\left(\frac{am}{c}\right)d(m)}{m^s}, \hspace{3mm} (a,c) = 1.
\end{equation}


It is known that the above series has a meromorphic continuation to the whole complex plane with a double pole at $s=1$ and satisfies a functional equation $s \to 1-s$. This allows us to shift integration contour and ultimately obtain certain hypergeometric function. It is important to note that the residue (when shifting contour) also contributes to the main term. 

Finally, we apply stationary phase analysis to the hypergeometric functions to obtain both the main term and a power-saving error term.


Another way to explain why the off-diagonal term contributes to the main term is as follows. For $(u,v)\in\mathbb{C}^2$ with $\Re(u+v),\Re(u-v)\gg1$ we investigate the weighted second moment 
\[\sum_{\pi\in\mathcal{F}(k,q)}\lambda_{\pi}(n)L(1/2+u+v,\pi)L(1/2+u-v,\pi)\]
and study its analytic continuation to $(u,v)=(0,0).$ To achieve this we apply the Hecke relation for $L(1/2+u+v,\pi)L(1/2+u-v,\pi)$, the Petersson trace formula and the Voronoi summation formula for the divisor function (see \cite[Theorem 4.2]{BF21} for details). The diagonal term contains a zeta factor $\zeta(1+2u)$ with a pole at $(u,v) = (0,0)$, which must be canceled by another term arising from the off-diagonal term. Further, the two constant terms in the Voronoi summation formula may also contribute to the main term. It is also worth noting that the Voronoi summation formula for the divisor function $d(n)$ is equivalent to the functional equation of the series \eqref{key series in classical case}, and both can be regarded as special cases of Poisson summation formulas.


We believe this method works over any totally real number field $F$. However, the calculation would be much more involved. The "Kloosterman sum" that appears in the Petersson trace formula is more complicated than the usual one, and products of $d_F$ Bessel functions would show up, instead of a single Bessel function. These all eventually lead to a series much more complicated than \eqref{key series in classical case}, whose analytic behavior is not as clear. 


One advantage of the \textbf{RTF} is that it gives a closed formula for the desired weighted second moment, without explicitly working out a functional equation of \eqref{key series in classical case}, even over a general totally real field $F$. In fact, the functional equation is hidden in the regularization of \textbf{RTF}. See \cite[Section 3]{Yan23c} for details.

We also note that, when $F=\bQ,$ the relative trace formula yields a closed formula for the weighted second moment that is nearly identical to that in \cite[Theorem 4.2]{BF21}. In particular, the singular orbital integral will correspond to the main term in \cite[Theorem 4.2]{BF21}. Additionally, the product of $\fp$-adic evaluation functions in the regular orbital integral simplifies to the product of divisor functions. (See Remark \ref{divisor function remark} for details.) This indicates that the regular orbital integrals correspond to the error term $E(l;u,v)$ in \cite[Theorem 4.2]{BF21}.


There are also some difficulties when dealing with the mollified first moment using the classical method. For instance, the presence of the global root number $\epsilon_\pi = \mu_F(\fq)i^{\sum_{v\mid\infty}k_v}\lambda_{\pi}(\fq)N(\fq)^{1/2}$ and the fact that contributions from oldforms cannot be treated trivially, since we allow both weight and level to vary simultaneously. These matters would cause a worse error term than the case $F = \bQ$ and $q=1$.  

In contrast, the advantage of the relative trace formula is that it begins in a region where absolute convergence is guaranteed and then extends to the central value through analytic continuation. This allows us to bypass the root number issue, and the discussion regarding oldforms becomes much clearer.

Finally, when applying the approximated functional equation, we encounter integrals of hypergeometric functions, which complicate the classical case. In the context of number fields, this becomes even more intricate due to the presence of more Archimedean places and units. The relative trace formula simplifies the analysis, as it is established through analytic continuation. The continuation transforms those integrals into other types of hypergeometric functions, which are significantly easier to study. The decay of the hypergeometric functions, combined with Dirichlet's unit theorem, effectively address the issue of units. (See Lemma \ref{lem6.11} for details.) This is another advantage of the relative trace formula in the context of number fields.

\bibliographystyle{alpha}	
\bibliography{WYZ}

\newcommand{\etalchar}[1]{$^{#1}$}
\begin{thebibliography}{BFK{\etalchar{+}}23}

\bibitem[BF21]{BF21}
Olga Balkanova and Dmitry Frolenkov.
\newblock Moments of {$L$}-functions and the {L}iouville-{G}reen method.
\newblock {\em J. Eur. Math. Soc. (JEMS)}, 23(4):1333--1380, 2021.

\bibitem[BFK{\etalchar{+}}23]{BlomerFouvryKowalskiMichelMilicevicSawin2023}
Valentin Blomer, \'{E}tienne Fouvry, Emmanuel Kowalski, Philippe Michel,
  Djordje Mili\'{c}evi\'{c}, and Will Sawin.
\newblock The second moment theory of families of {$L$}-functions---the case of
  twisted {H}ecke {$L$}-functions.
\newblock {\em Mem. Amer. Math. Soc.}, 282(1394):v+148, 2023.

\bibitem[BM15]{BM15}
Valentin Blomer and Djordje Mili\'cevi\'c.
\newblock The second moment of twisted modular {$L$}-functions.
\newblock {\em Geom. Funct. Anal.}, 25(2):453--516, 2015.

\bibitem[Bum97]{Bum97}
Daniel Bump.
\newblock {\em Automorphic forms and representations}, volume~55 of {\em
  Cambridge Studies in Advanced Mathematics}.
\newblock Cambridge University Press, Cambridge, 1997.

\bibitem[DK00]{DukeKowalski2000}
W.~Duke and E.~Kowalski.
\newblock A problem of {L}innik for elliptic curves and mean-value estimates
  for automorphic representations.
\newblock {\em Invent. Math.}, 139(1):1--39, 2000.
\newblock With an appendix by Dinakar Ramakrishnan.

\bibitem[FH95]{FriedbergHoffstein1995}
Solomon Friedberg and Jeffrey Hoffstein.
\newblock Nonvanishing theorems for automorphic {$L$}-functions on {${\rm
  GL}(2)$}.
\newblock {\em Ann. of Math. (2)}, 142(2):385--423, 1995.

\bibitem[ILS00]{IwaniecLuoSarnak2000}
Henryk Iwaniec, Wenzhi Luo, and Peter Sarnak.
\newblock Low lying zeros of families of {$L$}-functions.
\newblock {\em Inst. Hautes \'{E}tudes Sci. Publ. Math.}, (91):55--131, 2000.

\bibitem[IS99]{IS97}
H.~Iwaniec and P.~Sarnak.
\newblock Dirichlet {$L$}-functions at the central point.
\newblock In {\em Number theory in progress, {V}ol. 2
  ({Z}akopane-{K}o\'{s}cielisko, 1997)}, pages 941--952. de Gruyter, Berlin,
  1999.

\bibitem[IS00]{IwaniecSarnak2000}
Henryk Iwaniec and Peter Sarnak.
\newblock The non-vanishing of central values of automorphic {$L$}-functions
  and {L}andau-{S}iegel zeros.
\newblock {\em Israel J. Math.}, 120:155--177, 2000.

\bibitem[JK15]{JK15}
Julia Jackson and Andrew Knightly.
\newblock Averages of twisted {$L$}-functions.
\newblock {\em J. Aust. Math. Soc.}, 99(2):207--236, 2015.

\bibitem[KL08]{KL08}
Andrew Knightly and Charles Li.
\newblock Petersson's trace formula and the {H}ecke eigenvalues of {H}ilbert
  modular forms.
\newblock In {\em Modular forms on {S}chiermonnikoog}, pages 145--187.
  Cambridge Univ. Press, Cambridge, 2008.

\bibitem[KL10]{KL10}
Andrew Knightly and Charles Li.
\newblock Weighted averages of modular {$L$}-values.
\newblock {\em Trans. Amer. Math. Soc.}, 362(3):1423--1443, 2010.

\bibitem[KM99]{KowalskiMichel1999}
E.~Kowalski and P.~Michel.
\newblock The analytic rank of {$J_0(q)$} and zeros of automorphic
  {$L$}-functions.
\newblock {\em Duke Math. J.}, 100(3):503--542, 1999.

\bibitem[Luo99]{Luo1999}
Wenzhi Luo.
\newblock Values of symmetric square {$L$}-functions at {$1$}.
\newblock {\em J. Reine Angew. Math.}, 506:215--235, 1999.

\bibitem[Luo03]{Luo2003}
Wenzhi Luo.
\newblock Poincar\'{e} series and {H}ilbert modular forms.
\newblock volume~7, pages 129--140. 2003.
\newblock Rankin memorial issues.

\bibitem[Luo15]{Luo2015}
Wenzhi Luo.
\newblock Nonvanishing of the central {$L$}-values with large weight.
\newblock {\em Adv. Math.}, 285:220--234, 2015.

\bibitem[MOS66]{MOS66}
Wilhelm Magnus, Fritz Oberhettinger, and Raj~Pal Soni.
\newblock {\em Formulas and theorems for the special functions of mathematical
  physics}, volume Band 52 of {\em Die Grundlehren der mathematischen
  Wissenschaften}.
\newblock Springer-Verlag New York, Inc., New York, enlarged edition, 1966.

\bibitem[MV02]{MichelVanderKam2002}
P.~Michel and J.~Vanderkam.
\newblock Simultaneous nonvanishing of twists of automorphic {$L$}-functions.
\newblock {\em Compositio Math.}, 134(2):135--191, 2002.

\bibitem[MV10]{MV10}
Philippe Michel and Akshay Venkatesh.
\newblock The subconvexity problem for {${\rm GL}_2$}.
\newblock {\em Publ. Math. Inst. Hautes \'{E}tudes Sci.}, (111):171--271, 2010.

\bibitem[OS98]{OnoSkinner1998}
Ken Ono and Christopher Skinner.
\newblock Non-vanishing of quadratic twists of modular {$L$}-functions.
\newblock {\em Invent. Math.}, 134(3):651--660, 1998.

\bibitem[Roh89]{Rohrlich1989}
David~E. Rohrlich.
\newblock Nonvanishing of {$L$}-functions for {${\rm GL}(2)$}.
\newblock {\em Invent. Math.}, 97(2):381--403, 1989.

\bibitem[RR05]{RR05}
Dinakar Ramakrishnan and Jonathan Rogawski.
\newblock Average values of modular {$L$}-series via the relative trace
  formula.
\newblock {\em Pure Appl. Math. Q.}, 1(4):701--735, 2005.

\bibitem[Shi77]{Shimura1977}
Goro Shimura.
\newblock On the periods of modular forms.
\newblock {\em Math. Ann.}, 229(3):211--221, 1977.

\bibitem[Tro11]{Trotabas2011}
Denis Trotabas.
\newblock Non annulation des fonctions {$L$} des formes modulaires de {H}ilbert
  au point central.
\newblock {\em Ann. Inst. Fourier (Grenoble)}, 61(1):187--259, 2011.

\bibitem[TZ21]{ThornerZaman2021}
Jesse Thorner and Asif Zaman.
\newblock An unconditional {${\rm GL}_n$} large sieve.
\newblock {\em Adv. Math.}, 378:Paper No. 107529, 24, 2021.

\bibitem[Van99]{VanderKam1999}
Jeffrey~M. VanderKam.
\newblock The rank of quotients of {$J_0(N)$}.
\newblock {\em Duke Math. J.}, 97(3):545--577, 1999.

\bibitem[Yan23a]{Yan23c}
Liyang Yang.
\newblock Average of central {$L$}-values for
  {$\mathrm{GL}(2)\times\mathrm{GL}(1),$} hybrid subconvexity, and simultaneous
  nonvanishing.
\newblock {\em preprint}, 2023.

\bibitem[Yan23b]{Yan23a}
Liyang Yang.
\newblock Relative trace formula and {$L$}-functions for
  {$\mathrm{GL}(n+1)\times \mathrm{GL}(n)$}.
\newblock {\em arXiv preprint arXiv:2303.02225}, 2023.

\end{thebibliography}

\end{document}